\documentclass[11pt,a4paper]{amsart}
\usepackage[margin=1in]{geometry}
\usepackage{graphicx}
\usepackage{color}
\usepackage[english]{babel}
\usepackage[T1]{fontenc}
\usepackage{amsmath}
\usepackage{amscd}
\usepackage{amstext}
\usepackage{amsbsy}
\usepackage{amsopn}
\usepackage{amsthm}
\usepackage{amsxtra}
\usepackage{upref}
\usepackage{amsfonts}
\usepackage{amssymb}
\usepackage{mathrsfs}
\usepackage{euscript}
\usepackage{array}
\usepackage{verbatim}
\newtheorem{theorem}{Theorem}[section]
\newtheorem{lemma}[theorem]{Lemma}
\newtheorem{corollary}[theorem]{Corollary}
\newtheorem{proposition}[theorem]{Proposition}
\theoremstyle{definition}
\newtheorem{definition}[theorem]{Definition}
\newtheorem{definition-proposition}[theorem]{Definition-Proposition}
\newtheorem{example}[theorem]{Example}

\theoremstyle{hypothesis}

   \newtheorem{letterthm}{Theorem}

\theoremstyle{remark}

\newcounter{Step}
\newenvironment{step}[0]{\bigskip\addtocounter{Step}{1}\noindent\textbf{Step \theStep~:} }{\begin{flushright}\tiny \end{flushright}}

\numberwithin{equation}{section}

\DeclareFontFamily{U}{rcjhbltx}{}
\DeclareFontShape{U}{rcjhbltx}{m}{n}{<->rcjhbltx}{}
\DeclareSymbolFont{hebrewletters}{U}{rcjhbltx}{m}{n}
\DeclareMathSymbol{\lamed}{\mathord}{hebrewletters}{108}
\DeclareMathSymbol{\mem}{\mathord}{hebrewletters}{109}
\DeclareMathSymbol{\ayin}{\mathord}{hebrewletters}{96}
\DeclareMathSymbol{\tsadi}{\mathord}{hebrewletters}{118}
\DeclareMathSymbol{\qof}{\mathord}{hebrewletters}{114}
\DeclareMathSymbol{\shin}{\mathord}{hebrewletters}{152}

\def\o{\otimes }
\def\R{\mbox{I\hspace{-.15em}R} }
\def\Q{\mbox{l\hspace{-.47em}Q} }
\def\C{\hspace{.17em}\mbox{l\hspace{-.47em}C} }

\def\N{\mbox{I\hspace{-.15em}N} }

\newcommand{\gre}{}
\begin{document}

\title[Laplace Principle for Convex Functionals Of Hermitian Brownian Motion]{A Laplace Principle for Hermitian Brownian Motion and Free Entropy I: the convex functional case.}

%    Information for first author

\author{Yoann Dabrowski}
%    Address of record for the research reported here
\address{\gre Universit\'{e} de Lyon\\ 
Universit\'{e} Lyon 1\\
Institut Camille Jordan UMR 5208\\
43 blvd. du 11 novembre 1918\\
F-69622 Villeurbanne cedex\\
France
}
\email{\gre dabrowski@math.univ-lyon1.fr}
%    General info
\date{}

\begin{abstract}
This paper is part of a series aiming at proving that the $\limsup$ and $\liminf$ variants of Voiculescu's free entropy coincide. This is based on a Laplace principle (implying a large deviation principle) for hermitian brownian motion on $[0,1]$. In the current paper, we show that microstates free entropy $\chi(X_1,...,X_m)$ and non-microstate free entropy  $\chi^*(X_1,...,X_m)$ coincide for self-adjoint variables $(X_1,...,X_m)$ satisfying a Schwinger-Dyson equation for subquadratic, bounded below, strictly convex potentials with Lipschitz derivative sufficiently approximable by non-commutative polynomials. 
Our results are based on Dupuis-Ellis weak convergence approach to large deviations 
where one shows a Laplace principle in obtaining a stochastic control formulation for exponential functionals. In the non-commutative context, ultrapoduct analysis replaces weak-convergence of the stochastic control problems.  
\end{abstract}
%Suggestion for arxiv comment:
%Major revision: First part of a revised corrected version now split in two papers. It contains the most probabilistc part of the two. New presentation even for the material of the previous version, including better emphasis on technical tools used : control, FBSDE with more references to the classical litterature.
\maketitle
\section{Introduction}
In a fundamental series of papers \cite{V2,V3,Voi4,V5}, Voiculescu introduced analogues of entropy and Fisher information in the context of free probability theory. A first microstates free entropy $\chi(X_{1},...,X_{m})$ is defined as a normalized limit of the volume of sets of microstates i.e. matricial approximants (in moments) of the n-tuple of self-adjoints $X_{i}$ living in a (tracial) $W^{*}$-probability space $M$. Starting from a definition of a free Fisher information \cite{V5}, Voiculescu also defined a non-microstate free entropy $\chi^{*}(X_{1},...,X_{m})$, known by the fundamental work \cite{BCG} not to be smaller than the previous microstates entropy, and believed to be equal (at least modulo Connes' embedding conjecture). For more details, we refer the reader to the survey \cite{VoS} for a list of properties as well as applications of free entropies in the theory of von Neumann algebras. The technical definitions are recalled later in subsection \ref{FreeEntropyDef}.

As pointed out in the review article \cite{VoS}, the study of free entropy has been faced with several technical questions, among which the two most famous are the equality of microstates and non-microstate definitions (the so-called unification problem \cite[p22]{VoS}) and the equality of two variants of free entropy with a $\limsup$, we will call $\chi$, or a $\liminf$ over the size of matrix approximations $N$ \cite[Rmk (a) p9]{VoS}.  We will call $\underline{\chi}$ the $\liminf$ variant. %In absence of equality, several natural questions, most notably additivity of free entropy for free variables and the free entropy power inequality \cite{EntPow}, were solved using yet another ultrafilter variant $\chi^\omega$ in between the previous two definitions. 
Our goal in the current series of papers is to solve completely this second question about limits and partially the unification problem.

Our approach  is based on large deviations for hermitian brownian motion as in \cite{GuionnetCD1,GuionnetCD2,BCG,GZ2}. In the several variable case, the best result in \cite{BCG} gave a large deviation upper bound and a large deviation lower bound with two (a priori) different rate functions. Our main improvement in the first paper of the series is to find two (a priori) different rate functions for the lower and upper bound that can be proved equal for certain natural states. The second paper of this series will then obtain more general equalities of our lower bound and upper bound using different techniques. 
For the terminology and methodology of large deviation Theory we refer to \cite{DupuisEllis}. Since we follow their weak-convergence approach (that should maybe be called stochastic control approach), we also follow their terminology ``Laplace principle". We will refer to \cite{DZ} for more technical results. We refer to Theorem \ref{LaplaceNonConvex} for the Large deviation principles obtained in this paper.

 The key part in its proof is to work harder on the Large Deviation upper bound in order to get the rate function as a minimization problem on a smaller set in order to get more easily a lower bound with a similar rate function that can be proved equal, here in some cases, and in the sequel in all cases. It does not use any approximation of variables as in the final result in the one variable case \cite{GZ2}. Apart from the exponential tightness and the identification of the brownian bridge as reaching an infimum, it does not use much of the argument in \cite{BCG}, but the setting is similar. We obtain a variational formula for the rate function in the spirit of a free analogue of pressure defined recently by Hiai \cite{Hi}, but in a non-convex setting with the use of Bryc's inverse Varadhan lemma of large deviation theory (and this is necessary since Voiculescu's entropy is not concave and we don't use its concavification \cite{BianeD}). The formula is based on a free analogue of Bou\'e-Dupuis formula \cite{BD} for free pressure and is based on a recent improvement by \"Ust\"unel \cite{Ustunel} of the original formula, better suited for convex analysis, and applied to hermitian brownian motion (see also \cite{Lehec} for a nice introduction and other applications of this formula, and \cite{Hartman} for an extension to diffusions). This is our starting point for our use of the weak-convergence approach of Dupuis and Ellis (that motivated the discovery of their formula in \cite{BD}) that shows a Laplace principle (equivalent to large deviation principle, LDP for short, under appropriate assumptions) based on limits of a stochastic control problem. In our version, the weak-convergence part is replaced by ultraproduct analysis (see Subsection 2.9 for a short introduction and references).

Technically, our result is strongly based on convex analysis, this is why we can still get a partial answer about the unification problem under a convexity assumption. Note that beyond the general inequality in \cite{BCG}, and the free product of the single variable case, not much was known about the equality $\chi=\chi^*$. Even when an explicit computation of microstates free entropy was known as for small perturbation of semicircular variables \cite{GuionnetMaurelSegala}, or when implicit formulas are known thanks to the recent developments of free transport \cite{GS12}, no equality cases was known even for perturbations of semi-circular variables, to the best of our knowledge. This is due to the technicality of the definition of non-microstate free entropy in \cite{V5} for which no non-linear change of variable is known. Much more cases of equality were known for the related but much cruder microstates free entropy dimension \cite{ShMineyev,Dab10}. In the case of finite free entropy, it identically equals $n$, the number of variables.
 
We are in a setting close to \cite{GuionnetS07} (with a much less constraining notion of convexity defined in subsection \ref{NCFun}). The exact formulation of the function space for the potential $\mathcal{E}^{1,1}_{app}(  \mathcal{T}_{2}(\mathcal{F}^m_{1}),d_{2} )$ will be given in this subsection. A basic example of function, linear in the trace, is for instance for $\lambda$ small enough to get convexity:
$$g(\tau)=D+C \sum_{j=1}^k \sum_{l=1}^m\tau ((4i\frac{u_j^l+1}{u_j^l-1})^*(4i\frac{u_j^l+1}{u_j^l-1}))+\sum_\epsilon\Re(\lambda_\epsilon\tau((u_{j_1}^{l_1})^{\epsilon_1}...(u_{j_{m}}^{l_{m}})^{\epsilon_{m}}))$$
for the trace expressed in unitary variables $u_j^l=u(X_{t_j}^i)=\frac{X_{t_j}^i+4i}{X_{t_j}^i-4i}$ $t_j\in [0,1]$, $\epsilon_i\in\{+1,-1\}$

In words, this is a space of convex bounded below subquadratic functions with lipschitz derivative and a natural approximation of the derivative by non-commutative polynomials. 

Finally, we also deduce various other conjectures for free entropy for semi-circular perturbation in this case, most notably the technical equality of entropy to entropy in presence $\chi(X_1+\sqrt{t} S_1,...,X_m+\sqrt{t} S_m:\sqrt{t} S_1,...,\sqrt{t} S_m)$, appearing in the definition of microstates free entropy dimension, and the continuity of free Fisher information $\Phi^*$ under semicircular perturbation (see the question in \cite[p23]{VoS}). Our main result is thus the following:

\begin{letterthm}\label{ThmC}
Let $m\geq 2$ and $V\in \mathcal{E}^{1,1}_{app}(  \mathcal{T}_{2}(\mathcal{F}^m_{1}),d_{2} ).$ 

Let $X_1,...,X_m$ having law $\tau_V$, the unique solution of $(SD_V)$ obtained in Theorem \ref{SDVg} (or equivalently a conjugate variable in the sense of \cite{V5}  given by usual cyclic derivatives $X_i+\mathcal{D}_iV$). Then, we have the equality: $$\chi(X_1,...,X_m)= \underline{\chi}(X_1,...,X_m)= \chi^*(X_1,...,X_m).$$
If moreover $S_1,...,S_m$ is a free semicircular system free from $\{X_1,...,X_m\}$, then for any $t> 0$: $$\chi(X_1+\sqrt{t} S_1,...,X_m+\sqrt{t} S_m:\sqrt{t} S_1,...,\sqrt{t} S_m)=\chi^*(X_1+\sqrt{t} S_1,...,X_m+\sqrt{t} S_m),$$ the free Fisher information $t\mapsto \Phi^*(X_1+\sqrt{t} S_1,...,X_m+\sqrt{t} S_m)$ is H\"older continuous (of exponent $1/2$) on $[0,\infty)$ and for $(\xi_{1,t},...,\xi_{m,t})$ the conjugate variables in $W^*(X_1+\sqrt{t} S_1,...,X_m+\sqrt{t} S_m)$ and $\delta_t$ the corresponding free difference quotient, we have the integral formula for $0\leq s<t$:
$$\Phi^*(X_1+\sqrt{t} S_1,...,X_m+\sqrt{t} S_m)=\Phi^*(X_1+\sqrt{s} S_1,...,X_m+\sqrt{s} S_m)-\int_s^t\sum_{i=1}^m||\delta_u(\xi_{i,u})||_2^2du.$$
\end{letterthm}
The last equation is an improvement of an inequality in \cite{Dab14}. The equality version enables to compute the derivative of free Fisher information along free brownian motion in the convex potential case.

Let us finally give an overview of the paper. Section 2 deals with the various preliminaries and notation. The notation strongly needed in the paper about tracial state spaces is in subsection 2.1, the classes of non-commutative convex functions are defined in 2.2 and preliminaries on hermitian brownian motion in 2.4. Subsection 2.5 explains the two kinds of concentration of measure results we will use, the well-known one coming from convexity via logarithmic Sobolev inequality useful for the large deviation upper bound, and the less well-known one coming from a Poulsen Simplex property from \cite{BianeD} useful for the large deviation lower bound. 

 The other preliminaries are less critical, subsection 2.3 deals with Malliavin calculus needed to understand the statements of \cite{Ustunel}, the classical expression of regularity of derivatives of convex functions via second order difference quotients (better suited in stochastic control estimates) is explained in subsection 2.6, background on classical and free entropy are gathered in 2.7 and 2.8. References on ultraproduct analysis are given in subsection 2.9 with preliminaries on processes in ultraproduct filtrations.

Section 3 recalls the version of the Bou\'e-Dupuis formula due to \"Ust\"unel that we will need. As usual in stochastic control, the value function at time $t$ will be used to characterize the minimiser to the optimization problem and it is expressed with an optimization problem given in our case by an application of the Bou\'e-Dupuis-\"Ust\"unel formula. We use standard ideas from optimal control to give regularity properties in time and space. If the terminal cost function $g$ we start with is convex, the value function stay convex and if this cost function $g$ has lipschitz first derivative, one gets upper bounds on second order difference quotients of the value function $h_t$. They enable to get some H\"older continuity in time for the derivative that will give us the crucial uniform regularity to use ultraproduct techniques. This section emphasizes uniformity in the matrix size $N$ of constants.

Section 4 recalls solutions of Stochastic differential equations (SDE) under monotone drift assumptions we will use. It also gives the corresponding version in free probability which is new and of independent interest when using convex analysis jointly with free SDE techniques. We hope it may have future applications to free transport. The uniqueness and existence of strong solutions in the free case will be crucial to obtain a value function independent of the ultrafilter in our ultraproduct analysis and thus a real limit in our Laplace deviation with a rate function not dependent on any ultrafilter. This lack of uniqueness of potential solutions of a free SDE was the issue in \cite{BCG} since the (a priori) different values of the lower bound rate function was mainly associated with a supplementary uniqueness assumption of such a solution not known for the upper bound. The key point of our use of convex analysis is to get this uniqueness.

Section 5 uses the solution in \cite{Ustunel}  of the minimization problem appearing in Bou\'e-Dupuis formula. This is the second place where convexity is used crucially. We apply it in the case of hermitian brownian motion and check again that our estimates are sufficiently uniform in the matrix size $N$. The minimizer is expressed in terms of a solution of the SDE with the drift coming from the gradient of the value function. We estimated it earlier as Lipschitz in space and H\"older continuous in time. An alternative formula is given in subsection 5.1 in terms of a solution of an SDE. This will be this formula that will insure that we control the von Neumann algebra in which lives the drift of the SDE. One of the reasons of the failure of obtaining equality of $\chi=\chi^*$ previously was that some limit of matricial analogues of conjugate variables may still depend on entries of matrices and may not be a non-commutative functional calculus of the solution. In our framework, they could be in an ultraproduct of random matrix spaces rather than in a smaller von Neumann algebra generated by the solution. The formula in this section is the key to obtain the smaller von Neumann algebra via an existence and uniqueness of some free (forward-backward) SDE which is explained in section 6.

Section 7 then explains the Laplace principle for convex functions of the type of those appearing in Theorem \ref{ThmC}. Section 8 gives a formula for ultrafilter limits in the non-convex regular case and deduces our general Laplace lower and upper bounds, which is key to this series of papers. Note that following \cite{BD}, the lower bound in our Laplace principle corresponds to the usual large deviation upper bound. Section 9 gives our applications to free entropy including the proof of Theorem \ref{ThmC}.

 \bigskip\textbf{Acknowledgments :} The author wants to thank Ivan Gentil for showing him \cite{Kolesnikov} and discussing \cite{Gentil}
 , which motivated the use of convex analysis to the present large deviation problem. The author is grateful to the organizers of the ``Conference on von Neumann algebras and related topics" in January 2012 at RIMS, Kyoto University, Japan, where started this research on the applications of ultraproducts to free entropy theory.  The author also wants to thank Yoshimichi Ueda for interesting conversations on orbital free entropy, which started at that time too, and for communicating \cite{UedaRmq} pointing out on this basis an issue in the statements on orbital free entropy in the first preprint version  of this series of papers that led the author to discover a significant gap which enabled the current corrected version. He also thanks Alice Guionnet for numerous discussions and Jonathan Husson for pointing out later in a different way  the same gap and various more minor issues in the first preprint.

     \section{Preliminaries and Notation}

\subsection{Tracial states with second moments}
 \label{prelimStates}    
    We fix a setting similar to \cite{BCG}, except that we exploit convexity more crucially. We call $\mathcal{F}^m_{[0,1]}$ the group universal $C^*$-algebra in the free group with generators $\{u_t^i, i=1,...,m, t\in [0,1]\cap\Q\}.$ (note they considered only the algebra with the same notation, and we will need the tracial state space of a $C^*$-algebra and we prefer countably many generators, which won't matter because of the continuity restriction on states.) $\mathcal{F}^m_{[0,1],\R}$ will be the variant with generators $\{u_t^i, i=1,...,m, t\in [0,1]\}.$
    
%    Note that \textbf{we often consider $\mathbf{m\geq 2}$.}
    
    We call $\mathcal{T}(\mathcal{F}^m_{[0,1]})
   $ the set of tracial states on this $C^*$ algebra.
Consider also   $\mathcal{F}^m_{n}$ the group universal $C^*$-algebra in the free group with generators $\{u_j^i, i=1,...,m, j=2,...,n+1\}.$
We start indexing at $2$ to differentiate from time indices in $[0,1]$.  %When $n=1$ we only write $u^i$ the variables.

We will also consider the universal $C^*$-algebra free product for instance $\mathcal{F}^m_{[0,1]}*\mathcal{F}^{\mu}_{n}$ for $\mu\in \N^*$, of course we have for instance $\mathcal{F}^{m}_{1}*\mathcal{F}^{\mu}_{1}\simeq \mathcal{F}^{\mu+m}_{1}$. In case of $\mathcal{F}^{m}_{k}*\mathcal{F}^{\mu}_{\nu}\subset \mathcal{F}^{m+\mu}_{k+\nu}$ we consider the $C^*$ algebra variable with generators $\{u_j^i:\ i=1,...,m, j=2,...,k+1, \}\cup \{u_j^i:\ i=m+1,...,m+\mu, j=k+2,...,k+\nu+1 \}.$

 For any $\{t_1,...,t_n\}$, we have a canonical $*$-homomorphism $ I_{t_1,...,t_n}:\mathcal{F}^m_{n}\to \mathcal{F}^m_{[0,1]}$ with the following action on generators $$I_{t_1,...,t_n}(u^i_j)=u_{t_{j-1}}^i.$$
We denote  $\mathcal{T}^c(\mathcal{F}^m_{[0,1]})
   $  the set of uniformly continuous  tracial states namely those states $\tau$ such that for any $n$ the map $(t_1,...,t_n)\in ([0,1]\cap \Q)^n\to  \tau\circ I_{t_1,...,t_n}(f)$ is uniformly continuous for any $f\in \mathcal{F}^m_{n}$. Similarly we consider $\mathcal{T}^c(\mathcal{F}^m_{[0,1]}*\mathcal{F}^{\mu}_{n})
   $  with continuity on the uniformly continuous time variables and using $I_{t_1,...,t_k}*Id:\mathcal{F}^{m}_{k}*\mathcal{F}^{\mu}_{n}\to \mathcal{F}^m_{[0,1]}*\mathcal{F}^{\mu}_{n}$ instead of $I_{t_1,...,t_n}$ in the continuity requirement. If we adopt the same for the real variant, we have $\mathcal{T}^c(\mathcal{F}^m_{[0,1],\R}*\mathcal{F}^{\mu}_{n})\simeq\mathcal{T}^c(\mathcal{F}^m_{[0,1]}*\mathcal{F}^{\mu}_{n})$
   since uniform continuity enables to extend the moments to  the universal $C^*$-algebras in finitely many generators indexed by $[0,1]$ and the general one is obtained by inductive limit. We will use this extension without further notice.
   
Let $(M,\tau)$ be a (tracial) $W^*$-probability space, namely a von Neumann algebra $ M$ with a faithful normal trace $\tau$.  We will assume implicitly all our $W^*$-probability spaces to be tracial. We consider self-adjoint processes $X=(X_t^i,i=1,...,m,t\in [0,1])$, with  $X_t^i=(X_t^i)^*\in L^{2}(M,\tau)$, $i=1,...,m$, $t\in [0,1].$
   We write $\tau_X\in \mathcal{T}(\mathcal{F}^m_{[0,1]})$ the law of $$(u(X_t^l):=\frac{X_t^l+4i}{X_t^l-4i}, l=1,...,m,t\in [0,1]).$$

It is not hard to see that if $t\mapsto X_t^i$ are continuous in $ L^2_{sa}(M,\tau)$, then $\tau_X\in \mathcal{T}^c(\mathcal{F}^m_{[0,1]}).$  We even have $\tau_X\in \mathcal{T}^c_2(\mathcal{F}^m_{[0,1]})$ the space where for any $l$, 
$$t\in [0,1]\mapsto\tau ((4i\frac{u_t^l+1}{u_t^l-1})^*(4i\frac{u_t^l+1}{u_t^l-1}))<\infty$$ %and
%$$(t,s)\mapsto \tau ((4i\frac{u_t^l+1}{u_t^l-1}-4i\frac{u_s^l+1}{u_s^l-1})^*(4i\frac{u_t^l+1}{u_t^l-1}-4i\frac{u_s^l+1}{u_s^l-1}))$$ 
is continuous (since $\tau_X ((4i\frac{u_t^l+1}{u_t^l-1})^*(4i\frac{u_t^l+1}{u_t^l-1}))=\tau((X_t^l)^2)$. We use a similar notation $\tau_X$ for state in $\mathcal{T}(\mathcal{F}^m_{n})$. For $X$ self-adjoint variables.
Similarly, we call $\mathcal{T}_{2}(\mathcal{F}^m_{n})$ the set of tracial states with $\tau ((4i\frac{u_j^l+1}{u_j^l-1})^*(4i\frac{u_j^l+1}{u_j^l-1}))<\infty$, $j=1,...,k.$ 
% We finally define $\mathcal{T}^c_2(\mathcal{F}^m_{[0,1]})\subset \mathcal{T}^c_{2;\mathbf{t}}(\mathcal{F}^m_{[0,1]})\subset\mathcal{T}^c(\mathcal{F}^m_{[0,1]})$ the subset satisfying the same bound with $u_j^l$ replaced by  $u_{t_j}^l.$

%Similarly if $t\mapsto X_t^i\in L^2([0,t],L^2(M,\tau))$ (measurability in Bochner's sense equivalent to measurability in Lusin's sense [REF paper Naralenkov ?]), then $\tau_X\in \mathcal{T}^b(\mathcal{F}^m_{[0,1]}).$ [DETAIL, no Need $M$ separable if Lusin's measurability assumed].
% We even have $\tau_X\in \mathcal{T}^b_2(\mathcal{F}^m_{[0,1]})$ the space where for any $l$, $$\int_0^1dt\ \tau ((4i\frac{u_t^l+1}{u_t^l-1})^*(4i\frac{u_t^l+1}{u_t^l-1}))<\infty.$$

We first consider on $\mathcal{T}^c(\mathcal{F}^m_{[0,1]})$ the same topology as in \cite{BCG}, namely the one given by the distance:
$$d(\tau_1,\tau_2)=\sum_{k=1}^\infty 2^{-k}\sup_{(i_1,...,i_m) \in [\![1,m]\!]^k}\sup_{(\epsilon_1,...,\epsilon_m) \in \{-1,1\}^k}\sup_{(t_1,...,t_k)\in[0,1]^k}|(\tau_1-\tau_2)((u_{t_1}^{i_1})^{\epsilon_1}...(u_{t_k}^{i_k})^{\epsilon_k})|.$$
   
We also consider on    $\mathcal{T}_{2}^c(\mathcal{F}^m_{[0,1]})%,%\mathcal{T}^c_{2;\mathbf{t}}(\mathcal{F}^m_{[0,1]})
$ the topologies given by the distances: %(for $\mathbf{t}=t_0<t_1<...<t_k\leq 1$):
   \begin{align*}&d_1(\tau_1,\tau_2)=d(\tau_1,\tau_2)
   \\&+\sum_{k=1}^\infty 2^{-k} \sup_{(l,i_1,...,i_m) \in [\![1,m]\!]^{k+1}}\sup_{(\epsilon_1,...,\epsilon_m) \in \{-1,1\}^k}\sup_{(t,t_1,...,t_k)\in[0,1]^{k+1}}|(\tau_1-\tau_2)((\frac{u_t^l+1}{u_t^l-1})^*(u_{t_1}^{i_1})^{\epsilon_1}...(u_{t_k}^{i_k})^{\epsilon_k})|,
 \\&d_{2}(\tau_1,\tau_2)=d_1(\tau_1,\tau_2) +\sup_{l,\lambda=1,...,m}\sup_{(s,t)\in [0,1]^2} \left|(\tau_1-\tau_2) \left((\frac{u_t^l+1}{u_t^l-1})^*(\frac{u_s^\lambda+1}{u_s^\lambda-1})\right)\right|%,\\&d_{2;\mathbf{t}}(\tau_1,\tau_2)=d_1(\tau_1,\tau_2) +\sup_{l,\lambda=1,...,m}\sup_{(s,t)\in \{t_0,...,t_k\}^2} \left|(\tau_1-\tau_2) \left((\frac{u_t^l+1}{u_t^l-1})^*(\frac{u_s^\lambda+1}{u_s^\lambda-1})\right)\right|[UTILE?] 
 .\end{align*}

We put similar distances $d_1,d_{2}$ on $\mathcal{T}_{2}(\mathcal{F}^m_{n})$.

%\begin{align*}&d_1(\tau_1,\tau_2)\\&=\sum_{k=1}^\infty 2^{-k}\sup_{(i_1,...,i_m) \in [\![1,m]\!]^k}\sup_{(\epsilon_1,...,\epsilon_m) \in \{-1,1\}^k}\sup_{(j_1,...,j_k)\in\{2,....,n+1\}^k}|(\tau_1-\tau_2)((u_{j_1}^{i_1})^{\epsilon_1}...(u_{j_k}^{i_k})^{\epsilon_k})|\\&+\sum_{k=1}^\infty 2^{-k}\sup_{(l,i_1,...,i_m) \in [\![1,m]\!]^{k+1}}\sup_{(\epsilon_1,...,\epsilon_m) \in \{-1,1\}^k}\sup_{(j_1,...,j_k,j)\in\{2,....,n+1\}^{k+1}}|(\tau_1-\tau_2)((\frac{u_j^l+1}{u_j^l-1})^*(u_{j_1}^{i_1})^{\epsilon_1}...(u_{j_k}^{i_k})^{\epsilon_k})|,\\&d_{2+\alpha}(\tau_1,\tau_2)=d_1(\tau_1,\tau_2)+\sup_{l,\lambda=1,...,m}\sup_{j,\iota=1,...,n,\beta\in \{0,\alpha\}} \left|(\tau_1-\tau_2) \left(\left((4i\frac{u_j^l+1}{u_j^l-1})^*(4i\frac{u_j^l+1}{u_j^l-1})\right)^{\beta/2}(\frac{u_j^l+1}{u_j^l-1})^*(\frac{u_\iota^\lambda+1}{u_\iota^\lambda-1})\right)\right|.
%\end{align*}

Finally we have a variant space $\mathcal{T}^c_{2,0}(\mathcal{F}^m_{[0,1]}*\mathcal{F}^{m}_{\mu}),\mathcal{T}^c_{2,0,}(\mathcal{F}^m_{[0,1]}*\mathcal{F}^{m}_{\mu})$, (resp. $\mathcal{T}^c_{2,0}(\mathcal{F}^m_{n}*\mathcal{F}^{\nu}_{\mu}) $) where only the variables in $\mathcal{F}^m_{[0,1]}$ are associated with $X^i_t$ with finite second moment in the GNS representation (resp. only variables in the first copy $\mathcal{F}^m_{n}$). We put e.g. on the first space the distances:
\begin{align*}&%d_{2,0;\mathbf{t}}(\tau_1,\tau_2)=d_{1,0}(\tau_1,\tau_2)+\sup_{l,\lambda=1,...,m}\sup_{(s,t)\in  \{t_0,...,t_k\}^2} \left|(\tau_1-\tau_2) \left((\frac{u_t^l+1}{u_t^l-1})^*(\frac{u_s^\lambda+1}{u_s^\lambda-1})\right)\right|,
%\\&
d_{2,0}(\tau_1,\tau_2)=d_{1,0}(\tau_1,\tau_2)+\sup_{l,\lambda=1,...,m}\sup_{(s,t)\in [0,1]^2} \left|(\tau_1-\tau_2) \left((\frac{u_t^l+1}{u_t^l-1})^*(\frac{u_s^\lambda+1}{u_s^\lambda-1})\right)\right|,\\&
d_{1,0}(\tau_1,\tau_2)=d_{0,0}(\tau_1,\tau_2)+\sum_{k=1}^\infty 2^{-k}\\&\sup_{t\in [0,1]}\sup_{(l,i_1,...,i_m) \in [\![1,m]\!]^{k+1}}\sup_{(\epsilon_1,...,\epsilon_m) \in \{-1,1\}^k}\sup_{(j_1,...,j_k)\in(\{2,....,\mu+1\}\cup[0,1])^k}|(\tau_1-\tau_2)((\frac{u_t^l+1}{u_t^l-1})^*(u_{j_1}^{i_1})^{\epsilon_1}...(u_{j_k}^{i_k})^{\epsilon_k})|.\end{align*}
where 
\begin{align*}&d_{0,0}(\tau_1,\tau_2)=\sum_{k=1}^\infty 2^{-k}\\&\sup_{(i_1,...,i_m) \in [\![1,m]\!]^k}\sup_{(\epsilon_1,...,\epsilon_m) \in \{-1,1\}^k}\sup_{(j_1,...,j_k)\in(\{2,....,\mu+1\}\cup[0,1])^k}|(\tau_1-\tau_2)((u_{j_1}^{i_1})^{\epsilon_1}...(u_{j_k}^{i_k})^{\epsilon_k})|.\end{align*}

%We may also use the obvious variants on  $\mathcal{T}^c_{2,0}(\mathcal{F}^m_{[0,1]}*\mathcal{F}^\nu_{\mu})$ and on $\mathcal{T}_{2,0}(\mathcal{F}^m_{n}*\mathcal{F}^{\nu}_{\mu}) $.
 The extra unitary variables will be used %to generate the relative algebra $B$ appearing in Theorem \ref{ThmA} {\color{red} [UPDATE]} and will be used 
 as variables approximated by non-random matrices. We will only consider  classes of convex or regular functions in the next subsection to be uniform in any respect over these extra variables. For $X$ a process in $L^2_{sa}(M,\tau)^m$ continuous as before and $u\in \mathcal{U}(M)^{\mu\nu}$ (we call $\mathcal{U}(M)$ the set of unitaries), we write $\tau_{X,u}\in \mathcal{T}^c_{2,0}(\mathcal{F}^m_{[0,1]}*\mathcal{F}^\nu_{\mu})$ for the joint law and similar notations for laws on $\mathcal{T}_{2,0}(\mathcal{F}^m_{n}*\mathcal{F}^{\nu}_{\mu}) $.

%$$d_2(\tau_1,\tau_2)=d(\tau_1,\tau_2)+\sup_{l=1,...,m}\sup_{t\in [0,1]} \left|(\tau_1-\tau_2) \left((\frac{u_t^l+1}{u_t^l-1})^*(\frac{u_t^l+1}{u_t^l-1})\right)\right|.$$

%We call $\mathcal{T}_{2,N}(\mathcal{F}^m_{n})$ the closed convex subset realised as laws $\tau_X$ of $N\times N$ random hermitian matrices and $\mathcal{T}_{2,N, det}(\mathcal{F}^m_{n})$ the subset realised as laws of $N\times N$ (deterministic) hermitian matrices. {\color{red}[DO I USE  THIS?]}

\subsection{Some non-commutative continuous (semi)-convex functions}\label{NCFun}
 We fix $\mu,\nu\in\N,m\geq 1$.  We will consider two kinds of convex (or semi-convex) functions on $\mathcal{T}_{2,0}^c(\mathcal{F}^m_{[0,1]}*\mathcal{F}^\nu_{\mu})$ bounded from below and with subquadratic growth in the continuous time variables uniformly over the supplementary unitary variables. Of course, this includes the case $\mathcal{T}_2^c(\mathcal{F}^m_{[0,1]})$ if $\nu=\mu=0$. To deal with the growth condition, we will need to consider functions depending only on finitely many times. We thus define a subspace of the space of continuous functions $C^0(   \mathcal{T}_2^c(\mathcal{F}^m_{[0,1]}*\mathcal{F}^{\nu}_{\mu}),d_{i,0}),i=0,1,2$ defined as :
     \begin{align*}C^0_{(k)}(   \mathcal{T}_{2,0}^c(\mathcal{F}^m_{[0,1]}*\mathcal{F}^{\nu}_{\mu}),d_{i,0})= \{ f\in &C^0(   \mathcal{T}_{2,0}^c(\mathcal{F}^m_{[0,1]}*\mathcal{F}^{\nu}_{\mu}),d_{i,0}) : \exists t_1,...,t_k \in ]0,1]\\& \exists g\in  C^0(   \mathcal{T}_{2,0}(\mathcal{F}^m_{k}*\mathcal{F}^{\nu}_{\mu}),d_{i,0}),\ 
  f(\tau)= g(\tau\circ (I_{t_1,...,t_k}*Id))\}.
  \end{align*}

For $f\in C^0_{(k)}(   \mathcal{T}_{2,0}^c(\mathcal{F}^m_{[0,1]}*\mathcal{F}^{\nu}_{\mu}),d_{i,0})$. we then write $\mathbf{t}(f)$, $k(f)$ the minimal choice of time sets and index $k$.

It is convenient to define for $\mathbf{t}=(0<t_1<t_2<...<t_k)$ the functions (for $\alpha,\epsilon\geq 0$) \begin{align*}g_{2,\mathbf{t}}(\tau)&=  \frac{1}{2}\sum_{l=1}^m\left(\frac{16}{t_1}\tau\left((\frac{u_1^l+1}{u_1^l-1})^*(\frac{u_1^l+1}{u_1^l-1})\right)\right.\\&\left.+\sum_{L=2}^k \frac{16}{t_L-t_{L-1}}\tau\left((\frac{u_L^l+1}{u_L^l-1}-\frac{u_{L-1}^l+1}{u_{L-1}^l-1})^*(\frac{u_L^l+1}{u_L^l-1}-\frac{u_{L-1}^l+1}{u_{L-1}^l-1})\right)\right)%,
%\\ g_{\alpha,\epsilon}(\tau)&=  \frac{\epsilon}{2}\sum_{L=1}^k\sum_{l=1}^m\left((16)^{1+\alpha/2}\tau\left(\Big((\frac{u_L^l+1}{u_L^l-1})^*(\frac{u_L^l+1}{u_L^l-1})\Big)^{1+\alpha/2}\right)\right)
.\end{align*}
   We have $g_{2,\mathbf{t}}\in C^0(   \mathcal{T}_{2}(\mathcal{F}^m_{k}),d_2)%,g_{\alpha,\epsilon}\in C^0(   \mathcal{T}_{2+\alpha}(\mathcal{F}^m_{k}),d_{2+\alpha})
   .$ It corresponds to the potential giving the density of the finite dimensional distribution of hermitian brownian motion.
   
   \begin{definition}
A (real valued) continuous function $g\in  C^0(   \mathcal{T}_{2,0}(\mathcal{F}^m_{k}*\mathcal{F}^{\nu}_{\mu}),d_{2,0})$ %f\in C^0_{(k)}(   \mathcal{T}_2^c(\mathcal{F}^m_{[0,1]}),d_2)$
 is said to be \textbf{universally convex} if  for any tracial $W^*$ probability space $(M,\tau)$, $X\mapsto f(\tau_{X,u})$ is convex on $(L^2(M,\tau)^{mk})_{sa} $ for any unitaries $u\in \mathcal{U}(M)^{\mu\nu}$. It is said \textbf{matricially convex} if this holds only for $M\subset R^\omega$, some (countable) ultrapower of the hyperfinite $II_1$ factor $R$. A function $f\in C^0_{(k)}(   \mathcal{T}_{2,0}^c(\mathcal{F}^m_{[0,1]}*\mathcal{F}^{\nu}_{\mu}),d_{2,0})$ is said to be matricially or universally convex if $f(\tau)= g(\tau\circ (I_{t_1,...,t_k}*Id))
$ and if so is $g$. 
   \end{definition}

For a convex function    $g\in  C^0(   \mathcal{T}_2(\mathcal{F}^m_{k}*\mathcal{F}^{\nu}_{\mu}),d_{2,0}),$ %$f\in C^0(   \mathcal{T}_2^b(\mathcal{F}^m_{[0,1]}),d_2)$,
 we may sometimes call, for $u\in\mathcal{U}(M)^{\mu\nu},$ $g_\tau(u):X\mapsto f(\tau_{X,u})$ the convex function on $(L^2(M,\tau)^{mk})_{sa} $. We can consider its subdifferential, a multivalued map: $$\partial(g_\tau(u)):(L^2(M,\tau)^{mk})_{sa} \to P((L^2(M,\tau)^{mk})_{sa} )$$ (see e.g \cite{Ekeland} or \cite[ex 2.1.4]{BrezisMonotone})  which is defined by: 
\begin{align*}\partial(g_\tau(u))(X):=\{G\in (L^2(M,\tau)^{mk})_{sa} : &\forall Y\in L^2_{sa}(M,\tau)^{mk}, \\&\tau(\sum_{i=1}^n\sum_{j=1}^k(G_j^i)^*(Y_j^i-X_j^i))\leq g(\tau_{Y,u})-g(\tau_{X,u})\}.\end{align*}

Note that $g_\tau(u)$ is continuous on $L^2_{sa}(M,\tau)^{mk}.$ %(since $X\mapsto \tau_{X,u}$is continuous for an $d_{i,0}$ at the target.)
 Indeed, for any sequence $X_n\to X$, it is easy to see that $d_{2,0}(\tau_{X_n,u},\tau_{X,u})\to 0$. Thus, from \cite[ex 2.3.4]{BrezisMonotone}, $\partial(g_\tau(u))$ is a maximal monotone operator so that we will have available the theory of \cite{BrezisMonotone}. 

  \begin{definition}We fix $i\in\{0,1,2\},\alpha\in]0,1].$
  We call $\mathcal{C}(  \mathcal{T}_{2,0}(\mathcal{F}^m_{k}*\mathcal{F}^{\nu}_{\mu}),d_{i,0} )$, (resp. $\mathcal{E}(  \mathcal{T}_{2,0}(\mathcal{F}^m_{k}*\mathcal{F}^{\nu}_{\mu}),d_{i,0} )$) the set of  continuous (resp. universally convex, convex in $\tau$ and continuous )  functions $g\in C^0(   (  \mathcal{T}_{2,0}(\mathcal{F}^m_{k}*\mathcal{F}^{\nu}_{\mu}),d_{i,0} )$, bounded below, { always at least lower semi-continuous for $d_{1,0}$} %, lipschitz on bounded sets for $d_2$, 
and \textbf{subquadratic} in the sense that there is a $C>0$ such that for any $\tau\in \mathcal{T}_{2,0}(\mathcal{F}^m_{k}*\mathcal{F}^{\nu}_{\mu})$:
\begin{equation}\label{subquadorder0} g(\tau)\leq C (1+ \sum_{l=1}^m\sum_{j=1}^k \tau ((4i\frac{u_j^l+1}{u_j^l-1})^*(4i\frac{u_j^l+1}{u_j^l-1}))).\end{equation}
%and for all  $(M,\tau)$ and $X\in L^2_{sa}(M,\tau)^{mk}.$
%$$\forall G\in \partial(g_\tau)(X), ||G||_2\leq C (1+ \sum_{l=1}^m\sum_{j=1}^k \tau ((X_j^l)^*X_j^l)). $$

We call $\mathcal{E}^{1,\alpha}(  \mathcal{T}_{2,0}(\mathcal{F}^m_{k}*\mathcal{F}^{\nu}_{\mu}),d_{i,0} )$ (resp. $\mathcal{C}^{1,\alpha}(  \mathcal{T}_{2,0}(\mathcal{F}^m_{k}*\mathcal{F}^{\nu}_{\mu}),d_{i,0} )$) the subset of functions $g$ of $\mathcal{E}(  \mathcal{T}_{2,0}(\mathcal{F}^m_{k}*\mathcal{F}^{\nu}_{\mu}),d_{i,0} )$ (resp. $\mathcal{C}(  \mathcal{T}_{2,0}(\mathcal{F}^m_{k}*\mathcal{F}^{\nu}_{\mu}),d_{i,0} )$)
such that there is a constant $C$ such that for all  $(M,\tau)$ and $X,Y\in L^2_{sa}(M,\tau)^{mk}, u,v\in \mathcal{U}(M)^{\nu\mu}$:
\begin{equation}\label{Lop0} |g(\tau_{X,u})-g(\tau_{Y,u})| \leq C\left( \sum_{l=1}^m\sum_{j=1}^k \tau ((X_j^l-Y_j^l)^*(X_j^l-Y_j^l))\right)^{1/2} \left(1+ \sum_{l=1}^m\sum_{j=1}^k \tau ((X_j^l)^*X_j^l+(Y_j^l)^*Y_j^l)\right)^{1/2}.\end{equation}
\begin{equation}\label{Lip0} |g(\tau_{X,u})-g(\tau_{X,v})| \leq C\left( \sum_{l=1}^\mu\sum_{j=1}^\nu \tau ((u_j^l-v_j^l)^*(u_j^l-v_j^l))\right)^{1/2} .\end{equation}
and
\begin{equation}\label{C110} |g(\tau_{X+Y,u})+g(\tau_{X-Y,u})-2g(\tau_{X,u})| \leq C \left( \sum_{l=1}^m\sum_{j=1}^k \tau ((Y_j^l)^*Y_j^l)\right)^{\frac{1+\alpha}{2}}\left(1+ \sum_{l=1}^m\sum_{j=1}^k \tau ((X_j^l)^*X_j^l+(Y_j^l)^*Y_j^l)\right)^{\frac{1_{\{\alpha<1\}}}{2}}.\end{equation}
and also for some constant $C(g)>0$:\begin{equation}\label{LipReg}
 |g(\tau_1)-g(\tau_2)|\leq C(g) d_{i,0}(\tau_1,\tau_2).
 \end{equation}

% We call $\mathcal{C}^{2}_b(  \mathcal{T}_2(\mathcal{F}^m_{k}),d_2 )$ the subset of functions $g$ of $C^0(  \mathcal{T}_2(\mathcal{F}^m_{k}),d_2 )$ bounded below, subquadratic in the sense of \eqref{subquadorder0} and such that for all  $(M,\tau)$ and $X\in L^2_{sa}(M,\tau)^{mk}$ $X\mapsto g(\tau_X)$ is twice continuously differentiable and satisfy \eqref{Lop0}, \eqref{C110} and \eqref{LipReg}.
 
We call $\mathcal{E}_{(k)}(   \mathcal{T}_{2,0}^c(\mathcal{F}^m_{[0,1]}*\mathcal{F}^{\nu}_{\mu}),d_{2,0})$ the set of  universally convex functions $f\in C^0(   \mathcal{T}_{2,0}^c(\mathcal{F}^m_{[0,1]}*\mathcal{F}^{\nu}_{\mu}),d_{2,0})$ with $f(\tau)= g(\tau\circ (I_{t_1,...,t_k}*Id))
$ for  $g\in \mathcal{E}(  \mathcal{T}_{2,0}(\mathcal{F}^m_{k}*\mathcal{F}^{\nu}_{\mu}),d_{2,0} )$ and some $0<t_1<\dots <t_k\leq 1$. 

Finally, we write $\mathcal{E}(   \mathcal{T}_{2,0}^c(\mathcal{F}^m_{[0,1]}*\mathcal{F}^{\nu}_{\mu}),d_{2,0})=\cup_{k=1}^\infty\mathcal{E}_{(k)}(   \mathcal{T}_{2,0}^c(\mathcal{F}^m_{[0,1]}*\mathcal{F}^{\nu}_{\mu}),d_{2,0}),$
and similarly $\mathcal{E}_{(k)}^{1,\alpha}(   \mathcal{T}_{2,0}^c(\mathcal{F}^m_{[0,1]}*\mathcal{F}^{\nu}_{\mu}),d_{i,0})$, $\mathcal{E}^{1,\alpha}(   \mathcal{T}_{2,0}^c(\mathcal{F}^m_{[0,1]}*\mathcal{F}^{\nu}_{\mu}),d_{i,0})$,
$\mathcal{C}_{(k)}^{1,\alpha}(   \mathcal{T}_{2,0}^c(\mathcal{F}^m_{[0,1]}*\mathcal{F}^{\nu}_{\mu}),d_{i,0})$, $\mathcal{C}^{1,\alpha}(   \mathcal{T}_{2,0}^c(\mathcal{F}^m_{[0,1]}*\mathcal{F}^{\nu}_{\mu}),d_{i,0})$.

%We call $\mathcal{E}_m(   \mathcal{T}_{2,0}^c(\mathcal{F}^m_{[0,1]}*\mathcal{F}^{\nu}_{\mu}),d_{2,0})$ the corresponding space where functions are only matricially convex.% and $M\subset R^\omega$ in the condition on sub-differentials.
% We call $\mathcal{E}(   \mathcal{T}_{2,N}(\mathcal{F}^m_k),d_2),$ $\mathcal{E}(   \mathcal{T}_{2,N,det}(\mathcal{F}^m_k),d_2)$ the corresponding spaces of continuous subquadratic bounded below convex functions on the state spaces for hermitian $N\times N$ matrices, with all the conditions restricted to the states appearing in those state spaces. {\color{red}[DO I USE  THIS?]}
    \end{definition} 
   
  % Of course, we have an obvious variant for $t<1$, $\mathcal{E}_m(   \mathcal{T}_2^c(\mathcal{F}^m_{[0,t]}),d_2)$ with all suprema taken on $[0,t]$.
  
The following lemma will help checking semicontinuity for $d_{1,0}$  in our examples:
\begin{lemma}\label{quadraticsemicont}The function defined by $g_{j,k,\lambda}(\tau)=
 \tau ((\frac{u_j^l+1}{u_j^l-1}+\lambda\frac{u_k^l+1}{u_k^l-1})))^*(\frac{u_j^l+1}{u_j^l-1}+\lambda\frac{u_k^l+1}{u_k^l-1})))$ (value computed in the GNS representation), for $\lambda\in \R$ is in $\mathcal{C}^{1,1}(  \mathcal{T}_{2,0}(\mathcal{F}^m_{k}*\mathcal{F}^{\nu}_{\mu}),d_{2,0} ).$
\end{lemma}
\begin{proof} Continuity and even Lipschitzness for $d_{2,0}$ is in the definition of $d_{2,0}$, subquadratic behaviour is obvious by Cauchy-Schwarz, and non-negativity gives the lower bound. \eqref{Lop0} and \eqref{C110} are similar.
We must check lower semcontinuity for $d_{1,0}$. This comes from the formula:
$$\tau ((\frac{u_j^l+1}{u_j^l-1}+\lambda\frac{u_k^l+1}{u_k^l-1})))^*(\frac{u_j^l+1}{u_j^l-1}+\lambda\frac{u_k^l+1}{u_k^l-1})))=\sup \frac{\tau\left((\frac{u_j^l+1}{u_j^l-1}+\lambda\frac{u_k^l+1}{u_k^l-1})^*\sum_i\lambda_i(u_{j_1^i}^{l_1^i})^{\epsilon_1^i}...(u_{j_{m_i}^i}^{l_{m_i}^i})^{\epsilon_{m_i}^i}\right)}{\max(1,||\sum_i\lambda_i(u_{j_1^i}^{l_1^i})^{\epsilon_1^i}...(u_{j_{m_i}^i}^{l_{m_i}^i})^{\epsilon_{m_i}^i})||_{2,\tau})}$$
and each term in supremum in the right hand is by definition $d_{1,0}$ continuous, hence the supremum is lower semi-continuous.
\end{proof}

For our purpose, we also introduce  several classes of explicit functions:
   \begin{align*}&\mathcal{E}_{reg,\infty}(   \mathcal{T}_{2,0}^c(\mathcal{F}^m_{[0,1]}*\mathcal{F}^{\nu}_{\mu}),d_{2,0})\\&=\{ f\in \mathcal{E}(   \mathcal{T}_{2,0}^c(\mathcal{F}^m_{[0,1]}*\mathcal{F}^{\nu}_{\mu}),d_{2,0}):  \exists l\in \N,\exists g,g_1,...,g_l\in \mathcal{E}(  \mathcal{T}_{2,0}(\mathcal{F}^m_{k}*\mathcal{F}^{\nu}_{\mu}),d_{2,0} ),\\&\qquad \exists (D,D_1,...,D_l,\lambda_1,...,\lambda_l,C_1,...,C_l)\in \R^{l+1}\times\C^{l} \times]0,\infty[^l, \exists (\epsilon_1^i,...,\epsilon_{m_i}^i) \in \{-1,1\}^{m_i}, \\&\qquad\exists ((j_1^i,l_1^i),...,(j_{m_i}^i,l_{m_i}^i)) \in (\{1,...,k+1\}\times \{1,...,m\}\cup \{k+2,...,k+\mu+1\}\times \{m+1,...,m+\nu\})^{m_i} \\&\qquad f(\tau)= g(\tau\circ (I_{t_1,...,t_k}*Id))\ and \ g(\tau)=D+\left(\max_{i=1,...,l}g_i(\tau)\right)\\& \ and\  g_i(\tau)=D_i+C_i \sum_{j=1}^k\sum_{l=1}^m \tau (\left|4i\frac{u_j^l+1}{u_j^l-1})\right|^2)+\Re(\lambda_i\tau((u_{j_1^i}^{l_1^i})^{\epsilon_1^i}...(u_{j_{m_i}^i}^{l_{m_i}^i})^{\epsilon_{m_i}^i}))\geq 1
 \}
\end{align*}
   
with $\Re$ the real part to stay within real potentials and another $p$-norm variant for $p\in[ 1,\infty[$:  
   \begin{align*}&\mathcal{E}_{reg,p}(   \mathcal{T}_{2,0}^c(\mathcal{F}^m_{[0,1]}*\mathcal{F}^{\nu}_{\mu}),d_{2,0})\\&=\{ f\in \mathcal{E}(   \mathcal{T}_{2,0}^c(\mathcal{F}^m_{[0,1]}*\mathcal{F}^{\nu}_{\mu}),d_{2,0}):  \exists l\in \N,\exists g,g_1,...,g_l\in \mathcal{E}(  \mathcal{T}_{2,0}(\mathcal{F}^m_{k}*\mathcal{F}^{\nu}_{\mu}),d_{2,0} ),\\&\qquad \exists (D,D_1,...,D_l,\lambda_1,...,\lambda_l,C_1,...,C_l)\in \R^{l+1}\times \C^{l}\times ]0,\infty[^l, \exists (\epsilon_1^i,...,\epsilon_{m_i}^i) \in \{-1,1\}^{m_i}, \\&\qquad\exists ((j_1^i,l_1^i),...,(j_{m_i}^i,l_{m_i}^i)) \in (\{2,...,k+1\}\times \{1,...,m\}\cup \{k+2,...,k+\mu+1\}\times \{m+1,...,m+\nu\})^{m_i} \\&\forall\overline{ \tau}\in \mathcal{T}_{2,0}^c(\mathcal{F}^m_{[0,1]}*\mathcal{F}^{\nu}_{\mu}), \forall \tau\in \mathcal{T}_{2,0}(\mathcal{F}^m_{k}*\mathcal{F}^{\nu}_{\mu}):\\&\quad f(\overline{ \tau})= g(\overline{ \tau}\circ (I_{t_1,...,t_k}*Id))\ \ and \ g(\tau)=D+\left(\sum_{i=1,...,l}\left(g_i(\tau)\right)^p\right)^{1/p} \\&\qquad \ and\ g_i(\tau)=D_i+C_i \sum_{j=1}^k\sum_{l=1}^m \tau ((4i\frac{u_j^l+1}{u_j^l-1})^*(4i\frac{u_j^l+1}{u_j^l-1}))+\Re(\lambda_i\tau((u_{j_1^i}^{l_1^i})^{\epsilon_1^i}...(u_{j_{m_i}^i}^{l_{m_i}^i})^{\epsilon_{m_i}^i}))\geq 1
 \}
\end{align*}
We also write  
 $\mathcal{E}_{reg,p}(\mathcal{T}_{2,0}(\mathcal{F}^m_{k}*\mathcal{F}^{\nu}_{\mu}),d_{2,0})  $ the corresponding spaces  of functions of the type of $g$ in the definition above. Note immediately that for any $G\in\mathcal{E}_{reg,p}( \mathcal{T}_{2,0}^c(\mathcal{F}^m_{[0,1]}*\mathcal{F}^{\nu}_{\mu}),d_{2,0}) $, $p\in [1,\infty]$, there is a constant $C(G)$ such that \eqref{LipReg} holds.
  
We also consider variants without the convexity property but with the $\mathcal{C}^{1,1}$ property, for $i\in\{0,1,2\},p\in [1,\infty],C\in [0,\infty]$ : 
  
  \begin{align*}&\mathcal{C}_{reg,p,C}(   \mathcal{T}_{2,0}^c(\mathcal{F}^m_{[0,1]}*\mathcal{F}^{\nu}_{\mu}),d_{i,0})\\&=\{ f\in \mathcal{C}(   \mathcal{T}_{2,0}^c(\mathcal{F}^m_{[0,1]}*\mathcal{F}^{\nu}_{\mu}),d_{i,0}):  \exists (l,L)\in \N^2,\exists g_1,...,g_l\in \mathcal{C}^{1,1}(  \mathcal{T}_{2,0}(\mathcal{F}^m_{k}*\mathcal{F}^{\nu}_{\mu}),d_{i,0} ),\\&\exists (D,D_1,...,D_l,\lambda_1,...,\lambda_l)\in \R^{l+1}\times (\C^{L})^{l}, \exists (\epsilon_1^i,...,\epsilon_{m_i}^i) \in (\{-1,0,1\}^L)^{m_i}, \\&\exists ((j_1^i,l_1^i),...,(j_{m_i}^i,l_{m_i}^i)) \in ((\{2,...,k+1\}\times \{1,...,m\}\cup \{k+2,...,k+\mu+1\}\times \{m+1,...,m+\nu\})^L)^{m_i} \\&\forall\overline{ \tau}\in \mathcal{T}_{2,0}^c(\mathcal{F}^m_{[0,1]}*\mathcal{F}^{\nu}_{\mu}), \forall \tau\in \mathcal{T}_{2,0}(\mathcal{F}^m_{k}*\mathcal{F}^{\nu}_{\mu}):\\&\quad f(\overline{ \tau})= g(\overline{ \tau}\circ (I_{t_1,...,t_k}*Id))\ \ with \ g(\tau)=D+C \sum_{j=1}^k\sum_{l=1}^m \tau (\left|4i\frac{u_j^l+1}{u_j^l-1})\right|^2)+\left(\sum_{i=1,...,l}\left(g_i(\tau)\right)^p\right)^{1/p} \\&\qquad \ and\ g_i(\tau)=D_i+\sum_{I=1}^L\Re(\lambda_i(I)\tau((u_{j_1^i(I)}^{l_1^i(I)})^{\epsilon_1^i(I)}...(u_{j_{m_i}^i(I)}^{l_{m_i}^i(I)})^{\epsilon_{m_i}^i(I)}))\geq 1
 \}
\end{align*}
Note that since $C=0$ is now allowed, the functions can be continuous for $d_{1,0}$ so that none of these spaces are empty in this case $C=0$.

 In order to prove a Laplace principle \cite{DupuisEllis} (or a Large deviation principle modulo Bryc's theorem \cite[Th 4.4.2]{DZ}), we will need the following lemma to use the variant \cite[Th 4.4.10]{DZ}. Recall that a class $G$ of continuous real valued functions  on a topological space $X$ is said to be \textit{well-separating} if it contains constant functions, is closed by finite pointwise maxima and separates points of $X$ in the sense that for $x\neq y\in X$ and $a,b\in \R$ there exists $g\in G$ with $g(x)=a, g(y)=b.$ (note this would say $-G$ well-separating with the definition of \cite[Def 4.4.7]{DZ} but we prefer to discuss convex instead of concave functions.)

\begin{lemma}\label{wellsep}
The class %es $\mathcal{E}_{reg,\infty}(  \mathcal{T}_{2,0}^c(\mathcal{F}^m_{[0,1]}*\mathcal{F}^{\nu}_{\mu}),d_{2,0}) )$, and $\mathcal{E}( \mathcal{T}_{2,0}^c(\mathcal{F}^m_{[0,1]}*\mathcal{F}^{\nu}_{\mu}),d_{2,0}) $ %and $\mathcal{E}_m( \mathcal{T}_{2,0}^c(\mathcal{F}^m_{[0,1]}*\mathcal{F}^{\nu}_{\mu}),d_{2,0}) $ 
%are well-separating on $(    \mathcal{T}_{2,0}^c(\mathcal{F}^m_{[0,1]}*\mathcal{F}^{\nu}_{\mu}),d_{2,0}) ),$ and 
$\mathcal{C}_{reg,\infty,0}(  \mathcal{T}_{2,0}^c(\mathcal{F}^m_{[0,1]}*\mathcal{F}^{\nu}_{\mu}),d_{1,0}) )$ is well-separating on $(    \mathcal{T}_{2,0}^c(\mathcal{F}^m_{[0,1]}*\mathcal{F}^{\nu}_{\mu}),d_{1,0}) ).$
\end{lemma}   

\begin{proof}
It is obvious that space contains constant functions and is stable by finite maxima by definition.%, since convexity is stable by supremum and so is the subquadratic behaviour.%, one mostly have to check stability by maximum of the condition on subdifferentials. But since the subdifferential of a maxis included in the convex hull of the union of the subdifferentials [FIND REF],  it suffices to note the condition (with the maximum of $C$ involved) is convex.

%Since $\mathcal{E}_{reg,\infty}(   \mathcal{T}_{2,0}^c(\mathcal{F}^m_{[0,1]}*\mathcal{F}^{\nu}_{\mu}),d_{2,0}) )\subset\mathcal{E}(   \mathcal{T}_{2,0}^c(\mathcal{F}^m_{[0,1]}*\mathcal{F}^{\nu}_{\mu}),d_{2,0}) )$%\subset\mathcal{E}_m(   \mathcal{T}_{2,0}^c(\mathcal{F}^m_{[0,1]}*\mathcal{F}^{\nu}_{\mu}),d_{2,0}) )$
%, and the first space is also by definition stable by maximum, it only remains to check the separation property on the first space. But since this space is stable by translation (addition of constants) and multiplication by positive numbers, it suffices to find two functions with different values on $\tau_1\neq \tau_2$. But if all the values where the same, we would have (by the cases $l=1$ $(C_1,\lambda_1)=(C,\lambda)$ or $(C,0)$), $\tau_1((u_{j_1}^{i_1})^{\epsilon_1}...(u_{j_{m_i}}^{i_{m_i}})^{\epsilon_{m_i}})=\tau_2((u_{j_1}^{i_1})^{\epsilon_1}...(u_{j_{m_i}}^{i_{m_i}})^{\epsilon_{m_i}})$ for all possible choices and thus  $\tau_1=\tau_2$ since indeed for $C$ fixed, there is always a $\lambda$ small enough (real or imaginary) such that $g(\tau)=C \sum_{j=1}^k \sum_{l=1}^m\tau ((4i\frac{u_j^l+1}{u_j^l-1})^*(4i\frac{u_j^l+1}{u_j^l-1}))+\Re(\lambda\tau((u_{j_1}^{i_1})^{\epsilon_1}...(u_{j_{m_i}}^{i_{m_i}})^{\epsilon_{m_i}}))$ gives $g\in \mathcal{E}(  \mathcal{T}_{2,0}^c(\mathcal{F}^m_{[0,1]}*\mathcal{F}^{\nu}_{\mu}),d_{2,0}) )$. 
It separates points since:
Since this space is stable by translation (addition of constants) and multiplication by real numbers, it suffices to find two functions with different values on $\tau_1\neq \tau_2$. But if all the values were the same, we would have $\tau_1((u_{j_1}^{i_1})^{\epsilon_1}...(u_{j_{m_i}}^{i_{m_i}})^{\epsilon_{m_i}})=\tau_2((u_{j_1}^{i_1})^{\epsilon_1}...(u_{j_{m_i}}^{i_{m_i}})^{\epsilon_{m_i}})$ for all possible choices and thus  $\tau_1=\tau_2$ by density of linear combinations of these functional in the universal $C^*$-algebra.% since indeed for $C$ fixed, there is always a $\lambda$ small enough (real or imaginary) such that $g(\tau)=\Re(\lambda\tau((u_{j_1}^{i_1})^{\epsilon_1}...(u_{j_{m_i}}^{i_{m_i}})^{\epsilon_{m_i}}))$
%[DETAIL] The case of $\mathcal{C}_{reg,\infty,0}(  \mathcal{T}_{2,0}^c(\mathcal{F}^m_{[0,1]}*\mathcal{F}^{\nu}_{\mu}),d_{1,0}) )$ is similar and easier since the functions with all $C_i=0$ is allowed.
\end{proof}   
 
  \begin{definition}
 %{\color{red}[Really useful ?]}
We call $\mathcal{E}^{1,1}_{app}( \mathcal{T}_{2,0}(\mathcal{F}^m_{k}*\mathcal{F}^{\nu}_{\mu}),d_{2,0}) )\subset \mathcal{E}^{1,1}(  \mathcal{T}_{2,0}(\mathcal{F}^m_{k}*\mathcal{F}^{\nu}_{\mu}),d_{2,0}) )$ the subset of functions $g$  such that there are constants $M>0, $ with for all $I=1,...,m, l=1,...,k$
  all $\epsilon>0$, there are  $P_1,...,P_L\in \mathcal{F}^m_{k}*\mathcal{F}^{\nu}_{\mu}$,  $f_1,...,f_L, g_{1,1},...,g_{k,m}\in C^0(  \mathcal{T}_{2,0}(\mathcal{F}^m_{k}*\mathcal{F}^{\nu}_{\mu}),d_{2,0}) )$, depending on $I,l,\epsilon$, but with $g_{i,j}$ independent $\epsilon$ if $\mu\nu\neq 0$ and  such that for all  $(M,\tau)$ and $X\in L^2_{sa}(M,\tau)^{mk},u\in \mathcal{U}(M)^{\mu\nu}$ we have the approximation:
\begin{equation}\label{App} \left\|\nabla_{X_l^{(I)}}g_\tau(u)(X)-\sum_{i=1}^L P_i(u(X),u)f_i(\tau_{X,u})- \sum_{i=1}^m\sum_{j=1}^kX_j^{(i)}g_{j,i}(\tau_{X,u})\right\|_2 \leq \epsilon.\end{equation}
Moreover if $\mu\nu\neq 0$, we assume $g_{j,i}=0$ if $l\neq j$ or $i\neq I$ and $g_{l,I}(\tau_{X,u})\geq 0$,
$g_\tau(u)(.)$ is twice differentiable %$\frac{D}{(C+\tau(\sum_{i=1}^m\sum_{j=1}^k(X_j^{(i)})^2))^\beta}$ 
and the operator norm bound $||\sum_{i=1}^L P_i(u(X),u)f_i(\tau_{X,u})||\leq M.$
 
  \end{definition} 
  We then define $\mathcal{E}^{1,1}_{app(k)}( \mathcal{T}_{2,0}^c(\mathcal{F}^m_{[0,1]}*\mathcal{F}^{\nu}_{\mu}),d_{2,0}) )$ and $\mathcal{E}^{1,1}_{app}( \mathcal{T}_{2,0}^c(\mathcal{F}^m_{[0,1]}*\mathcal{F}^{\nu}_{\mu}),d_{2,0}) )$ as before.
  
  Note that we will see later (proposition \ref{C11}) that the assumption implies $g_\tau$ differentiable so that the subdifferential $\partial g_\tau(X)=\{(\nabla_{X_j^{(I)}}g_\tau(X))_{j,I}\}$ is singleton valued.
  
  Clearly   $\mathcal{E}_{reg,p}(  \mathcal{T}_{2,0}^c(\mathcal{F}^m_{[0,1]}*\mathcal{F}^{\nu}_{\mu}),d_{2,0}) )\subset \mathcal{E}^{1,1}_{app}( \mathcal{T}_{2,0}^c(\mathcal{F}^m_{[0,1]}*\mathcal{F}^{\nu}_{\mu}),d_{2,0}) )$ for $p\in [2,\infty[$ (cf. the proof of lemma \ref{UniformN} below for an explicit computation of derivatives making that point more explicit).
 
 \subsection{Malliavin calculus and distributional Clark-Ocone formula}  
   
We give in this subsection preliminaries on Malliavin calculus as background to understand the statements from \cite{Ustunel} that we will use.    
   
 Let $\mathbb W\subset C^0([0,1],\R^d)$ the set of continuous paths  starting at $0$ and $\gamma$ Wiener measure on it. $B$ will be the canonical coordinate process. As usual $\mathbb H\subset \mathbb W$ is the Calderon-Martin space of functions $U_t=\int_0^t u_sds$ with $\|U\|_{\mathbb H}=\int_0^1 |u_s|^2ds<\infty.$  We then write $u_s=\dot{U}_s.$

 Since the translations of $\gamma$ with the elements of $\mathbb H$ induce measures
equivalent to $\gamma$, the G\^ateaux  derivative in $\mathbb H$ direction of the
random variables is a closable operator on $L^p(\gamma)$-spaces and  this
closure will be denoted by $\nabla:L^p(\mathbb W,\gamma)\to L^p((\mathbb W,\gamma:\mathbb H)$ (cf. e.g. \cite{ASU, ASU-1,Nualart}). 
It is also useful to point out the explicit formula for $f\in C^1(\R^{kd})$ a $C^1$ function, $t_1<t_2<...<t_k\in[0,1]$ and $F=f(B_{t_1},...,B_{t_k})$, we have (if $h\in \mathbb H$ is the coordinate and $d_if$ is the partial differential in the directions of the $i$-th $d$-uple of variables, $\nabla_if$ the corresponding gradient with $d_if(x).h=\langle \nabla_if(x), h\rangle$):
\[\nabla F=\sum_{i=1}^n (d_if)(B_{t_1},...,B_{t_k}).h_{t_i}\]
It is then common to write for $h\in \mathbb H$: \[\langle \nabla F,h\rangle=\nabla_h F=\int_0^1\dot{h}_t D_tF dt\]
where $D_tF$ is the Lebesgue density of the process $\nabla F$ seen in $L^p(\Omega,\gamma:\mathbb H)$ defined $d\gamma\otimes dt$ almost surely. Note also that $\nabla_h:L^p(\mathbb W,\gamma)\to L^2(\mathbb W,\gamma)$ is also a closable operator for $h\in \mathbb H.$ Especially for $F=f(B_{t_1},...,B_{t_k})$ as above, it is easy to see that :
\begin{equation}\label{DtF}D_t F=\sum_{i=1}^n 1_{[0,t_i]}(t)(\nabla_if)(B_{t_1},...,B_{t_k}).\end{equation}

The corresponding Sobolev spaces
 of (the equivalence classes of)  real random variables
will be denoted as $\mathbb D_{p,k}$, where $k\in \N$ is the order of
differentiability and $p>1$ is the order of integrability. If the
random variables are with values in some separable Hilbert space, say
$\Phi$, then we can define similarly the corresponding Sobolev
spaces and they are denoted as $\mathbb D_{p,k}(\Phi)$, $p>1,\,k\in
\N$. Since $\nabla:\mathbb D_{p,k}\to\mathbb D_{p,k-1}(H)$ is a continuous and
linear operator its adjoint is a well-defined operator which we
represent by $\delta$.  $\delta$ coincides with the It\^o
integral of the Lebesgue density of the adapted elements of
$\mathbb D_{p,k}(H)$ (cf.\cite{ASU,ASU-1,Nualart}).  
   
We denote by $\mathbb D_{p,k}^a(\mathbb H)$ the subspace defined by 
$$
\mathbb D_{p,k}^a(\mathbb H)=\{\xi\in\mathbb D_{p,k}(\mathbb H):\,\dot{\xi}\mbox{ is adapted}\}
$$
for $p> 1,\,k\in\N$, for $p=2,\,k=0$, we shall write $L^2_a(\mu,\mathbb H)$.

To use the results of \cite{Ustunel}, we will need supplementary background from \cite{repres}. We consider $\mathbb D(\Phi)=\bigcap_{p,k}\mathbb D_{p,k}(\Phi), \mathbb D^a(\mathbb H)=\bigcap_{p,k}\mathbb D_{p,k}^a(\mathbb H)$ with projective limit topology and the continuous duals $\mathbb D^\prime(\Phi), (\mathbb D^a(\mathbb H))^\prime$ (if $\Phi=\C$ we write only $\mathbb D,\mathbb D^\prime.$) We also let $\mathbb D_0= \overline{\{ \psi-\langle\psi,1\rangle, \psi \in\mathbb D\}}$ In \cite[Corol II.1]{repres}, it is shown that $J:\mathbb D^a(\mathbb H)\to \mathbb D_0$ defined by stochastic integration $$J(\xi)=\int_0^1\dot{\xi}_sdB_s$$
is continuous and has a continuous inverse $\partial_B:\mathbb D_0\to \mathbb D^a(\mathbb H)$ which enables to express Clark-Ocone's formula as $$\psi=\langle\psi,1\rangle+J(\partial_B(\psi-\langle\psi,1\rangle)).$$
Then, the adjoint of $J$ extends $\partial_B$, and the adjoint of $\partial_B$ extends $J$ so that if we still call in the same way the extensions, the previous formula holds for $\psi \in  \mathbb D^\prime$ 
\cite[Prop II.2]{repres}. For $\xi\in \mathbb D(\mathbb H)$, if $\pi\xi$ is an $\mathbb H$-valued process with Lebesgue derivative $E(\dot{\xi_s}|\mathcal{F}_s)$. Then \cite[Prop III.1]{repres}, $\pi:\mathbb D(\mathbb H)\to \mathbb D^a(\mathbb H)$ is continuous and has a unique continuous extension $\hat{\pi}:\mathbb D^\prime(\mathbb H)\to (\mathbb D^a(\mathbb H))^\prime$ \cite[Prop III.1]{repres} which coincides with the restriction map on $\mathbb D^a(\mathbb H)\subset \mathbb D(\mathbb H)$. Then on $\mathbb D_0^\prime= \{ \psi-\langle\psi,1\rangle, \psi \in\mathbb D^\prime\}$, \cite[Th IV.1]{repres} gives the representation $$\partial_B=\hat{\pi}\nabla.$$

Especially, if $F\in L^{1+\epsilon}(\gamma)$, $\epsilon>0$, we can consider $\nabla F\in \mathbb D_{1+\epsilon,-1}(\mathbb H)$ and $\partial_B(F)=\hat{\pi}\nabla(F)\in L^{1+\epsilon}(\gamma;\mathbb H)$ using Bucholder-Davis-Gundy inequality giving the equivalence of norms induced by $J:L^{1+\epsilon}_{ad}(\gamma;\mathbb H)\to L^{1+\epsilon}(\gamma).$ 
In that case, we will write $$[\hat{\pi}\nabla(F)]_s=E(D_sF|\mathcal{F}_s).$$ By the extension properties mentioned before, this expression coincides when $F=f(B_{t_1},...,B_{t_k})$ with the previous expression if $f$ is $C^1$. We will thus be able to use results in \cite{Ustunel} expressed in terms of $[\hat{\pi}\nabla(F)]_s$ in using classical Malliavin calculus formulas they extend.

%We will also need a classical result on Skorohod integration even with the variant where $\gamma$, is replaced by $\gamma_\mu$ the law of a brownian motion of covariance $\mu t Id_d$ for $\mu>0$. We can also consider $\delta:L^2(\mathbb W,\gamma_\mu:\mathbb H)\to L^2(\mathbb W,\gamma_\mu)$ as an unbounded operator adjoint of $\nabla:L^2(\mathbb W,\gamma_\mu)\to L^2(\mathbb W,\gamma_\mu:\mathbb H).$ The next result is standard, see e.g. \cite[Prop 1.3.1,1.3.2]{Nualart} (where we take $\mu \langle .,.\rangle_{\mathbb H}$ as scalar product to recover the law $\gamma_\mu.$ )

%\begin{proposition}\label{DomainSkorohod} 
%$\mathbb D_{2,1}(\mathbb H)\subset D(\delta)$ and we have the isometry relation for $u,v\in\mathbb  D_{2,1}(\mathbb H)$:
%$$E(\delta(u)\delta(v))=\mu E(\langle u,v\rangle_{\mathbb H})+\mu^2\int_{[0,1]^2}D_s\dot{u}_tD_t\dot{u}_s ds dt.$$
%Moreover, if $u\in \mathbb D_{2,2}(\mathbb H)$ then  $\delta(u)\in D(\nabla_h), \nabla_hu\in \mathbb D_{2,1}(\mathbb H)\subset D(\delta)$ for any $h\in \mathbb H$ and :
%\begin{equation}\label{commutation}\nabla_h(\delta(u))=\mu \langle u,h\rangle _{\mathbb H} + \delta ( \nabla_h u).\end{equation}
%\end{proposition}

\noindent
Finally, a measurable  function
$f:\mathbb W\to \R\cup\{+\infty\}$  is called $\alpha$-convex, $\alpha\in\R$, if the map
$$
h\to f(x+h)+\frac{\alpha}{2}|h|_H^2=F(x,h)
$$
is convex on the Cameron-Martin space $\mathbb H$ with values in
$L^0(\mu)$. Note that this notion is compatible with the
$\mu$-equivalence classes of random variables thanks to the
Cameron-Martin theorem (cf. \cite{F-U1}).

  \subsection{Hermitian brownian motion and its exponential tightness}  
   \label{hermitianB}
For $(N,m,\mu,\nu)\in(\N^*)^4,$ and $d=N^2m$, we write $\mathbb W_{sa,N}\subset C^0([0,1],\R^d)=C^0([0,1],(M_N(\C)_{sa})^m)$ the Wiener space for paths on $m$-tuples of hermitian matrices in $M_N(\C)_{sa}\simeq \R^{N^2}$ ($m$ is fixed throughout the paper and does not appear in the notation). If $\gamma$ is the law of the standard brownian measure making $B_t$ into an hermitian brownian motion, we write $\gamma_{sa,N,m}=\gamma_N$ the law of $\frac{B_t}{\sqrt{N}}$. We define a random state $\widehat{\sigma}^N\in \mathcal{T}^c_2(\mathcal{F}^m_{[0,1]})$ by $\widehat{\sigma}^{N}=\tau_{H^N}$ with $H^N=(\frac{B_t}{\sqrt{N}})_{t\in 0,1}$.
Note that the normalisation is as usual made to insure :$$E\left(\widehat{\sigma}_{N}\left((4i\frac{u_t^l+1}{u_t^l-1})^*(4i\frac{u_t^l+1}{u_t^l-1})\right)\right)=E(\frac{1}{N}Tr((\frac{B_t^l}{\sqrt{N}})^2))=\frac{tN^2}{N^2}=t.$$

% We also define as a technical tool the random state $\widehat{\xi}_{\alpha,\epsilon,\mathbf{t}}^N=\tau_{H^N_{\alpha,\epsilon,\mathbf{t}}}$ with $H^N_{\alpha,\epsilon,\mathbf{t}}$ a random hermitian process of density with respect to the law of $H^N$ given by $Z_Nexp(-N^2g_{\alpha,\epsilon}\circ I_{t_1,....,t_k}(\frac{1}{N}Tr).$

If $\Upsilon\in \mathcal{U}(M_N(\C))^{\mu\nu}$ is a bunch of (deterministic) unitary matrices, we call $$\widehat{\sigma}^{N}_{\Upsilon}=\tau_{H^N,\Upsilon}\in \mathcal{T}_{2,0}^c(\mathcal{F}^m_{[0,1]}*\mathcal{F}^\nu_{\mu}).$$
We will be interested in large deviation results for this $\widehat{\sigma}^{N}_{\Upsilon}$ (this includes if $\mu=\nu=0$ the case $\widehat{\sigma}^{N}$ considered in \cite{BCG}).

We can also define the image of the Gaussian Unitary Ensemble in Unitary variables as above $\mathfrak{G}_N=\tau_{G^N}\in\mathcal{T}_{2}(\mathcal{F}^m_{1}),$ with $G^N=(\frac{B_1}{\sqrt{N}}).$

The only statement we will use from \cite{BCG} to prove our large deviation principle is their lemma 5.4 (or rather actually a slight improvement, consequence of their proof) giving exponential tightness of $\widehat{\sigma}^{N}_{\Upsilon}$ in $\mathcal{T}_{2,0}^c(\mathcal{F}^m_{[0,1]}*\mathcal{F}^\nu_{\mu}).$
 Let us recall the appropriate notation. For any $g:[0,1]\to \R^+$ with $\lim_{x\to 0} g(x)=0$. We let :
 $$K_g=\left\{ \tau \in \mathcal{T}^c(\mathcal{F}^m_{[0,1]}*\mathcal{F}^\nu_{\mu})
   : \forall s\leq t\in [0,1] \max_{i=1...m} \tau(|u_t^i-u_s^i|^2)\leq g(t-s) \right\},$$
      $$K_{g,2}=\left\{ \tau \in \mathcal{T}^c_{2,0}(\mathcal{F}^m_{[0,1]}*\mathcal{F}^\nu_{\mu})
   : \forall s\leq t\in [0,1] \max_{l=1...m} \tau(|4i(\frac{u_t^l+1}{u_t^l-1})-4i(\frac{u_s^l+1}{u_s^l-1})|^2)\leq g(t-s) \right\}\subset K_{100g},$$
$$\Gamma_L   =\left\{ \tau \in \mathcal{T}^c_{2,0}(\mathcal{F}^m_{[0,1]}*\mathcal{F}^\nu_{\mu})
   : \forall t\in [0,1] \max_{l=1...m} \tau(\left|4i\frac{u_t^l+1}{u_t^l-1}\right|^2)\leq L\right\},$$
%   $$\Gamma_{L,\mathbf{t},\alpha}   =\left\{ \tau \in \mathcal{T}^c_{2+\alpha,0}(\mathcal{F}^m_{[0,1]}*\mathcal{F}^\nu_{\mu})\cap \Gamma_L
%   : \forall t\in \{t_1,...,t_k\}, \max_{l=1...m} \tau(\left|4i\frac{u_t^l+1}{u_t^l-1}\right|^{2+\alpha})\leq L\right\}.$$
   
   \begin{lemma}\label{compact1}
   For any $L>0, g$ as above, $K_{g,2}\cap \Gamma_L$ is a compact set of $(\mathcal{T}^c_{2,0}(\mathcal{F}^m_{[0,1]}*\mathcal{F}^\nu_{\mu}),d_{1,0}).$ %Similarly, for $\alpha>0,%\beta\in [0,\alpha[,
%   $ $K_{g,2}\cap \Gamma_{L,\mathbf{t},\alpha}$ is a compact set of $(\mathcal{T}^c_{2+\alpha,0;\mathbf{t}}(\mathcal{F}^m_{[0,1]}*\mathcal{F}^\nu_{\mu}),d_{2,0;\mathbf{t}}).$
   \end{lemma}
   Note that we cannot prove the corresponding statement with $d_{2,0}$ since there is no reason that the limit of the norm of $X_t$ along a subsequence needed for the convergence in $d_{2,0}$ to be the norm of the limiting linear form $X_t$ defined by the moments. 
\begin{proof} First, the set is obviously closed in using Lemma \ref{quadraticsemicont} for the lower semicontinuity of involved quadratic functions.
Since $K_{g,2}\subset K_{g}$, %the set is precompact in $\mathcal{T}^c(\mathcal{F}^m_{[0,1]})$ from 
one can argue as in \cite[lemma 2.1]{BCG} to check the set is precompact for $d_{0,0}$. Similarly as in their proof, the family of maps $(t,t_1,...,t_n)\mapsto \tau ((4i\frac{u_t^l+1}{u_t^l-1})^*u_{t_1}^{l_1}\cdots u_{t_n}^{l_n}), \tau\in K_{g,2}\cap \Gamma_L$ is equicontinuous and pointwise bounded, thus the result follows from Arz%\"
el\`a-Ascoli theorem. This gives precompactness for $d_{1,0}$ (in using that the relation for $\tau ((4i\frac{u_t^l+1}{u_t^l-1})^*(u_t^l-1)^*u_{t_1}^{l_1}\cdots u_{t_n}^{l_n})$ insures that the maps corresponding to a limit point state has the same interpretation in terms of law of $u_t^l$ and computed in each GNS representation). 
\end{proof} 

From the proof of their result, % and our very first Laplace deviation bound we will obtain later, 
one  readily deduces (since $\Upsilon_N$ does not appear in the sets above):

\begin{lemma}[Lemma 5.4 in \cite{BCG}]\label{Exptight}
For any sequence $\Upsilon_N\in \mathcal{U}(M_N(\C))^{\mu\nu}$, $\widehat{\sigma}^N_{\Upsilon_N}$ is exponentially tight in $(\mathcal{T}^c_{2,0}(\mathcal{F}^m_{[0,1]}*\mathcal{F}^\nu_{\mu}),d_{1,0}).$ since :
\[ \limsup_{L\to \infty}\limsup_{N\to \infty}\frac{1}{N^2}\log P\left(\widehat{\sigma}^N_{\Upsilon_N}\in (K_{L\sqrt{.},2}\cap \Gamma_L)^c\right)=-\infty.\]
\end{lemma}

\subsection{Concentration of measure for Random matrices}
\label{Concentration}
We will need two kinds of concentration of measure results, one for the lower bound (usual LDP upper bound) and one for the upper bound, the first one being mostly used to obtain an appropriate set up to be able to apply the second one.

The first result uses convexity of a potential and Brascamp-Lieb inequality. It was first used in the proof of \cite[Theorem 3.4]{GuionnetMaurelSegala}. We follow their method and only give the proof for the reader's convenience. Following the probabilistic tradition, we state the almost sure result, but the interesting bound for us with our use of ultraproducts techniques is the uniform integrability like bound \eqref{IntegralUnifBound}. In a second part, we also include a concentration property coming from logarithmic Sobolev inequality (see e.g. \cite[Th 4.4.17]{AGZ}).

For brevity, we write $\mathcal{C}_{k}^m=\C\langle X_1^1,...,X_1^m,X_2^1,...,X_k^m\rangle$ the algebra of non-commutative polynomials in selfadjoint variables and $$\mathcal{C}_{k,\mu}^{m,\nu}=\C\langle X_1^1,...,X_1^m,X_2^1,...,X_k^m,  u_1^1,...,u_1^\nu,u_2^1,...,u_\mu^\nu\rangle,$$ the algebra of non-commutative polynomials in the same selfadjoint variables and supplementary unitary variables.

\begin{proposition}\label{ConcentrationNorm}Let $\Upsilon_N\in \mathcal{U}(M_N(\C))^{\mu\nu}$ a sequence of unitary matrices.
Let $g\in \mathcal{E}(  \mathcal{T}_{2,0}(\mathcal{F}^m_{k}*\mathcal{F}^\nu_{\mu}),d_{2,0} )$ %and $\epsilon,\alpha\geq0$ 
if $\mu\nu=0$ or $g\in \mathcal{E}^{1,1}_{app}(  \mathcal{T}_{2,0}(\mathcal{F}^m_{k}*\mathcal{F}^\nu_{\mu}),d_{2,0} )$ %and $\epsilon=\alpha=0$
 if $\mu\nu\neq0.$
Let $\mathbf{t}=(0<t_1<t_2<...<t_k)$. Consider the probability on $(M_N(\C)_{sa})^{km}$ given (for some normalization constant $Z_{g,\mathbf{t},N}$) by :
$$\mu_{g,\mathbf{t},N}(dx)=\mu_{g,\mathbf{t},N}(dx)=\frac{1}{Z_{g,\mathbf{t},N}}e^{-N^2g(\tau_{x,\Upsilon_N})-N^2g_{2,\mathbf{t}}(\tau_{x})}dLeb_{(M_N(\C)_{sa})^{km}}(dx)$$
 Let $A_{1}^N,..., A_{k}^N=(A_{k,1}^N,...,A_{k,m}^N)$ of law $\mu_{g,\mathbf{t},N}$ (on a same probability space), we have a constant $C>0$ such that a.s.: $$\limsup _{N\to \infty}\max_i||A_{i}^N||_\infty\leq C,$$
 and for $K\in\N^*$\begin{equation}\label{IntegralUnifBound}
 \limsup _{N\to \infty}E_{\mu_{g,\mathbf{t},N}}(1_{\{||A_{i,l}^N||_\infty\geq C\}}\frac{1}{N}Tr(((A_{i,l}^N)^{2K}))=0.
 \end{equation}
Moreover, %, if we also assume $x\mapsto g(\tau_x)$ is $C^2$ on any matrix spaces, then 
for any non-commutative polynomial $P\in \mathcal{C}_{k,\mu}^{m,\nu}\otimes_{alg}\mathcal{C}_{k,\mu}^{m,\nu}$
 $$\lim_{N\to \infty}\left|E_{\mu_{g,\mathbf{t},N}}(\frac{1}{N^2}(Tr\otimes Tr)(P(A_1,...,A_k))-\frac{1}{N^2}\left[(E_{\mu_{g,\mathbf{t},N}}\circ Tr)\otimes (E_{\mu_{g,\mathbf{t},N}}\circ Tr)\right](P)\right|=0 .$$
  \end{proposition}
\begin{proof}
%If $\delta=\min(t_1,t_2-t_1,...,t_k-t_{k-1})$, write $t_\delta=(\delta,2\delta,...,k\delta)$ and note that $g_{2,\mathbf{t}}=g_{2,t_\delta}+g_{2,t_\delta}$
We follow the beginning of the proof of \cite[Theorem 3.4]{GuionnetMaurelSegala}. By Brascamp-Lieb inequality \cite[theorem 1.1]{Harge} (and \cite[Th 2 p 709]{Soshnikov} for the second inequality), we have (if $E(A_{i,l}^N)= E_{\mu_{g,\mathbf{t},N}}(A_{i,l}^N)$ entrywise) $$\mu_{g,\mathbf{t},N}(\frac{1}{N}Tr((A_{i,l}^N-E(A_{i,l}^N))^{2k}))\leq \mu_{0,\mathbf{t},N}(\frac{1}{N}Tr((A_{i,l}^N)^{2k})))\leq C 4^k t_i^{k},\quad k= \sqrt{N},$$
and for $k\geq K$ (using also in the second line \cite[Th 7.5]{MingoNica} and Jensen's inequality):
\begin{align}\label{BLunifinteg}\begin{split}\mu_{g,\mathbf{t},N}&\left(\left(\frac{1}{N}Tr((A_{i,l}^N-E(A_{i,l}^N))^{2k})\right)^{1+K/k}\right)\leq \mu_{0,\mathbf{t},N}\left(\left(\frac{1}{N}Tr((A_{i,l}^N)^{2k}))\right)^{1+K/k}\right)\\&\leq\left[ \mu_{0,\mathbf{t},N}\left(\left(\frac{1}{N}Tr((A_{i,l}^N)^{2k}))\right)^{2}\right)\right]^{1/2+K/2k}\leq C_K 4^{k+K} t_i^{k+K}.\end{split}\end{align}
And thus by Markov inequality, one gets: 
\begin{align*}\mu_{g,\mathbf{t},N}(||A_{i,l}^N-E(A_{i,l}^N)||_\infty\geq 3 \sqrt{t_i})&\leq \mu_{g,\mathbf{t},N}(\frac{1}{N}Tr((A_{i,l}^N-E(A_{i,l}^N))^{2\sqrt{N}})\geq \frac{1}{N}(3 \sqrt{t_i})^{2\sqrt{N}})
\\&\leq CN\left(\frac{2}{3}\right)^{2\sqrt{N}}.\end{align*}
Moreover, we also have the uniform integrability type bound in using the same Markov inequality type argument for $N$ large enough:
\begin{align*}&E_{\mu_{g,\mathbf{t},N}}(1_{\{||A_{i,l}^N-E(A_{i,l}^N)||_\infty\geq 3 \sqrt{t_i}\}}\frac{1}{N}Tr(((A_{i,l}^N)^{2K}))\\&\leq E_{\mu_{g,\mathbf{t},N}}(1_{\{\frac{1}{N}Tr((A_{i,l}^N-E(A_{i,l}^N))^{2\sqrt{N}})\geq \frac{1}{N}(3 t_i)^{2\sqrt{N}}\}}\left(\frac{2^K}{N}Tr(((A_{i,l}^N-E(A_{i,l}^N))^{2K})+\frac{2^K}{N}Tr((E(A_{i,l}^N))^{2K})\right))
\\&\leq \frac{2^K}{N}Tr((E(A_{i,l}^N))^{2K})CN\left(\frac{2}{3}\right)^{2\sqrt{N}}\\&+\frac{2^KN}{(3 t_i)^{2\sqrt{N}}}E_{\mu_{g,\mathbf{t},N}}\left(\frac{1}{N}Tr((A_{i,l}^N-E(A_{i,l}^N))^{2\sqrt{N}})\frac{1}{N}Tr((A_{i,l}^N-E(A_{i,l}^N))^{2K})\right)
\end{align*}
\begin{align*}&E_{\mu_{g,\mathbf{t},N}}(1_{\{||A_{i,l}^N-E(A_{i,l}^N)||_\infty\geq 3 \sqrt{t_i}\}}\frac{1}{N}Tr(((A_{i,l}^N)^{2K}))
\\&\leq \frac{2^K}{N}Tr((E(A_{i,l}^N))^{2K})CN\left(\frac{2}{3}\right)^{2\sqrt{N}}\\&+\frac{2^KN}{(3 t_i)^{2\sqrt{N}}}E_{\mu_{g,\mathbf{t},N}}\left(\left(\frac{1}{N}Tr((A_{i,l}^N-E(A_{i,l}^N))^{2\sqrt{N}})\right)^{1+K/\sqrt{N}}\right)
\\&\leq \frac{2^K}{N}Tr((E(A_{i,l}^N))^{2K})CN\left(\frac{2}{3}\right)^{2\sqrt{N}}+\frac{2^KN}{(3 \sqrt{t_i})^{2\sqrt{N}}}C_K 4^{\sqrt{N}+K} t_i^{\sqrt{N}+K}.
\end{align*}
where the next-to-last inequality comes from H\"older inequality for the normalized trace and the last inequality comes from \eqref{BLunifinteg}.

We estimate $E(A_{i,l}^N)$ in different ways depending on whether $\mu\nu=0$ or not. The case $\mu\nu=0$ is similar to \cite[Theorem 3.4]{GuionnetMaurelSegala} since in this case we can use unitary invariance to get $E(A_{i,l}^N)=E(\frac{1}{N}Tr(A_{i,l}^N))Id_N$ so that $||E(A_{i,l}^N)||_\infty=|E(\frac{1}{N}Tr(A_{i,l}^N))|.$
%First note that since $g$ s bounded below by d$=\sup(-g)$, Z_{g,\mathbf{t}}

First note that from the subquadratic growth condition \eqref{subquadorder0}, applied to $g$, %+g_{2,\mathbf{t}}$
 one deduces \begin{align*}\frac{Z_{g,\mathbf{t}}}{Z_{0,\mathbf{t}}}&\geq \frac{1}{Z_{0,\mathbf{t}}}\int e^{-N^2(C (1+ \sum_{l=1}^m\sum_{j=1}^k \tau_{x} ((x^l_j)^*x^l_j)-N^2g_{2,\mathbf{t}}(\tau_{x})%-N^2g_{\alpha,\epsilon}(\tau_{x})
 }dLeb_{(M_N(\C)_{sa})^{km}}(dx)
 \\&\geq \exp\left(-\int [N^2(C (1+ \sum_{l=1}^m\sum_{j=1}^k \tau_{x} ((x^l_j)^*x^l_j)))%-N^2g_{\alpha,\epsilon}(\tau_{x})
 ]\frac{1}{Z_{0,\mathbf{t}}} e^{-N^2g_{2,\mathbf{t}}(\tau_{x})}dLeb_{(M_N(\C)_{sa})^{km}}(dx)\right)\\&\geq e^{-N^2C(1+m\sum_{j=1}^kt_j)%-N^2\epsilon m 16^{1+\alpha/2}\sum_{j=1}^k(1+m_{\lceil 1+\alpha/2\rceil}t_j^{\lceil 1+\alpha/2\rceil})/2
 }=:e^{-N^2D}\end{align*}
where the second inequality comes from Jensen's inequality (and $m_k$ is the $2k$-th moment of a standard gaussian).
%$ g(\tau)\leq C (1+ \sum_{j=1}^k \tau ((4i\frac{u_j^l+1}{u_j^l-1})^*(4i\frac{u_j^l+1}{u_j^l-1}))),$
Now using that $g$ is bounded below by $\sup(-g)$%and $g_{\alpha,\epsilon}(\tau_{x})\geq 0$
, one deduces from Markov's inequality for any $y>0,\lambda>0$ :
\begin{align*}&\mu_{g,\mathbf{t},N}(|\frac{1}{N}Tr(A_{i,l}^N)|\geq y)=
\mu_{g,\mathbf{t},N}(e^{\lambda N^2|\frac{1}{N}Tr(A_{i,l}^N)|}\geq e^{\lambda N^2y})
\\&\leq e^{-\lambda N^2y+(D-\sup(-g))N^2}\frac{1}{Z_{0,\mathbf{t}}}\int e^{\lambda N^2|\frac{1}{N}Tr(A_{i,l}^N)|-N^2g_{2,\mathbf{t}}(\tau_{x})}dLeb_{(M_N(\C)_{sa})^{km}}(dx)
\\&\leq  e^{-\lambda N^2y+(D-\sup(-g))N^2}\frac{1}{Z_{0,\mathbf{t}}}\int e^{-\lambda N^2\frac{1}{N}Tr(A_{i,l}^N)-N^2g_{2,\mathbf{t}}(\tau_{x})}dLeb_{(M_N(\C)_{sa})^{km}}(dx)
\\&+  e^{-\lambda N^2y+(D-\sup(-g))N^2}\frac{1}{Z_{0,\mathbf{t}}}\int e^{\lambda N^2\frac{1}{N}Tr(A_{i,l}^N)-N^2g_{2,\mathbf{t}}(\tau_{x})}dLeb_{(M_N(\C)_{sa})^{km}}(dx)
\\&\leq  2e^{-\lambda N^2y+(D-\sup(-g))N^2+\frac{N^2}{2}\lambda^2t_i}
\end{align*}Thus optimizing in $\lambda=y/t_i$, one gets a bound by $2e^{AN^2-\frac{y^2N^2}{2t_i}}$ for $A=(D-\sup(-g))$. It is pertinent to cut integrals at $y=\sqrt{2t_i A}$ in order to get:
$$\mu_{g,\mathbf{t},N}(|\frac{1}{N}Tr(A_{i,l}^N)|)\leq \sqrt{2t_i A}+2e^{AN^2}\int_{\sqrt{2t_i A}}^\infty dy e^{-\frac{\sqrt{A}N^2y}{\sqrt{2t_i}}}%\leq \sqrt{2t_i A}+2e^{AN^2}\frac{\sqrt{2t_i}}{\sqrt{A}} e^{-AN^2}
 \leq\sqrt{2t_i }\left(\sqrt{A}+2\frac{1}{\sqrt{A}}\right).$$
 From Borel-Cantelli's lemma, this concludes to the almost sure statement with $C=\sqrt{2t_i }\left(\sqrt{A}+2\frac{1}{\sqrt{A}}+3\right).$ The same $C$ works for the second statement.
 
We now treat the case $\mu\nu\neq 0$ where a different bound is needed for $E(A_{i,l}^N)$ since the sequence $\Upsilon_N$ prevents unitary invariance of the model. In that case, since the entry-wise expectation $E$ is the trace preserving conditional to $M_N(\C)$, thus a completely bounded map, the first use Schwarz inequality for completely positive maps (cf e.g. \cite[Prop 3.3]{Paulsen}, we have the operator inequality $E(A_{i,l}^N)^*E(A_{i,l}^N)\leq E(A_{i,l}^{N*}A_{i,l}^N)$, and thus from selfadjointness $||E(A_{i,l}^N)||^2\leq ||E((A_{i,l}^N)^2)||.$

Recall that by the definition of $\mathcal{E}^{1,1}_{app}$ that in our case $\mu\nu\neq 0$, we know that the operator norm bound (in any finite von Neumann algebra)
$$\left\|\nabla_{X_l^{(I)}}g_\tau(u)(X) -X_l^{(I)}g_{l,I}(\tau_{X,u})\right\| \leq M+1,$$
with $g_{l,I}(\tau_{X,u})\geq 0$.%\frac{D}{(C+\tau(\sum_{i=1}^m\sum_{j=1}^k(X_j^{(i)})^2))^\beta}$.
We will use this in conjunction with the fact that the score function of $\mu_{g,\mathbf{t},N}$ with gradient in variable $A_{i,l}$ is $$\Xi_{i,l}=-N \nabla_{X_i^{(l)}}g_\tau(\Upsilon_N)(A^N)-N\frac{1}{t_{i}-t_{i-1}}(A_{i,l}^N-A_{i-1,l}^N)+N\frac{1}{t_{i+1}-t_{i}}(A_{i+1,l}^N-A_{i,l}^N)$$ with $t_0=0, t_{k+1}=\infty$, $A_{0,l}^N=A_{N+1,l}^N=0$.
 Let us define the block tridiagonal matrix $\Theta(X,u)$ with blocks $\Theta_{i,l}^{j,k}$ all zero (i.e. $\Theta_{i,l}^{I,L}=0$ if $L\neq l$ or $|I-i|>1$) but $\Theta_{i,l}^{i\pm 1,l}(X,u)=\frac{1}{|t_{i\pm 1}-t_{i}|}$ and $\Theta_{i,l}^{i,l}(X,u)=-\frac{1}{t_{i}-t_{i-1}}-\frac{1}{t_{i+1}-t_{i}}-g_{i,l}(\tau_{X,u})$ so that (uniformly in $A^N$):
 $$\left\|\frac{1}{N}\Xi_{i,l}-\sum_{I,L}\Theta_{i,l}^{I,L}(A^N, \Upsilon_N)A_{I,L}^N\right\|\leq M+1.$$
Let us call  $B_{i,l}^N=\frac{1}{N}\Xi_{i,l}-\sum_{I,L}\Theta_{i,l}^{I,L}(A^N, \Upsilon_N)A_{I,L}^N.$
 
%Consider the matrix  $\mathcal{A}^N=(E(A_{i,l}^NA_{j,k}^N))_{((i,l)(j,k))}\in M_{km}(M_N(\C))$ and we will estimate its operator norm.
Note that $-\Theta(A^N, \Upsilon_N)\in M_{km}(\R)$ is symmetric positive and for $\tau=\max(t_1,t_2-t_1,...,t_{k}-t_{k-1})$\begin{align*}-\sum_{i,l,I,L}
 &\Theta_{i,l}^{I,L}(A^N, \Upsilon_N)\lambda_{i,l}\lambda_{I,L}=\sum_{i=1}^k\sum_{l=1}^m\lambda_{i,l}^2g_{i,l}(\tau_{A^N,\Upsilon_N})\\&-\sum_{i=1}^{k-1}\sum_{l=1}^m\frac{1}{(t_{i+ 1}-t_{i})}\lambda_{i,l}(\lambda_{i+ 1,l}-\lambda_{i,l})-\sum_{i=2}^{k}\sum_{l=1}^m\frac{1}{(t_{i}-t_{i-1})}\lambda_{i,l}(\lambda_{i- 1,l}-\lambda_{i,l})+\frac{1}{t_{1}}\sum_{l=1}^m\lambda_{1,l}^2
 \\&=\sum_{i=1}^k\sum_{l=1}^m\lambda_{i,l}^2g_{i,l}(\tau_{A^N,\Upsilon_N})+\sum_{i=1}^{k-1}\sum_{l=1}^m\frac{1}{(t_{i+ 1}-t_{i})}(\lambda_{i+ 1,l}-\lambda_{i,l})^2+\frac{1}{t_{1}}\sum_{l=1}^m\lambda_{1,l}^2
%\\& \geq\frac{1}{\tau}\sum_{l=1}^m\left(\lambda_{1,l}^2+\sum_{i=1}^{k-1}(\lambda_{i+ 1,l}-\lambda_{i,l})^2\right)\geq\frac{1}{\tau}\sum_{l=1}^m\left(\lambda_{1,l}^2+\sum_{i=1}^{k-1}\lambda_{i+ 1,l}^2+\lambda_{i,l}^2-\frac{2i+2}{2i+3}\lambda_{i+ 1,l}^2-\frac{2i+3}{2i+2}\lambda_{i,l}^2\right)
\\& \geq%\frac{1}{\tau}\sum_{l=1}^m\left(\lambda_{1,l}^2+\sum_{i=1}^{k-1}\frac{1}{2i+3}\lambda_{i+ 1,l}^2-\frac{1}{2i+2}\lambda_{i,l}^2\right)=
\frac{1}{\tau}\sum_{l=1}^m\left(\frac{3}{4}\lambda_{1,l}^2+\sum_{i=2}^{k-1}(\frac{1}{2i+1}-\frac{1}{2i+2})\lambda_{i,l}^2+\frac{1}{2k+1}\lambda_{k,l}^2\right)\geq \frac{1}{\tau 2k(2k+1)}\sum_{l=1}^m\sum_{i=1}^{k}\lambda_{i,l}^2.
 \end{align*}(We used in the next-to-last ineaquality the elementary bound $(\lambda_{i+ 1,l}-\lambda_{i,l})^2\geq \frac{1}{2i+3}\lambda_{i+1,l}^2-\frac{1}{2i+2}\lambda_{i,l}^2$. 
Thus we have a uniform operator bound  $(-\Theta(A^N, \Upsilon_N))^{-1}\leq2k(2k+1)\tau $. 

One bounds in introducing conjugate variables to reduce by their definition the degree of potentiallly unbounded terms:
\begin{align*}||E(A_{\iota,\lambda}^NA_{\iota,\lambda}^N)||&=||\sum_{j,L,i,l}E((-\Theta(A^N, \Upsilon_N))^{-1}_{(\iota,\lambda),(j,L)}(-\Theta(A^N, \Upsilon_N))_{(j,L),(i,l)}A_{i,l}^NA_{\iota,\lambda}^N)||
\\&=||\sum_{j,L}E((-\Theta(A^N, \Upsilon_N))^{-1}_{(\iota,\lambda),(j,L)}(-\frac{1}{N}\Xi_{j,L}^N+B_{j,L}^N)A_{\iota,\lambda}^N)||
\\&\leq \sum_{j,L}||E(A_{\iota,\lambda}^N\frac{1}{N}\Xi_{j,L}^{N*}|(-\Theta(A^N, \Upsilon_N))^{-1}_{(\iota,\lambda),(j,L)}|^2\frac{1}{N}\Xi_{j,L}^NA_{\iota,\lambda}^N)||^{1/2}
\\&+ \sum_{j,L}||E(A_{\iota,\lambda}^N(B_{j,L}^N)^*|(-\Theta(A^N, \Upsilon_N))^{-1}_{(\iota,\lambda),(j,L)}|^2B_{j,L}^NA_{\iota,\lambda}^N)||^{1/2}
\\&\leq \tau(2k)(2k+1)\sum_{j,L}||E(A_{\iota,\lambda}^N\frac{1}{N}\Xi_{j,L}^{N*}\frac{1}{N}\Xi_{j,L}^NA_{\iota,\lambda}^N)||^{1/2}
\\&+ \tau(2k)(2k+1)(M+1)\sum_{j,L}||E(A_{\iota,\lambda}^NA_{\iota,\lambda}^N)||^{1/2}
\end{align*}
It remains to compute the first expectation in using integration by parts defining score functions several times:
\begin{align*}\\&E((A_{\iota,\lambda}^N\frac{1}{N}\Xi_{j,L}^{N*}\frac{1}{N}\Xi_{j,L}^NA_{\iota,\lambda}^N)_{ae})
=\sum_{b,c,d=1}^NE(\frac{1}{N}(\Xi_{j,L}^{N})_{bc}\frac{1}{N}(\Xi_{j,L}^N)_{cd}(A_{\iota,\lambda}^N)_{de}(A_{\iota,\lambda}^N)_{ab})
\\&=-\sum_{b,c,d=1}^N1_{j=\iota,L=\lambda,b=e,d=c}E(\frac{1}{N^2}(\Xi_{j,L}^N)_{cc}(A_{\iota,\lambda}^N)_{ab})
-\sum_{b,c,d,e=1}^N1_{j=\iota,L=\lambda,a=c}E(\frac{1}{N^2}(\Xi_{j,L}^N)_{cd}(A_{\iota,\lambda}^N)_{de})
\\&+\frac{1}{N}\sum_{b,c,d=1}^NE(Tr\left(e_{1,c}d_{X_j^{(L)}}{\nabla_{X_j^{(L)}}}g_\tau(\Upsilon_N)(A^N).(e_{cb})e_{d1}+\frac{1_{b=d}}{(t_{i}-t_{i-1})N}+\frac{1_{b=d}}{N(t_{i+1}-t_{i})}\right)
(A_{\iota,\lambda}^N)_{de}(A_{\iota,\lambda}^N)_{ab})\\&=(1+\frac{1}{N^2})1_{a=e}+\left(\frac{1}{t_{i}-t_{i-1}}+\frac{1}{t_{i+1}-t_{i}}\right)E(((A_{\iota,\lambda}^N)^2)_{ae})\\&+\frac{1}{N}\sum_{b,c,d=1}^NE(d^2_{X_j^{(L)},X_j^{(L)}}g_\tau(\Upsilon_N)(A^N).(e_{cb},e_{dc})
(A_{\iota,\lambda}^N)_{de}(A_{\iota,\lambda}^N)_{ab})
\end{align*}

Let us write the matrix $B_c(A^N)=(d^2_{X_j^{(L)},X_j^{(L)}}g_\tau(\Upsilon_N)(A^N).(e_{cb},e_{dc}))_{bd}$ and note that $||B_c(A^N)||\leq C$ with the constant in \eqref{C110} (and in using also proposition \ref{C11}).
%Recall the use of the estimate C110 require suing the norm 1/NTr(A^2)=||A||_{HS}^2 on matrices as GNS representation, 1/NTr(B_c(A^N)BC)=\sum_j 1/NTr(B_c(A^N)B_{.j}C_{j,.})\leq \sum_j  1/N(d^2_{X_j^{(L)},X_j^{(L)}}g_\tau(\Upsilon_N)(A^N).(e_{cb}C_{j,b},B_{dj}e_{dc}))(i.e. gradient evaluated with the 1/N in the scalar product: \leq C\sum j||e_{cb}C_{j,b}||_{HS}||B_{dj}e_{dc}||_{HS}\leq C/N  \sqrt{\sum_{b}|C_{j,b}|^2\sum_{d}|B_{dj}|^2}\leq C ||C||_{HS}||B||_{HS} Hence the operator norm bound
 We thus have the bound \begin{align*}&||E((A_{\iota,\lambda}^N\frac{1}{N}\Xi_{j,L}^{N*}\frac{1}{N}\Xi_{j,L}^NA_{\iota,\lambda}^N))||
\\&\leq(1+\frac{1}{N^2})+\left(\frac{1}{t_{i}-t_{i-1}}+\frac{1}{t_{i+1}-t_{i}}\right)||E((A_{\iota,\lambda}^N)^2)||+\frac{1}{N}\sum_{c=1}^N||E(
A_{\iota,\lambda}^NB_cA_{\iota,\lambda}^N)||
\\&\leq2+\left(\frac{1}{t_{i}-t_{i-1}}+\frac{1}{t_{i+1}-t_{i}}+C\right)||E((A_{\iota,\lambda}^N)^2)||
\end{align*}
Combining our estimates, we obtained:
\begin{align*}||E(A_{\iota,\lambda}^NA_{\iota,\lambda}^N)||&\leq \tau(2k)(2k+1)km\left(\sqrt{2}+\left(\frac{2}{\tau}+C\right)^{1/2}||E(A_{\iota,\lambda}^NA_{\iota,\lambda}^N)||^{1/2}\right)
\\&+ \tau(2k)(2k+1)(M+1)km||E(A_{\iota,\lambda}^NA_{\iota,\lambda}^N)||^{1/2}
\\||E(A_{\iota,\lambda}^NA_{\iota,\lambda}^N)||&\leq D':=2(\tau(2k)(2k+1)km)^2(M+1+\left(\frac{2}{\tau}+C\right)^{1/2})^2+2\sqrt{2}\tau(2k)(2k+1)km.
\end{align*}
As before in the first case, from Borel-Cantelli lemma, this concludes with the bound $C=\sqrt{D'}+\sqrt{9t_i}.$

 For the supplementary statement, it clearly suffices to check for any non-commutative polynomial $Q\in \mathcal{C}_{k,\mu}^{m,\nu}$ :
 \begin{equation}\label{ConcentrationLogSob}\lim_{N\to \infty}\frac{1}{N}E_{\mu_{g,\mathbf{t},N}}\left(|Tr(Q)-(E_{\mu_{g,\mathbf{t},N}}\circ Tr)(Q)|
\right)=0.\end{equation}
If $x\mapsto g(\tau_{x,\upsilon_N})$ were $C^2$ on any matrix spaces (this is in our assumption only if $\mu\nu\neq 0$), we could deduce that from \cite[Theorem 4.4.17]{AGZ} which is based on Bakry-Emery criterion. Instead, we use \cite[Proposition 3.1]{BobkovLedoux} which only uses convexity. Let $c=\max(t_1,t_2-t_1,...,t_{k}-t_{k-1})>0$, then $x\mapsto N^2g_{2,\mathbf{t}}(\tau_{x})$ has second derivative bounded below by $N/c>0$ and thus satisfy (3.1) in \cite{BobkovLedoux} (with euclidean norm), thus adding a convex potential so does the potential for $\mu_{g,\mathbf{t},N}$. Thus, from their proposition 3.1, $\mu_{g,\mathbf{t},N}$ satisfies logarithmic Sobolev inequality with constant $c/N$ and thus the Poincar\'e inequality with constant $m=N/c$, (in the sense of \cite[Definition 4.4.2]{AGZ}, see their \cite[(3.5)]{BobkovLedoux} ), namely if $H_{P,N}=\frac{1}{N}Tr(Q)$:
$$E_{\mu_{g,\mathbf{t},N}}\left(\left|H_{P,N}-E_{\mu_{g,\mathbf{t},N}}(H_{P,N})\right|^2\right)\leq \frac{c}{N^2}E\left(\sum_{i,l}\frac{1}{N}Tr((\mathcal{D}_i^lP)^*(\mathcal{D}_i^lP)\right)$$
and since from our previous result the written expectation has bounded $\limsup$, using $C$ as almost sure bound of our variables, one gets the result.
\end{proof}

Our second concentration result is a variant adapted to our context of \cite[lemma 6.1]{BianeD}. Recall that we call $\mathcal{S}_R^m$ the convex set of tracial states on the universal $C^*$ algebra free product $\bigstar_{i=1}^nC^0([-R,R]).$ The key fact at the basis of our concentration result is that $\mathcal{S}_R^m$ is a Poulsen Simplex  in the sense of \cite{L}. This property has been proved in \cite[Corollary 5]{Dab08} using free entropy techniques.

We will need a variant for another simplex. We call $\mathcal{S}_R^m*\mathcal{T}(\mathcal{F}^\nu_{\mu})$ the convex set of tracial states on the universal free product $\bigstar_{i=1}^nC^0([-R,R])\star \mathcal{F}^\nu_{\mu}.$ Mixing the quoted result with the unitary variant (similar to \cite[lemma 5.2 and Th 5.3]{DDM}) one obtains:

\begin{lemma}\label{Poulsen}
If $m+\mu\nu\geq 2, m\geq 1$, then $\mathcal{S}_R^m*\mathcal{T}(\mathcal{F}^\nu_{\mu})$ is the Poulsen Simplex.
\end{lemma}
\begin{proof}
The only potentially non-well known statement to check is that the extreme points are dense. %If $m=0$, use \cite[lemma 5.2]{DDM}, and 
If $\mu\nu=0$ use \cite[Corollary 5]{Dab08}, thus assume $ \mu\nu\geq 1$. For $X_1,...,X_m,u_1^1,...,u_\mu^\nu$ variables in the GNS representation of a state. If $m\geq 2$ consider $Y_{i,t}=\frac{R(X_i+tS_i)}{R+2t}$, with $S_i$ free semicircular variables free from $X_1,...,X_m,u_1^1,...,u_\mu^\nu$, as in \cite[Corollary 5]{Dab08}, then use \cite{V5} to get $\Phi^*(Y_{1,t},...,Y_{m,t}:W^*(u_1^1,...,u_\mu^\nu))<\infty$ then since $Y_{2,t}$ has finite entropy \cite{V5}, it is diffuse, and using \cite[Th 4]{Dab08}, $W^*(Y_{1,t},...,Y_{m,t},u_1^1,...,u_\mu^\nu)$ is a factor thus correspond to an extremal state in $\mathcal{S}_R^m*\mathcal{T}(\mathcal{F}^\nu_{\mu})$. If $m=1$, approximate $u_1^1$ by a diffuse random variable $v_t$ and conclude in the same way with $Y_{1,t},v_t,u_2^1,...,u_\mu^\nu.$
%One can assume $u_1^1\not\in \C$ (otherwise approximate the state by a commutative unitary random variable). 
%Consider also $u_t$ a unitary free brownian motion free  $X_1,...,X_m,u_1^1,...,u_\mu^\nu$. Let $v_{1,t}=u_tu_1^1u_t^*$.  Then from \cite[lemma 5.2]{DDM}, the commutant $$W^*(v_{1,t})'\cap W^*(X_1,...,X_m,v_{1,t},u_2^1,...,u_\mu^\nu)=\C$$ and thus $W^*(X_1,...,X_m,v_{1,t},u_2^1,...,u_\mu^\nu)$ is a factor. Thus the state corresponding to 
%$X_1,...,X_m,v_{1,t},u_2^1,...,u_\mu^\nu$ in $\mathcal{S}_R^m*\mathcal{T}(\mathcal{F}^\nu_{\mu})$ is extremal, and tends to our given state when $t\to 0.$
\end{proof}

 We then quote a variant of \cite[Corollary 5.4]{BianeD}.

The only change is that we consider for $\tau\in \mathcal{S}_R^m*\mathcal{T}(\mathcal{F}^\nu_{\mu})$ another neighbourhood basis of the weak-* topology. We call $U_{\epsilon, K}(\tau)$ the set of tracial states $\sigma$ such that for all $k\leq K,$ $((j_1,i_1),...,(j_m,i_m)) \in ([\![1,m]\!]\times \{0\}\cup\{1,...,\mu\}\times \{1,...,\nu\})^k$ $(\epsilon_1,...,\epsilon_m) \in ( \{-1,1\})^k$, we have $$|(\sigma-\tau)((u_{j_1}^{i_1})^{\epsilon_1}...(u_{j_k}^{i_k})^{\epsilon_k})|\leq \epsilon,$$
where $u_j^0=u(X_j)$ is obtained from the canonical variable $X_j=X_j(\tau), j=1,...,m$ in the GNS representation of $L^2(\tau)$ as in subsection \ref{prelimStates}, and $u(\tau)=(u_j^i)_{(j,i)\in\{2,...,\mu+1\}\times \{1,...,\nu\} }$ the corresponding unitary variable in the GNS representation. 
We define $V_{\epsilon, K}(\tau)$ as in \cite{BianeD} in considering instead ordinary monomials in variables $X_j$ and $(u_j^i)_{(j,i)\in\{1,...,\mu\}\times \{1,...,\nu\} }$ of order less than $K$.

This defines a map $X(\tau)$ and recall we also defined $\tau_X$ as a tracial state on $\mathcal{F}^m_1*\mathcal{T}(\mathcal{F}^\nu_{\mu})$ in subsection \ref{prelimStates}. Note that the map $\tau\mapsto \tau_{X(\tau),u(\tau)}$ induces a homeomorphism for the weak-* topology to the topology given by $d_{0,0}$ or $d_{2,0}$ which is equivalent on the image since $\frac{1}{u_j-1}=\frac{1}{(u_j+1)-2}$ has a power expansion since $||u_j+1||\leq 2 \frac{R}{\sqrt{R^2+16}} <2$. Note that the homeomorphism statement has a similar proof, the topology of $d_{2,0}$ is clearly  weaker than the image of weak-* topology since $u_j\in \bigstar_{i=1}^nC^0([-R,R])$ and stronger in reasoning as above. 

%We will even use the equivalence of $d,d_2$ as distances coming from 
%$$\left|(\tau_1-\tau_2) \left((\frac{u_j^l+1}{u_j^l-1})^*(\frac{u_j^l+1}{u_j^l-1})\right)\right|\leq \sum_{k,K=0}^\infty  \frac{1}{2^{K+k+2}}\left|(\tau_1-\tau_2)([(u_j^l+1)^*]^{k+1}(u_j^l+1)^{K+1})\right|.$$
We thus deduce from the same proof as \cite[Corollary 5.4]{BianeD}:

\begin{lemma}\label{sep}
Let $\tau$ be an extremal state in $\mathcal S_R^n*\mathcal{T}(\mathcal{F}^\nu_{\mu})$ with $m+\mu\nu\geq 2, m\geq 1$, and $\epsilon>0$. For any $\eta>0$, there exists a self adjoint polynomial 
$$Q_{\eta}\in\Bbb C\langle u(X_1),u(X_1)^*,\ldots, u(X_n),u(X_n)^*, (u_j^i)_{(j,i)\in\{2,...,\mu+1\}\times \{1,...,\nu\} }\rangle$$ 
such that for every $\sigma\in \mathcal S_R^n$ one has
$$\tau(Q_{\eta})>\sigma(Q_{\eta})-\eta$$ 
and for all $\sigma\notin V_{\epsilon,K}(\tau)$ (resp. $\sigma\notin U_{\epsilon,K}(\tau)$) one
has $$\sigma(Q_{\eta})<\tau(Q_{\eta})-1.$$
\end{lemma}

We deduce the concentration of measure in the form we need it.

\begin{proposition}\label{ConcentrationPoulsen}
If $\tau$ is an extremal state in $\mathcal S_R^m\star\mathcal{T}(\mathcal{F}^\nu_{\mu})$ with $m+\mu\nu\geq 2, m\geq 1$, then for any $\epsilon>0$ there exists $\eta>0$, such that for all $N\in \N^*$, for any probability measure $\mu$ law of $X\in(M_N(\C))_{sa}^n$ supported on the ball of operator norm $R$ and $\Upsilon_N\in \mathcal{U}(M_N(\C))^{\mu\nu}$, if 
$ d_{2,0}(E_\mu \circ \tau_{.,\Upsilon_N}, \tau_{X(\tau),u(\tau)})\leq \eta$ then $$ E_\mu\left(d_{2,0}(\tau_{.,\Upsilon_N}, \tau_{X(\tau),u(\tau)})\right)\leq \epsilon.$$

\end{proposition}
\begin{proof}
Since $d_{2,0}$ and $d_{0,0}$ give equivalent topologies on the image of $\mathcal S_R^m\star\mathcal{T}(\mathcal{F}^\nu_{\mu})$ by $\tau_{X(.),u(.)}$ as explained before, it is easy to see that it suffices to prove the corresponding statement for $d_{0,0}$ instead of $d_{2,0}$.
Since the distance $d_{0,0}$ is bounded by $1$, it suffices to prove that with probability greater than $1-\epsilon/2$ we have $d(\tau_{.,\Upsilon_N}, \tau_{X(\tau),u(\tau)})\leq \epsilon /2$ and for that, it suffices to bound sufficiently many moments with probability greater than $1-\epsilon/2$. 
The conclusion follows from our previous lemma as in \cite[lemma 6.1]{BianeD}.
\end{proof}

\subsection{A lipschitzness criterion for directional derivatives, Subdifferentials and Lax-Hopf-Yosida semigroup\label{Yosida}} 
 
 It is well-known (see e.g. \cite{FS}) that it is easier to estimate difference quotients of optimal control problems than derivatives. This is why the relation between regularity bounds of difference quotients (or so-called higher order modulus of continuity) and derivatives will be crucial for us. A first idea would be to use a general kind of spaces between H\"older-Zygmund spaces and Besov spaces, the so-called Nikol'ski\u{\i} Spaces% (see \cite[section 8.2]{Kufner} and \cite{Nikolskii})
 , but most of the available results are strongly dimension dependent. Looking for a dimension independent result, we will rather rely on the following well-known result in convex analysis \cite{Plazanet}. % (see. also \cite{Borwein}). 
 In recent terminology, a function both para-convex and para-concave is G\^ateaux differentiable with Lipschitz derivative (especially $f$ is $C^1$ and thus Fr\'echet-differentiable). Of course, as usual, for $f:H\to \R$, we look at $df:H\to H'\simeq H$ so that we write $\nabla f(x)\in H$ the vector corresponding to $df=\langle\nabla f,.\rangle.$
 
\begin{proposition}\label{C11}[2.2.1 in \cite{Plazanet}]
Let $f: H\to \R$ be a function on a Hilbert space $H$ with $\alpha>0$ such that both $\frac{\alpha}{2}||.||^2-f$ and $\frac{\alpha}{2}||.||^2+f$ are convex, then $f$ is G\^ateaux differentiable on $H$ and we have \begin{equation}
\label{GradLip}||\nabla f(x)-\nabla f(y)||\leq \alpha ||x-y||.
\end{equation}
\end{proposition}

%We will also need a local version for instance proved in finite dimension in \cite[lemma 2.1]{Eberhard}:
%\begin{proposition}\label{C11Ball}
%Let $H=\R^n$ with standard euclidean norm and $x\in \R^n, \delta>0$, $f: B(x,\delta)\to \R$ be a function defined on the ball centred at $x$ with $\alpha>0$ such that both $\frac{\alpha}{2}||.||^2-f$ and $\frac{\alpha}{2}||.||^2+f$ are convex on $B(x,\delta)$ , then $f$ is G\^ateaux differentiable on $B(x,\delta)$ and we have \begin{equation}
%\label{GradLipBall}||\nabla f(x)- \nabla f(y)||\leq \alpha ||x-y||.
%\end{equation}
%\end{proposition}
Even if the previous result is the most crucial one for us, we will also use %(in order to limit the regularity of the potential in  Theorem \ref{ThmC} to having an H\"older continuous gradient instead of a Lipschitz one)
an infinite dimensional local version written in a slightly more general context in  \cite[Thm 4]{Rolewicz} (see also \cite[Thm 6.1]{JTZ} for a generalization to more general uniform continuity classes of the gradient and the proof there for the exact constants in the result below).

\begin{proposition}\label{C1alphaBall}
Let $H$ be a normed space and $x\in H$, $\delta>0$.

\begin{enumerate}
\item If a lipschitz function $f: B(x,\delta)\to \R$ is $(1+\alpha)$-paraconvex and $(1+\alpha)$-paraconcave on $B(x,\delta)$ with constant $C\geq 0$, i.e. for all $t\in ]0,1[,y,z\in B(x,\delta)$:
\begin{equation}
\label{alphaparaconvex}f(ty+(1-t)z)\leq tf(y)+(1-t)f(z)+Ct(1-t)||y-z||^{1+\alpha}\end{equation}
and similarly with $f$ replaced by $-f$, 
  then $f$ is G\^ateaux differentiable on $B(x,\delta)$ and we have for all $y,z\in B(x,\frac{\delta}{4})$ \begin{equation}
\label{GradLipBall}||D_G f(z) -  D_G f(y)||_{H'}\leq 2^{\alpha+2}C||z-y||^\alpha.
\end{equation}
\item Conversely, if $f$ is Fr\'echet differentiable on $B(x,\delta)$ and for all $y,z\in B(x,\delta)$ :$$||D_G f(z) -  D_G f(y)||_{H'}\leq C||z-y||^\alpha,$$
then $f$ is $(1+\alpha)$-paraconvex and $(1+\alpha)$-paraconcave on $B(x,\delta)$ with constant $2C.$
\end{enumerate}
\end{proposition}
We also recall a formula for Clarke's subdifferential \cite{Clarke} $\partial_cf$ in this case. Even though we will mostly use it in the Fr\'echet differentiable case. We quote \cite[Thm 3.1, Corol 7.1] {J}

\begin{proposition}\label{Clarke}
If $f:X\to \R$ is $\gamma=1+\alpha$-paraconvex with constant $C>0,\alpha>0$ in the sense of \eqref{alphaparaconvex}, then for all $x\in X$ 
$$\partial_cf(x)=\{ x^*\in X^* : \langle x^*,h\rangle \leq f(x+h)-f(x)+C||h||^\gamma, \forall h\in X\}.$$
Moreover, in this case, $\partial_cf$ is $\gamma$-monotone, in the sense that, for all $x^*\in \partial_cf(x), u^*\in \partial_cf(u)$, then:
$$\langle x^*-u^*,u-x\rangle \leq 2C||u-x||^\gamma.$$
\end{proposition}

We finally recall basic facts from convex analysis about the Hopf-Lax-Yosida semigroup (also called Moreau envelopes). 
\begin{definition}
Given a convex lower semicontinuous function $g:H\to \R$ on a Hilbert space $H$, the Hopf-Lax-Yosida semigroup is the family $g_\lambda:H\to \R,\lambda >0$ defined by:
$$g_\lambda(x)=\inf_{y\in H}\frac{1}{2\lambda}||x-y||^2+g(y).$$
\end{definition}

The key result is the following:
\begin{proposition}\label{Yosida}
For any continuous convex function $g:H\to \R$, $g_\lambda\in C^{1,1}(H)$ and is convex  and for any $x\in H$, $g_\lambda(x)\to_{\lambda \to 0} g(x)$. In fact, $A_\lambda=\nabla g_\lambda$ is Lipschitz with constant $\frac{1}{\lambda}$, $||A_\lambda(x)||$ increases to $||A^0(x)||$ where $A^0(x)$ is the unique element of $\partial g(x)$ of minimal norm and we have the inequalities:
$$|g_\lambda(x)-g_\lambda(y)-\langle A_\lambda, x-y\rangle|\leq \frac{1}{\lambda}||x-y||^2,$$
$$||A_\lambda(x)-A^0(x)||^2\leq ||A^0(x)||^2-||A_\lambda(x)||^2.$$
Finally, for any $x\in H$ there is a unique solution $J_\lambda(x)$  such that $x-J_\lambda(x)\in\lambda \partial g(J_\lambda(x))$ and $J_\lambda:H\to H$ is a contraction such that $A_\lambda(x)=\frac{x-J_\lambda(x)}{\lambda}\in \partial g(J_\lambda(x))$ and $J_\lambda$ reaches the infimum defining $g_\lambda$.
\end{proposition}
The proof is contained in \cite[Propositions 2.6,2.11]{BrezisMonotone}.
Note also that from the characterization of the minimum defining $g_\lambda$, $y=J_\lambda(0)$ is such that $\frac{1}{2\lambda} ||y||_2^2+g(y)\leq g(0)$ so that 
\begin{equation}\label{boundJlambda}
||J_\lambda(0)||_2^2\leq 2\lambda (g(0)-g(J_\lambda(0))).
\end{equation}
We also prove the following regularity lemma in terms of parameters. It will be used to solve free SDEs with gradient drift coming from a convex potential of weak regularity via Yosida approximation in section 4.2.
\begin{lemma}\label{RegularityYosida}
If $g^t:H\to \R, t\in [a,b] |b-a|\leq 1$ is a family of convex $C^1$ maps uniformly bounded below by $c\leq 0$ satisfying for some $ \alpha,\beta\in ]0,1],C,D>0$ and all $x,y\in H, t,s\in [a,b]$:
$$||\nabla g^t(x)-\nabla g^s(x)||\leq |t-s|^\alpha (D||x||+C),$$
$$||\nabla g^t(x)-\nabla g^t(y)||\leq ||x-y||^\beta (D||x||+D||y||+C),$$
then we have for all $\lambda\leq 1$:
$$||\nabla g_\lambda^t(x)-\nabla g_\lambda^s(x)||\leq 2 |t-s|^{\alpha\beta}[(2D||x||_2+2D\sqrt{2 (\sup_{s\in[a,b]}g^s(0)+|c|)}+C+1)^{1+\beta}],$$
$$||\nabla g_\lambda^t(x)-\nabla g_\lambda^t(y)||\leq ||x-y||^\beta (D||x||+D||y||+2D \sqrt{2 (\sup_{s\in[a,b]}g^s(0)+|c|)}+C).$$
\end{lemma}
\begin{proof}
From the previous proposition, if we call $J_{t,\lambda}=(1+\lambda \partial g_t)^{-1}$, we have since $g_t$ is $C^1$: $$\nabla g_\lambda^t(x)=\nabla g^t(J_{t,\lambda}(x))$$
Note that $$(1+\lambda \partial g_t)(\partial g_t)[J_{t,\lambda}(x)]=(\partial g_t)(1+\lambda \partial g_t)[J_{t,\lambda}(x)]=(\partial g_t)(x)$$ and therefore by uniqueness of the equation characterizing $J_{t,\lambda}$ (which comes from the uniqueness of the minimizer defining $g_\lambda^t$) we have: $(\partial g_t)[J_{t,\lambda}(x)]= J_{t,\lambda}(\partial g_t(x)).$ We can write a resolvent like equation and deduce bounds from contractivity of $J_{t,\lambda}$, estimate \eqref{boundJlambda} and the assumption for $\lambda\leq 1$:
\begin{align*}\|J_{t,\lambda}(x)-J_{s,\lambda}(x)\|_2&=\|J_{t,\lambda}[(1+\lambda \partial g_s)(J_{s,\lambda}(x))]-J_{t,\lambda}[(1+\lambda \partial g_t)(J_{s,\lambda}(x))]\|_2
\\&\leq \|\lambda (\partial g^ss- \partial g^t)(J_{s,\lambda}(x))\|_2
\\&\leq \lambda |t-s|^\alpha(D\|(J_{s,\lambda}(x))\|_2+C)
\\&\leq \lambda |t-s|^\alpha(D\|x\|_2+D\sqrt{2 (\sup_{s\in[a,b]}g^s(0)+|c|)}+C).
\end{align*}

Combining this and the defining equation, one gets the expected result:\begin{align*}&\|\nabla g_\lambda^t(x)-\nabla g_\lambda^s(x)\|_2\leq \|\nabla g^t(J_{t,\lambda}(x))-\nabla g^s(J_{t,\lambda}(x))\|_2+
\|\nabla g^t(J_{t,\lambda}(x))-\nabla g^t(J_{s,\lambda}(x))\|_2
\\&\leq \lambda |t-s|^\alpha(D\|x\|_2+D\sqrt{2 (\sup_{s\in[a,b]}g^s(0)+|c|)}+C)+||J_{s,\lambda}(x)-J_{t,\lambda}(x)||^\beta (D||J_{t,\lambda}(x)||+D||J_{s,\lambda}(x)||+C)
\\&\leq \lambda |t-s|^\alpha(D\|x\|_2+D\sqrt{2 (\sup_{s\in[a,b]}g^s(0)+|c|)}+C)\\&+\lambda^\beta|t-s|^{\alpha\beta}(D\|x\|_2+D \sqrt{2 (\sup_{s\in[a,b]}g^s(0)+|c|)}+C)^\beta (2D||x||_2+2D\sqrt{2 (\sup_{s\in[a,b]}g^s(0)+|c|)}+C)
\\&\leq 2\max(\lambda,\lambda^\beta) |t-s|^{\alpha\beta}[(2D||x||_2+2D\sqrt{2 (\sup_{s\in[a,b]}g^s(0)+|c|)}+C+1)^{1+\beta}].
\end{align*}
%Note that for any $x,y\in H$, one can bound in using the fundamental theorem of calculus and our assumpton: \begin{align*}
%|g^t(x)-g^t(y)-g^s(x)+g^s(y)|&=\left|\int_0^1du\langle \nabla g^t(x+u(y-x))-\nabla g^s(x+u(y-x)),y-x\rangle\right|\\&\leq |t-s|^\alpha(D||x||+D||y||+C)||y-x||.\end{align*}
%Then taking $X=J_{s,\lambda}(x),Y=J_{t,\lambda}(y)$ (it is known they are the map defined in the previous definition, \cite[lemme 2.1]{Brezis}) reaching the infimum defining $g_\lambda^s(x),g_\lambda^t(y)$ one gets :
%\begin{align*}
%&g_\lambda^s(y)+g_\lambda^t(x)\leq \frac{1}{2\lambda}||x-X||^2+\frac{1}{2\lambda}||y-Y||^2+g^s(Y)+g^t(X)\leq
%\frac{1}{2\lambda}||x-X||^2+\frac{1}{2\lambda}||y-Y||^2\\&+g^t(Y)+g^s(X)+ |t-s|^\alpha(D||X||+D||Y||+C)||Y-X||
%\leq ,\end{align*}
%since $||J_{s,\lambda}(x)||\leq ||x||$ and $||J_{t,\lambda}(y)-J_{s,\lambda}(x)||$
where the last expected conclusion is for $\lambda\leq 1.$
Finally note that the H\"older continuity in space is obvious:
\begin{align*}||\nabla g_\lambda^t(x)-\nabla g_\lambda^t(y)||&=||\nabla g^t(J_{t,\lambda}(x))-\nabla g^t(J_{t,\lambda}(y))||\\&\leq ||x-y||^\beta (D||J_{t,\lambda}(x)||+D||J_{t,\lambda}(y)||+C)\\&\leq ||x-y||^\beta (D||x||+D||y||+2D \sqrt{2 (\sup_{s\in[a,b]}g^s(0)+|c|)}+C).\end{align*}

\end{proof}
\subsection{Classical Entropy}
 Recall that the entropy of a probability measure $\mu$ on $\R^p$ is the
      quantity
      $$\text{Ent}(\mu)=\left\lbrace\begin{array}{l}-\int_{\R^p} f(x)\log f(x)dx\quad
      \text{if}\ \mu(dx)=f(x)dx\\ \\
      -\infty \quad\text{if $\mu$ is not absolutely
      continuous}\end{array}\right .$$
      The entropy is a concave upper semi-continuous function of $\mu$.
      
      Moreover, there is also a well known notion of relative entropy of two probability measures, say on a locally compact space $\Omega$ (also called Kullback-Leibler divergence, cf. \cite{K}).
      $$\text{Ent}(\mu|\nu)=\left\lbrace\begin{array}{l}-\int_{\Omega} f(x)\log f(x)d\nu(x)\quad
      \text{if}\ \mu(dx)=f(x)d\nu(x)\\ \\
      -\infty \quad\text{if $\mu$ is not absolutely
      continuous with respect to $\nu$}\end{array}\right .$$
Note that, by Jensen inequality, $\text{Ent}(\mu|\nu)\leq 0$.

We shall need another characterization of entropy, through its Legendre
      transform. Indeed one has, for any  probability measure $\mu$ supported by a set $E$, of finite Lebesgue measure,
      $$\text{Ent}(\mu)=\inf_{\phi\in C_b(E)}
       \left(\log\left(\int_E\exp \phi(x)dx\right)-\int_E\phi(x)\mu(dx)\right).$$
Likewise (see e.g. \cite[section 6.2]{DZ} )    for any  probability measures $\mu,\nu$ supported on $E$,
      \begin{equation}\label{dualEnt}\text{Ent}(\mu|\nu)=\inf_{\phi\in C_b(E)}
       \left(\log\left(\int_E\exp \phi(x)d\nu(x)\right)-\int_E\phi(x)\mu(dx)\right).\end{equation} 
We will even need a stronger representation in case of gaussian measures on $\R^p$.
Let $C_{sq}(E)$ the space of continuous functions $\phi$ subquadratic in the sense that there is $C$ such that $|\phi(x)|\leq C||x||^2$.

If $\nu$ is a gaussian measure, $\int_E\exp(C||x||^2)d\nu(x)<\infty$ and $-M\vee\phi\wedge M\to \phi$ so that by dominated convergence $\int_E\exp (-M\vee\phi(x)\wedge M )d\nu(x)\to_{M\to \infty} \int_E\exp \phi(x)d\nu(x)$. Moreover, if $\mu$ has a second moment by dominated convergence theorem again $\int_E-M\vee\phi(x)\wedge M\mu(dx)\to_{M\to \infty} \int_E\phi(x)\mu(dx).$ 
Let $$\mathscr{P}^2(\R^p)=\{\mu\in \mathscr{P}(\R^p): \int_{\R^p}||x||_2^2\ d\mu(x)<\infty\}$$
Thus one deduces that for $\nu$ gaussian measure, $\mu\in \mathscr{P}^2(\R^p)$:
\begin{equation}\label{dualEntsq}\text{Ent}(\mu|\nu)=\inf_{\phi\in C_{sq}(E)}
       \left(\log\left(\int_E\exp \phi(x)d\nu(x)\right)-\int_E\phi(x)\mu(dx)\right).\end{equation} 

Finally, if $\mu$ is the restriction of $\nu$ to $E$, renormalized into a probability measure, then 
$$\text{Ent}(\mu|\nu)=\log(\nu(E)))$$
and again this is the maximum value of $\text{Ent}(.|\nu)$ on the set of
 all probability
      measures supported by $E$.

\subsection{Free entropy}
\label{FreeEntropyDef}

  Let $\tau\in\mathcal S_R^{m+\mu}\star \mathcal T(\mathcal F^\nu_1)$ (cf subsection \ref{Concentration}), let $\epsilon>0$ be a real number and
     $K,N$ be positive integers. Let $\Upsilon\in \mathcal{U}(M_N(\C))^\nu$. We denote by
     $\Gamma_{R,\Upsilon}(\tau,\epsilon,K,N)$ the set of $n+\mu$-tuples of hermitian matrices
     $M_1,\ldots,M_{m+\mu}\in H_N^R$ such that
     for all monomials $m(X_1,\ldots,X_{m+\mu},\upsilon_1,...,\upsilon_\nu)\in\C\langle X_1,\ldots,X_{m+\mu},\upsilon_1,...,\upsilon_\nu \rangle$ of degree less than $K$ one has
     $$|\tau\left(m(X_1,\ldots,
     X_{m+\mu},\upsilon)\right)-\frac{1}{N}Tr\left(m(M_1,\ldots,M_{m+\mu},\Upsilon)\right)|<
     \epsilon$$
     Equivalently $\Gamma_{R,\Upsilon}(\tau,\epsilon,K,N)$ is the set 
     of $n+\mu$-tuples of hermitian matrices
     $M_1,\ldots,M_{m+\mu}\in H_N^R$ whose associated state $\sigma_{M_1,\ldots, M_{m+\mu},\Upsilon}\in \mathcal S_R^{m+\mu}\star \mathcal T(\mathcal F^\nu_1)$, defined by $$\sigma_{M_1,\ldots, M_m,\Upsilon}(P)=\frac{1}{N}Tr\left(P(M_1,\ldots,M_{m+\mu},\Upsilon)\right),$$ 
     is in $V_{\epsilon,K}(\tau)$, defined in subsection 2.7. We will write similarly as $\sigma_{X,\upsilon}$ any mixed law of self-adjoint variables $X$ and unitaries $\upsilon$ from any tracial von Neumann algebra. 
Voiculescu defined free entropy in term of Lebesgue measure $Leb$ on hermitian matrices, but a related definition can be made in terms of a Gaussian measure $P$ law of $H_1^N$ for our hermitian Brownian motion $H_t^N.$ One can also define a third version associated to the unitary transformation we used following \cite{BCG} : $$u(X)=\frac{X+4i}{X-4i}.$$

Then the relevant set of unitaries are defined using the notation $U_{\epsilon,K}(\sigma)$ in subsection \ref{Concentration}, for $\sigma\in\mathcal{T}_R(\mathcal{F}^{m+\mu}_{1})\star \mathcal{T}(\mathcal{F}^\nu_{1})$ by:
      \begin{align*}&\Gamma_{R,\Upsilon}^U(\sigma,\epsilon,K,N)\\& =\{(U_1,...,U_m)\in U(N)^{m+\mu}: \tau_{U_1,...,U_{m+\mu},\Upsilon}\in U_{\epsilon,K}(\sigma), \frac{U_j+U_j^*}{2}\leq 1-\frac{2}{R^2+1}\}.
      \end{align*}
      Of course in this context we need to consider $\Psi(X_1,...,X_m)=(u(X_1),...,u(X_m))$ and the push-forward measure $\Psi_*P$ of our gaussian measure $P.$
      It is reasonable to give a name $\mathcal{T}_R(\mathcal{F}^{m+\mu}_{1})\star \mathcal{T}(\mathcal{F}^\nu_{1})$ to the set of states $\sigma$ such that in the GNS representation $\frac{U_j+U_j^*}{2}\leq 1-\frac{2}{R^2+1}$. Note this is a closed set in the topology given by either $d$ or $d_2$.
Voiculescu introduced $\limsup$ and $\liminf$      variants. Moreover \cite{S02} defined a notion of relative entropy, we define a variant where the extra variables are unitaries instead of self-adjoints (variant which is of course completely equivalent, thanks to the von Neumann algebra invariance in the relative variable, and only better suited with our framework using unitary variables). We call $p_{m,\mu}:M_N(\C)^{m+\mu}\to M_N(\C)^{m}$ the projection on the $m$ first coordinates.
      
    \begin{definition}\cite{V2,S02}\label{MfreeEnt} Let $(M,\tau)$ a finite von Neumann algebra, $R\in [0,\infty]$, $X_1,...,X_m,Y_1,...,Y_\mu\in (M,\tau)$ self-adjoints with $||X_i||,||Y_i||\leq R,$ $U_1,...,U_m,V_1,...,V_\mu, \upsilon_1,...,\upsilon_\nu\in \mathcal{U}(M)$. We also fix $\Upsilon_N\in\mathcal{U}(M_N(\C))^\nu$ a sequence approximating in law $\upsilon$. %$\tau\in \mathcal S_R^m, \sigma\in\mathcal{T}_R(\mathcal{F}^m_{1})$
    Define the various \textit{free entropies of $X=(X_1,...,X_m)$ in presence of $Y=(Y_1,...,Y_\mu)$ (resp. of $U$ in the presence of $V$) relative to $\upsilon=(\upsilon_1,...,\upsilon_\nu)$ (resp. $(\Upsilon_N)$) with bound $R$}:
        $$\chi_R(X:Y|\upsilon)=\lim_{K\to\infty,\epsilon\to 0}
     \limsup_{N\to\infty}\left(\frac{1}{N^2}\sup_{\tau_{\Upsilon}\in V_{\epsilon,K}(\tau_\upsilon)}\log\left(
     \text{Leb}(p_{m,\mu}\Gamma_{R,\Upsilon}(\sigma_{X,Y,\upsilon},\epsilon,K,N)\right))+\frac{m}{2}\log N\right)$$   $$\underline{\chi}_R(X:Y|\upsilon)=\lim_{K\to\infty,\epsilon\to 0}
     \liminf_{N\to\infty}\left(\frac{1}{N^2}\sup_{\tau_{\Upsilon}\in V_{\epsilon,K}(\tau_\upsilon)}\log\left(
     \text{Leb}(p_{m,\mu}\Gamma_{R,\Upsilon}(\sigma_{X,Y,\upsilon},\epsilon,K,N)\right))+\frac{m}{2}\log N\right)$$
     $$\chi_R(X:Y|(\Upsilon_N)_{N\in\N})=\lim_{K\to\infty,\epsilon\to 0}
     \limsup_{N\to\infty}\left(\frac{1}{N^2}\log\left(
     \text{Leb}(p_{m,\mu}\Gamma_{R,\Upsilon_N}(\sigma_{X,Y,\upsilon},\epsilon,K,N)\right))+\frac{m}{2}\log N\right)$$   $$\underline{\chi}_R(X:Y|(\Upsilon_N)_{N\in\N})=\lim_{K\to\infty,\epsilon\to 0}
     \liminf_{N\to\infty}\left(\frac{1}{N^2}\log\left(
     \text{Leb}(p_{m,\mu}\Gamma_{R,\Upsilon_N}(\sigma_{X,Y,\upsilon},\epsilon,K,N)\right))+\frac{m}{2}\log N\right)$$
     
     $$\chi_R^G(X:Y|\upsilon)=\lim_{K\to\infty,\epsilon\to 0}
     \limsup_{N\to\infty}\left(\frac{1}{N^2}\sup_{\tau_{\Upsilon}\in V_{\epsilon,K}(\tau_\upsilon)}\log\left(
     P(p_{m,\mu}\Gamma_{R,\Upsilon}(\sigma_{X,Y,\upsilon},\epsilon,K,N)\right))\right)$$   $$\underline{\chi}_R^G(X:Y|\upsilon)=\lim_{K\to\infty,\epsilon\to 0}
     \liminf_{N\to\infty}\left(\frac{1}{N^2}\sup_{\tau_{\Upsilon}\in V_{\epsilon,K}(\tau_\upsilon)}\log\left(
     P(p_{m,\mu}\Gamma_{R,\Upsilon}(\sigma_{X,Y,\upsilon},\epsilon,K,N)\right))\right)$$
     $$\chi_R^G(X:Y|(\Upsilon_N)_{N\in\N})=\lim_{K\to\infty,\epsilon\to 0}
     \limsup_{N\to\infty}\left(\frac{1}{N^2}\log\left(
     P(p_{m,\mu}\Gamma_{R,\Upsilon_N}(\sigma_{X,Y,\upsilon},\epsilon,K,N)\right))\right)$$   $$\underline{\chi}_R^G(X:Y|(\Upsilon_N)_{N\in\N})=\lim_{K\to\infty,\epsilon\to 0}
     \liminf_{N\to\infty}\left(\frac{1}{N^2}\log\left(
     P(p_{m,\mu}\Gamma_{R,\Upsilon_N}(\sigma_{X,Y,\upsilon},\epsilon,K,N)\right))\right)$$

     $$\widetilde{\chi}_R(U:V|\upsilon)=\lim_{K\to\infty,\epsilon\to 0}
     \limsup_{N\to\infty}\left(\frac{1}{N^2}\sup_{\tau_{\Upsilon}\in V_{\epsilon,K}(\tau_\upsilon)}\log\left(
     \Psi_*P(p_{m,\mu}\Gamma_{R,\Upsilon}^U(\tau_{U,V,\upsilon},\epsilon,K,N)\right))\right)$$   $$\widetilde{\underline{\chi}}_R(U:V|\upsilon)=\lim_{K\to\infty,\epsilon\to 0}
     \liminf_{N\to\infty}\left(\frac{1}{N^2}\sup_{\tau_{\Upsilon}\in V_{\epsilon,K}(\tau_\upsilon)}\log\left(
     \Psi_*P(p_{m,\mu}\Gamma_{R,\Upsilon}^U(\tau_{U,V,\upsilon},\epsilon,K,N)\right))\right)$$
     $$\widetilde{\chi}_R(U:V|(\Upsilon_N)_{N\in\N})=\lim_{K\to\infty,\epsilon\to 0}
     \limsup_{N\to\infty}\left(\frac{1}{N^2}\log\left(
     \Psi_*P(p_{m,\mu}\Gamma_{R,\Upsilon_N}^U(\tau_{U,V,\upsilon},\epsilon,K,N)\right))\right)$$   $$\widetilde{\underline{\chi}}_R(U:V|(\Upsilon_N)_{N\in\N})=\lim_{K\to\infty,\epsilon\to 0}
     \liminf_{N\to\infty}\left(\frac{1}{N^2}\log\left(
     \Psi_*P(p_{m,\mu}\Gamma_{R,\Upsilon_N}^U(\tau_{U,V,\upsilon},\epsilon,K,N)\right))\right)$$

The \textit{free entropy of $X=(X_1,...,X_m)$ in presence of $Y=(Y_1,...,Y_\mu)$ (resp. of $U$ in the presence of $V$
    ) relative to $\upsilon=(\upsilon_1,...,\upsilon_\nu)$ (resp. $ \Upsilon=(\Upsilon_N)_{N\in\N}$, resp. a subalgebra $B\subset M$)} is for $\chi\equiv\chi^L$ and $p\in\{L,G\}$:
$$\chi^p(X:Y|\upsilon)=\sup_{R>0}\chi_R^p(X:Y|\upsilon),%\ {\chi^G}(X:Y|\upsilon)=\sup_{R>0}\chi_R^G(X:Y|\upsilon),
\ \widetilde{\chi}(U:V|\upsilon)= \sup_{R>0}\widetilde{\chi}_R(U:V|\upsilon),$$
$$%\underline{\chi}(X:Y|\upsilon)=\sup_{R>0}\underline{\chi}_R(X:Y|\upsilon),\ 
{\underline{\chi}^p}(X:Y|\upsilon)=\sup_{R>0}\underline{\chi}_R^p(X:Y|\upsilon),\ \widetilde{\underline{\chi}}(U:V|\upsilon)= \sup_{R>0}\widetilde{\underline{\chi}}_R(U:V|\upsilon),$$

$$\chi^p(X:Y|\Upsilon)=\sup_{R>0}\chi_R^p(X:Y|(\Upsilon_N)_{N\in\N}),
\ \widetilde{\chi}(U:V|\Upsilon)= \sup_{R>0}\widetilde{\chi}_R(U:V|(\Upsilon_N)_{N\in\N}),$$
$${\underline{\chi}^p}(X:Y|\Upsilon)=\sup_{R>0}\underline{\chi}_R^p(X:Y|(\Upsilon_N)_{N\in\N}),\ \widetilde{\underline{\chi}}(U:V|\Upsilon)= \sup_{R>0}\widetilde{\underline{\chi}}_R(U:V|(\Upsilon_N)_{N\in\N}).$$

$$\chi^p(X:Y|B)=\inf_{\nu\in\N}\inf_{\upsilon_1,...,          \upsilon_\nu\in\mathcal{U}(B)}\chi^p(X:Y|\upsilon),
\ \widetilde{\chi}(U:V|\upsilon)= \inf_{\nu\in\N}\inf_{\upsilon_1,...,          \upsilon_\nu\in\mathcal{U}(B)}\widetilde{\chi}(U:V|\upsilon),$$
$${\underline{\chi}^p}(X:Y|B)=\inf_{\nu\in\N}\inf_{\upsilon_1,...,          \upsilon_\nu\in\mathcal{U}(B)}\underline{\chi}^p(X:Y|\upsilon),\ \widetilde{\underline{\chi}}(U:V|B)= \inf_{\nu\in\N}\inf_{\upsilon_1,...,          \upsilon_\nu\in\mathcal{U}(B)}\widetilde{\underline{\chi}}(U:V|\upsilon).$$

\end{definition}  

Of course we will write $\chi(X|.)$ if $\mu=0$ and $\chi(X:Y)$ if $\nu=0$ and similar variants.

    As is well-known and as was noticed e.g. in \cite[section 7]{BCG}, the  above entropies are related by a universal constant $C$ such that:
    \begin{equation}\label{chiGchi}
    \chi(X:Y|.)={\chi^G}(X:Y|.)+\frac{1}{2}\sum_{i=1}^m\tau(X_i^2)+mC.
    \end{equation}
    and their lemma 7.1 shows that %for any state $\tau\in \mathcal{S}_R^n$,
     if $\Psi(X_1,...,X_m)=(u(X_1),...,u(X_m))$% and $\Psi_*(\tau)(P)=\tau(P(\Psi(X)))$
    \begin{equation}\label{chiGchiU}\chi^G(X:Y|.)=\widetilde{\chi}(\Psi(X):\Psi(Y)|.).\end{equation}
    The analogue formulas for $\liminf$ variants and for $\chi_\infty$ variants are also true.

As in \cite[Theorem 2.15]{S02} ( a consequence of Kaplansky density theorem), \begin{equation}\chi(X:Y|u_1,...,u_\nu)= \chi(X:Y|W^*(u_1,...,u_\nu)),\end{equation} and this last version is the same as the variant defined using self-adjoint variables in this paper.

    The following result from \cite{BelinschiBercovici} will be crucial to apply large deviation principle to free entropy. The proof gives right away the result for $\chi,\chi^G$ and their $\liminf$ variants, and then \eqref{chiGchiU} deals with the remaining case.
    
    \begin{proposition}[Proposition 2.1 in \cite{BelinschiBercovici}]\label{chiinfty}
    We have in the setting of the previous definition, $p\in\{L,G\}$:
$$ % \chi(X:Y|.)=\chi_\infty(X:Y|.), \ \ 
\chi^p(X:Y|.)=\chi_\infty^p(X:Y|.), \ \   \widetilde{{\chi}}(\Psi(X):\Psi(Y)|.)=\widetilde{{\chi}}_\infty(\Psi(X):\Psi(Y)|.),$$
    $$  % \underline{\chi}(X:Y|.)=\underline{\chi}_\infty(X:Y|.), \ \   
    \underline{\chi}^p(X:Y|.)=\underline{\chi}_\infty^p(X:Y|.), \ \   \widetilde{\underline{\chi}}(\Psi(X):\Psi(Y)|.)=\widetilde{\underline{\chi}}_\infty(\Psi(X):\Psi(Y)|.)).$$
    \end{proposition}

    In order to recall the definition of Voiculescu's non-microstate free entropy, we recall first the definition of free Brownian motion. Here $Id_{\R^m}$ is the unit in $M_m(\R)$
    
\begin{definition}
Let $B_{s}$ be an increasing filtration of von Neumann algebras in a non-commutative tracial probability space $(M,\tau)$. $S_{s}=(S_{s}^{1},...,S_{s}^{m}), s\in \R_{+}$ an m-tuple of self-adjoint processes adapted to this filtration with $Z_{0}=0$ is a \textit{free brownian motion} adapted for $B_s$ if:
\begin{enumerate}
\item $(S_t-S_s)$ are free semi-circular variables of covariance $(t-s) Id_{\R^m}$.
\item $\{(S_u-S_s),u\geq s\}$ are free from $B_s$.
\end{enumerate}
\end{definition}    
  There is an important characterization of free brownian motion in the spirit of Paul L\'evy's characterization of ordinary brownian motion. It is due to \cite{BCG}.  
    \begin{theorem}[Theorem 6.2 in \cite{BCG}]\label{LevyBCG}
Let $B_{s}$ be an increasing filtration of von Neumann algebras in a non-commutative tracial probability space $(M,\tau)$
$Z_{s}=(Z_{s}^{1},...,Z_{s}^{m}), s\in \R_{+}$ an m-tuple of self-adjoint processes adapted to this filtration with $Z_{0}=0$ and~:
\begin{enumerate}
\item $\tau(Z_{t}|B_{s})=Z_{s}$
\item $\tau(|Z_{t}-Z_{s}|^{4})\leq K (t-s)^{2}$ for some constant $K>0$.
\item $\tau(Z_{t}^{k}AZ_{t}^{l}B)=\tau(Z_{s}^{k}AZ_{s}^{l}B)+(t-s)\tau(A)\tau(B)1_{\{k=l\}}+o(t-s)$ for any $A,B\in B_s$.
\end{enumerate}
Then $Z$ is a free brownian motion adapted to $B_s$. %i.e. for each $s$ $(Z_{t}^{l}-Z_{s}^{l})$ are free with $B_{s}$ and have a semicircular distribution of covariance $1_{k=l}(t-s)$.
\end{theorem}
 
 Let us recall the definition of free entropy relative to a subalgebra $B$ from \cite{V5}. Here $S_1,...,S_m$ are free semicircular variables free from $X_1,...,X_m,B$:
\begin{align*} \chi^{*}(X_{1},...,X_{m}:B)=\frac{1}{2}&\int_0^\infty\left(\frac{m}{1+t}-\Phi^*(X_1+\sqrt{t}S_1,...,X_m+\sqrt{t}S_m:B)\right)dt+\frac{m}{2}\log (2\pi e).
\end{align*}
 Let us also remind the Fisher information $\Phi^{*}(Y_1,...,Y_m:B)=\sum_{i=1}^m||\xi_i||_2^2$ where $\xi_i$ are the conjugate variables relative to $B$ (which exists for $Y=X+\sqrt{t}S$ as above, and are (when they exist) the unique $\xi_i\in L^2(W^*(B,Y_1,...,Y_m)$ such that for all $P\in B\langle X_1,...,X_m\rangle$, if $\partial_i$ is the free difference quotient, unique derivation with $\partial_i(b)=0,b\in B$ and $ \partial_iX_j=1\otimes 1_{i=j}\in L^2(W^*(B,Y_1,...,Y_m)\otimes W^*(B,Y_1,...,Y_m))$, then:
$$\langle 1\otimes 1,\partial_i P\rangle=\tau(\xi_iP).$$
The gaussian variant is defined using the gaussian variant of Fisher's information, $\Phi^{G*}(Y_1,...,Y_m:B)=\sum_{i=1}^m||\xi_i-Y_i||_2^2$ by:
 \begin{align*} \chi^{G*}(X_{1},...,X_{m}:B)=-\frac{1}{2}&\int_0^1\frac{dt}{t}\Phi^{G*}(\sqrt{t}X_1+\sqrt{1-t}S_1,...,\sqrt{t}X_m+\sqrt{1-t}S_m:B).
\end{align*}
It is easy to see in using linear changes of variables for the score function from \cite{V5} that $\chi^{G^*},\chi^*$  are related by \eqref{chiGchi}.
 
 \subsection{Ultraproducts and processes valued in them} \label{ultra}
 
In this final preliminary subsection, we recall backgrounds on ultraproducts. We refer to \cite{PisierBook} for more details in the tracial von Neumann algebra context and to \cite{FarahII,CapraroLupini} in the model theory context. Let $(M_n,\tau_n)$ a sequence of tracial von Neumann algebras. We will mainly use the case $M_n=M_{n}(L^\infty(\Omega_n, P_n))$ of matrix algebras over a classical probability space $(\Omega_n, P_n)$, with $\tau_n=E\circ \frac{1}{n}Tr_n$. Let $\omega\in \beta\N-\N$ a non-principal ultrafilter (or equivalently a non-integer  point in the Stone-Cech compactification $\beta\N$ of $\N$).

The ultraproduct  of this sequence is defined as the following  quotient of the set of bounded sequences with the $n$-th term of the sequence in $M_n$, noted $\ell^\infty(M_n,n\in \N)$:$$(M_n,\tau_n)^\omega=\ell^\infty(M_n,n\in \N)/\{(x_n): \lim_{n\to\omega}\tau_n(x_n^*x_n)=0\}.$$
 
It is known that  $(M_n,\tau_n)^\omega$ is a tracial von Neumann algebra with trace :
$$\tau_\omega((x_n)^\omega)=\lim_{n\to \omega}\tau_n(x_n).$$
There is a recent model theoretic proof of this fact \cite{FarahII}, but the classical proof (see e.g.\cite[section 9.10]{PisierBook}) gives an explicit action on a Hilbert space that we will need. Actually, in writing $(M_n,\tau_n)^\omega$ we consider always given the trace $\tau_\omega.$

Let $H_n=L^2(M_n,\tau_n)$ the Hilbert space of the GNS representation. For instance, if $M_n=M_{n}(L^\infty(\Omega_n, P_n))$, $L^2(M_n,\tau_n)=M_n(L^2(\Omega_n, P_n))$ with its canonical Hilbert space structure with scalar product $\langle u,v\rangle_2 =E(\frac{1}{n}Tr(u^*v)).$

Consider the Hilbert space ultraproduct :
$$(L^2(M_n,\tau_n))^\omega=\ell^\infty(L^2(M_n,\tau_n); n\in \N)/\{(h_n), \lim_{n\to \omega}||h_n||\to 0\}$$
Then $(M_n,\tau_n)^\omega\subset B((L^2(M_n,\tau_n))^\omega)$ with the action given by $$(x_n)^\omega(h_n^\omega)=(x_nh_n)^\omega.$$ Then the GNS construction $L^2((M_n,\tau_n)^\omega)$ is a subset of  $(L^2(M_n,\tau_n))^\omega$ that can be described as follows (see e.g. \cite[Rmk 9.10.2]{PisierBook}). This is either the closure of $(M_n,\tau_n)^\omega$ (included by its action on $(1)^\omega$). Alternatively, a sequence $(h_n)^\omega\in (L^2(M_n,\tau_n))^\omega$ belongs to $L^2((M_n,\tau_n)^\omega)$ if and only if the following uniform integrability like condition holds :
$$\lim_{c\to\infty} \lim_{n\to \omega} \tau_n(h_n^*h_n1_{h_n^*h_n\geq c})=0.$$
Here $1_{h\geq c}$ denotes the spectral projection of the positive operator $h\in L^1(M_n,\tau_n)$.

In the second paper of this series, we will need more results from continuous model theory for processes, but we will be content here to only discuss some terminology and its behaviour under ultraproduct.

All our filtrations will be filtrations on $[0,1]$ of tracial $W^*$-probability spaces (i.e. von Neumann algebras with a given trace) with traces compatible for inclusion, and therefore trace preserving conditional expectations. A \textit{standard martingale} in a filtration will be a martingale, $1/2$-H\"older continuous in $L^4$ norm such that the process is uniformly bounded by a fixed constant $C$ (which is fixed throughout the paper) in operator norm. Theorem \ref{LevyBCG} explains why free brownian motion is a standard martingale and what is needed for a standard martingale to be a free brownian motion. For a \textit{subfiltration} $\mathcal{G}\subset \mathcal{F}$, we require  not only  $\mathcal{G}_s\subset \mathcal{F}_s$ for all times but also for any $x\in \mathcal{G}_1$ any $s$, $E_{\mathcal{F}_s}(x)=E_{\mathcal{G}_s}(x)$ so that martingales coincide for both filtrations. Said otherwise, conditional expectation are part of the data of the filtrations, and a data preserving map should preserve them. For instance, $\mathcal{F}_s\subset \mathcal{F}_{s+t}$ is not  subfiltration for us in general.

All those notions are "universally axiomatizable" in first order continuous logic (cf. \cite{FarahII} and the second paper of this series) and therefore stable by ultraproduct and subfiltrations (which coincide with submodels) as can be checked easily directly. We will sometimes use right continuous filtrations, a notion which need not be stable by ultraproduct.

\begin{example}\label{ExFiltration}
\begin{enumerate}
\item $m$ hermitian brownian motion $H_t^N,t\in \Q\cap[0,1]$ in the canonical filtration of $M_N(L^\infty(\Omega,P))$ is not a standard martingale  because of the lack of boundedness.
\item The ultraproduct of previous filtrations $(M_N(L^\infty(\Omega,P)),\tau_n)^\omega$ is a filtration and $S_t=(H_t^N)^\omega$ is a standard martingale for $C\geq 3$. Indeed, from proposition \ref{ConcentrationNorm}, take $C$ satisfying \eqref{IntegralUnifBound}  so that if $ A_N=\{||H_t^N||_{\infty}\leq C\}$, we have: $$\limsup_{N\to \infty}||H_t^N1_{A_N}-H_t^N||_2^2=0,$$ then $S_t=(H_t^N 1_{A_N})^\omega\in \mathcal{L}_P^\omega$ and since by construction $||W_t^N 1_{A_N}||_{\infty}\leq C$ one deduces that $S_t\in \mathcal{M}_P^\omega$, as expected. %Indeed, in the ultraproduct of $L^2$ spaces $(H_t^N)^\omega=(H_t^N1_{||H_t^N||\leq 3})^\omega$ for $t\in [0,1]$ by Proposition \ref{ConcentrationNorm}.
From the law of $S_t$ as semicircular variable, one deduces $C\geq 3$ is enough. However, it does not give a free brownian motion. We will call $\mathcal{M}_P^\omega$ this filtration with process $S$ depending on the non-principal ultrafilter $\omega$.
\item  If we define $\mathcal{F}_s$ generated by $S_t$, $t\leq s$ (or $\mathcal{F}_s^\omega$ generated by $\mathcal{F}_s$ and $\mathcal{M}_{S,0}^\omega$, $S$ is a free brownian motion in it by standard freeness results. Note that these are subfiltrations in the sense above inside the ultraproduct filtration. 
This is an application of Clark-Ocone's formula to check e.g. that any element of $L^2(\mathcal{F}_1)\ominus L^2(\mathcal{F}_s)$, being a stochastic integral is orthogonal to the ultraproduct filtration at time $s$.  More precisely we will need to note that $S_t$ and its stochastic integrals are still martingales  adapted to $\mathcal{L}_{P,s}^\omega$, namely:
\begin{equation}\label{orthofbm}\forall U\in \mathcal{L}_{P,s}^\omega, \forall V\in L^2(\mathcal{F}_1)\ominus L^2(\mathcal{F}_s)L,\qquad \langle U,V\rangle=0.
 \end{equation}
 From Clarck-Ocone's formula (a slight extension of the one \cite{BianeSpeicher} with extra initial conditions),  $V\in L^2(\mathcal{F}_1)\ominus L^2(\mathcal{F}_s)$ is a stochastic integral and thus can be approximated by a sum of terms of the form $P\#(S_t-S_T)$, $t\geq T\geq s$ $P\in \C\langle S_v,v\leq T,\upsilon\rangle \otimes_{alg}  \C\langle S_v,v\leq T,\upsilon\rangle$ with $\upsilon=(\Upsilon_N)^\omega$ a finite sequence of say unitaries in $\mathcal{F}_0$. But, for such a polynomial (in abstract variables), we have (in inserting $1_{A_N}$ thanks to the case $K=2$ in \eqref{IntegralUnifBound} to use the definition of product in von Neumann algebra ultraproduct):
 $$(P(W_s^N, s\leq T,\Upsilon_N)\#(W_t^N-W_T^N))^\omega=P\#(S_t-S_T)$$ and thus the stated orthogonality is obvious from the martingale property of matrix stochastic integrals of hermitian brownian motion.
%\item %\label{ExFiltration4}
%This is an application of Clark-Ocone's formula to check e.g. that any element of $L^2(\mathcal{F}_1)\ominus L^2(\mathcal{F}_s)$, being a stochastic integral is orthogonal to the ultraproduct filtration at time $s$ and thus the conditional expectation computed in the small filtration agrees with the one computed in the ultraproduct filtration.
\end{enumerate}
\end{example}

Even though this is not a result about ultraproducts, we write here a substitute to Ito formula for our standard martingales. Recall that for a filtration $\mathcal{F}$, $L^2_{ad}([0,1],L^2(\mathcal{F}))$ is the set of (Bochner measurable) square integrable processes $V$ such that for Lebesgue almost all $s\in[0,1]$, $V_s\in L^2(\mathcal{F}_s)$. This result is not optimal and probably well-known to experts of (non-commutative) martingales:

\begin{proposition}\label{Ito}
Let $\mathcal{F}$ be a filtration as above,  $(L_s)_{s\in[0,1]}$ a  martingale with $L_1\in L^2(\mathcal{F}_1)$, $V\in L^2_{ad}([0,1],L^2(\mathcal{F})), Y_0\in L^2(\mathcal{F}_0)$ and $Y_s=Y_0+\int_0^sV_udu+L_u-L_0$, then :
$$||Y_t||_2^2=||Y_0||_2^2+\int_0^t2\Re\langle Y_u,V_u\rangle du+||L_t-L_0||_2^2.$$
\end{proposition}

\begin{proof} Note that $L_u$ is bounded in $L^2$, $||L_u||_2^2$ is non-decreasing, hence has at most countably many discontinuity points, and the orthogonality $||L_u||_2^2=||L_v||_2^2+||L_u-L_v||_2^2, v<u$ implies any continuity point of $||L_u||_2^2$ is a continuity point of $L$. Hence the values at rational points and the countably many values at discontinuity points approximate any value, hence the image of $L$ or or at fortiori $Y$ is almost surely separably valued. The continuity condition also implies weak measurability hence Bochner measurability with the previous result (by Pettis Theorem). Thus  $Y\in L^2_{ad}([0,1],L^2(\mathcal{F}))$ so that the integral of the formula makes sense.

Write $Y_t=Y_0+\sum_{k=1}^n(Y_{tk/n}-Y_{t(k-1)/n})$ so that:
$$||Y_t||_2^2 =||Y_0||_2^2 + \sum_{k=1}^n||(Y_{tk/n}-Y_{t(k-1)/n})||_2^2+ 2 \sum_{k=1}^n\Re \langle (Y_{tk/n}-Y_{t(k-1)/n}), Y_{t(k-1)/n}\rangle.$$
Since $Y_t$ is adapted one can use the martingale property to note that $$\langle (Y_{tk/n}-Y_{t(k-1)/n}), Y_{t(k-1)/n}\rangle=\int_{t(k-1)/n}^{tk/n}\langle V_u,Y_{t(k-1)/n}\rangle du=\int_{t(k-1)/n}^{tk/n}\langle V_u,Y_{u}\rangle+\int_{t(k-1)/n}^{tk/n}\langle V_u,Y_{t(k-1)/n}-Y_u\rangle.$$
But we can bound using Cauchy-Schwarz inequality and the martingale property: $$\sum_{k=1}^n\left|\int_{t(k-1)/n}^{tk/n}\langle V_u,Y_{t(k-1)/n}-Y_u\rangle\right|\leq \frac{\sqrt{t}}{\sqrt{n}}\sqrt{\sum_{k=1}^n\int_{t(k-1)/n}^{tk/n}||V_u||_2^2du}\sqrt{\sum_{k=1}^n||Y_{tk/n}-Y_{t(k-1)/n}||_2^2}.$$
It remains to compute the quadratic variation term using the orthogonality of martingale increments:
\begin{align*}&\left|\sum_{k=1}^n||(Y_{tk/n}-Y_{t(k-1)/n})||_2^2-||L_t-L_0||_2^2\right|\\&
=\left|%\sum_{k=1}^n||(L_{tk/n}-L_{t(k-1)/n})||_2^2+
\sum_{k=1}^n\Big\|\int_{t(k-1)/n}^{tk/n}V_udu\Big\|_2^2+2\sum_{k=1}^n\Re\langle \int_{t(k-1)/n}^{tk/n}V_udu, (L_{tk/n}-L_{t(k-1)/n})\rangle%-||L_t-L_0||_2^2
\right|
\\&\leq \frac{t}{n}\int_{0}^{1}||V_u||_2^2du+2\frac{\sqrt{t}}{\sqrt{n}}\sqrt{\int_{0}^{1}||V_u||_2^2du}||L_t-L_0||_2\to_{n\to \infty} 0.
\end{align*}
Combining all our estimates, our first equation converges exactly to the expected equation.
\end{proof}

\section{Application of the Bou\'e-Dupuis-\"Ust\"unel Formula}   
  
In \cite{BD}, Bou\'e and Dupuis proved a formula for exponential functionals of brownian motion. They deduced from it large deviation results, and we will use an improvement with exactly the same goal.  Recall that a process on the Wiener space $\Omega=\mathbb W$
 is said progressively measurable if its restriction to $[0,t]\times \Omega$ is (jointly) measurable with respect to the canonical brownian filtration $\mathcal{F}_t$ (tensor Borel sets on $[0,t]$). We refer to \cite{BD} for the following result (and also \cite{Lehec} for an enlightening explanation).
  
 \begin{theorem}
\label{boue}
For every function $f\colon \mathbb W \to \mathbb R$ measurable and bounded 
from above, we have
\[
 -\log \left( \int_{\mathbb W} e^{-f} \ \mathrm{d} \gamma \right) = \inf_{U\in L^2_a(\gamma,\mathbb H)} \left[
 \mathbf{E}_\gamma  \left( f ( B + U) + 
  \frac{1}{2} \|U\|_{\mathbb H}^2  \right) \right] , 
   \]
where the supremum is taken over $L^2_a(\gamma,\mathbb H)$ of  all progressively measurable processes $U$  which belongs to $\mathbb H$ almost surely.
\end{theorem}
  
The boundedness assumption will be annoying for our purposes since the typical convex functions we considered on matrix hermition brownian motion 
$\mathcal{E}^{1,1}_{app}( \mathcal{T}_{2,0}(\mathcal{F}^m_{1}*\mathcal{F}^\nu_{\mu}),d_{2,0})$% $\mathcal{E}_{reg,p}(   \mathcal{T}_2^c(\mathcal{F}^m_{[0,1]}),d_2)$
 are only subquadratic. Fortunately, \"Ust\"unel extended recently this formula to a wider class of functionals \cite{Ustunel}. He studied the class of functionals satisfying this theorem under the name ``tame functionals''.
 
 \begin{definition}
A measurable  map $f: W\to \R\cup\{\infty\}$, with the property
$E_\gamma[(1+|f|)e^{-f}]<\infty$,  is called a {\bf{tamed functional}} if 
\[
-\log \left( \int_{\mathbb W} e^{-f} \ \mathrm{d} \gamma \right)=\inf_{U\in L^2_a(\gamma,\mathbb H)}\left[\mathbf{E}_\gamma  \left(
  [f(B+U)+\frac{1}{2}\|U\|_{\mathbb H}^2\right)\right]
\]
\end{definition}
  
The result we need is:
  \begin{theorem}[Theorem 7 in \cite{Ustunel}]
\label{ustunel}
Every measurable function $f\colon \mathbb W \to \mathbb R$ such that $f\in L^p(\gamma)$ and $e^{-f}\in L^q(\gamma)$ with $\frac{1}{p}+\frac{1}{q}=1$ is a tame functional.\end{theorem}
  This result can straightforwardly be applied to hermitian brownian motions and functionals $W\mapsto f(\tau_{W,\Upsilon_N})$ with $f\in\mathcal{E}%^{1,1}_{app}
  ( \mathcal{T}_{2,0}(\mathcal{F}^m_{1}*\mathcal{F}^\nu_{\mu}),d_{2,0}) %\mathcal{E}_{m}(   \mathcal{T}_2^c(\mathcal{F}^m_{[0,1]}),d_2)
  .$ We will develop in the next section a more explicit formulation in this case, better suited for ultraproducts techniques. 
  
We will also need the following consequence when we apply it to a brownian motion on $[t,1]$  for $t\in ]0,1[$. Let $\mathbb W_{[0,t]},  \mathbb W_{[t,1]}$ the spaces of continuous functions starting at zero and $\gamma_{[0,t]},\gamma_{[t,1]}$ the standard Wiener measure on them (for the second, we write $(B-B_t)$ the brownian variable for consistency). For $\omega\in \mathbb W_{[0,t]}, \nu\in  \mathbb W_{[t,1]}$, there is a process $\omega+\nu\in \mathbb W$ such that $(\omega+\nu)_s=\omega_s$ for $s\leq t$ and $(\omega+\nu)_s=\omega_t+\nu_s$ for $s\geq t$. Of course we also use the ordinary sum with same notation (for instance in convex combinations) and the reader will understand the meaning depending on the context. For $f:\mathbb W\to \R$ measurable, $\nu\mapsto f(\omega+\nu)$ defines a measurable function.
 Recall that we defined $\alpha$-convex functions at the end of  subsection 2.3.
\begin{corollary}\label{UstunelConditional}
Fix $t\in ]0,1[$ and a measurable function $f\colon \mathbb W \to \mathbb R$ such that $f\in L^p(\gamma)$ and $e^{-f}\in L^q(\gamma)$ with $\frac{1}{p}+\frac{1}{q}=1$, then for $\gamma_{[0,t]}$ almost all $\omega\in   \mathbb W_{[0,t]}$ we have the equality :
\begin{align} \label{UstunelCond} \begin{split}\lambda_t(f)(\omega)&:=-\log \left( \int_{\mathbb W_{[t,1]}} e^{-f(\omega+\nu)} \ \mathrm{d} \gamma_{[t,1]}(\nu) \right)\\&=\inf_{U\in L^2_a(\gamma_{[t,1]},\mathbb H)}\left[\mathbf{E}_{\gamma_{[t,1]}}  \left(
  [f(\omega+(B-B_t)+U)+\frac{1}{2}\|U\|_{\mathbb H}^2\right)\right].\end{split}\end{align}
  Moreover the condition holds for every $\omega \in \mathbb W_{[0,t]}$ if $e^{-f(\omega+.)}\in L^q(\gamma_{[t,1]})$ and $f(\omega+.)\in L^p(\gamma_{[t,1]})$ for every $\omega\in \mathbb W_{[0,t]}.$ 
As a consequence, in this case, if  $f$ is moreover convex or $\alpha$-convex for $\alpha\leq 1$, so is $\lambda_t(f).$
\end{corollary}
\begin{proof}
By Fubuni theorem, since $\mathbf{E}_\gamma(e^{-qf})=\int d\gamma_{[0,t]}(\omega)\int d\gamma_{[t,1]}(\nu)e^{-qf(\omega+\nu)}$, we have for $\gamma_{[0,t]}$ almost all $\omega$  , $e^{-f(\omega+.)}\in L^q(\gamma_{[t,1]})$ and similarly $f(\omega+.)\in L^p(\gamma_{[t,1]})$; The formula is thus deduced from \"Ust\"unel's theorem. For the convexity result, only take $U_1,U_2\in L^2_a(\gamma_{[t,1]},\mathbb H)$, and note that by convexity of $f$ and $|.|_H^2$, one gets for $\lambda\in [0,1]$: \begin{align*}\mathbf{E}_{\gamma_{[t,1]}}  &\left(
  [f(\lambda\omega_1+(1-\lambda)\omega_2+(B-B_t)+\lambda U_1+(1-\lambda)U_2)+\frac{1}{2}\|\lambda U_1+(1-\lambda)U_2)\|_{\mathbb H}^2\right)\\&\leq \lambda\left[\mathbf{E}_{\gamma_{[t,1]}}  \left(
  [f(\omega_1+(B-B_t)+U_1)+\frac{1}{2}\|U_1\|_{\mathbb H}^2\right)\right]\\&+(1-\lambda)\left[\mathbf{E}_{\gamma_{[t,1]}}  \left(
  [f(\omega_2+(B-B_t)+U_2)+\frac{1}{2}\|U_2\|_{\mathbb H}^2\right)\right].\end{align*}
  This concludes to the convexity once taken various infima.
For the $\alpha$-convexity result take $\omega_1=\omega+V_1,  \omega_2=\omega+V_2$ with $V_i\in\mathbb H_{[0,t]}$ and write $W_i=V_i+U_i$ for the $\mathbb H_{[0,1]}$ valued process with the sum with successive time introduced before the lemma and obtain the similar bound (based on $\alpha\leq 1$), $\alpha$-convexity of $f$,$ \|W_i\|_{\mathbb H}^2=\|V_i\|_{\mathbb H}^2+\|U_i\|_{\mathbb H}^2$ and convexity of $\frac{1}{2}\|.\|_{\mathbb H}^2$:
  \begin{align*}&\frac{\alpha}{2}\|\lambda V_1+(1-\lambda)V_2\|_{\mathbb H}^2\\&+\mathbf{E}_{\gamma_{[t,1]}}  \left(
  [f(\lambda\omega_1+(1-\lambda)\omega_2+(B-B_t)+\lambda U_1+(1-\lambda)U_2)+\frac{1}{2}\|\lambda U_1+(1-\lambda)U_2)\|_{\mathbb H}^2\right)
  \\&=\mathbf{E}_{\gamma_{[t,1]}}  \left(
  [f(B+\lambda W_1+(1-\lambda)W_2)+\frac{\alpha}{2}\|\lambda  W_1+(1-\lambda)W_2\|_{\mathbb H}^2+\frac{1-\alpha}{2}\|\lambda U_1+(1-\lambda)U_2)\|_{\mathbb H}^2\right)
  \\&\leq \lambda\left[\frac{\alpha}{2}\|V_1\|_{\mathbb H}^2+\mathbf{E}_{\gamma_{[t,1]}}  \left(
  [f(\omega_1+(B-B_t)+U_1)+\frac{1}{2}\|U_1\|_{\mathbb H}^2\right)\right]\\&+(1-\lambda)\left[\frac{\alpha}{2}\|V_2\|_{\mathbb H}^2+\mathbf{E}_{\gamma_{[t,1]}}  \left(
  [f(\omega_2+(B-B_t)+U_2)+\frac{1}{2}\|U_2\|_{\mathbb H}^2\right)\right].\end{align*}
   This concludes to the $\alpha$-convexity once taken various infima.
\end{proof}
  
The use of $\lambda_s$ is standard in optimal control (cf. e.g. \cite{FS}) to solve the minimization problem in Bou\'e-Dupuis-\"Ust\"unel formula. It has a (non-linear) semigroup property in the form $\lambda_s(\lambda_{s+t}(f))=\lambda_s(f),$ for $s,t>0$, $s+t<1$. We gather this and several basic properties in our next result.

We first fix some preliminary notation. Let $g\in C^0(\R^{dk},\R)$ a continuous function and $t_1<...<t_k\in [0,1]$. Let $J_{t_1,...,t_k}:\mathbb W\to \R^{dk}$ such that $J_{t_1,...,t_k}(\omega)=(\omega_{t_1},...,\omega_{t_k}).$ We let  $\mathcal{E}(\R^{dk})$  the set of convex  continuous functions $g$, bounded from below  and such that \begin{equation}\label{SubquadE} g(x_1,...,x_k)\leq c\left(d+||x||^{2}\right),\end{equation} 

(as usual we wrote $||x||^2=\sum_{i=1}^k\sum_{j=1}^d(|x_i|^{(j)})^2$ the euclidean norm) and with the following local lisphitzness condition for the euclidean norm on $\R^{dk}$, with some $C\geq 1,D\geq 0$ such that for all $x,y$: %We call $\mathcal{E}_l(\R^{dk})\subset \mathcal{E}(\R^{dk})$ the set of functions which are also lipschitz ...
\begin{equation}\label{lipE}g(x)\leq g(y)+ (C ||y||^{1}+C||x||^{1}+D\sqrt{d})||x-y||\end{equation}

For $\alpha\in [1,2]$, we call $\mathcal{E}_\alpha(\R^{dk})$  the subset of $\mathcal{E}(\R^{dk})$ such that there exists $C_\alpha, D_\alpha>0$ with for all $x,y\in \R^{dk}, t\in [0,1]$:% with $y=(y_1,...,y_n) y_i\in\R^d||y_i||\leq 1:$  
\begin{equation}\label{CalphaE}tg(x+(1-t)y)+(1-t)g(x-ty)-g(x)\leq  d^{1-\alpha/2}(C_\alpha +D_\alpha \frac{||x||+||x+(1-t)y||+||x-ty||}{\sqrt{d}})t(1-t)||y||^{\alpha}.\end{equation}
Note that in the  convex case, the left hand side is positive. The fact that one can only obtain one sided bounds for second order difference quotients of value functions is standard in optimal control (see e.g. \cite[section IV.9]{FS}). By %\cite{Nikolskii} or
 proposition \ref{C1alphaBall}, this class of functions are G\^ateaux-differentiable if $\alpha>1$ so that we can write $D^y_Hg(x_1,...,x_{n-1},y)=\sum_{i=1}^d \frac{\partial}{\partial y^{(j)}}g(x_1,...,x_{n-1},y)H_j$ for $H\in \R^d.$
Note also we made appear the dimension $d$ explicitly  since we will use later families of models where the constants involved will be dimension independent once used the conventions above. %Similarly we call $\mathcal{C}^2_b(\R^{dk})$ the space of $C^2$ functions on $\R^{dk}$, bounded from below, satisfying \eqref{SubquadE}, \eqref{lipE} and such that there exists $C_2,D_2>0$ with for all $x,y\in \R^{dk}$:\begin{equation}\label{C2E}|g(x+y)+g(x-y)-2g(x)|\leq  (C_2 +D_2 \frac{||x||+||x+y||+||x-y||}{\sqrt{d}})||y||^{2}.\end{equation}

%We call for $t\leq 1$, $\mathfrak{c}_t:\mathbb W\to \R$ the function defined by $$\mathfrak{c}_t(\omega)=\int_0^t\|\omega_s\|^2ds, \ \ \ \mathfrak{c}_{u,t,x}(\omega)=\int_u^t\|x+\omega_s\|^2ds, u<t.$$

\begin{proposition}\label{OptimalValueFct}Let $t_0=0<t_1<...<t_k\in [0,1]$, $g\in \mathcal{E}(\R^{dk}),% \ell\in]-\infty, 1/2[
$ and $f =g\circ J_{t_1,...,t_k}%-\ell \mathfrak{c}_1
:\mathbb W\to \R$, then \eqref{UstunelCond} holds for every $t\in [0,1]$ and $\omega \in \mathbb W_{[0,t]}.$ Moreover, if $t_i<t\leq t_{i+1}$, then there is $h_{t}\in \mathcal{E}%^{2\max(0,\ell)(1-t),L(\ell)(1-t)}
(\R^{d(i+1)})%,L(\ell)=\frac{2\max(0,\ell)}{1-2\ell}
$ with $\lambda_t(f)=h_{t}\circ J_{t_1,...,t_i,t}%-\ell \mathfrak{c}_t
, $ and for all $s,t>0, s+t<1:$
\[\lambda_s(\lambda_{s+t}(f))=\lambda_s(f).\]
Moreover, there are constants $c_1,d_1>0$ such that for all $t,t+s\in ]t_i,t_{i+1}], s>0$ we have :
\begin{equation}\label{TimeBound}| h_{t}(x_1,...,x_i,x)-h_{t+s}(x_1,...,x_i,x)|\leq \sqrt{s}(c_1||(x_1,...,x_i,x)||^2+d_1(d+\sup(-g)))\end{equation}
and for $t=t_i,s>0,t+s\in ]t_i,t_{i+1}]$,
\begin{align*}&| h_{t}(x_1,...,x_i)-h_{t+s}(x_1,...,x_i,x)|\\&\leq \sqrt{(c_1||(x_1,...,x_i,x)||^2+d_1(d+\sup(-g)))}\left(\sqrt{s(c_1||(x_1,...,x_i,x)||^2+d_1(d+\sup(-g)))}+ ||x_i-x||\right)\end{align*}
Finally, for $\alpha\in [1,2],$ if $g\in \mathcal{E}_\alpha(\R^{dk})$ so is $h_{t}\in \mathcal{E}_\alpha%^{2\max(0,\ell)(1-t),L(\ell)(1-t)}
(\R^{d(i+1)})$, and if $\alpha>1,$ $h_{t}$ is G\^ateaux-differentiable and if $D_\alpha=0$ for $g$ %and $\ell\geq 0$ 
we can take $D_\alpha=0$ for $h_{t}$, and  for all
$||H||=1, H\in \R^d$, the following bound on the directional derivative holds: 
\begin{align*}&\left|D^y_Hh_t(x_1,...,x_i,y)-D^y_Hh_{t+s}(x_1,...,x_i,y)\right|\leq  s^{(\alpha-1)/2}\sqrt{d}%s^{1/2}\sqrt{d}%\\&\times 
\left(c_1+d_1\frac{||(x_1,...,x_i,y)||^2+\sup(-g)}{d}\right).\end{align*}

\end{proposition}

\begin{proof}
\begin{step} Formula and Stability of subquadratic behaviour\end{step} First we fix $t$ and write $h=h_t$. The assumptions from the previous corollary to check \eqref{UstunelCond} hold from the boundedness and at most quadratic growth assumption. It suffices to define: \[h(x_1,...,x_i,x)=-\log \left( \int_{\mathbb W_{[t,1]}} e^{-g(x_1,...,x_i,x+\nu_{t_{i+1}},...,x+\nu_{t_{k}})} \ \mathrm{d} \gamma_{[t,1]}(\nu) \right).\]
The expected formula $\lambda_t(f)=h\circ J_{t_1,...,t_i,t}$ follows by an examination of our notation. The convexity then follows from theorem \ref{ustunel} as in the previous corollary. This results also yields the alternative formula :
\begin{align}\begin{split}h_t&(x_1,...,x_i,x)=\inf_{U\in L^2_a(\gamma_{[t,1]},\mathbb H)}\\&\left[\mathbf{E}_{\gamma_{[t,1]}}  \left(
  [g(x_1,...,x_i,x+[(B-B_t)+U]_{t_{i+1}},...,x+[(B-B_t)+U]_{t_{k}})+\frac{1}{2}\|U\|_{\mathbb H}^2\right)\right].\end{split}\label{htinf}\end{align}
Obviously, $h=h_t$ is bounded from bellow with the same constant as $g$ and taking $U=0$, \begin{align*}h(x_1,...,x_i,x)&\leq c\left(d+\sum_{l=1}^i\sum_{j=1}^d(|x_l|^{(j)})^2+\mathbf{E}_{\gamma_{[t,1]}}  \left(\sum_{l=i+1}^k\sum_{j=1}^d(|x+(B-B_t)_l|^{(j)})^2\right)\right)\\&\leq c\left(d+\sum_{l=1}^i\sum_{j=1}^d(|x_l|^{(j)})^2+  2(k-i)\left(d+\sum_{j=1}^d|x^{(j)}|^2\right)\right)\\&\leq c(1+2(k-i))\left(d+\|(x_1,...,x_i,x)\|^2\right).\end{align*}
Thus $h$ is subquadratic too.

\begin{step} Stability of lipschitzness\end{step}
 We now check the stability of \eqref{lipE}. Fix $U\in L^2_a(\gamma_{[t,1]},\mathbb H)$ reaching the inf up to $\epsilon>0$ for $h(y_1,...,y_i,y)$ and use $\eqref{lipE}$ to get as before for $X=(x_1,...,x_i,x, ....,x),Y=(y_1,...,y_i,y,...,y)$ that :
\begin{align*}&\mathbf{E}_{\gamma_{[t,1]}}  \left(
  g(x_1,...,x_i,x+[(B-B_t)+U]_{t_{i+1}},...,x+[(B-B_t)+U]_{t_{k}})\right)\\&\leq \mathbf{E}_{\gamma_{[t,1]}}  \left(
  g(y_1,...,y_i,y+[(B-B_t)+U]_{t_{i+1}},...,y+[(B-B_t)+U]_{t_{k}})\right)+ \sqrt{k-i}||x-y||\\&\times\mathbf{E}_{\gamma_{[t,1]}}  \left(D\sqrt{d}+C \|X||+2C ||([(B-B_t)+U]_{t_{i+1}},...,[(B-B_t)+U]_{t_{k}})\|+C\|Y\|\right)
\\&\leq \mathbf{E}_{\gamma_{[t,1]}}  \left(
  g(y_1,...,y_i,y+[(B-B_t)+U]_{t_{i+1}},...,y+[(B-B_t)+U]_{t_{k}})\right)+ \sqrt{k-i}||x-y||\\&\times\mathbf{E}_{\gamma_{[t,1]}}  \left(D\sqrt{d}+C \|X\|+2C \sqrt{k-i}(\sqrt{d}+\|U\|_{\mathbb H}) +C\|Y\|\right).
\end{align*}
We then note that we can use the definition as an infimum and the subquadratic bound above %jointly with \eqref{boundU2}
 to show that a $U,$ $\epsilon$-close to the infimum as we chose, satisfies: \begin{align}\label{boundU}\begin{split}\mathbf{E}(\|U\|_{\mathbb H}^2)&\leq 4\epsilon+ 4 h(y_1,...,y_i,y) + 4\sup (-g)%+ 4L(\ell)\left(d+\|(y_1,...,y_i,y)\|^2\right)
 \\&\leq 4\epsilon+ [4 c(1+2(k-i))%+4L(\ell)+8|\ell|
 ]\left(d+\|(y_1,...,y_i,y)\|^2\right)+ 4\sup (-g)=:4\epsilon+A(y_1,...,y_i,y)\end{split}\end{align}

  Considering such an $U$, we can thus combine our previous equations   and after adding $\frac{1}{2}\|U\|_{\mathbb H}^2$, 
 and taking the infimum on the left hand side, one gets:
    \begin{align*}&h(x_1,...,x_i,x)\leq h(y_1,...,y_i,y) +\sqrt{k-i}||x-y||\left(D\sqrt{d}+C \|X\|+C \|Y\|\right)+\epsilon\\&+2C(k-i)||x-y|| \left(\sqrt{d}+\sqrt{2\epsilon+A(y_1,...,y_i,y)} \right)
    %\\&+|\ell|\ ||x-y||\Big(||x||+||y||+2\sqrt{d}+2\sqrt{4\epsilon+A(y_1,...,y_i,y)}\Big)
    \end{align*}
   thus letting $\epsilon\to 0$ %and using our previous bound for $h$
  one gets  \eqref{lipE} for $h$ with $C$ replaced by $C_{k,i}=C(k-i)+ %[
  2C(k-i)%+2|\ell|]
  \sqrt{2 c(1+2(k-i))%+4L(\ell)+8|\ell|
  }$ and $D$ replaced by $D_{k,i}=D\sqrt{k-i}+[2C(k-i)%+2|\ell|
  ](1+\sqrt{2 c(1+2(k-i))%+4L(\ell)+8|\ell|
  }+\sqrt{\frac{2}{d}\sup (-g)})).$
 
 \begin{step}``Semigroup'' formula\end{step}
In decomposing Wiener measure as a product measure by independence and using the defining formula twice and Fubini theorem :\begin{align*}\lambda_s(\lambda_{s+t}(f))&=-\log \left( \int_{\mathbb W_{[s,t+s]}} e^{-\lambda_{s+t}(f)(\omega+\nu)} \ \mathrm{d} \gamma_{[s,t+s]}(\nu) \right)\\&=-\log \left( \int_{\mathbb W_{[s,t+s]}} \int_{\mathbb W_{[t+s,1]}} e^{-f(\omega+\nu+\mu)} \ \mathrm{d} \gamma_{[t+s+1]}(\mu) \ \mathrm{d} \gamma_{[s,t+s]}(\nu)\right)\\&=\lambda_s(f).\end{align*}
This will be the key to all our time estimates below. 
 
\begin{step}Second order bound\end{step} Finally, let us prove the statement for $g\in \mathcal{E}_\alpha(\R^{dk})$. Consider a typical element in the infimum formula defining $h$ for $x,Y$ with $Y=(y_1,...,y_i,y), Y'=(y_1,...,y_i,y,...,y),X'=(x_1,...,x_i,x,...,x)$%, ||y_j||\leq 1, ||y||\leq 1$ 
and apply \eqref{CalphaE} to $g$ in order to get for $u\in [0,1], V_s=[(B-B_t)+U]_s$ (using the parallelogram identity for integral terms):
 
\begin{align*} &\left[\mathbf{E}_{\gamma_{[t,1]}}  \left(
  [g(x_1,...,x_i,x+[(B-B_t)+U]_{t_{i+1}},...,x+[(B-B_t)+U]_{t_{k}})+\frac{1}{2}\|U\|_{\mathbb H}^2\right)\right] 
  %\\& \mathbf{E}_{\gamma_{[t,1]}}  \left(
  %-\ell \int_t^1||x+[(B-B_t)+U]_{s}||^2ds\right)
  \geq\\& \left[\mathbf{E}_{\gamma_{[t,1]}}  \left(
  [ug(x_1+(1-u)y_1,...,x_i+(1-u)y_i,x+(1-u)y+V_{t_{i+1}},...,x+(1-u)y+V_{t_{k}})+\frac{u}{2}\|U\|_{\mathbb H}^2\right)\right]\\&+\left[\mathbf{E}_{\gamma_{[t,1]}}  \left(
  [(1-u)g(x_1-uy_1,...,x_i-uy_i,x-uy+V_{t_{i+1}},...,x-uy+V_{t_{k}})+\frac{1-u}{2}\|U\|_{\mathbb H}^2\right)\right]
  \\&-\left[d^{1-\alpha/2}(D_\alpha \frac{||X'||+||X'+(1-u)Y'||+||X'-uY'|| }{\sqrt{d}})u(1-u)||Y'||^\alpha\right]
  \\&-\left[\mathbf{E}_{\gamma_{[t,1]}}  \left(d^{1-\alpha/2}(C_\alpha+D_\alpha \frac{3||[(B-B_t)+U]_{t_{i+1}},...,x+[(B-B_t)+U]_{t_{k}})||}{\sqrt{d}})u(1-u)||Y'||^\alpha\right)\right]
 % \\&+\ell(1-t)u(1-u)||y||^2+\mathbf{E}_{\gamma_{[t,1]}}  \left(
%  -\ell u\int_t^1||x+(1-u)y+[(B-B_t)+U]_{s}||^2ds\right)\\&+\mathbf{E}_{\gamma_{[t,1]}}  \left(
%  -\ell (1-u)\int_t^1||x-uy+[(B-B_t)+U]_{s}||^2ds\right)
  \\&\geq u h(x_1+(1-u)y_1,...,x_i+(1-u)y_i,x+(1-u)y)+(1-u)h(x_1-uy_1,...,x_i-uy_i,x-uy)
  \\&-\left[d^{1-\alpha/2}(D_\alpha \frac{||X'||+||X'+(1-u)Y'||+||X'-uY'|| }{\sqrt{d}})u(1-u)||Y'||^\alpha\right]%+\ell(1-t)u(1-u)||y||^2
  \\&-\left[d^{1-\alpha/2}(C_\alpha+D_\alpha \frac{3(\sqrt{d}(k-i)+(k-i)\sqrt{\mathbf{E}_{\gamma_{[t,1]}}  \left(||U||_{\mathbb H}^2\right)}}{\sqrt{d}}))u(1-u)||Y'||^\alpha\right].\end{align*}
  
But $||Y'||^\alpha\leq (k-i)^{\alpha/2} ||Y||^\alpha$, $||y||^2\leq d^{1-\alpha/2}||y||^\alpha(\frac{||x-y||+||x||}{\sqrt{d}})^{2-\alpha}\leq d^{1-\alpha/2}||y||^\alpha(1+\frac{||x-y||+||x||}{\sqrt{d}})$ and one can bound from \eqref{boundU} $\mathbf{E}_{\gamma_{[t,1]}}  \left(||U||_{\mathbb H}^2\right)\leq 2\epsilon+A(x_1,...,x_i,x)$ so that taking the infimum, $h$ satisfies \eqref{CalphaE} with $C_\alpha$ replaced by $C_{\alpha,k,i}$ defined as  $$\left(C_\alpha+3 D_\alpha(k-i)[1+\sqrt{4c(1+2(k-i))%+4L(\ell)+8|\ell|
}]+\frac{3(k-i)D_\alpha\sqrt{2\sup(-g)}}{\sqrt{d}}\right)(k-i)^{\alpha/2}%+\max(0,-2\ell)
,$$ and $D_\alpha$ replaced by (a value which is $0$ if $D_\alpha=0, %\ell\geq 0
$): $$D_{\alpha,k,i}:=\left(2D_\alpha\sqrt{k-i}+3D_\alpha(k-i)\sqrt{4c(1+2(k-i))%+4L(\ell)+8|\ell|
}\right)(k-i)^{\alpha/2}%+\max(0,-2\ell)
.$$
 
\begin{step}Regularity in time within intervals\end{step}
For the regularity in time, we first consider $t,s>0$ with $t+s\leq t_{i+1}, t\in ]t_i,t_{i+1}[$. Then we use a variant of the composition formula for $\lambda_t$ to get
 
 \begin{align*}h_t&(x_1,...,x_i,x)=\inf_{U\in L^2_a(\gamma_{[t,t+s]},\mathbb H)}%\\&
 \left[\mathbf{E}_{\gamma_{[t,1]}}  \left(
  [h_{t+s}(x_1,...,x_i,x+[(B-B_t)+U]_{t+s})+\frac{1}{2}\|U\|_{\mathbb H}^2%-\ell \int_t^{t+s}||x+[(B-B_t)+U]_{u}||^2du
  \right)\right]\\&\leq \left[\mathbf{E}_{\gamma_{[t,1]}}  \left(
  [h_{t+s}(x_1,...,x_i,x+(B_{t+s}-B_t))%-\ell \int_t^{t+s}||x+(B_u-B_t)||^2du
  \right)\right]
  \\&\leq h_{t+s}(x_1,...,x_i,x) +%+ 2|\ell| s(||x||^2+d)
  \\&\left[\mathbf{E}_{\gamma_{[t,1]}}  \left(
  (C_{k,i} ||(x_1,...,x_i,x)||+C_{k,i} ||(x_1,...,x_i,x+(B_{t+s}-B_t))|| +D_{k,i}\sqrt{d}) ||(0,...,0,B_{t+s}-B_t)||\right)\right]\\&\leq \left[
  [h_{t+s}(x_1,...,x_i,x) + (2C_{k,i} ||(x_1,...,x_i,x)||+(C_{k,i}+D_{k,i})\sqrt{d}) \sqrt{ds}%+ 2|\ell| s(||x||^2+d)
  \right].\end{align*}
  Conversely, one obtains in considering $U$ and noting that by Cauchy-Schwarz for $U\in L^2_a(\gamma_{[t,t+s]},\mathbb H)$  we have $||U_{t+s}||\leq \sqrt{s}\|U\|_{\mathbb H}$ and also  using \eqref{lipE},\eqref{boundU} for $U$ achieving enough the infimum of the left hand side:
  \begin{align*}&\mathbf{E}_{\gamma_{[t,1]}}  \left(
  h_{t+s}(x_1,...,x_i,x+[(B-B_t)+U]_{s})+\frac{1}{2}\|U\|_{\mathbb H}^2\right)\geq 
  \left[h_{t+s}(x_1,...,x_i,x) -\right. \\&\left.\mathbf{E}_{\gamma_{[t,1]}}  \left((C_{k,i} ||[(B_{t+s}-B_t)+U_{t+s}]||+2C_{k,i} ||(x_1,...,x_i,x)||+D_{k,i}\sqrt{d}) ||B_{t+s}-B_t+U_{t+s}||\right)\right]
  \\&\geq\left[
  h_{t+s}(x_1,...,x_i,x) - C_{k,i}\mathbf{E}_{\gamma_{[t,1]}}  \left([\|B_{t+s}-B_t\|+\sqrt{s}\|U\|_{\mathbb H}]^2\right)\right.\\&-(2C_{k,i} ||(x_1,...,x_i,x)||+D_{k,i}\sqrt{d})\sqrt{\mathbf{E}_{\gamma_{[t,1]}}  \left([\|B_{t+s}-B_t\|+\sqrt{s}\|U\|_{\mathbb H}]^2\right)}
 \\&\geq
  h_{t+s}(x_1,...,x_i,x) - 2C_{k,i}s \left(d+2\epsilon+A(x_1,...,x_i,x)\right)\\&-2(2C_{k,i} ||(x_1,...,x_i,x)||+D_{k,i}\sqrt{d})\sqrt{s(2\epsilon+A(x_1,...,x_i,x))} .\end{align*}

Taking %the sum and 
an infimum concludes to \eqref{TimeBound} in the present case $t,s>0, t>t_i$ with $t+s\leq t_{i+1}$. 

\begin{step}Regularity in time of increments and derivatives within intervals\end{step}
We are now ready to estimate time variation of increments for  $g\in \mathcal{E}_\alpha(\R^{dk})$, $\alpha>1$. From  $h_{t+s}\in \mathcal{E}_\alpha(\R^{dk})$, we know that $h_{t+s}$ is convex and locally Lipschitz and for $x,z=x+(1-u)y,t=x-uy\in B(0,R),u\in [0,1]$:
$$uh_{t+s}(x+(1-u)y)+(1-u)h_{t+s}(x-uy)-h_{t+s}(x)\leq  d^{1-\alpha/2}(C_{\alpha,k,i} +D_{\alpha,k,i} \frac{3R}{\sqrt{d}})u(1-u)||y||^{\alpha}.$$
Thus if $K(R)=d^{1-\alpha/2}(C_{\alpha,k,i} +D_{\alpha,k,i} \frac{3R}{\sqrt{d}})$
one deduces from proposition \ref{C1alphaBall} that 
$h_{t+s}$ is Gateaux-differentiable  with $\alpha-1$ H\"older derivative with constant $2^{1+\alpha}K(R)$ on the ball of radius $R/4$  so that we can compute in applying the fundamental theorem of calculus along lines for $Y,X+h,X,Y+h\in B(0,R)$:
\begin{align}\label{incrementsstep6}\begin{split}&\left|h_{t+s}(Y)+h_{t+s}(X+h)-h_{t+s}(X)-h_{t+s}(Y+h)\right|\\&=\left|\int_0^1 d\lambda dh_{t+s}(X+\lambda(Y-X)).(Y-X) -dh_{t+s}(X+\lambda(Y-X)+h).(Y-X) \right|\\&\leq 2^{1+\alpha}K(4R)||Y-X||\ ||h||^{\alpha-1}\\&\leq 8d^{1-\alpha/2}(C_{\alpha,k,i} +12D_{\alpha,k,i} \frac{||Y||+||X||+||h||}{\sqrt{d}})||Y-X||\ ||h||^{\alpha-1}\end{split}\end{align}
(with the last inequality obtained in minimizing $R$).

 Take a $U$ giving a value $\epsilon$-close to the infimum in the formula for $h_t(x_1,...,x_i,y).$ We first take in the infimum definition for $h_t(x_1,...,x_i,x)$ the value given at this $U$, apply the estimate just obtained and finally H\"older inequality, the bound \eqref{boundU}%, and a variant of \eqref{boundell}
 . We obtain for $X=(x_1,...,x_i,x), Y=(x_1,...,x_i,y)$ :

\begin{align*}&h_t(X)+h_{t+s}(Y)\leq h_{t+s}(Y)+ \\&\left[\mathbf{E}_{\gamma_{[t,1]}}  \left(
  [h_{t+s}(x_1,...,x_i,x+[(B-B_t)+U]_{t+s})+\frac{1}{2}\|U\|_{\mathbb H}^2%-\ell\int_t^{t+s}||x+[(B-B_t)+U]_{s}||^2ds
  \right)\right]
  \\&\leq h_{t+s}(X)+ \mathbf{E}_{\gamma_{[t,1]}}  \left(
  [h_{t+s}(x_1,...,x_i,y+[(B-B_t)+U]_{t+s})+\frac{1}{2}\|U\|_{\mathbb H}^2
  %-\ell\int_t^{t+s}||y+[(B-B_t)+U]_{s}||^2ds
  \right)\\&+\mathbf{E}_{\gamma_{[t,1]}}  \left(
  8d^{1-\alpha/2}(C_{\alpha,k,i} +12D_{\alpha,k,i} \frac{||Y||+||X||+||[(B-B_t)+U]_{t+s}||}{\sqrt{d}})||Y-X||\ ||[(B-B_t)+U]_{t+s}||^{\alpha-1}\right)
  %\\&+|\ell|s\ ||x-y||\Big(||x||+||y||+2\sqrt{d}+2\mathbf{E}_{\gamma_{[t,1]}} \left(\|U\|_{\mathbb H}\right)\Big)
    \\&\leq h_{t+s}(x_1,...,x_i,x)+ h_t(x_1,...,x_i,y)+\epsilon \\&+ ||y-x||8d^{1-\alpha/2}(C_{\alpha,k,i} +12D_{\alpha,k,i} \frac{||(x_1,...,x_i,y)||+||(x_1,...,x_i,x)||}{\sqrt{d}})\mathbf{E}_{\gamma_{[t,1]}}  \left(
 [\|B_{t+s}-B_t\|+\sqrt{s}\|U\|_{\mathbb H}]^2\right)^{(\alpha-1)/2}\\&+||y-x||\frac{96D_{\alpha,k,i}}{\sqrt{d}} \mathbf{E}_{\gamma_{[t,1]}}  \left([\|B_{t+s}-B_t\|+\sqrt{s}\|U\|_{\mathbb H}]^2\right)^{\alpha/2}% +|\ell|s\ ||x-y||\Big(||x||+||y||+2\sqrt{d}+2\mathbf{E}_{\gamma_{[t,1]}} \left(\|U\|_{\mathbb H}\right)\Big)
 \\&\leq h_{t+s}(x_1,...,x_i,x)+ h_t(x_1,...,x_i,y)+\epsilon \\&+ ||y-x||8d^{1-\alpha/2}(C_{\alpha,k,i} +12D_{\alpha,k,i} \frac{||(x_1,...,x_i,y)||+||(x_1,...,x_i,x)||}{\sqrt{d}})\left[2sd+8\epsilon s+2sA(x_1,...,x_i,y)\right]^{(\alpha-1)/2}\\&+||y-x||\frac{96d^{1-\alpha/2}D_{\alpha,k,i}}{\sqrt{d}} \left[2sd+8\epsilon s+2sA(x_1,...,x_i,y)\right]^{\alpha/2}%\\&+|\ell|s\ ||x-y||\Big(||x||+||y||+2\sqrt{d}+2\left[4\epsilon +A(x_1,...,x_i,y)\right]^{1/2}\Big)
 .\end{align*}

letting $\epsilon \to 0$ and exchanging $x,y$ one obtains the bound on increments :

\begin{align}\label{Incrementsht}\begin{split}&\left|h_t(x_1,...,x_i,x)-h_t(x_1,...,x_i,y)-h_{t+s}(x_1,...,x_i,x)+ h_{t+s}(x_1,...,x_i,y)\right| \leq  8||y-x||s^{(\alpha-1)/2}\\&\times (C_{\alpha,k,i} +%[
12D_{\alpha,k,i} %+|\ell|]
\frac{||(x_1,...,x_i,y)||+||(x_1,...,x_i,x)||+ \left[2d+2\max(A(x_1,...,x_i,y),A(x_1,...,x_i,x))\right]^{1/2}}{\sqrt{d}})\\&\times
 \left[2d+2\max(A(x_1,...,x_i,y),A(x_1,...,x_i,x))\right]^{1/2}
%\\&+ |\ell|s\ ||x-y||\Big(||x||+||y||+2\sqrt{d}+2\left[\max(A(x_1,...,x_i,y),A(x_1,...,x_i,x))\right]^{1/2}\Big).
\end{split}\end{align}
and since $h_t$ admits partial derivatives, in making in making $x\to y$, one obtains for $||H||=1, H\in \R^d$, a bound on the directional derivative $D^y_Hh_t(x_1,...,x_i,y)=\sum_{i=1}^d \frac{\partial}{\partial y^{(j)}}h_t(x_1,...,x_i,y)H_j$:
\begin{align*}&\left|D^y_Hh_t(x_1,...,x_i,y)-D^y_Hh_{t+s}(x_1,...,x_i,y)\right|\\&\leq  s^{(\alpha-1)/2} 8(C_{\alpha,k,i} +%[
12D_{\alpha,k,i}%+|\ell|] 
\frac{2||(x_1,...,x_i,y)||+ \left[2d+2A(x_1,...,x_i,y)\right]^{1/2}}{\sqrt{d}})\left[2d+2A(x_1,...,x_i,y)\right]^{1/2}.\end{align*}

\begin{step}Regularity in time around a $t_i$\end{step}

 For the second and last case, $t=t_i,s>0, t+s<t_{i+1}$ 
 
 \begin{align*}h_t&(x_1,...,x_i)=\inf_{U\in L^2_a(\gamma_{[t,s]},\mathbb H)}\\&\left[\mathbf{E}_{\gamma_{[t,1]}}  \left(
  [h_{t+s}(x_1,...,x_i,x_i+[(B-B_t)+U]_{t+s})+\frac{1}{2}\|U\|_{\mathbb H}^2%-\ell \int_t^{t+s}||x_i+[(B-B_t)+U]_{u}||^2du
  \right)\right]\\&\leq \left[\mathbf{E}_{\gamma_{[t,1]}}  \left(
  [h_{t+s}(x_1,...,x_i,x_{i+1}+(x_i-x_{i+1})+(B_{t+s}-B_t))%-\ell \int_t^{t+s}||x_i+(B_u-B_t)||^2du
  \right)\right]
  \\&\leq h_{t+s}(x_1,...,x_i,x_{i+1}) \\& + (2C_{k,i} ||(x_1,...,x_i,x_{i+1})||+(C_{k,i}+D_{k,i})(\sqrt{d}+||x_i-x_{i+1}||)) (\sqrt{ds}+||x_i-x_{i+1}||)%+2|\ell|s(||x_i||^2+d)
  .\end{align*}
and similarly:  
  \begin{align*}&h_{t}(x_1,...,x_i)\geq 
  %\left[h_{t+s}(x_1,...,x_i,x_{i+1}) -\right. %\\&\left.\mathbf{E}_{\gamma_{[t,1]}}  \left((C_{k,i} ||(x_{i}-x_{i+1})+[(B_{t+s}-B_t)+U_{t+s}]||+2C_{k,i} ||(x_1,...,x_i,x_{i+1})||+D_{k,i}\sqrt{d}) ||(x_i-x_{i+1})+B_{t+s}-B_t+U_{t+s}||\right)\right]
%  \\&\geq\left[
%  h_{t+s}(x_1,...,x_i,x) - C_{k,i}\mathbf{E}_{\gamma_{[t,1]}}  \left([||x_{i}-x_{i+1}||+\|B_{t+s}-B_t\|+\sqrt{s}\|U\|_{\mathbb H}]^2\right)\right.\\&-(2C_{k,i} ||(x_1,...,x_i,x_{i+1})||+D_{k,i}\sqrt{d})\sqrt{\mathbf{E}_{\gamma_{[t,1]}}  \left([\|x_{i}-x_{i+1}\|+\|B_{t+s}-B_t\|+\sqrt{s}\|U\|_{\mathbb H}]^2\right)}
% \\&\geq
  h_{t+s}(x_1,...,x_i,x) - 3C_{k,i} \left(\|x_{i}-x\|+sd+sA(x_1,...,x_i,x_i)\right)\\&-3(2C_{k,i} ||(x_1,...,x_i,x)||+D_{k,i}\sqrt{d})\sqrt{sA(x_1,...,x_i,x_i)+\|x_{i}-x\|}
 % \\&-|\ell| s(2||x_i||^2+4d+4A(x_1,...,x_i,x_i)) 
 .\end{align*}
\end{proof}

\section{Stochastic differential equations with monotone drift and their free variant}  
  \subsection{Classical Case}
We quote here the main result of the chapter 3 in \cite{RocknerConcise} (coming from \cite{Krylov}) and apply it to the setting we need.

We start by quoting their Theorem 3.1.1. We consider $W_t$ a Wiener process in $\R^d$ in a normal filtration $\mathcal{F}_t$ (i.e. for instance the completed filtration generated by this brownian motion on Wiener space $\Omega$, to insure $\mathcal{F}_0$ contains every null-sets and the filtration is right continuous, cf. \cite[Prop 2.1.13]{RocknerConcise}). We fix $$\sigma:[0,1]\times \R^d\times \Omega\to M_d(\R),\ \  b:[0,1]\times \R^d\times \Omega\to \R^d$$ continuous in $x\in \R^d$ for $t\in [0,1], w\in \Omega$ fixed and progressively measureable in the sense that their restriction to $[0,t]\times \Omega$ is $B([0,t])\otimes \mathcal{F}_t$ measurable. We give the target space  their  usual euclidean norms $||.||$.

\begin{theorem}[\cite{RocknerConcise} Theorem 3.1.1 ]\label{KrylovMonotone}
Consider $b,\sigma$ as above, and assume moreover that on $\Omega$ for all, $R\in [0,\infty[$ we have the integrability condition:
\begin{equation}\label{integrability}\int_0^1dt\sup_{|x|\leq R}||\sigma(t,x)||^2+||b(t,x)||<\infty\end{equation}
and for also $t\in [0,1]$, $x,y\in \R^d$, with $||x||,||y||\leq R$ the local weak monotonicity:
\begin{equation}\label{monotonicity}2\langle x-y,b(t,x)-b(t,y)\rangle+||\sigma(t,x)-\sigma(t,y)||^2\leq K_t(R)||x-y||^2
\end{equation}
and the weak coercivity:
\begin{equation}\label{coercivity}2\langle x,b(t,x)\rangle+||\sigma(t,x)||^2\leq K_t(1)(d+||x||^2)
\end{equation}
where for each $R>0$, $K_t(R)$ is an $\R_+$-valued, $\mathcal{F}_t$ adapted process satisfying $\alpha_s(R)=\int_0^sK_t(R)dt<\infty$ on $\Omega$. Then, 
Then for any $\mathcal{F}_0$ -measurable map $X_0 : \Omega \to \R^d$ there exists a (up to P -
indistinguishability) unique solution to the stochastic differential equation
$$dX(t) = b(t, X(t)) dt + \sigma(t, X(t)) dB (t).$$
Here solution means that $(X(t))_{ t\geq 0}$ is a P -a.s. continuous $R^d$ -valued $\mathcal{F}_t$-
adapted process such that P -a.s. for all $t \in [0, 1]$

\begin{equation}X(t) = X_0 +
\int_0^t b(s, X(s)) ds +
\int_0^t \sigma(s, X(s)) dB (s).\end{equation}

Furthermore, for all $t \in [0, 1]$
$$E(||X(t)||^2 e^{-\alpha_t (1)} )\leq 
E(||X_0 ||^2 ) + d.$$
\end{theorem}

%\begin{proof}

%Only the fourth order estimate is not contained in the stated result (but it is well-known, see e.g. \cite[Th 3.5 p 59]{BensoussanLyons})

%\end{proof}

We now apply this result in the simple case we will use in taking care of dimensional dependence of constants as before.
\begin{corollary}\label{MonotoneSol} We consider the previous setting with $\Omega=C^0([0,1],\R^d)$ the pathspace with Wiener measure $\gamma$ with its canonical normal filtration. Let $g\in \mathcal{E}_\alpha(\R^{dk}),\alpha>1$  and $t_0=0\leq t_1<...<t_k\in [0,1 ]$  and $h_{t,\ell}$ defined in lemma \ref{OptimalValueFct}. We define $b(t,x,\omega)$ for $\omega \in \Omega$, for $t_i<t\leq t_{i+1}$ by :
\[b_j(t,x,\omega)=b_j^{g,\ell}(t,x,\omega):=-\frac{\partial}{\partial x^{(j)}} h_{t,\ell}(X_{t_1}(\omega),...,X_{t_i}(\omega),x),\]
and $b(t,x,\omega)=0$ if $t> t_k$. Fix also $\sigma(t,x,\omega)=I_d$. Then, they satisfy \eqref{integrability},\eqref{monotonicity} and \eqref{coercivity} for $K_t(R)=8\max(0,\ell)$ for $R\neq 1$,  $K_t(1)=K>0$ a fixed constant so that the conclusion of theorem \eqref{KrylovMonotone} follows on each interval, with a solution $X_s$ which is then used to define $b$ on the next interval. Moreover, for the solution $X_s$, we have the regularity bounds for some $C_4$ for all  $t,s\in ]t_i,t_{i+1}], t>s$: \[E(||b(t,X_t)- b(s,X_s)||)\leq  C_4\sqrt{d(t-s)^{\alpha-1}}\left(1+\frac{E(||X_0 ||^2 ) +\sup(-g)}{d}\right),\]
\[E(||b(t,X_t)- b(s,X_s)||^2)\leq  C_4d\sqrt[4]{(t-s)^{\alpha-1}}\sqrt{\left(1+\frac{E(||X_0 ||^2 ) +\sup(-g)}{d}\right)}\left(1+\frac{E(||X_0 ||^4 )}{d^2}\right)^{3/8},\]
and the estimate for $t\in [0,1]$:
$$E(||X(t)||^4)\leq \left[E\left(||X(0)||^4\right)+(3K+1)d^2\right]e^{(3K+1)t}.$$
\end{corollary}
The last H\"older-continuity estimate will be crucial to recover some continuity in ultraproducts.
\begin{proof}
By induction, one can assume the previously built solution $X_t$ to be measurable. From \eqref{CalphaE}, the 2-paraconvexity and proposition \ref{C1alphaBall}, $\frac{\partial}{\partial x^{(j)}} h_t(X_{t_1}(\omega),...,X_{t_i}(\omega),x)$ is $(\alpha-1)$ H\"older-continuous in $x$, thus continuous. From proposition  \ref{OptimalValueFct}, it is also continuous in $t$ on $]t_i, t_{i+1}]$ for each $\omega, x$ fixed and for each $x,t\in [0,1] $ the formula is clearly $\sigma(X_s, s\leq t)$-measurable  thus, by e.g. \cite[lemma 9.2]{Mackey}, on $]t_i, t_{i+1}]\times \Omega$, $b_j(.,x,.)$ is $B(]t_i, t])\otimes  \mathcal{F}_t$-measurable, and thus $b$ is progressively measurable. We have from \eqref{lipE}: $$\sup_{\|x\|_2\leq R}||\sigma(t,x)||^2+||b(t,x)||\leq 1+(2CR+D\sqrt{d}),$$
so that \eqref{integrability} is verified. The monotony of $b(t,.,\omega)$ follows from the 2-paraconvexity of $h$ so that \eqref{monotonicity} holds with $K_t(R)\geq 8\max(0,\ell)(1-t)$. \eqref{coercivity} then follows from \eqref{lipE} as before since $\langle x,b(t,x,\omega)\rangle=dh_t(\omega_{t_1},...,\omega_{t_i},x).(x_1,...,x_n)$ with $K_t(1)\geq\max(5C, D^2)+1.$

For the continuity property, we decompose the inequality into two terms. First a bound on $E(||b(t,X_t)- b(s,X_t)||)$ is obtained by using the last inequality in proposition \ref{OptimalValueFct} and the estimates on $E(||X_t||^2).$
\begin{align*}&E(||b(t,X_t)- b(s,X_t)||)\leq  \sqrt{|t-s|^{\alpha-1}d}%\\&\times 
\left(c_1+d_1\frac{e^KE(||X_0||^2+d)+\sup(-g)}{d}\right).\end{align*}

 A similar bound holds from the application of  proposition \ref{C1alphaBall} as in step 6 of the proof of proposition \ref{OptimalValueFct}: \begin{align*}E(||b(s,X_t)- b(s,X_s)||)&\leq d^{1-\alpha/2}E\left[(C +D \frac{||X_t||+||X_s||}{\sqrt{d}})||X_t-X_s||^{(\alpha-1)}\right]\\&\leq d^{1-\alpha/2}(C +D \frac{\sqrt{E[||X_t||^2]}+\sqrt{E[||X_s||^2]}}{\sqrt{d}})(E[||X_t-X_s||^2])^{(\alpha-1)/2},\end{align*}
and since, we have: 
$E(||b(u,X_u)||^2)\leq E( 2C||X_u||+D\sqrt{d})^2\leq 8C^2 e^K(E( ||X_0||^2)+d)+ 2D^2d,$ and say for $t>s$  \begin{align*}E(||X_t- X_s||^{2})&\leq 2d(t-s)+2\sqrt{d(t-s)}\sqrt{E(\int_s^t||b(u,X_u)||^2du)}\\&\leq 2d(t-s)+2\sqrt{d}(t-s)\sqrt{8C^2 e^K(E( ||X_0||^2)+d)+ 2D^2d}=C_3d(t-s).\end{align*}
Thus, we obtain the expected bound:
\begin{align*}E(||b(s,X_t)- b(s,X_s)||)&\leq (C +2De^{K/2} \frac{\sqrt{E[||X_0||^2]+d}}{\sqrt{d}})\sqrt{d}[C_3(t-s)]^{(\alpha-1)/2}.\end{align*}

 The fourth order estimate is not contained in the previous stated result (but it is well-known, see e.g. \cite[Th 3.5 p 59]{BensoussanLions})
Ito's formula and our weak coercivity assumptions gives : \begin{align*}||X(t)||^2&=||X(0)||^2+\int_0^t(2\langle X(s),b(s,X(s))\rangle +1)ds + \int_0^t\langle X(s), dB_s\rangle\\&\leq ||X(0)||^2+K\int_0^t(1+||X(s)||^2)ds + \int_0^t\langle X(s), dB_s\rangle\end{align*}
In taking the square and applying Ito's formula again and taking mathematical expectation, one gets:
\begin{align*}E(||X(t)||^4)&\leq E\left(||X(0)||^4+2K\int_0^t(d+||X(s)||^2)||X(s)||^2ds + \int_0^t||X(s)||^2ds\right)
\\&\leq E\left(||X(0)||^4+(3K+1)\int_0^t(d^2+||X(s)||^4)ds.\right)\end{align*}
Gronwall's lemma concludes to the fourth order bound and then Cauchy-Schwartz inequality gives the $L^2$ bound on increments:
\begin{align*}E(||b(t,X_t)- b(s,X_s)||^2)&\leq E(||b(t,X_t)- b(s,X_s)||^3)^{1/2}E(||b(t,X_t)- b(s,X_s)||)^{1/2}\\&\leq  4 E(||b(t,X_t)- b(s,X_s)||)^{1/2}\sup_{t\in [0,1]} E(||b(t,X_t)||^4)^{3/8}.\end{align*}
 \end{proof}

\subsection{Free case}
\setcounter{Step}{0}
We now obtain a result similar to the previous corollary in the free case in order to describe the limit of the application to matrices of this corollary. Instead of using Euler approximation as in \cite{RocknerConcise}, we will use a Yosida approximation as in \cite[Theorem 4.3]{StephanLiu}. The necessary  preliminaries were recalled in subsection 2.6.
The following result is of independent interest for the study of free SDEs. That's why we assume a slightly more general setting that what we need later.

\begin{theorem}\label{FreeMonotone}
Let $M_t\subset (M,\tau)$ a filtration of finite von Neumann algebras containing an adapted free brownian motion $S_t=(S_t^1,...,S_t^m).$ Let $T>0$ and $t_0=0< t_1< t_2...<t_k\leq t_{k+1}=T |t_{i+1}-t_i|\leq 1$, and for $t\in ]t_i,t_{i+1}]$, let $h_t:L^2_{sa}(M_t,\tau)^{m (i+1)}\to \R$ a convex function bounded below by $c\in\R$ (uniformly in $t$), subquadratic with bound:
$$ h_t(x)\leq |c|(1+||x||_2^2),$$
 and satisfying for some $C,D>0$, for $X,Y\in L^2_{sa}(M_t,\tau)^{m (i+1)}$:$$|h_t(X)-h_t(Y)|\leq ||X-Y||\left(C||X||_2+C||Y||_2+D\right).$$
Assume $h_t$ is G\^ateaux differentiable such that $t\mapsto \nabla_{i+1} h_t(X)$ is continuous with value $L^2_{sa}(M,\tau)^m$  on any  $]t_i,t_{i+1}[$, that for $X\in L^2_{sa}(M_u)^{m(i+1)}, u<t$ we have $\nabla_{i+1} h_t(X)\in L^2_{sa}(M_u)^m$ and even satisfying for some $\alpha,\beta\in ]0,1]$ and any $t<s\in ]t_i,t_{i+1}[, X,Y\in L^2_{sa}(M_t,\tau)^{m (i+1)}:$
\begin{equation}\label{Holderht}||\nabla_{i+1} h_t(X)-\nabla_{i+1} h_s(X)||_2\leq |t-s|^\alpha (C+D||X||_2).\end{equation}
\begin{equation}\label{Holderhtspace}||\nabla_{i+1} h_t(X)-\nabla_{i+1} h_t(Y)||_2\leq ||X-Y||_2^\beta (C+D||X||_2+D||Y||_2).\end{equation}
Assume finally given $X_0\in L^2_{sa}(M_0,\tau)^m, L>0.$ Then there is $X_t\in L^2_{sa}(M_t,\tau)^m$  continuous in $t$ satisfying :
$$X_t=X_0+S_t+\int_0^t LX_s-u_s(X)ds,\ \ u_s(X)=\sum_{i=0}^k\nabla_{i+1} h_s(X_{t_1},...,X_{t_i},X_s) 1_{]t_{i},t_{i+1}]}(s).$$
Moreover, any continuous solution $Y_t\in L^2_{sa}(M_t,\tau)^m$ of this equation equals $X_t$ for every $t\in [0,T]$ if $Y_0=X_0$.
\end{theorem}
\begin{proof}
\begin{step}Uniqueness\end{step}Let us start by checking uniqueness. Note that by the assumption $||u_s(X)||_2\leq (2C||(X_{t_1},...,X_{t_i},X_s)||_2+D)$ which is bounded on $[0,T]$ and 
$$X_t-Y_t=X_0-Y_0+\int_0^tL(X_s-Y_s)-(u_s(X)-u_s(Y))ds=\int_0^tL(X_s-Y_s)-(u_s(X)-u_s(Y))ds$$
so that $X_t-Y_t$ is absolutely continuous in $t$ in $L^2_{sa}(M)$ and we have:
$$||X_t-Y_t||_2^2=%||X_0-Y_0||_2^2
2\int_0^tds\langle L(X_s-Y_s)-((u_s(X)-u_s(Y)),X_s-Y_s\rangle \leq 2\int_0^tds L||X_s-Y_s||_2^2%||X_0-Y_0||_2^2
$$
 since $\langle (u_s(X)-u_s(Y)),X_s-Y_s\rangle\geq 0$ since $u_s$ is monotone from the convexity of $h_s$ and first since for $s\in [0,t_1]$ $u_s(X)\in L^2_{sa}(M_s)$ and then since by induction on the time intervals for $s\in ]t_i,t_{i+1}]$ one can use $ X_{t_i}= Y_{t_i}$. Gronwall's lemma then concludes the induction step to prove equality.
 
\begin{step}Definition of the Yosida approximation\end{step}
We now turn to the proof of the existence result. We follow a Yosida approximation scheme (see e.g. \cite[Theorem 4.3]{StephanLiu} in the classical case). By induction on $i$, we aim at finding a solution on $]t_i,t_{i+1}]$ and thus we consider given $X_{t_1},...X_{t_i}$ and $H_t=h_t(X_{t_1},...X_{t_i},.):L^2_{sa}(M_t)^m\to \R.$ This is a convex continuous (even locally H\"older continuous) function and we can consider $H_{t,\lambda}$ its Hopf-Lax-Yosida approximation from Proposition \ref{Yosida}. We now fix $0<\lambda\leq 1$. From a Picard iteration argument (see e.g. \cite[Lemma 3.2]{Gao}, or the proof of Theorem \ref{MainTechnical} Step 2(ix) in a more complicated context), since $\nabla H_{t,\lambda}$ is globally $1/\lambda$-Lipschitz, one gets a solution on $[t_i,t_{i+1}]$ with $X_{t,\lambda}\in L^2_{sa}(M_t)^m$:
$$X_{t,\lambda}=X_{t_i}+S_t-S_{t_i}+\int_{t_i}^tds LX_{s,\lambda}-\nabla H_{s,\lambda}(X_{s,\lambda}).$$
Note that this argument uses the regularity in time obtained from lemma \ref{RegularityYosida} based on the assumptions \eqref{Holderhtspace} and \eqref{Holderht} so that the Picard iterates are well-defined and adapted. Note that using the assumption that for $X\in L^2_{sa}(M_u)^{m(i+1)}, u<t$ we have $\nabla_{i+1} h_t(X)\in L^2_{sa}(M_u)^m$, the quoted lemma can be used with $H=L^2_{sa}(M_t)^m$ and the appropriate restriction to this space of $h_s,s>t.$ Even in the case where $\beta=1$, the Yosida approximation is useful to get a global lipschitzness. 

\begin{step}Second moment estimate\end{step}

From Ito's formula \cite{BianeSpeicher} and taking traces, one gets the a priori estimate:
\begin{align*}&||X_{t,\lambda}||_2^2=||X_{t_i}||_2^2+2\int_{t_i}^tds\langle LX_{s,\lambda} -\nabla H_{s,\lambda}(X_{s,\lambda}),X_{s,\lambda}\rangle+(t-t_i)\\&\leq ||X_{t_i}||_2^2+2\int_{t_i}^tds L||X_{s,\lambda}||_2^2+D||X_{s,\lambda}||_2+(t-t_i)
\\&\leq ||X_{t_i}||_2^2+\int_{t_i}^tds (2L+1)||X_{s,\lambda}||_2^2 +(D+1)(t-t_i)
\\&\leq e^{(2L+1)(t-t_i)} \left(||X_{t_i}||_2^2+(D+1)(t-t_i)\right)=:E(t),\end{align*}
%\begin{align*}&||X_{t,\lambda}||_2^2=||X_{t_i}||_2^2+2\int_{t_i}^tds\langle LX_{s,\lambda} -\nabla H_{s,\lambda}(X_{s,\lambda}),X_{s,\lambda}\rangle+(t-t_i)\\&\leq ||X_{t_i}||_2^2+2\int_{t_i}^tds L||X_{s,\lambda}||_2^2+(2C||J_{s,\lambda}(0)||_2+D)||X_{s,\lambda}||_2+(t-t_i)
%\\&\leq ||X_{t_i}||_2^2+2\int_{t_i}^tds (4C\sqrt{|c|+H_s(0)}+D)^2 +(L+1)||X_{s,\lambda}||_2^2+(t-t_i)
%\\&\leq e^{(L+1)(t-t_i)} \left(||X_{t_i}||_2^2+(t-t_i)(1+ (32|c|(2+||(X_{t_1},...X_{t_i})||_2^2)+2D^2))\right)=:E(t),\end{align*}
The last inequality comes from Gronwall's lemma and we used: $$
\langle -\nabla H_{s,\lambda}(X_{s,\lambda}),X_{s,\lambda}\rangle=-\lambda ||\nabla H_{s,\lambda}(X_{s,\lambda})||_2^2+ 
\langle -\nabla H_{s}J_{s,\lambda}(X_{s,\lambda}),J_{s,\lambda}X_{s,\lambda}\rangle\leq \langle -\nabla H_{s}(0),J_{s,\lambda}X_{s,\lambda}\rangle$$
from the identities $\nabla H_{s,\lambda}=\nabla H_{s}J_{s,\lambda}$,$\lambda\nabla H_{s,\lambda}(x)+J_{s,\lambda}(x)=x$ and   by monotony of the gradient $\nabla H_{s}$ (coming from $H_s$ convex).

%where the inequality comes from the lipschitzness of $h_t$, $\langle -\nabla H_{s,\lambda}(X_{s,\lambda}),X_{s,\lambda}\rangle\leq \langle -\nabla H_{s,\lambda}(0),X_{s,\lambda}\rangle$ by monotony of the gradient  $\nabla H_{s,\lambda}=\nabla H_{s}J_{s,\lambda}$ for the canonical contraction $J_{s,\lambda}=(I+\lambda \nabla H_s)^{-1}$ satisfying the bound \eqref{boundJlambda}. 
% and we defined $C(X_{t_1},...,X_{t_i})=2(2C+D+4C(2|c|+|c|\ ||(X_{t_1},...X_{t_i})||_2^2)+2C||(X_{t_1},...X_{t_i})||_2)$.

\begin{step}Convergence in $\lambda$ in $L^2$.\end{step}
Then, one deduces similarly for $t\in [t_i,t_{i+1}]$:
$$||X_{t,\lambda}-X_{t,\mu}||_2^2=||X_{t_i,\lambda}-X_{t_i,\mu}||_2^2+2\int_{t_i}^tds \ \left(\ L||X_{s,\lambda}-X_{s,\mu}||_2^2 -\langle \nabla H_{s,\lambda}(X_{s,\lambda})-\nabla H_{s,\mu}(X_{s,\mu}),X_{s,\lambda}-X_{s,\mu}\rangle\right) $$
Now, note that, from proposition \ref{Yosida}, if we call $J_{s,\lambda}=(I+\lambda \nabla H_s)^{-1}$ the contraction, we have $\nabla H_{s,\lambda}=\nabla H_s\circ J_{s,\lambda}$ and $\lambda\nabla H_{s,\lambda}(x)+J_{s,\lambda}(x)=x$, thus we can decompose 
\begin{align*}&\langle \nabla H_{s,\lambda}(X_{s,\lambda})-\nabla H_{s,\mu}(X_{s,\mu}),X_{s,\lambda}-X_{s,\mu}\rangle\\&\qquad =
\langle \nabla H_s( J_{s,\lambda}(X_{s,\lambda}))-\nabla H_s( J_{s,\mu}(X_{s,\mu})),J_{s,\lambda}(X_{s,\lambda})-J_{s,\mu}(X_{s,\mu})\rangle \\&\qquad+\langle \nabla H_{s,\lambda}(X_{s,\lambda})-\nabla H_{s,\mu}(X_{s,\mu}),\lambda\nabla H_{s,\lambda}(X_{s,\lambda})-\mu\nabla H_{s,\mu}(X_{s,\mu})\rangle\\&\qquad \geq \langle \nabla H_{s,\lambda}(X_{s,\lambda})-\nabla H_{s,\mu}(X_{s,\mu}),\lambda\nabla H_{s,\lambda}(X_{s,\lambda})-\mu\nabla H_{s,\mu}(X_{s,\mu})\rangle
\end{align*}
where we use the convexity of $H$ in the inequality. Thus one gets in using more the relations above (and the inequality $|ab|\leq a^2/4+b^2$) :
\begin{align*}&-\langle \nabla H_{s,\lambda}(X_{s,\lambda})-\nabla H_{s,\mu}(X_{s,\mu}),X_{s,\lambda}-X_{s,\mu}\rangle\\& \leq -\lambda|| \nabla H_{s,\lambda}(X_{s,\lambda})||^2 -\mu ||\nabla H_{s,\mu}(X_{s,\mu})||^2+(\lambda+\mu)||\nabla H_{s,\lambda}(X_{s,\lambda})||\ ||\nabla H_{s,\mu}(X_{s,\mu})||\\&\leq \frac{\lambda}{4}||\nabla H_{s,\mu}(X_{s,\mu})||^2+\frac{\mu}{4}||\nabla H_{s,\lambda}(X_{s,\lambda})||^2.
\end{align*}
Finally, from the Lipschitzness of $H$, one gets from the contractivity of $J_{s,\mu}$: \begin{align*}||\nabla H_{s,\mu}(X_{s,\mu})||^2%=||\nabla H_{s}(J_{s,\mu}(X_{s,\mu})||^2
&\leq (2C||(X_{t_1},...X_{t_i},J_{s,\mu}(X_{s,\mu}))||+D)^2\\&\leq (2C||(X_{t_1},...X_{t_i},X_{s,\mu})||+2C||J_{s,\mu}(0)||+D)^2\end{align*}
Gathering all our estimates and using again \eqref{boundJlambda}, we have thus obtained the inequality (for $\mu,\lambda\leq 1$):
\begin{align*}
&||X_{t,\lambda}-X_{t,\mu}||_2^2\leq e^{2L(t-t_i)}||X_{t_i,\lambda}-X_{t_i,\mu}||_2^2+\frac{\lambda+\mu}{2} e^{2L(t-t_i)}(t-t_i)\\&\times (8C^2\sum_{j=1}^i||X_{t_j}||_2^2+8C^2 E(t_{i+1})+4D^2+32C^2(\sup H_s(0)+|c|)).
\end{align*}
Thus $X_{t,\lambda}$ is Cauchy in $\lambda$ and by induction on $i$ converges in $C^0(|t_i,t_{i+1}],L^2_{sa}(M)^m)$ to $X_t$ such that $X_t\in L^2_{sa}(M_t)^m.$

\begin{step}Checking the limiting equation.\end{step}

Moreover, using \eqref{Holderhtspace}, triangular inequality and the definition of $J_{s,\mu}$, one gets:
\begin{align*}&||\nabla H_{s,\mu}(X_{s,\mu})-\nabla H_{s}(X_{s})||_2= ||\nabla H_{s}J_{s,\mu}(X_{s,\mu})-\nabla H_{s}(X_{s})||_2
\\&\leq ||J_{s,\mu}(X_{s,\mu})-X_{s}||_2^\beta (C+D||J_{s,\mu}(X_{s,\mu})||_2+D||X_{s}||_2)
\\&\leq \left(\mu^\beta||\nabla H_{s,\mu}(X_{s,\mu})||_2^\beta +||X_{s,\mu}-X_{s}||_2^\beta\right) (C+D||X_{s,\mu}||_2+D||J_{s,\mu}(0)||_2+D||X_{s}||_2)\to_{\mu\to 0} 0\end{align*}
uniformly on $[t_i,t_{i+1}]$ from the bounds in the previous steps. We can now take the limit in the SDE satisfied by $X_{s,\mu}$ to get the expected SDE for $X_{s}.$
\end{proof}
\setcounter{Step}{0}
We now apply this result in the spirit of \cite{GuionnetS07} in order to get limit states of convex potential matrix models with limited regularity.

\begin{theorem}\label{SDVg}
Let $g\in \mathcal{E}^{1,1}_{app}( \mathcal{T}_{2,0}(\mathcal{F}^m_{1}*\mathcal{F}^\nu_{\mu}),d_{2,0})$ %{\color{red}[EXTEND TO $\mathcal{E}^{1,\alpha}$ ]}%satisfying for $c_g>0$ the strict convexity condition for any $X,Y$ in  any $L^2_{sa}(M,\tau)^m, u\in\mathcal{U}(M)^{\mu\nu}$ (any tracial $(M,\tau)$):
%$$g(\tau_{X+Y,u})+g(\tau_{X-Y,u})-2g(\tau_{X,u})\geq c_g||Y||_2^2$$
 and consider, for $\Upsilon_N\in \mathcal{U}(M_N(\C))$, the law absolutely continuous with respect to the law $P_{G^N}$ of GUE $G^N$:$$d\mu_{g,N}(X)=\frac{1}{Z_{g,N,\Upsilon_N}}e^{-N^2g(\tau_{X,\Upsilon_N})}dP_{G^N}(X).$$
Assume finally that the non-commutative law $\tau_{\Upsilon_N}$ converges to some $\mu_\Upsilon\in (\mathcal{T}(\mathcal{F}^\nu_{\mu}),d)$. Then $E_{\mu_{g,N}}\circ\tau_{.,\Upsilon_N}$ converges  in $(\mathcal{T}_{2,0}(\mathcal{F}^m_{1}*\mathcal{F}^\nu_{\mu}),d_{2,0} )$ to a tracial state $\tau_g$ which is law of self-adjoint and unitary variables $X(g),u$ (of norm bounded by some $R$) and the unique solution, such that the law of $u$ is $\mu_\Upsilon$, to the equation $(SD_g)$, for $G(X)=g(\tau_{X,u})$:
$$ \forall P\in \C\langle X_1,...,X_n,u_1^1,...,u_\mu^\nu\rangle , \ (\tau_g\otimes \tau_g)(\partial_{X_i}(P))=\tau_g(X_iP)+d_{X_i}G(X(g)).P(X).$$
Moreover, there is a solution on $\R_+$ given by the theorem \ref{FreeMonotone} for $h_t(x_1,...,x_{i+1})=g(\tau_{x_{i+1},u})+g_{2,\mathbf{(1)}}(\tau),L=0$ for all $t$ (with $\mathbf{(1)}$ the list of times with only one time equal to 1)., it satisfies for the solutions $X_t(X),X_t(Y)$ with initial condition $X,Y$:
\begin{equation}\label{expdecay}||X_t(X)-X_t(Y)||_2^2\leq e^{-t}||X_0(X)-X_0(Y)||_2^2\end{equation}
and $\tau_g$ is the unique stationary state for this free SDE. 
\end{theorem}
\begin{proof}
\begin{step}Defining limit variables in a von Neumann algebra ultraproduct.\end{step}
Let law $\mu_{g,N}$ be the marginal at time $1$ of a law considered in Proposition \ref{ConcentrationNorm}.
Consider a non-principal ultrafilter $\omega$ on $\N$ and the ultraproducts $\mathcal{L}^\omega=L^2(M_N(L^\infty(\mu_{g,N}))^\omega$, $\mathcal{M}^\omega=M_N(L^\infty(\mu_{g,N}))^\omega$ (tracial von Neumann algebra ultraproduct). Considering $A_1^N,...,A_n^N$ the canonical hermitian variables in $M_N(L^\infty(\mu_{g,N}))$, we know from \eqref{IntegralUnifBound} that $||A_i^N1_{||A_i^N||\leq C}- A_i^N||_2\to 0$ so that $X_i^\omega=(A_i^N)^\omega=(A_i^N1_{\{||A_i^N||\leq C\}})^\omega\in \mathcal{M}^\omega$. We thus also fix $B_i^N=A_i^N1_{||A_i^N||\leq C}.$ We can also consider $u_i^j=((\Upsilon_N)_i^j)^\omega\in \mathcal{M}^\omega.$ Of course, $u$ has law $\mu_\Upsilon$.

This gives a tracial state $\tau_{X^\omega,u}\in \mathcal{S}_C^m\star \mathcal{T}_\mu^\nu.$  Let us check that any such state satisfies $(SD_g)$. 
%\begin{step}Limits for $G$.\end{step}

\begin{step}Showing $(SD_g)$.\end{step}
First, note that for $P_1,...,P_m\in \C\langle X_1,...,X_m,u_1^1,...,u_\mu^\nu\rangle$,
\begin{align*}\lim_{N\to\omega}&E_{\mu_{g,N}}(g(\tau_{B_1^N+P_1(B_1^N,...,B_m^N,\Upsilon_N),...,B_m^N+P_m(B_1^N,...,B_m^N,\Upsilon_N),\Upsilon_N})
\\&=G(X_1^\omega+P_1(X_1^\omega,...,X_m^\omega)
,...,X_m^\omega+P_m(X_1^\omega,...,X_m^\omega,u,u)).\end{align*}
Indeed, from the lipschitzness of $g$ for the metric $d_2$ it is easy to see that $$g(\tau_{X_1+P_1(X_1,...,X_m),...,X_m+P_1(X_1,...,X_m),u})=:G_P(\mu_{X,u})$$ is uniformly continuous on $S_C^m\star \mathcal{T}_\mu^\nu$ for the induced metric $D_2(\mu,\nu)=d_{2,0}(\tau_{X(\mu),u(\mu)},\tau_{X(\nu),u(\nu)})$ and from the second concentration \eqref{ConcentrationLogSob} in Proposition \ref{ConcentrationNorm}, $E_{\mu_{g,N}}(d_{2,0}(\tau_{B^N,\Upsilon_N},E_{\mu_{g,N}}(\tau_{B^N,\Upsilon_N}))\to_{N\to \infty} 0$, so that $\lim_{N\to \omega} E_{\mu_{g,N}}(d_{2,0}(\tau_{B^N,\Upsilon_N},\tau_{X^\omega,u}))=0$ and as a consequence for any $\eta>0$  $$\lim_{N\to \omega}P(d_{2,0}(\tau_{B^N,\Upsilon_N},\tau_{X^\omega,u})>\eta)=0.$$
If we fix $\epsilon>0$ and find $ \eta>0$ such that if $d_{2,0}(\tau_{X(\mu),u(\mu)},\tau_{X(\nu),u(\nu)})\leq \eta$ then $|G_P(\tau_{X(\mu),u(\mu)})-G(\tau_{X(\nu),u(\nu)})|\leq \epsilon$ one deduces as claimed that:
\begin{align*}&\lim_{N\to \omega}\left|E_{\mu_{g,N}}\left(g(\tau_{B_1^N+P_1(B_1^N,...,B_m^N),...,B_m^N+P_m(B_1^N,...,B_m^N),\Upsilon_N}))
\right.\right.\\&\qquad\qquad\left.\left.-G(X_1^\omega+P_1(X_1^\omega,...,X_m^\omega,u)
,...,X_n^\omega+P_n(X_1^\omega,...,X_m^\omega,u)\right)\right|\leq \epsilon+\lim_{N\to \omega}
\\&
\sqrt{P(d_2(\tau_{B^N,\Upsilon_N},\tau_{X^\omega,u})>\eta)(c+E(G^2(B_1^N+P_1(B_1^N,...,B_m^N,\Upsilon_N),...,B_m^N+P_1(B_1^N,...,B_m^N,\Upsilon_N),\Upsilon_N))}=\epsilon\end{align*}
where the last equality comes from the previously found convergence in probability and the subquadratic behaviour of $G$ in conjunction with $||B_i^N||\leq C$ and $c>0$ is another constant.

As in \cite{GuionnetMaurelSegala}, we use an integration by parts formula on $\mu_{g,N}$ which gives $\forall P\in \C\langle X_1,...,X_m,u_1^1,...,u_\mu^\nu\rangle$:
\begin{align*}\ E_{\mu_{g,N}}&\left(\frac{1}{N}Tr(A_i^NP(A_1^N,...,A_m^N,\Upsilon_N))+ \frac{1}{N}Tr(N\nabla_{A_i^N}G(A_1^N,...,A_m^N)P(A_1^N,...,A_m^N,\Upsilon_N))\right)\\&=E_{\mu_{g,N}}\left((\frac{1}{N}Tr\otimes\frac{1}{N}Tr)(\partial_{X_i}P)(A_1^N,...,A_m^N,\Upsilon_N)\right)\end{align*}
and the second concentration result in Proposition \ref{ConcentrationNorm} implies that the right hand side converges when $N\to \omega$ to $(\tau_{X^\omega,u}\otimes \tau_{X^\omega,u})(\partial_i(P))$.

But from the lipschitzness condition \eqref{Lop0} on $G$, on deduces 
$$Tr(\nabla_{A_i^N}G(A_1^N,...,A_m^N)^*\nabla_{A_i^N}G(A_1^N,...,A_m^N))
\leq\frac{C}{N}(1+\frac{1}{N}Tr(\sum_{i=1}^m(A_i^N)^2)
$$

and thus if $Z_i^N= N\nabla_{A_i^N}G(A_1^N,...,A_m^N)$ one gets $E(||Z_i^N||_2^2)\leq C(1+E(\frac{1}{N}Tr(\sum_{i=1}^m(A_i^N)^2))$, so that $Z=(Z^N)^\omega\in \mathcal{L}^\omega$ and one obtains the relation in taking of limit to $\omega$ of the integration by parts relation:
\begin{equation}\label{SDVgeq}\langle X_i^\omega+Z_i^*,P(X_1^\omega,...,X_n^\omega,u)\rangle=(\tau_{X^\omega,u}\otimes \tau_{X^\omega,u})(\partial_{X_i}(P)).\end{equation}

But since from the definition of the gradient \begin{align*}&\left|E_{\mu_{g,N}}\left(G(B_1^N+tP_1(B_1^N,...,B_m^N,\Upsilon_N),...,B_m^N+
tP_m(B_1^N,...,B_m^N,\Upsilon_N)-G(B_1^N,...,B_m^N)\right)\right.\\&\left.-tE_{\mu_{g,N}}(\sum_{i=1}^m\frac{1}{N}Tr(Z_i^NP_1(B_i^N,...,B_m^N,\Upsilon_N))\right|\leq t^2\sum_{i=1}^m\frac{c}{N}Tr(P_i^2(B_1^N,...,B_m^N,\Upsilon_N))\end{align*}
Thus taking the limit $N\to\Omega$ one deduces $$\langle Z_i^*,P(X_1^\omega,...,X_m^\omega,u)\rangle=
d_{X_i^{\omega}}G(X_1^\omega,...,X_m^\omega,u).
P(X_1^\omega,...,X_m^\omega,u).$$
This shows \eqref{SDVgeq} was the equation $(SD_g)$ we were aiming at.
Moreover, note that this implies $\tau_{X^\omega,u}$ has finite Fisher information.

\begin{step}Properties and use of the SDE.\end{step}
Since $h_t$ does not depend on time, assumption \eqref{Holderht} is obvious and all the remaining assumptions in Theorem \ref{FreeMonotone} are contained in $g\in \mathcal{E}^{1,1}_{app}( \mathcal{T}_{2,0}(\mathcal{F}^m_{1}*\mathcal{F}^\nu_{\mu}),d_{2,0})$ (with $\beta=1,D=0$ in \eqref{Holderhtspace} using proposition \ref{C11}). Note also that the approximation property in the definition  implies $\nabla_xg(\tau_{x,u})\in L^2(W^*(x,u)).$

The application of our Theorem thus gives a unique solution $X_t(X_0)$ on $[0,\infty[$ solving
$$X_t(X_0)=X_0-\int_0^t \nabla G(X_s(X_0)) ds-\int_0^t X_s(X_0) ds +S_t.$$
Considering another solution starting at $Y_0$, one obtains :

\begin{align*}||X_t(X_0)-X_t(Y_0)||_2^2&=||X_0-Y_0||_2^2-\int_0^t \langle \nabla G(X_s(X_0))-\nabla G(X_s(Y_0)),X_s(X_0)-X_s(Y_0) \rangle ds \\&-\int_0^t \langle X_s(X_0)-X_s(Y_0),X_s(X_0)-X_s(Y_0) \rangle ds\\&\leq ||X_0-Y_0||_2^2-\int_0^t \langle X_s(X_0)-X_s(Y_0),X_s(X_0)-X_s(Y_0) \rangle ds,\end{align*}
where the last inequality comes from  $g(\tau_{X,u})$ convex. Applying Gronwall's lemma, one gets the stated exponential decay. But \cite[Theorem 28]{Dab10b} since $\tau_{X^\omega,u}$ has finite Fisher information, there is a stationary solution to the same equation. But by the uniqueness of our solution in Theorem \ref{FreeMonotone}, the solution must be this same stationary process. But exponential decay implies that the laws  $\tau_{X_t(X^\omega)}$ and
$\tau_{X_t(X^{\omega'})}$ are arbitrarily close for $t\to\infty$ and since they are equal to $\tau_{X^\omega}$ and
$\tau_{X^{\omega'}}$ by stationarity, one deduces that $X^\omega$ have the same law for any ultrafilter. Similarly, $(SD_g)$ has a unique solution and the exponential decay implies a stationary state for the SDE is unique too.

\begin{step}Conclusion on the limit of $E_{\mu_{g,N}}\circ\tau_.$.\end{step}
The law $E_{\mu_{g,N}}\circ\tau_{.,\Upsilon_N}$ is close to $E_{\mu_{g,N}}\circ\tau_{{B^N,\Upsilon_N}}$ for $N$ large enough and this second law lies in the compact set $S_C^m*\mathcal{F}^\nu_{\mu}$ (for the weak-* topology induced by $d_{2,0}$) and from the result on ultrafilter limits the sequence has a unique limit point there (any such limit point being a $\tau_{X^\omega,u}$). We thus deduce by compactness the expected convergence.

\end{proof}

\section{Minimization in Bou\'e-Dupuis-\"Ust\"unel formula for hermitian brownian motion}  

The key for our large deviation estimate is to use the results from \cite{Ustunel} on the minimization problem in Bou\'e-Dupuis-\"Ust\"unel formula (applied to hermitian brownian motion) for specifically nice functionals coming from $\mathcal{E}_\alpha(\R^{dk})$, and deduce an equivalent minimization problem better suited to take the large $N$ limit. Note that we make a sign correction in the SDE agreeing with \cite[Theorems 2,4]{Lehec}, which differs from \cite{Ustunel}.

\begin{theorem}[Theorem 11 in \cite{Ustunel}]
\label{conv-1}
Assume that $f\in L^0(\gamma)$ is $1$-convex and that $f^-=\max(-f,0)$ is
exponentially integrable, i.e., 
$E[\exp cf^-]<\infty$ for some $c>1$. Then there is a unique  $u\in L^2_a(\gamma,H)$ reaching the infimum appearing  in the definition of tame functionals, provided that
$E[f\circ (I_W+\xi)]<\infty$ for at least one $\xi\in
L_a^2(\gamma,H)$. Moreover, if $f\in L^{1+\epsilon}(\gamma)$ for  $\epsilon=\frac{1}{c-1}>0$,
then $f$ is a tamed functional and if
$$
\dot{v}_t=-\frac{E[D_te^{-f}|\mathcal F_t]}{E[e^{-f}|\mathcal F_t]}
$$
 (where formally $E[D_te^{-f}|\mathcal F_t]=[\hat{\pi}\nabla(e^{-f})]_t$ as in subsection 2.3), then $U_t=B_t+u_t$  is the
unique strong solution of the following stochastic differential equation:
$$
dU_t=-\dot{v}_t\circ U dt+dB_t\,.
$$
As a consequence, for this solution $U^f:=U$, we have:\[
-\log \left( \int_{\mathbb W} e^{-f} \ \mathrm{d} \gamma \right)=\mathbf{E}  \left(
  f(U)+\frac{1}{2}\int_0^1\|\dot{v}_t\circ U\|_2^2dt\right).
\]
\end{theorem}
We deduce from that crucial description of the minimizer in the convex case the result we need :

\begin{corollary}\label{UstunelFinal}
Let $t_0=0<t_1<...<t_k\in [0,1]$, $g\in \mathcal{E}_\alpha(\R^{dk}),\alpha>1$ %and% for $\ell\in ]-\infty,\frac{1}{2}[$
: $$f=g\circ J_{t_1,...,t_k}%-\ell \mathfrak{c}_1
:\mathbb W\to \R.$$ Let $b(t,.)$ defined in corollary \ref{MonotoneSol} and $X(t)$ the unique strong solution starting at $X_0=0$ defined there of \[X(t) = X_0 +
\int_0^t b(s, X(s)) ds +
B_t.\]
Then, we have the formula:
\[
-\log \left( \int_{\mathbb W} e^{-f} \ \mathrm{d} \gamma \right)=\mathbf{E}  \left(
  f(X)+\frac{1}{2}\int_0^1\|b(t,X_t)\|_2^2dt\right).
\]
\end{corollary}
\begin{proof} Since $g$ is convex, $f$ is $0$-convex thus $1$-convex and since $g$ is bounded from bellow, $f^-$ is exponentially integrable. If $\xi=0$, the subquadratic bound  of $g$ implies $E[f\circ (I_W+\xi)]<\infty$ and $f\in L^{1+\epsilon}(\gamma)$ for any $\epsilon>0$. Thus, the previous theorem applies, $f$ is a tame functional. It essentially remains to identify $X(t)=U^f_t$ with the solution considered in the theorem. 

For, recall that by definition, we have the inductive definition (using the solution $X_{t_k}(\omega)$ for small times with $k\leq i$) for $t_i<t\leq t_{i+1}$ : \[b_j(t,x,\omega)=-\frac{\partial}{\partial x^{(j)}} h_t(X_{t_1}(\omega),...,X_{t_i}(\omega),x),\]
and we have from the proof of proposition \ref{OptimalValueFct}: \[h_t(x_1,...,x_i,x)=-\log \left( \int_{\mathbb W_{[t,1]}} e^{-g(x_1,...,x_i,x+\nu_{t_{i+1}},...,x+\nu_{t_{k}})} \ \mathrm{d} \gamma_{[t,1]}(\nu) \right).\]

From the differentability of $g\in \mathcal{E}_\alpha(\R^{dk}),\alpha>1$ and the lipschitzness bound \eqref{lipE} implying the boundedness of derivatives. We have the bound :\begin{align*}&\left|\frac{\partial}{\partial x^{(j)}}e^{-g(x_1,...,x_i,x+\nu_{t_{i+1}},...,x+\nu_{t_{k}})}\right|\\&\qquad=e^{-g(x_1,...,x_i,x+\nu_{t_{i+1}},...,x+\nu_{t_{k}})}\left|\frac{\partial}{\partial x^{(j)}}(g(x_1,...,x_i,x+\nu_{t_{i+1}},...,x+\nu_{t_{k}}))\right|\\&\qquad\leq e^{\sup(-g)}\sum_{l=i+1}^k\left(2C||x_1,...,x_i,x+\nu_{t_{i+1}},...,x+\nu_{t_{k}})||_2+D\sqrt{d}\right).
\end{align*}
which is an integrable dominating function on compact sets for $x$, so that one can compute the derivative by derivation of integral depending on parameters under Lebesgue domination condition and using that $$\left(\frac{\partial}{\partial x^{(j)}}e^{-g(B_{t_1},...,B_{t_i},x+\nu_{t_{i+1}},...,x+\nu_{t_{k}})}\right)_{x=B_t}=D_t(e^{-g\circ J_{t_1,...,t_k}})]_{\omega=B},$$ one gets the identity (taking $D_t$ with respect to the process $\widetilde{\omega}+\nu$ in path space):
\begin{equation}\label{bjasderivative}
 b_j(t,X_t(\omega),\omega)=\frac{\left( \int_{\mathbb W_{[t,1]}}D_t e^{-g(\widetilde{\omega}_{t_1},...,\widetilde{\omega}_{t_i}(\omega),\widetilde{\omega}_t+\nu_{t_{i+1}},...,\widetilde{\omega}_t+\nu_{t_{k}})} \ \mathrm{d} \gamma_{[t,1]}(\nu) \right)_{\widetilde{\omega}=X_.(\omega)}}{\left( \int_{\mathbb W_{[t,1]}} e^{-g(X_{t_1}(\omega),...,X_{t_i}(\omega),x+\nu_{t_{i+1}},...,x+\nu_{t_{k}})} \ \mathrm{d} \gamma_{[t,1]}(\nu) \right)},
\end{equation}
which is rewritten $b(t,X_t(\omega),\omega)=-\dot{v}_t\circ X$ and thus implies that the $U$ of the theorem, satisfies the same equation as $X$, thus, $U^f=X$ and from the equality of the drift, one obtains the stated equality of integrals reaching the infimum.
\end{proof}

We finally apply our results to hermitian brownian motion of section \ref{hermitianB}. We start by a lemma relating our various classes of functionals.

\begin{lemma}\label{UniformN}
Let $p\in [2,\infty[, d=N^2m$ and $G\in\mathcal{E}_{reg,p}(  \mathcal{T}_{2,0}(\mathcal{F}^m_{k}*\mathcal{F}^{\nu}_{\mu}),d_{2,0}) $. Seeing $\R^{dk}=((M_N(\C)_{sa})^m)^k$,  for $x=(H_1,...,H_k)\in \R^{dk}, H_i\in (M_N(\C)_{sa})^m$, and $\Upsilon_N\in\mathcal{U}((M_N(\C))^{\mu\nu}$, we let $\tau_{x,\Upsilon_N}\in \mathcal{T}_{2,0}(\mathcal{F}^m_{k}*\mathcal{F}^{\nu}_{\mu}))$ the corresponding state. Then $g:x\mapsto N^2G(\tau_{x/\sqrt{N},\Upsilon_N})\in \mathcal{E}_{\alpha}(\R^{dk})$ for any $\alpha\in ]1,2]$ with constants independent of $N,\Upsilon_N$ in \eqref{SubquadE}, \eqref{lipE} and \eqref{CalphaE}. In case $\alpha=2$, one  can even take $D_2=0$.  The same result holds for each fixed $\alpha$ for $G\in \mathcal{E}^{1,\alpha-1}(  \mathcal{T}_{2,0}(\mathcal{F}^m_{k}*\mathcal{F}^{\nu}_{\mu}),d_{2,0}) )$.
\end{lemma}
\begin{proof}
The case $G\in \mathcal{E}^{1,\alpha-1}(  \mathcal{T}_{2,0}(\mathcal{F}^m_{k}*\mathcal{F}^{\nu}_{\mu}),d_{2,0}) )$ is immediate by change of variable: \eqref{SubquadE} comes from \eqref{subquadorder0}, \eqref{lipE} comes from \eqref{Lop0} and \eqref{CalphaE} from \eqref{C110}. 

The convexity of $g$ comes from the universal convexity of $G$ and the continuity from the one of $G$ once noted that $x\mapsto \tau_{x/\sqrt{N},\Upsilon_N}$ is continuous with value $(\mathcal{T}_{2,0}(\mathcal{F}^m_{k}*\mathcal{F}^{\nu}_{\mu}),d_{2,0}).$ 

From the subquadratic behaviour in the definition of $\mathcal{E}(  \mathcal{T}_{2,0}(\mathcal{F}^m_{k}*\mathcal{F}^{\nu}_{\mu}),d_{2,0}), $ one gets:
$$ g(x)=N^2G(\tau_{x/\sqrt{N}})\leq C (N^2+ N^2\sum_{j=1}^k \sum_{l=1}^m\frac{1}{N^2}Tr( ((4i\frac{u_j^l+1}{u_j^l-1})^*(4i\frac{u_j^l+1}{u_j^l-1})))=C (N^2+||x||_2^2),$$
as expected with a constant $C$ independent of $N.$
Moreover, for $G(\tau)=D+\left(\sum_{i=1,...,l}\left(g_i(\tau)\right)^p\right)^{1/p}$ as in the definition of $\mathcal{E}_{reg,p}(  \mathcal{T}_{2,0}(\mathcal{F}^m_{k}*\mathcal{F}^{\nu}_{\mu}),d_{2,0})$ one can differentiate in $x=(x_{i,j}^k)_{i,j=1,....,N; k=1,...,m}$\footnote{\label{note1}so that the $M_n(\C)$ matrix entries are $x_{i,i}^k$ and $z_{i,j}^k=x_{i,j}^k+\sqrt{-1}x_{j,i}^k$ for $i<j$ in index $(i,j)$ and deduced by hermitianity in index $(j,i)$, and we call $$\lambda_{i,j}(x^k)=x_{i,j}^k=\sqrt{(-1)^{1_{j<i}}}(z_{i,j}^k+(-1)^{1_{j<i}}z_{j,i}^k)/2$$ the linear application realizing this choice of matrix entries for $i\neq j$ and $\lambda_{i,i}$ the matrix entry also defined on non-hermitian matrices if necessary in that way.} since the values of $g_i(\tau)\geq1$ are not close of the point where the root is not differentiable: 
\begin{align*}&\left\|\left(\frac{\partial}{\partial x_{i,j}^k}g(x)\right)_{i,j,k}\right\|\\&=N^2\frac{1}{p}\left(\sum_{\iota=1,...,l}\left(g_\iota(\tau_{x/\sqrt{N},\Upsilon_N})\right)^p\right)^{(1-p)/p}
\left\|\sum_{\iota=1,...,l}p\left(g_\iota(\tau_{x/\sqrt{N},\Upsilon_N})\right)^{p-1}%\frac{1}{\sqrt{N}}
\frac{\partial}{\partial x_{i,j}^k}g_\iota(\tau_{x/\sqrt{N},\Upsilon_N})\right\|
\\&\leq N^2
\left(\sum_{\iota=1,...,l}
\left\|\left(\frac{\partial}{\partial x_{i,j}^k}g_\iota(\tau_{x/\sqrt{N},\Upsilon_N})\right)_{i,j,k}\right\|^{p}\right)^{1/p}\end{align*}

where we applied H\"older inequality of exponents $p$, $q=p/(p-1)$. 
Then, recall that $g_\iota(\tau)=D_\iota+C_\iota \sum_{j=1}^k \sum_{l=1}^m\tau ((4i\frac{u_j^l+1}{u_j^l-1})^*(4i\frac{u_j^l+1}{u_j^l-1}))+\Re(\lambda_\iota\tau((u_{j_1}^{i_1})^{\epsilon_1}...(u_{j_{m_\iota}}^{i_{m_\iota}})^{\epsilon_{m_\iota}}))$ so that as in \cite[section 3.2]{BCG}, one gets:
\[\frac{\partial}{\partial x_{i,j}^k}g_\iota(\tau_{x/\sqrt{N},\Upsilon_N})=\frac{2C_\iota}{N^2}x_{i,j}^k +\Re\left(\frac{2C_\iota\lambda_\iota}{N\sqrt{N}}\lambda_{i,j}\left(\sum_{l=1}^{m_\iota}((u_{j_l}^{i_l})^{\epsilon_l}-1) ... (u_{j_{m_\iota}}^{i_{m_\iota}})^{\epsilon_{m_\iota}}(u_{j_1}^{i_1})^{\epsilon_1}...((u_{j_l}^{i_l})^{\epsilon_l}-1)\frac{\epsilon_l}{8\sqrt{-1}}1_{k=j_l}\right)\right)
\]

so that one can estimate $\left\|\left(\frac{\partial}{\partial x_{i,j}^k}g_\iota(\tau_{x/\sqrt{N},\Upsilon_N})\right)_{i,j,k}\right\|^2\leq \frac{8C_i^2}{N^4}||x||^2+\frac{8C_\iota^2|\lambda_\iota|^2m_\iota^2Nm}{4N^3}$ and get finally :
\begin{align*}&\left\|\left(\frac{\partial}{\partial x_{i,j}^k}g(x)\right)_{i,j,k}\right\|\leq \max_\iota(8C_\iota^2||x||^2+2C_\iota^2|\lambda_\iota|^2m_\iota^2N^2m)^{1/2}l^{1/p}\end{align*}
Thus applying the fundamental theorem of calculus, one gets \eqref{lipE} with $d=N^2m$ and constants $C,D$ independent on $N$.

To get the  second order bound, we are going to compute the second order derivative.
\begin{align*}&\left|\sum_{i,j,k,I,J,K}\lambda_{i,j,k}\lambda_{I,J,K}\left(\frac{\partial^2}{\partial x_{i,j}^k\partial x_{I,J}^K}g(x)\right)\right|\leq N^2\frac{1}{p}\left(\sum_{\iota=1,...,l}\left(g_\iota(\tau_{x/\sqrt{N},\Upsilon_N})\right)^p\right)^{1/p-1}
\\&\times\left[\ \left|\sum_{i,j,k,I,J,K}\lambda_{i,j,k}\lambda_{I,J,K}\sum_{\iota=1,...,l}p(p-1)\left(g_\iota(\tau_{x/\sqrt{N},\Upsilon_N})\right)^{p-2}
\frac{\partial}{\partial x_{i,j}^k}g_\iota(\tau_{x/\sqrt{N},\Upsilon_N})\frac{\partial}{\partial x_{I,J}^K}g_\iota(\tau_{x/\sqrt{N},\Upsilon_N})\right|\right.\\&+ %N^2\frac{1}{p}
\left|\frac{1}{p}-1\right|\left(\sum_{\iota=1,...,l}\left(g_\iota(\tau_{x/\sqrt{N},\Upsilon_N})\right)^p\right)^{-1}
\left\|\sum_{\iota=1,...,l}p\left(g_\iota(\tau_{x/\sqrt{N},\Upsilon_N})\right)^{p-1}%\frac{1}{\sqrt{N}}
\frac{\partial}{\partial x_{i,j}^k}g_\iota(\tau_{x/\sqrt{N},\Upsilon_N})\right\|^2||\lambda||^2
\\&+\left.%N^2\frac{1}{p}\left(\sum_{\iota=1,...,l}\left(g_\iota(\tau_{x/\sqrt{N},\Upsilon_N})\right)^p\right)^{1/p-1}
\left|\sum_{\iota=1,...,l}p\left(g_\iota(\tau_{x/\sqrt{N},\Upsilon_N})\right)^{p-1}%\frac{1}{\sqrt{N}}
\sum_{i,j,k,I,J,K}\lambda_{i,j,k}\lambda_{I,J,K}\frac{\partial^2}{\partial x_{i,j}^k\partial x_{I,J}^K}g_\iota(\tau_{x/\sqrt{N},\Upsilon_N})\right|\ \right]
%\\&\leq N^2
%\left(\sum_{\iota=1,...,l}
%\left\|\left(\frac{\partial}{\partial x_{i,j}^k}g_\iota(\tau_{x/\sqrt{N},\Upsilon_N})\right)_{i,j,k}\right\|^{p}\right)^{1/p}
\end{align*}

Thus applying Holder's inequality as before (for $p\geq 2$ for the first inequality which yields the same type of terms as in the second line and are thus gathered) :
\begin{align*}&\left|\sum_{i,j,k,I,J,K}\lambda_{i,j,k}\lambda_{I,J,K}\left(\frac{\partial^2}{\partial x_{i,j}^k\partial x_{I,J}^K}g(x)\right)\right|\\&\leq 2N^2p
\left(\sum_{\iota=1,...,l}\left(g_\iota(\tau_{x/\sqrt{N},\Upsilon_N})\right)^p\right)^{-1/p}
\left(\sum_{\iota=1,...,l}
\left\|\left(\frac{\partial}{\partial x_{i,j}^k}g_\iota(\tau_{x/\sqrt{N},\Upsilon_N})\right)_{i,j,k}\right\|^{p}\right)^{2/p}||\lambda||^2
\\&+N^2\left(\sum_{\iota=1,...,l}\left|\sum_{i,j,k,I,J,K}\lambda_{i,j,k}\lambda_{I,J,K}\frac{\partial^2}{\partial x_{i,j}^k\partial x_{I,J}^K}g_\iota(\tau_{x/\sqrt{N},\Upsilon_N})\right|^{p}\right)^{1/p}
\end{align*}
Thus, we have to compute the second derivative
\begin{align*}&\frac{\partial^2}{\partial x_{i,j}^k\partial x_{I,J}^K}g_\iota(\tau_{x/\sqrt{N},\Upsilon_N})=\frac{2C_\iota}{N^2}1_{i=I}1_{j=J}1_{k=K} +\Re\left(\frac{2C_\iota\lambda_\iota}{N^2}\sum_{l=1}^{m_\iota}\sum_{L\in[1,m_\iota]-\{l\}}\frac{-\epsilon_L\epsilon_l}{8^2}\right.\\&\left.\lambda_{I,J}(\left[\lambda_{i,j}(\left(\left[((u_{j_l}^{i_l})^{\epsilon_l}-1) ... ((u_{j_L}^{i_L})^{\epsilon_L}-1)\right]_{i'I} \left[((u_{j_L}^{i_L})^{\epsilon_L}-1)1_{K=j_L}...%(u_{j_{m_\iota}}^{i_{m_\iota}})^{\epsilon_{m_\iota}}(u_{j_1}^{i_1})^{\epsilon_1}...
((u_{j_l}^{i_l})^{\epsilon_l}-1)1_{k=j_l}
\right]_{J'j'}\right)_{i'j'}\right]_{I'J'})\right)
\end{align*}
so that $\left|\sum_{i,j,k,I,J,K}\lambda_{i,j,k}\lambda_{I,J,K}\frac{\partial^2}{\partial x_{i,j}^k\partial x_{I,J}^K}g_\iota(\tau_{x/\sqrt{N},\Upsilon_N})\right|\leq \frac{2C_\iota}{N^2}||\lambda||^2+\frac{cC_\iota\lambda_\iota m_\iota^2}{N^2} ||\lambda||^2$  for some constant $c$ not depending on $\iota, N$.

Finally, to get a better estimate, one needs a lower bound on $\left(\sum_{\iota=1,...,l}\left(g_\iota(\tau_{x/\sqrt{N},\Upsilon_N})\right)^p\right)\geq l$ but also $g_\iota(\tau_{x/\sqrt{N},\Upsilon_N})\geq C_i||\frac{x}{\sqrt{N}}||^2-E_i$ so that one deduces for $||\frac{x}{\sqrt{N}}||^2\geq 2E_i/C_i+1/C_i$, one deduces $g_\iota(\tau_{x/\sqrt{N},\Upsilon_N})\geq C_i||\frac{x}{\sqrt{N}}||^2/2+1/2$ and for $C_i||\frac{x}{\sqrt{N}}||^2/2\leq E_i+1/2$ one can fix $q\geq 1$ such that $(E_i+1/2)\leq q/2$ so that $C_i||\frac{x}{\sqrt{N}}||_2^2/2q\leq 1/2$ and $g_\iota(\tau_{x/\sqrt{N},\Upsilon_N})\geq 1\geq 1/2+C_i||\frac{x}{\sqrt{N}}||_2^2/2q$ and   in any case : $$\left(\sum_{\iota=1,...,l}\left(g_\iota(\tau_{x/\sqrt{N},\Upsilon_N})\right)^p\right)\geq l-1+\left( \frac{1}{2}+ \frac{C_1}{2q}\left\|\frac{x}{\sqrt{N}}\right\|_2^2\right)^p.$$
Thus gathering all our inequalities, one gets for a constant $c_1$ independent of $N,x$:
 
 \begin{align*}&\left|\sum_{i,j,k,I,J,K}\lambda_{i,j,k}\lambda_{I,J,K}\left(\frac{\partial^2}{\partial x_{i,j}^k\partial x_{I,J}^K}g(x)\right)\right|\leq%(
  c_1
  %\frac{||x||^2}{N^2}+c_2)
  ||\lambda||^2%\\&:= p
%l^{-1/p}
%\frac{1}{N^2}\max_\iota(8C_\iota^2||x||^2+2C_\iota^2|\lambda_\iota|^2m_\iota^2N^2m)l^{2/p}||\lambda||^2
%+\max_\iota(2C_\iota+cC_\iota\lambda_\iota m_\iota^2)l^{1/p}||\lambda||^2
\end{align*}

Using the fundamental theorem of calculus, one deduces a  bound on second order difference quotients (using in the  next-to-last line convexity of the suared euclidian norm): \begin{align*}|g(x+y)&+g(x-y)-2g(x)|%=\int_0^1d\lambda dg(x+\lambda y)(y)+dg(x-y+\lambda y)(y)=\int_0^1d\lambda dg(x+\lambda y)(y)-dg(x-\lambda y)(y)
=|\int_0^1d\lambda\int_0^1d\mu d^2g(x-\lambda y+2\lambda\mu y)(y,2\lambda y)|\\&\leq \int_0^1d\lambda\int_0^1d\mu c_1%( c_1\frac{||x-\lambda y+2\lambda\mu y||^2}{N^2}+c_2)
\lambda||y||^2
%\\&\leq \int_0^1d\lambda\int_0^1d\mu( c_1\frac{(1-\lambda)||x||^2+\lambda(1-\mu)||x- y||^2+\lambda\mu||x+ y||^2}{N^2}+c_2)||y||^2\\&
\leq c_1%( c_1\frac{||x||^2+||x- y||^2+||x+ y||^2}{N^2}+c_2)
||y||^2\end{align*}
Combining this with the Lipschitz bound previously obtained of the form $$|g(x+y)+g(x-y)-2g(x)|\leq N ( c_2\frac{||x||+||x- y||+||x+ y||}{N}+c_3)||y||$$ one gets \eqref{CalphaE} with constants independent of $N, d=N^2m$ in taking the power $\alpha-1$ of our second order bound and multiplied by the power $(2-\alpha)$ of our lipschitz bound.\end{proof}

We can now gather our results in the matricial case, as needed to be applied      to prove the Laplace principle in the next section.
   Recall that the natural euclidean norm for $x\in(( M_N(\C))_{sa})^m $ is given by the following notation we will use in our next result:
        $$||x||_2^2=\sum_{k=1}^m\frac{1}{N}Tr(x_k^*x_k).$$
     
It will also be convenient to use the matricial cyclic gradient for functions $h:     ( M_N(\C))_{sa})^{mi}\to \R$, for $x_i\in(( M_N(\C))_{sa})^m $, $ I<J\in[\![1,N]\!], k=1,...,m, j=1,...,i$ (with the notation for coordinates of footnote \footnotemark[1]) 
     \[(\mathscr{D}_{j}^kh)_{II}:=\left(\frac{\partial}{\partial (x_j)_{II}^k} h\right)(x_1,...,x_i),\]
     \[(\mathscr{D}_{j}^kh)_{IJ}:=\left(\frac{\partial}{\partial (x_j)_{IJ}^k} h+\sqrt{-1} \frac{\partial}{\partial (x_j)_{JI}^k}h\right)(x_1,...,x_i),\]
     \[(\mathscr{D}_{j}^kh)_{JI}:=\left(\frac{\partial}{\partial (x_j)_{IJ}^k} h-\sqrt{-1} \frac{\partial}{\partial (x_j)_{JI}^k}h\right)(x_1,...,x_i),\]

\begin{theorem}\label{minimisationHermitian}
Let $p\in [2,\infty[, d=N^2m$ $\Upsilon_N\in \mathcal{U}( M_N(\C))$ and $G\in\mathcal{E}_{reg,p}(  \mathcal{T}^c_{2,0}(\mathcal{F}^m_{[0,1]}*\mathcal{F}^\nu_{\mu}),d_{2,0})\cup \mathcal{E}^{1,\alpha}( \mathcal{T}^c_{2,0}(\mathcal{F}^m_{[0,1]}*\mathcal{F}^\nu_{\mu}),d_{2,0}),$ $\alpha\in ]0,1]$ so that $g_N(x)=N^2G(\tau_{x/\sqrt{N},\Upsilon_N})\in \mathcal{E}_{1+\alpha}(\R^{dk})$ as in the previous lemma. Let $t_0=0< t_1<...<t_k\in [0,1]$ and $f_N=g_N\circ J_{t_1,...,t_k}%- \ell \mathfrak{c}_1
: \mathbb W_{sa,N}\to \R%, \ell\in ]-\infty,\frac{1}{2}[
.$

Let, for $x_i,x\in(( M_N(\C))_{sa})^m $ and $t\in [0,1]$: \[h^{N,\Upsilon_N}_{t}(x_1,...,x_i,x)=-\log \left( \int_{\mathbb W_{sa,N,[t,1]}} e^{-g_N(x_1,...,x_i,x+\nu_{t_{i+1}},...,x+\nu_{t_{k}})%+\ell\mathfrak{c}_{t,1,x}(\nu)
} \ \mathrm{d} \gamma_{[t,1]}(\nu) \right)\]
and for $\omega\in \mathbb W_{sa,N}, k=1,...,m$ define inductively on $i$ for $t_i<t\leq t_{i+1}$:
\[(b_k^{G,N,\Upsilon_N}(t,x,\omega)):=-\frac{1}{\sqrt{N}}\left(\mathscr{D}_{i+1}^k h_{t%,\ell
}^{N,\Upsilon_N}\right)(\sqrt{N}X_{t_1}^{G,N,\Upsilon_N}(\omega),...,\sqrt{N}X_{t_i}^{G,N,\Upsilon_N}(\omega),\sqrt{N}x),\]
%\[(b_k^{G,N}(t,x,\omega))_{IJ}:=-\frac{1}{\sqrt{N}}\left(\frac{\partial}{\partial x_{IJ}^k} h_t+\sqrt{-1} \frac{\partial}{\partial x_{ji}^k} h_t^N\right)(\sqrt{N}X_{t_1}^{G,N}(\omega),...,\sqrt{N}X_{t_i}^{G,N}(\omega),\sqrt{N}x),\]
%for $i<j$ and then extended to get $b^G$ hermitian
 so that there is a unique (strong) solution to the m-tuple of matrix-valued SDE driven by $H_t$ brownian motion of law $\gamma_{sa,N,m}$:
\[X^{G,N,\Upsilon_N}(t) = 
\int_0^t b^{G,N,\Upsilon_N}(s, X^{G,N,\Upsilon_N}(s)) ds +
H_t^N.\]
Then, we have the formula :
\begin{align*}
-\frac{1}{N^2}\log& \left( \int_{\mathbb W_{sa,N}} e^{-g_N(\nu_{t_1},...,\nu_{t_k})%+\ell \mathfrak{c}_1(\nu)
} \ \mathrm{d} \gamma(\nu) \right)\\&=\mathbf{E}  \left(
  G(\tau_{X^{G,N,\Upsilon_N},\Upsilon_N})+\frac{1}{2}\int_0^1||b^{G,N,\Upsilon_N}(t,X^{G,N,\Upsilon_N}(t))||_2^2%-2\ell ||X^{G,N,\Upsilon_N}(t)||_2^2\ 
  dt\right).
\end{align*}
Moreover we have
%$t,s\in ]t_i,t_{i+1}]$: \[E(||b(t,X_t)- b(s,X_s)||_2^2)\leq  C_2(t-s)^{(\alpha-1)}\left(E(||X_0 ||_2^2 ) + d+(\sup(-g))\right)^{(\alpha-1)}.\]
for some $C_4=C_4(G)$ independent of $N$ and for all  $t,s\in ]t_i,t_{i+1}]$: \[E(||b^{G,N,\Upsilon_N}(t,X_t^{G,N,\Upsilon_N})- b^{G,N,\Upsilon_N}(s,X_s^{G,N,\Upsilon_N})||_2^2)\leq  C_4\sqrt[4]{|t-s|^\alpha}%\sqrt{\left(1+ \sup(-G)\right)}
,\]
\[|h^{N,\Upsilon_N}_t(x_1,...,x_i,x)-h^{N,\Upsilon_N'}_t(x_1,...,x_i,x)|\leq C(G)||\Upsilon_N-\Upsilon_N'||_2.\]
Finally, the law of $\sqrt{N}X^{G,N,\Upsilon_N}(t)$ on the pathspace $\mathbb W_{sa, N}$ is exactly the Gibbs law:
$$\frac{e^{-g_N(\nu_{t_1},...,\nu_{t_k})%+\ell \mathfrak{c}_1(\nu)
}}{\int d\gamma(\nu)e^{-g_N(\nu_{t_1},...,\nu_{t_k})%+\ell \mathfrak{c}_1(\nu)
})} \ \mathrm{d} \gamma(\nu).$$

\end{theorem}     
     
\begin{proof}
$\sqrt{N}X_{s}^{G,N,\Upsilon_N}$ is the solution from Corollary \ref{MonotoneSol} with our $g$ which satisfies the assumption by our previous lemma \ref{UniformN}. The formula is then exactly the one given by Corollary \ref{UstunelFinal}. The first bound is also given in Corollary \ref{MonotoneSol}. The independence of $N$ comes from the dimension independence of the constants given in lemma \ref{UniformN} that make all the corresponding constants in proposition \ref{OptimalValueFct} also independent of dimension.
For the final statement, we can use for instance the well-known \cite[Theorem 5]{Ustunel} to get:
\begin{align*}\mathbf{E}  &\left(
  G(\tau_{X^{G,N,\Upsilon_N},\Upsilon_N})+\frac{1}{2}\int_0^1||b^{G,N,\Upsilon_N}(t,X^{G,N,\Upsilon_N}(t))||_2^2%-2\ell ||X^{G,N,\Upsilon_N}(t)||_2^2
  \ dt\right)\\&=\inf\{ \int G(\tau_{.,\Upsilon_N})%-\ell \mathfrak{c}_1(./\sqrt{N}) 
  d\mu +\frac{1}{N^2}Ent(\mu| \gamma), \mu\in Prob(\mathbb W_{sa,N})\}\end{align*}
where $Ent(\mu| \gamma)$  is the usual relative entropy so that for instance \cite[Theorem 2]{Ustunel} gives that for $\mu_{G,N,\Upsilon_N}$ the law of $\sqrt{N}X^{G,N,\Upsilon_N}$  $$Ent(\mu_{G,N,\Upsilon_N}| \gamma)\leq \frac{N^2}{2}\int_0^1||b^{G,N,\Upsilon_N}(t,X^{G,N,\Upsilon_N}(t))||_2^2\ dt.$$ Thus the law $\mu_{G,N,\Upsilon_N}$ reaches the infimum above and the uniqueness of the infimum and its form as a Gibbs state are given in the already quoted \cite[Theorem 5]{Ustunel}.

Finally, the lipschitz bound in $\Upsilon_N$ is obtained as in step 3 of the proof of proposition \ref{OptimalValueFct} but using \eqref{Lip0} instead of \eqref{Lop0}.
\end{proof}     

\subsection{Alternative formula for the drift $b^{G,N,\Upsilon_N}$}

\begin{proposition}\label{minht} Consider the setting of Corollary \ref{UstunelFinal}. Let $t_i<t\leq t_{i+1}$ and $(x_1,...,x_i,x)\in\R^{d(i+1)}$.

We define $b^t(s,y,\omega)$ for $s\in [t,1]$ and $\omega \in \Omega$, in considering the case $\max(t_I,t)<s\leq t_{I+1}, I\geq i$ by :
\[b_j^t(s,y,\omega):=-\frac{\partial}{\partial y^{(j)}} h_{s,\ell}(x_1,...,x_i,X_{t_{i+1}}^t(\omega),...,X_{t_I}^t(\omega),y),\]
and $b^t(s,y,\omega)=0$ if $s> t_k$.
We define simultaneously  $X^t(s)=X^t(s,x)$ the unique strong solution starting at $X_t^t=x$ defined in Corollary \ref{MonotoneSol} of \[X^t(s) = X_t^t +
\int_t^s b(u, X^t(u)) du +
B_s-B_t.\]
Also define for convenience $X^t_{t_I}=x_I$, $x_0=0$ for $I\leq i$ (and say interpolate linearly values on $[0,t]$).
Then, we have the formulas:
\[
h_{t%,\ell
}(x_1,...,x_i,x)=\mathbf{E}  \left(
  g\circ J_{t_1,...,t_k}(X^t)+\frac{1}{2}\int_t^1\|b(s,X^t(s))\|_2^2%-2\ell ||X^t(s)||^2ds
  \right),
\]
\begin{align*}
&\frac{\partial}{\partial x^{(j)}}h_{t%,\ell
}(x_1,...,x_i,x)\\&=\mathbf{E}  \left(
  \left[\frac{\partial}{\partial y^{(j)}}g(x_1,...,x_i,y+(X^t(t_{i+1},x)-x),...,y+(X^t(t_{k},x)-x))%-\ell\mathfrak{c}_{t,1,y}(X^t(.,x)-x)
  \right]_{y=x}\right),
\end{align*}
\begin{align*}
\frac{\partial}{\partial x_k^{(j)}}&h_{t%,\ell
}(x_1,...,x_i,x)\\&=\mathbf{E}  \left(
  \left[\frac{\partial}{\partial y_k^{(j)}}
  g(x_1,...,y_k,x_{k+1},...,x_i,X^t(t_{i+1},x),...,X^t(t_{k},x))
  %-\ell\mathfrak{c}_{t,1,x}X^t(.,x)
  \right]_{y_k=x_k}\right).
\end{align*}
\end{proposition}
It will be important for us in the next subsection to note that this last formula defines $X^t(u),Y_u=b(u, X^t(u))$ as a solution of a forward-backward stochastic differential equation on $[t,1]$.
\begin{proof}
Consider the minimization problem in \eqref{htinf}. Since for each $X=(x_1,...,x_i,x)$ fixed, $g_X(y_{i+1},...,y_k)=g(x_1,...,x_i,x+y_{i+1},...,x+y_k)$ defines a function $g_X\in \mathcal{E}_\alpha(\R^{d(k-i)})$, we can apply Corollary \ref{UstunelFinal} to it with $[0,1]$ replaced by $[t,1]$. The first formula for $h_t$ is then exactly the result of this corollary. Arguing as in the proof of Theorem \ref{minimisationHermitian}, one knows that the law of $(X^t(s)-x)_{s\in[t,1]}$ is $$ \frac{e^{-g(x_1,...,x_i,x+\nu_{t_{i+1}},...,x+\nu_{t_{k}})%+\ell \mathfrak{c}_{t,1,x}(\nu)
} \ \mathrm{d} \gamma_{[t,1]}(\nu)}{\int_{\mathbb W_{[t,1]}} e^{-g(x_1,...,x_i,x+\nu_{t_{i+1}},...,x+\nu_{t_{k}})%+\ell \mathfrak{c}_{t,1,x}(\nu)
} \ \mathrm{d} \gamma_{[t,1]}(\nu)}.$$
But this measures appears in \eqref{bjasderivative} (and its variant with $X_t$ replaced by a generic point $(x_1,...x_i,x)$) which we can interpret as the second expected formula. The other derivtives are similar.
\end{proof}
For convenience, we state separately the obvious application in the matricial case. %[DISTINGUISH $||.||_2^2$ and ordinary one everywhere]
     \begin{corollary}\label{minhthermitian}Fix the setting of Theorem \ref{minimisationHermitian}.
     Let $t_i<t\leq t_{i+1}$ and $(x_1,...,x_i,x)\in(( M_N(\C))_{sa})^{m(i+1)}$.

We define $b^{G,N,\Upsilon_N,t}(s,y,\omega)$ for $s\in [t,1]$ and $\omega \in \Omega$, in considering the case $\max(t_I,t)<s\leq t_{I+1}, I\geq i$ by :
\[b_k^{G,N,\Upsilon_N,t}(s,y,\omega):=-\frac{1}{\sqrt{N}}\left(\mathscr{D}_{i+1}^k h_{s,\ell}^{N,\Upsilon_N}\right)(\sqrt{N}x_1,...,\sqrt{N}x_i,\sqrt{N}X_{t_{i+1}}^{G,N,\Upsilon_N,t}(\omega),...,\sqrt{N}X_{t_I}^{G,N,\Upsilon_N,t}(\omega),\sqrt{N}y),\]
and $b^{G,N,\Upsilon_N,t}(s,y,\omega)=0$ if $s> t_k$.
We define simultaneously  $X^{G,N,\Upsilon_N,t}(s)=X^{G,N,\Upsilon_N,t}(s,x)$ the unique strong solution starting at $X^{G,N,\Upsilon_N,t}(t)=x$ defined in Corollary \ref{MonotoneSol} driven by $H_t^N$ brownian motion of law $\gamma_{sa,N,m}$ of \[X^{G,N,\Upsilon_N,t}(s) = X^{G,N,\Upsilon_N,t}(t) +
\int_t^s b^{G,N,\Upsilon_N,t}(u, X^{G,N,\Upsilon_N,t}(u)) du +
H_s^N-H_t^N.\]
 Also define for convenience $X^{G,N,\Upsilon_N,t}(t_I)=x_I$, $x_0=0$ for for $I\leq i$ (and say interpolate linearly values on $[0,t]$).
Then, we have the formulas for $j\geq i$:
\begin{align*}
\frac{1}{N^2}h_t^{N,\Upsilon_N}&(\sqrt{N}x_1,...,\sqrt{N}x_i,\sqrt{N}x)\\&=\mathbf{E}  \left(
  G(\tau_{X^{G,N,\Upsilon_N,t},\Upsilon_N})+\frac{1}{2}\int_t^1||b^{G,N,\Upsilon_N,t}(u,X^{G,N,\Upsilon_N,t}(u))||_2^2%-2\ell ||X^{G,N,\Upsilon_N,t}(u)||_2^2\ du
  \right),
\end{align*}
\begin{align*}
&(\mathscr{D}_{i+1}h_{t%,\ell
}^{N,\Upsilon_N})(\sqrt{N}x_1,...,\sqrt{N}x_i,\sqrt{N}x)=%-\mathbf{E}  \left(
%  2\ell \sqrt{N}\int_t^1 X^{G,N,\Upsilon_N,t}(u)du\right)\\&+
 \\& \mathbf{E}  \left(
  \sum_{j=i+1}^k\mathscr{D}_{j}g_N
  (\sqrt{N}x_1,...,\sqrt{N}x_i,\sqrt{N}X^{G,N,\Upsilon_N,t}(t_{i+1},x),...,\sqrt{N}X^{G,N,\Upsilon_N,t}(t_{k},x))\right),
\end{align*}
\begin{align*}
(\mathscr{D}_{j}h_{t,\ell}^{N,\Upsilon_N})&(\sqrt{N}x_1,...,\sqrt{N}x_j,...,\sqrt{N}x_i,\sqrt{N}x)=\\&\mathbf{E}  \left(
  \mathscr{D}_{j}g_N
  (\sqrt{N}x_1,...,\sqrt{N}x_i,\sqrt{N}X^{G,N,\Upsilon_N,t}(t_{i+1},x),...,\sqrt{N}X^{G,N,\Upsilon_N,t}(t_{k},x))\right).
\end{align*}
\end{corollary}    
 \subsection{An extension to the non-convex bounded case}
We extract the following result from \cite{Lehec}.
\begin{theorem}\label{minimisationHermitianNonconvex}
Let $p\in [2,\infty[, d=N^2m$ $\Upsilon_N\in \mathcal{U}( M_N(\C))$ and $G\in\mathcal{C}_{reg,p,0}(  \mathcal{T}^c_{2,0}(\mathcal{F}^m_{[0,1]}*\mathcal{F}^\nu_{\mu}),d_{1,0}),$  so that $g_N(x)=N^2G(\tau_{x/\sqrt{N},\Upsilon_N})\in C^{2}(\R^{dk})$. Let $t_0=0< t_1<...<t_k\in [0,1]$ and $f_N=g_N\circ J_{t_1,...,t_k}: \mathbb W_{sa,N}\to \R,.$

Let, for $x_i,x\in(( M_N(\C))_{sa})^m $ and $t\in [0,1]$: \[h^{N,\Upsilon_N}_t(x_1,...,x_i,x)=-\log \left( \int_{\mathbb W_{sa,N,[t,1]}} e^{-g_N(x_1,...,x_i,x+\nu_{t_{i+1}},...,x+\nu_{t_{k}})} \ \mathrm{d} \gamma_{[t,1]}(\nu) \right)\]
which is differentiable.
We also define $b^{G,N,\Upsilon_N,t}(s,y,\omega), b^{G,N,\Upsilon_N}(s,y,\omega)=b^{G,N,\Upsilon_N,0}(s,y,\omega)$ for $s\in [t,1]$ and $\omega \in \Omega$, in considering the case $\max(t_I,t)<s\leq t_{I+1}, I\geq i$ by :
\[b_k^{G,N,\Upsilon_N,t}(s,y,\omega):=-\frac{1}{\sqrt{N}}\left(\mathscr{D}_{i+1}^k h_s^{N,\Upsilon_N}\right)(\sqrt{N}x_1,...,\sqrt{N}x_i,\sqrt{N}X_{t_{i+1}}^{G,N,\Upsilon_N,t}(\omega),...,\sqrt{N}X_{t_I}^{G,N,\Upsilon_N,t}(\omega),\sqrt{N}y),\]
and $b^{G,N,\Upsilon_N,t}(s,y,\omega)=0$ if $s> t_k$.
We define simultaneously  $X^{G,N,\Upsilon_N,t}(s)=X^{G,N,\Upsilon_N,t}(s,x)$ the unique strong solution starting at $X^{G,N,\Upsilon_N,t}(t)=x$ defined in Corollary \ref{MonotoneSol} driven by $H_t^N$ brownian motion of law $\gamma_{sa,N,m}$ of \[X^{G,N,\Upsilon_N,t}(s) = X^{G,N,\Upsilon_N,t}(t) +
\int_t^s b^{G,N,\Upsilon_N,t}(u, X^{G,N,\Upsilon_N,t}(u)) du +
H_s^N-H_t^N.\]

Also define for convenience $X^{G,N,\Upsilon_N,t}(t_I)=x_I$, $x_0=0$ for for $I\leq i$ (and say interpolate linearly values on $[0,t]$).
Then, we have the formulas :
\begin{align*}
\frac{1}{N^2}h_t^{N,\Upsilon_N}&(\sqrt{N}x_1,...,\sqrt{N}x_i,\sqrt{N}x)\\&=\mathbf{E}  \left(
  G(\tau_{X^{G,N,\Upsilon_N,t},\Upsilon_N})+\frac{1}{2}\int_t^1||b^{G,N,\Upsilon_N,t}(u,X^{G,N,\Upsilon_N,t}(u))||_2^2\ du\right),
\end{align*}
\begin{align*}
&(\mathscr{D}_{i+1}h_t^{N,\Upsilon_N})(\sqrt{N}x_1,...,\sqrt{N}x_i,\sqrt{N}x)=\\&\mathbf{E}  \left(
  \sum_{j=i+1}^k\mathscr{D}_{j}g_N
  (\sqrt{N}x_1,...,\sqrt{N}x_i,\sqrt{N}X^{G,N,\Upsilon_N,t}(t_{i+1},x),...,\sqrt{N}X^{G,N,\Upsilon_N,t}(t_{k},x))\right).
\end{align*}

Finally, if we write $X^{G,N,\Upsilon_N}(s)=X^{G,N,\Upsilon_N,0}(s,x=0)$, the law of $\sqrt{N}X^{G,N,\Upsilon_N}(t)$ on the pathspace $\mathbb W_{sa, N}$ is exactly the Gibbs law:
$\frac{e^{-g_N(\nu_{t_1},...,\nu_{t_k})}}{\int d\gamma(\nu)e^{-g_N(\nu_{t_1},...,\nu_{t_k})}} \ \mathrm{d} \gamma(\nu).$

\end{theorem}
\begin{proof}
Recall that the class $\mathcal{S}$ in \cite[Def 5]{Lehec} contains densities as the one of the above Gibbs law $\mu_{G,N}$ depending on finitely many times, bounded away from zero (since the potential $g_N$ is bounded), Lipschitz and with Lipschitz derivative. Both lipchitzness conditions come as in lemma \ref{UniformN} from the first and second bounded derivatives of $g_N$ (the constant $C=0$ in the condition $G\in\mathcal{C}_{reg,p,C=0}$ is crucial for that). Then note that our $b^{G,N,\Upsilon_N,0}(t,.)$ is exactly $u_t$ on \cite[Thm 4]{Lehec} based on the formula for F\"ollmer's drift in is lemma 3 which is similar to our corollary \ref{bjasderivative}. From his lemma 6, the SDE for $X^{G,N,\Upsilon_N,0}$ has the pathwise uniqueness property and thus his Thm 4 applies and we know the relative entropy of $\mu_{G,N}$. From the classical \cite[Thm 5]{Ustunel} we already used, this gives exactly our formula for $h_0^{N,\Upsilon_N}$. The general result comes from a similar argument for brownian motion on $[t,1]$ and the formula for the derivative then follows as in corollary \ref{bjasderivative}.
\end{proof}     
    
\section{Free Forward-Backward stochastic differential equations}     
     
The two previous subsections made appear the following kind of forward-backward stochastic differential equations (see \cite{MaYong} for an overview of the subject):
$$X^{G,N,\Upsilon_N,t}(s) = X^{G,N,\Upsilon_N,t}(t) -
\int_t^s Y^{G,N,\Upsilon_N,t}(u) du +
H_s^N-H_t^N,$$
     \begin{align*}
&Y^{G,N,\Upsilon_N,t}(u)=%-\mathbf{E}  \Big(
  %2\ell \int_u^1 X^{G,N,\Upsilon_N,t}(v)dv\Big| \mathcal{F}_u\Big)
  \\&%+
  \mathbf{E}  \Big(
  \sum_{j=i+1}^k\frac{1}{\sqrt{N}}\mathscr{D}_{j}g_N
  (\sqrt{N}x_1,...,\sqrt{N}x_i,\sqrt{N}X^{G,N,\Upsilon_N,t}(t_{i+1},x),...,\sqrt{N}X^{G,N,\Upsilon_N,t}(t_{k},x))\Big| \mathcal{F}_u\Big).
\end{align*}
     The formulation above is similar to the weak formulation of \cite{MaAntonelli}. But since all processes here are adapted to a classical brownian filtration, one case use the representation theorem on martingales in this filtration to rewrite the second equation in the most standard way :
     \begin{align*}
&Y^{G,N,\Upsilon_N,t}(u)=%-
  %2\ell \int_u^1 X^{G,N,\Upsilon_N,t}(v)dv 
  - \int_u^tZ^{G,N,\Upsilon_N,t}(v) \#dH_v^N\\&+
  \sum_{j=i+1}^k\frac{1}{\sqrt{N}}\mathscr{D}_{j}g_N
  (\sqrt{N}x_1,...,\sqrt{N}x_i,\sqrt{N}X^{G,N,\Upsilon_N,t}(t_{i+1},x),...,\sqrt{N}X^{G,N,\Upsilon_N,t}(t_{k},x)).
\end{align*}
Note that $Z^{G,N,\Upsilon_N,t}(v)     
   $ is valued in $(M_N(\C)\otimes_{alg}  M_N(\C))^{m}$ and thus $\#$ is as usual in free probability the side multiplication $(a\o c)\#b=abc$ for biprocesses.
   
   We  need to study the limit $N\to\infty$ of theses equations in a free probability framework. First, we recall sort time existence results similar to \cite{Delarue}. Then, in the convex case, one can use the relation to the optimal control problem and the a priori estimates this provides to argue to an existence and uniqueness result on arbitrary $[0,T]$ with estimate uniform in $N$ enabling to take a limit. The method will be close to \cite{FurhmanTessitore} and will be the object of the second subsection. In a previous preprint version, before we discover this reference, the argument was hidden in the proof of our Laplace principle lower bound. We think that it deserves being highlighted in a specific subsection.
   
     \subsection{Local Existence to free Forward-Backward SDE}

\begin{proposition}\label{localFBSDE}
Let $\mathcal{G}_t$ a filtration containing a standard martingale $(S_t)$  in the sense of subsection \ref{ultra}. Suppose that $\mathcal{F}_t$ %=W^*(S_s,s\leq t, \mathcal{F}_0)$ 
is a right continuous subfiltration% for some $\mathcal{F}_0\subset \mathcal{G}_0$ gives a subfiltration 
, (resp.  %
and moreover $S_t$ is a free brownian motion adapted to $\mathcal{F}_t$). Let $G%=h_1
: [L^2(\mathcal{G}_1)]^{(k+1)m}\to \R$ a 2-paraconvex and $2$-paraconcave function such that $\nabla_{X_o}G(X_0,...,X_k)\in\mathcal{F}_1$ (resp. $\nabla_{X_o}G(X_0,...,X_k)\in L^2(W^*(X_0,...,X_k,\mathcal{F}_0))$) for any $o\in [\![1,k]\!]$ and $X_i\in L^2(\mathcal{F}_1)$.  Assume given $0<t\leq T\leq 1$%<t_1<...<t_k\leq 1=t_{k+1}$ %and moreover that there exists a family $h_t:[L^2(\mathcal{G}_t)]^{(i+1)m}\to \R$ of 2-paraconvex and $\gamma$-paraconcave functions  with constants independent of $t$ for $t_i<t<s\leq t_{i+1}$ (written $i(t)$) with $||\nabla_{X_{i+1}}h_t(X)-\nabla_{X_{i+1}}h_s(X)||\leq C(X)(t-s)^{(\gamma-1)/2}.$

Then, there is $\tau_0>0$ such that for $T<\tau_0$, there is exactly one solution (i.e. families of adapted functions $X^t(s,.)$,$Y_s(.), Z_s(.)$ satisfying the semigroup relation  $X^{t}_{s}(\mathbf{X},X)-X^{v}_{s}(\mathbf{X},X^{t}_v(\mathbf{X},X))=0$ for all $t<v<s$) to the equations for $T\geq s\geq t$:
$$X^t_s=X^t(s,Z,X^t(t))=X^t(t)
+S_s-S_t-\int_t^s Y_u(Z,X^t_u)du, $$
$$Y_s(Z,X^s(s))=E_{\mathcal{G}_s}
\Big[ \nabla_{X^s_{T}}G(Z,X^s_{T})%- 2\ell\int_s^1X^s(u)du
\Big],$$
$$Z_s(Z,X^s(s))=E_{\mathcal{G}_s}
\Big[ \nabla_{Z}G(Z,X^s_{T})%- 2\ell\int_s^1X^s(u)du
\Big],$$
%$$\nabla_{X_{j}}h_s(X_1,...,X_{{i(s)}},X^s(s))=E_{\mathcal{G}_s}
%\Big[ \nabla_{X^s_{t_j}}G(X_1,...,X_{{i(s)}},
%X^s_{t_{i(s)+1}},...,X^s_{t_{k+1}})\Big], j\leq i(s),$$
and it necessarily satisfies for $s\geq t$ for any $X\in \mathcal{F}_t^{km},X^t(t)\in \mathcal{F}_t^m$:
\begin{equation}\label{Xtadapted}X^t_s\in\mathcal{F}_s \ \Big(resp.\  X^t_s\in L^2(W^*(Z,X^t(t), S_u-S_t, u\in [t,s], \mathcal{F}_0))\Big),\end{equation}
\begin{equation}\label{nablahadapted}Y_s(Z,X^s(s))\in\mathcal{F}_s \ \Big(resp.\  Y_s(Z,X^s(s))\in L^2(W^*(Z,X^s(s)),  \mathcal{F}_0)\Big),\end{equation}
\begin{equation}\label{nablahadaptedextra}Z_s(Z,X^s(s))\in\mathcal{F}_s \ \Big(resp.\  Z_s(Z,X^s(s))\in L^2(W^*(Z,X^s(s)),  \mathcal{F}_0)\Big).\end{equation}
\end{proposition}    
    %In practice, the function $h_t$ will be obtained by stochastic control techniques or by taking ultraproducts thereof.
   \setcounter{Step}{0}
  
\begin{proof}  

The existence of the supplementary component $Z_s$ is treated as the second $Y_s$ and since it is not involved in the equation, its uniqueness is clear and we don't detail it more. 

   \begin{step} Existence for small $T$, with $c\max(1,||\nabla G||_{lip})T<1$ for some $c>0$ independent of $G$.\end{step}

%We will analyse it in an induction on $k-i$ starting at $k-i=0$,  $t_k<t\leq 1$.
%This is the trivial case since for such $t$, $\nabla_{X}h_t^\omega=0$ (the sum is empty in the last formula from (vii) and thus $Y^t_s(X_1,...,X_k,X)=X+S_s-S_t.$ Similarly later, we can disregard the case $s>t_k$, since $Y^t_s(X_1,...,X_k,X,\upsilon)=Y^t_{t_k}(X_1,...,X_k,X,\upsilon)+S_s-S_{t_k}$.

 On the interval $[0,T]$ for $T$ to be chosen small enough soon, we will approximate our system of equations by a Picard iteration. We define 
inductively for $X,X_l\in L^2(\mathcal{G}_{t})^m, s\in [t,T]$ (written $\mathbf{X}=(X_1,...,X_k)$): $$Y^{t,(0)}_s(X)=Y^{t,(0)}_s(\mathbf{X},X)=X+(S_s-S_t),$$
for $l\geq 0:$
$$DH^{(l)}_t(\mathbf{X},X)=E_{\mathcal{G}_t}
(\nabla_{Y^{t,(l)}_{T}}G(X_1,...,
X_{k},Y^{t,(l)}_{T}(\mathbf{X},X))$$
and for $l\geq 1$:
\begin{align*}&Y^{t,(l)}_s(\mathbf{X},X)=X+(S_s-S_t)+\int_t^sdv DH^{(l-1)}_v(\mathbf{X},Y^{t,(l-1)}_v(\mathbf{X},X)).
\end{align*}

We first want to check  by induction on $l$ that the integrals above are well-defined and $$Y^{t,(l)}_s(\mathbf{X},X)\in L^2(\mathcal{F}_s)^m,\  \Big(resp\ L^2(W^*( \mathcal{F}_0,X_1,...,X_k,X, S_u-S_t, u\in [t,s] )))^{m}\Big)$$ if $X,X_i\in(L^2(\mathcal{F}_t))^m$ (We also see inductively simultaneously the lipschitzness of $Y^{t,(l)}_s$ but we will detail the bounds later). The initialization $l=0$ is obvious and then if the result is true at level $l$, $X_1,...,
X_{k},Y^{t,(l)}_{T}(\mathbf{X},X))\in L^2(\mathcal{F}_1)$ (resp. in $L^2(W^*( \mathcal{F}_0,X_1,...,X_k,X, S_u-S_t, u\in [t,T] )))$) and applying the assumption on the gradients of $G$, so is the application $\nabla_{Y^{t,(l)}_{T}}G$ in the definition of $DH^{(l)}_t(X_1,...,X_{k},X,\upsilon)$. But on this space $L^2(\mathcal{F}_1)$, the agreement of conditional expectations assumed in the submodel/subfiltration property, one deduces $E_{\mathcal{G}_t}=E_{L^2(\mathcal{F}_t)}$ which concludes to the first case $DH^{(l)}_t(\mathbf{X},X)\in L^2(\mathcal{F}_t)$ and respectively,  from freeness of $\mathcal{F}_t$ with the free brownian motion, we even deduce $DH^{(l)}_t(X_1,...,X_{k-1},X)\in L^2(W^*(X_1,...,X_k,X,\mathcal{F}_0))$. From the Lipschitzness of $Y^{t,l}_T$ and the one for $\nabla G$ one also deduces the one of $DH^{(l)}$.

Let us explain why this implies the right continuity in $v$ of the integrand proved inductively. For $w>v$,  one gets:
\begin{align*}&||DH^{(l)}_v(\mathbf{X},Y^{t,(l)}_v(\mathbf{X},X))-DH^{(l)}_w(\mathbf{X},Y^{t,(l)}_w(\mathbf{X},X))||_2\\&\leq ||DH^{(l)}_v(\mathbf{X},Y^{t,(l)}_v(\mathbf{X},X))-DH^{(l)}_w(\mathbf{X},Y^{t,(l)}_v(\mathbf{X},X))||_2\\&+ ||DH^{(l)}_w(\mathbf{X},Y^{t,(l)}_v(\mathbf{X},X))- DH^{(l)}_w(\mathbf{X},Y^{t,(l)}_w(\mathbf{X},X))||_2
\\&\leq ||E_{\mathcal{G}_v}
\Big(\nabla_{Y^{v,(l)}_{T}}G(X_1,...,
X_{k},Y^{v,(l)}_{T}(\mathbf{X},Y^{t,(l)}_v(\mathbf{X},X)))\Big)-E_{\mathcal{G}_w}
\Big(\nabla_{Y^{w,(l)}_{T}}G(X_1,...,
X_{k},Y^{w,(l)}_{T}(\mathbf{X},Y^{t,(l)}_v(\mathbf{X},X)))\Big) ||_2
\\&+||DH^{(l)}_w||_{lip}||Y^{t,(l)}_v(\mathbf{X},X)-Y^{t,(l)}_w(\mathbf{X},X)||_2\\&\leq ||(E_{L^2(\mathcal{F}_v)}-E_{L^2(\mathcal{F}_w)})
\Big(\nabla_{Y^{v,(l)}_{1}}G(X_1,...,
X_{k},Y^{v,(l)}_{1}(\mathbf{X},Y^{t,(l)}_v(\mathbf{X},X)))\Big) ||_2\\&+ ||\nabla G||_{lip}||Y^{v,(l)}_{T}(\mathbf{X},Y^{t,(l)}_v(\mathbf{X},X)))-Y^{w,(l)}_{T}(\mathbf{X},Y^{t,(l)}_v(\mathbf{X},X)))||_2
\\&+||DH^{(l)}_w||_{lip}||Y^{t,(l)}_v(\mathbf{X},X)-Y^{t,(l)}_w(\mathbf{X},X)||_2
\end{align*}
and using the  right continuity of the filtration $L^2(\mathcal{F}_t)$ this tends to $0$ when $w\to v$ as soon as we checked by induction on $l$ that for $Z=Y^{t,(l)}_v(\mathbf{X},X)$, uniformly for $\tau\leq T$ :
\begin{align*}||Y^{v,(l)}_{\tau}(\mathbf{X},Z)&-Y^{w,(l)}_{\tau}(\mathbf{X},Z)||_2\leq ||S_w-S_v+\int_v^wd\sigma DH^{(l-1)}_\sigma(\mathbf{X},Y^{v,(l-1)}_\sigma(\mathbf{X},Z))||_2\\& +\int_w^\tau d\sigma   ||DH^{(l-1)}_\sigma||_{lip}||Y^{v,(l-1)}_\sigma(\mathbf{X},Z))-  Y^{w,(l-1)}_\sigma(\mathbf{X},Z))||_2\to_{w\to v} 0\end{align*}
 Especially, at each inductive step, the integral is well-defined. Applying this with the induction hypothesis to the integrand defining $Y^{t,(l+1)}_s$, one obtains the induction step concerning the adaptedness result.

Let us compute bounds on Lipschitzness constants : $||Y^{t,(0)}_s||_{Lip}\leq 1$ and knowing that $\nabla G$ is C-Lipschitz, we have by composition for $0\leq t\leq T$
\begin{align*}&||DH^{(l)}_t(\mathbf{X},X)-DH^{(l)}_t(Z_1,...Z_{k},Z)||_2\\&%\leq C ||(X_1,...X_i,Y^{t,(l)}_{t_{k}}(\mathbf{X},X,\upsilon))-(Z_1,...Z_{k-1},Y^{t,(l)}_{t_{k}}(Z_1,...Z_{k-1},Z,\upsilon))||
\leq C(1+||Y^{t,(l)}_{T}||_{Lip}^2)^{1/2}||(X_1,...,X_{k},X)-(Z_1,...Z_{k},Z)||_2\end{align*}

%For $DH$ by composition from the Lipschitzness of $\nabla G^\omega$ coming from the one of $G$, for $Y^t$ once defined also by composition. 
And similarly %$v\upsilon$
$$||Y^{t,(l)}_{s}||_{Lip}\leq 1+\int_t^sdv  ||DH^{(l-1)}_v||_{Lip}(1+||Y^{t,(l-1)}_{v}||_{Lip}^2)^{1/2}.$$
If $5CT<1$, an immediate induction yields for $s,t$ as above $||Y^{t,(l)}_{s}||_{Lip}\leq 1+5CT\leq 2$ %\leq E:=(2+C)e^C$, 
%(where the supplementary constant with $E-2$ is to keep the choice of $\epsilon$ independent of our next induction step)
 and  $||DH^{(l-1)}_s||_{Lip}\leq  D:=\sqrt{5}C.$

Let us turn to proving bounds on differences in $l$, towards proving convergence. 
We start by bounding the increments:
\begin{align*}&||DH^{(l)}_t(\mathbf{X},X)-DH^{(l-1)}_t(\mathbf{X},X)||_2
\\&\leq C\sup_{v\in [t,T]}||Y^{t,(l)}_{v}(\mathbf{X},X))-Y^{t,(l-1)}_{v}(\mathbf{X},X))||_2,\end{align*}
We then use the relation for $s\geq v\geq t$, $l\geq 1$: \begin{align}\label{semigroupY}\begin{split}&Y^{t,(l)}_{s}(\mathbf{X},X)=Y^{v,(l)}_{s}\left(\mathbf{X},Y^{t,(l-1)}_v(\mathbf{X},X)\right)+(Y^{t,(l)}_v- Y^{t,(l-1)}_v)(\mathbf{X},X)\\&+ \int_v^sdw DH^{(l-1)}_w(\mathbf{X},Y^{t,(l-1)}_w(\mathbf{X},X))- DH^{(l-1)}_w(\mathbf{X},Y^{v,(l-1)}_w(\mathbf{X},Y^{t,(l-1)}_v(\mathbf{X},X)))\end{split}
\end{align}
%at the third line below,using \eqref{semigroupY}, 
 and this reads for $l=1$ as $Y^{t,(1)}_{s}(\mathbf{X},X)=Y^{v,(1)}_{s}(\mathbf{X},Y^{t,(0)}_v(\mathbf{X},X))+(Y^{t,(1)}_v- Y^{t,(0)}_v)(\mathbf{X},X),$ since $Y^{t,(0)}_w(\mathbf{X},X)= Y^{v,(0)}_w(\mathbf{X},Y^{t,(0)}_v(\mathbf{X},X))$ so that one gets for $l\geq 2$, $s\geq v$ by an elementary induction (using we already chose $DT\leq 1$):
 \begin{align*}&%\sup_{t\leq v\leq s}
 f_l(s,v):=||Y^{t,(l)}_{s}(\mathbf{X},X)-Y^{v,(l)}_{s}(\mathbf{X},Y^{t,(l-1)}_v(\mathbf{X},X)||_2\leq %\sup_{t\leq v\leq s}
 ||(Y^{t,(l)}_v- Y^{t,(l-1)}_v)(\mathbf{X},X))||_2\\&+ %\sup_{t\leq v\leq s}
 \int_v^sdw D ||Y^{t,(l-1)}_w(\mathbf{X},X)- Y^{v,(l-1)}_w(\mathbf{X},Y^{t,(l-2)}_v(\mathbf{X}X))||_2+ 2D|| Y^{t,(l-1)}_v(\mathbf{X},X)-Y^{t,(l-2)}_v(\mathbf{X},X)||_2
 \\&\leq C_l(v)+2D(s-v) C_{l-1}(v)+ \int_v^sdw D f_{l-1}(w)\qquad \qquad \qquad \Big(C_l(v):= ||(Y^{t,(l)}_v- Y^{t,(l-1)}_v)(\mathbf{X},X))||_2\Big)
 %\\&\leq C_l(v)+3D(s-v) C_{l-1}(v)+ \int_v^sdw 2D^2(w-v)C_{l-2}(v)+  D^2\int_v^sdw \int_v^wd\sigma f_{l-2}(\sigma)
 \\&\leq C_l(v)+\sum_{k=1}^{l-1} 3\frac{(D(s-v))^{l-k}}{(l-k)!}
 C_{k}(v)\leq C_l(v)+\sum_{k=1}^{l-1} 3\frac{(D(s-v))^{l-1-k}}{(l-1-k)!}
 C_{k}(v)
% \\&\leq \sum_{k=1}^l4(DT)^{l-k}\sup_{t\leq v\leq T}||(Y^{t,(k)}_v- Y^{t,(k-1)}_v)(\mathbf{X},X))||_2%\\&+ (DT)^{l-1}\sup_{t\leq v\leq w\leq T} ||Y^{t,(1)}_w(\mathbf{X},X)- Y^{v,(1)}_w(\mathbf{X},Y^{t,(0)}_v(\mathbf{X}X))||_2, .
\end{align*}
 Thus, one can use this to bound increments of $Y$ and thus $C_{k}(v)$:
 
\begin{align*}&||Y^{t,(l+1)}_s(\mathbf{X},X)-Y^{t,(l)}_s(\mathbf{X},X)||_2\\&\leq\int_t^sdv ||DH^{(l)}_v(\mathbf{X},Y^{t,(l)}_v(\mathbf{X},X)) -DH^{(l-1)}_v(\mathbf{X},Y^{t,(l-1)}_v(\mathbf{X},X))||_2
\\&\leq C\int_t^sdv \sup_{w\in [v,T]}||Y^{v,(l)}_{w}(\mathbf{X},Y^{t,(l-1)}_v(\mathbf{X},X))%\\&\qquad \qquad \qquad \qquad \qquad\qquad \qquad
-Y^{v,(l-1)}_{w}(\mathbf{X},Y^{t,(l-1)}_v(\mathbf{X},X))||_2
+D\int_t^sdv  C_l(v)%%\sup_{w\in [v,T]}
%||Y^{t,(l)}_v(\mathbf{X},X))-Y^{t,(l-1)}_v(\mathbf{X},X))||_2
\\&\leq C\int_t^sdv \sup_{w\in [v,T]}(f_l(w,v)+f_{l-1}(w,v))
+\int_t^sdv  (D+C)\sup_{w\in [v,T]}C_l(w)%||Y^{t,(l)}_w(\mathbf{X},X)-Y^{t,(l-1)}_w(\mathbf{X},X)||_2
+2CC_{l-1}(v)\\&\leq 6C\int_t^sdv \sum_{k=1}^{l-1} \frac{(D(T-v))^{l-1-k}}{(l-1-k)!}
 C_{k}(v)
+\int_t^sdv  (D+2C)\sup_{w\in [v,T]}C_l(w)%||Y^{t,(l)}_w(\mathbf{X},X)-Y^{t,(l-1)}_w(\mathbf{X},X)||_2
+3CC_{l-1}(v)%||Y^{t,(l-1)}_v(\mathbf{X},X)-Y^{t,(l-2)}_v(\mathbf{X},X)||_2
\end{align*}
Let $E=\frac{(D+2C)T+ \sqrt{(D+2C)^2T^2+12CT(1+2e)}}{2}\geq DT$ and one shows by induction $A_{l+1}:=\sup_{t\leq s\leq T}C_{l+1}(s)\leq E^{l+1}A_1$ since our inequality gives combined with induction assumption:
$$A_{l+1}\leq (D+2C)T A_{l}+6CT\sum_{k=1}^{l-1} 3\frac{E^{l-1-k}}{(l-1-k)!}A_{k}+3CTA_{l-1}%\leq A_1 ((D+3C)T +3e+1) E^l
\leq A_1 ((D+3C)T +\frac{3CT(1+2e)}{E}) E^l%\leq A_1 ((D+3C)T +3Te+\frac{2CT}{E}) E^l 
= A_1 E^{l+1}$$
since we chose $E$ to get $((D+2C)T +\frac{3CT(1+2e)}{E})=E.$
%Thus taking a supremum, one gets 
%\begin{align*}&\sup_{t<s\in [1-\delta,1]^2}||Y^{t,(l+1)}_s(\mathbf{X},X)-Y^{t,(l)}_s(\mathbf{X},X)||_2
%\\&\qquad\leq(3(C+2|\ell|)+D)\delta \sup_{t<v\in [1-\delta,1]^2}||Y^{t,(l)}_v(X_1,...,X_{k-1},X)-Y^{t,(l-1)}_v(X_1,...,X_{k-1},X)||_2
%\\&\qquad\leq[(3(C+2|\ell|)+D)\delta]^l \sup_{t<v\in [1-\delta,1]^2}||Y^{t,(1)}_v(\mathbf{X},X)-Y^{t,(0)}_v(\mathbf{X},X)||_2.
%\end{align*}
If we moreover assume $T$ small enough so that $2(D+2C)T<1,84CT<1 $ so that $E\leq 1/4+\sqrt{2}/2<1$ %(again the supplementary $e^C$ for the next induction step) 
the series of term $Y^{t,(l+1)}_s(\mathbf{X},X)-Y^{t,(l)}_s(\mathbf{X},X)$ converges and thus $Y^{t,(l+1)}_s(\mathbf{X},X)\to X^t_s(\mathbf{X},X)$ uniformly on $s\geq t\in [0,T]^2,$ and as a consequence $DH^{(l)}_t(\mathbf{X},X)\to Y_t(\mathbf{X},X)$ uniformly on $t\in [0,T].$ The equations for $Y,X$ are obvious from the uniform Lipschitz constants, uniform convergences and defining inductive relations for $Y^{t,(l+1)}_s,DH^{(l)}_t.$

Moreover, taking the limit to equation \eqref{semigroupY}, $t<v<s$:
\begin{align}\begin{split}&X^{t}_{s}(\mathbf{X},X)-X^{v}_{s}(\mathbf{X},X^{t}_v(\mathbf{X},X))\\&= \int_v^sdw\ \Big( Y_w(\mathbf{X},X^{t}_w(\mathbf{X},X))- 
Y_w(\mathbf{X},X^{v}_w(\mathbf{X},X^{t}_v(\mathbf{X},X)))
\Big)\end{split}
\end{align}
and thus since $||Y_w||_{lip}\leq D$ and Gronwal's inequality gives $X^{t}_{s}(\mathbf{X},X)-X^{v}_{s}(\mathbf{X},X^{t}_v(\mathbf{X},X))=0$ for all $t<v<s$

\begin{step} Uniqueness and adapteness.\end{step}

Take a second solution $X_s',Y_s'$ also satisfying the semigroup relation (note this is the case as soon as we know a lipschitz bound for $Y_s'$ by uniqueness to the ODE satisfied by $X$), we have the estimates (using our previous Lipschitz bounds for $||Y_s||_{lip}\leq D, ||X_s^t||_{lip}\leq 2$ for $T$ chosen as before obtained in step 1) 
\begin{align*}||(Y_s-Y_s')(Z,X^s(s))||_2&\leq|| \nabla_{X^s_{T}}G(Z,X^s_{T}(Z,X^s(s))-\nabla_{X^s_{T}}G(Z,X^{s\prime}_{T}(Z,X^s(s))||_2\\&\leq ||\nabla G||_{lip}|| (X^s_{T}-X^{s\prime}_{T})(Z,X^s(s))||_2,\end{align*}
\begin{align*}&|| X^s_{t}(Z,X^s(s))-X^{s\prime}_{t}(Z,X^s(s))||_2\leq \int_s^t ||Y_u(Z,X^s_u(Z,X^s(s)))- Y_u'(Z,X^{s\prime}_u(Z,X^s(s)))||_2du\\&\leq \int_s^t ||\nabla G||_{lip}||(X^{u\prime}_T- X^{u}_{T})(Z,X^{s\prime}_u(Z,X^s(s)))||_2du+\int_s^t D||(X^s_u-X^{s\prime}_u)(Z,X^s(s)))||_2du\\&\leq \int_s^t ||\nabla G||_{lip}||X^{s\prime}_T(Z,X^s(s)))- X^{s}_{T}(Z,X^s(s)))||_2du+\int_s^t( D+2||\nabla G||_{lip})||(X^s_u-X^{s\prime}_u)(Z,X^s(s)))||_2du, \end{align*}

and thus taking a supremum over $s<t$:
\begin{align*}&\sup_{s<t\leq T}|| X^s_{t}(Z,X^s(s))-X^{s\prime}_{t}(Z,X^s(s))||_2\\&\leq ( D+3C)T \sup_{s<t\leq T}|| X^s_{t}(Z,X^s(s))-X^{s\prime}_{t}(Z,X^s(s))||_2, \end{align*}
and since we assumed in the previous step $( D+3C)T<1$ this is possible only if the supremum vanishes, i.e. if the solution is unique. Then adaptedness follows from the construction in step $1$ which has been shown to be adapted.
\end{proof}

\subsection{Global Existence in the convex case by stochastic control estimates}

\begin{proposition}\label{convexFBSDE}
Let $\mathcal{G}_t$ a filtration containing a standrad martingale $(S_t)$. Suppose that $\mathcal{F}_t=W^*(S_s,s\leq t, \mathcal{F}_0)$ for some $\mathcal{F}_0\subset \mathcal{G}_0$ gives a subfiltration (hence with $S_t$ free brownian motion adapted to $\mathcal{F}_t$). Let $G=h_1: [L^2(\mathcal{G}_1)]^{(k+1)m}\to \R$ a 2-paraconvex and $2$-paraconcave function such that $\nabla_{X_o}G(X_1,...,X_k)\in L^2(W^*(X_1,...,X_k,\mathcal{F}_0))$ for any $o\in [\![1,k]\!]$ and $X_i\in L^2_{sa}(\mathcal{F}_1)^m$.  %Let $\ell\in \R$.
 Assume given $0<t_1<...<t_k\leq 1=t_{k+1}$ and moreover that there exists a family $h_t:[L^2(\mathcal{G}_t)]^{(i+1)m}\to \R$ of 2-paraconvex and $2$-paraconcave functions  with constants independent of $t$ for $t_i<t<s\leq t_{i+1}$ (written $i(t)$) with $||\nabla_{X_{i+1}}h_t(X)-\nabla_{X_{i+1}}h_s(X)||_2\leq C(X)(t-s)^{\alpha}.$

Then, there is at most one solution (families of adapted functions $X^t(s,.)$,$\nabla_{X_{i(s)+1}}h_s(.)$) to the equations for $1\geq s\geq t$:
$$X^t_s=X^t(s,X_1,...,X_{t_{i(t)}},X^t(t))=X^t(t)+S_s-S_t-\int_t^s \nabla_{X^t_{t_{i(u)+1}}}h_u(X_1,...,X_{t_{i(t)}},X^t_{t_i(t)+1},...,X^t_{t_i(u)},X^t_u)du, $$
$$\nabla_{X_{i(s)+1}}h_s(X_1,...,X_{t_{i(s)}},X^s(s))=E_{\mathcal{G}_s}
\Big[\sum_{j=i(s)+1}^{k+1} \nabla_{X^s_{t_j}}G(X_1,...,X_{t_{i(s)}},X^s_{t_{i(s)+1}},...,X^s_{t_{k+1}})\Big],$$
$$\nabla_{X_{j}}h_s(X_1,...,X_{{i(s)}},X^s(s))=E_{\mathcal{G}_s}
\Big[ \nabla_{X^s_{t_j}}G(X_1,...,X_{{i(s)}},
X^s_{t_{i(s)+1}},...,X^s_{t_{k+1}})\Big], j\leq i(s),$$
and it necessarily satisfies for $s\geq t$ for any $X_{t_i},X^t(t)\in \mathcal{F}_t$:
\begin{equation}\label{Xtadapted}X^t_s\in L^2(W^*(X_1,...,X_{t_{i(t)}},X^t(t), S_u-S_t, u\in [t,s], \mathcal{F}_0)),\end{equation}
\begin{equation}\label{nablahadapted}\nabla_{X_{i(s)+1}}h_s(X_1,...,X_{t_{i(s)}},X^s(s))\in L^2(W^*(X_1,...,X_{t_{i(s)}},X^s(s),  \mathcal{F}_0)).\end{equation}
\end{proposition}    
    In practice, the function $h_t$ will be obtained by stochastic control techniques or by taking ultraproducts thereof. The key is of course the uniformity of the paraconvexity constants that this related control problem enables to get.
   \setcounter{Step}{0}
  
\begin{proof}  
\begin{step} Uniqueness \end{step}
The uniqueness part in Theorem \ref{FreeMonotone} does not use $S_t$ is a free brownian motion but only the 2-paraconvexity condition. Thus, if $h_t$ is given, there is at most one solution (whatever the filtration is, especially in $\mathcal{G}_t$) of the SDE for $X^t_s$. Then the second equation determines $\nabla_{X_{i(s)+1}}h_s$ (and $h_s$ also determines it) thus there is at most one solution. The hard part is to show the adaptedness property, which is based on an existence result.

   \begin{step} Solution on $[1-\delta,1]$ for $\delta\in]0,1-t_k[$ only depending on the paraconvexity and paraconcavity constants for $G$.\end{step}  
  This case reduces to Proposition \ref{localFBSDE} after translation of the filtration to get solution on short time intervals.

\begin{step}Iteration\end{step}

Note that from the equations we can reduce to a similar problem on $[0,1-\delta]$ in using $E_{\mathcal{G}_s}=E_{\mathcal{G}_s}E_{\mathcal{G}_{1-\delta}}$ and combining the last term in the sum with the end of the integral and the other terms alone

$$\nabla_{X_{i(s)+1}}h_s(X_1,...,X_{{i(s)}},X^s(s))=E_{\mathcal{G}_s}
\Big[\sum_{j=i(s)+1}^{k+1} \nabla_{X^s_{t_j}}h_{1-\delta}(X_1,...,X_{{i(s)}},X^s_{t_{i(s)+1}},...,
X^s_{t_{k+1}})\Big]$$

$$\nabla_{X_{j}}h_s(X_1,...,X_{{i(s)}},X^s(s))=E_{\mathcal{G}_s}
\Big[ \nabla_{X_{j}}h_{1-\delta}(X_1,...,X_{{i(s)}},X^s_{t_{i(s)+1}},...,X^s_{t_{k+1}})\Big],j\leq i(s).$$

Thus (up to reindexing time) this solve the same problem with $h_{1-\delta}$ replacing $G$, and we can thus go on to solve on $[1-2\delta,1]$ (using the uniformity of lipshitz contants). and thus on $[t_k,1]$. At time $t_k$, one need a slight change in the computation. Indeed the two last terms of the sum have to be gathered to obtain the change of number of time indices :

$$\nabla_{X_{i(s)+1}}h_s(X_1,...,X_{{i(s)}},X^s(s))=E_{\mathcal{G}_s}
\Big[\sum_{j=i(s)+1}^{k} \nabla_{X^s_{t_j}}h_{t_k}(X_1,...,X_{{i(s)}},X^s_{t_{i(s)+1}},...,X^s_{t_{k}})\Big],$$

 and the reasoning goes on until reaching [0,1].
%\begin{step}Case $1<\gamma<2$\end{step}
\end{proof}
  
   \setcounter{Step}{0}   
     
\begin{corollary}
\label{convexFBSDEmatricial}
     Let   $G\in \mathcal{E}^{1,1}_{app}(\mathcal{T}_{2,0}(\mathcal{F}^m_{k}*\mathcal{F}^\nu_{\mu}),d_{2,0})$
(e.g.  $G\in \mathcal{E}_{reg,p}(\mathcal{T}_{2,0}(\mathcal{F}^m_{k}*\mathcal{F}^\nu_{\mu}),d_{2,0})
\cup , p\in [2,\infty[),$ %if $\ell=0$, or $G\in\big(\mathcal{C}_{reg,p,C}(\mathcal{T}_{2,0}^c(\mathcal{F}^m_{k}*\mathcal{F}^\nu_{\mu}),d_{2,0})
%\big)\cap \mathcal{E}^{1,1}_{app}(\mathcal{T}_{2,0}^c(\mathcal{F}^m_{k}*\mathcal{F}^\nu_{\mu}),d_{2,0})
%, p\in [2,\infty[,$ for some $C>0$ if $\ell>0$.
 Fix $0<t_1<...<t_k\leq t_{k+1}=1$ and let  $f=G\circ (I_{t_1,...t_k}*Id)$%-\ell \mathfrak{c}_1$
 and consider for each $N, \Upsilon_N\in\mathcal{U}(M_N(\C))$ converging in law and $X^{G,N,\Upsilon_N,t}$ the solution in Corollary \ref{minhthermitian} (which suppose given various $X^{G,N,\Upsilon_N,s}$ for $s=t$ or $s=t_i\leq t$). Then %there is a constant $\ell(C)$ such that if $\ell \leq \ell(C)$,
  for any ultrafilter $\omega\in \beta\N-\N, \upsilon=(\Upsilon_N)^\omega, S_t=(H_t^N)^\omega$, then $Y^t_s=(X_s^{G,N,\Upsilon_N,t})^\omega\in W^*(\upsilon,S_u-S_t, u\leq s, Y^t_{t_j}, t_j\leq t, Y^t_t)$ as soon as $Y^t_{t_j}\in W^*(\upsilon,S_s,s\leq t_j)^m, t_j\leq t, Y^t_t\in W^*(\upsilon,S_s,s\leq t)^m$. Moreover it  satisfies an SDE with respect to the canonical brownian motion  $S_s\in\mathcal{M}_P^\omega$. There is $u^t_s=u_s^{t,G}:=(b^{G,N,\Upsilon_N,t}(s,X^{G,N,\Upsilon_N,t}(s)))^\omega\in L^2(W^*(\upsilon,Y_{t_1}, ...,Y_{t_i},Y_s))$ for $t_i<s\leq t_{i+1}$ ($i\leq k$) such that $$Y^t_s=Y^t_t+S_s-S_t+\int_t^su_v^Gdv.$$
    \end{corollary} 
     
   \begin{proof}
 \begin{step}Estimate showing $Y^t_s\in \mathcal{M}_P^\omega$ if we assume $X^{G,N,\Upsilon_N,s}$ for $s=t$ and $s=t_i\leq t$ uniformly operator norm bounded%by constants $Y^t_{t_j}\in \mathcal{M}_P^\omega, t_j\leq t, Y^t_t\in \mathcal{M}_P^\omega$
 .\end{step}
 We know from the proof of proposition \ref{minht}  the finite dimensional distribution of $Y^t_s, s\in[t,1]$, and considering $u(X^{G,N,\Upsilon_N,t}_t), u(X^{G,N,\Upsilon_N,t}_{t_i}), t_i\leq t$ as extra unitary variables we can add to $\Upsilon_N$, it is not hard to see it comes from a law of the form considered in an obvious variant of Proposition \ref{ConcentrationNorm}   thus one deduces $Y^t_s\in \mathcal{M}_P^\omega$ in this case.
 
 \begin{step} Limit along $\omega$ of the value function $h_{t,\ell}^{N,\Upsilon_n}(\sqrt{N}.)/N^2$%{\color{red}[Remove all remaining $\ell$]}%, \ell\in [0,1/2|$
 \end{step}
 
 We consider the function of $h_{t,\ell}^{N,\Upsilon_n}$ from Theorem \ref{minimisationHermitian}.

We first examine the limit along $N\to \omega$ of the value function $h_{t,N}(x,\Upsilon_N):=\frac{1}{N^2}h_{t,\ell}{N,\Upsilon_N}(\sqrt{N}x)$. We will be later able to get better convergence results, but, for now, we will be content of the $\omega$ dependent result.

From lemma \eqref{UniformN}, the bounds for $g_N$ in \eqref{lipE} are independent of $d$ and thus from the proof of Proposition \ref{OptimalValueFct}, so are the bounds for $h_t^{N,\Upsilon_n}$ so that one gets constants $C,D>0$ such that for all $N$: 
\begin{align*}&\left|h_{t,N}(x_1,...,x_i,x_{i+1},\Upsilon_N)-h_{t,N}(y_1,...,y_i,y_{i+1},\Upsilon_N)\right|\\&\leq 
\left(C\left(\sum_{K=1}^{i+1}\frac{1}{N}Tr(x_K^*x_K)\right)^{1/2}+C\left(\sum_{K=1}^{i+1}\frac{1}{N}Tr(y_K^*y_K)\right)^{1/2} +D\right)\left(\sum_{K=1}^{i+1}\frac{1}{N}Tr((x_K-y_K)^*(x_K-y_K))\right)^{1/2}\end{align*}
Thus applying this inequality to hermitian random variables $X^N=(X_1^N,...,X_{i+1}^N),Y^N=(Y_1^N,....,Y_{i+1}^N)$, taking expectation and using Cauchy-Schwartz inequality, we have:
\begin{align*}&\left|E(h_{t,N}(X^N,\Upsilon_N))-E(h_{t,N}(Y^N,\Upsilon_N))\right|\leq \left(\sum_{K=1}^{i+1}E\left(\frac{1}{N}Tr((X_K^N-Y_K^N)^2)\right)\right)^{1/2}\\&\times 
\sqrt{E\left(3C^3\left(\sum_{K=1}^{i+1}\frac{1}{N}Tr((X_K^N)^2)\right)+3C^2\left(\sum_{K=1}^{i+1}\frac{1}{N}Tr((Y_K^N)^2)\right) +3D^2\right)}\end{align*}
Thus considering $X=(X^N)^\omega\in (\mathcal{L}_P^\omega)^{m(i+1)}$, one can define :$$h_t^\omega(X,\upsilon)=\lim_{N\to \omega}E(h_{t,N}(X^N,\Upsilon_N)).$$
Indeed, our previous inequality insures this is well-defined, namely, this does not depend on the way  $X=(X^N)^\omega=(Y^N)^\omega\in (\mathcal{L}_P^\omega)^{m(i+1)}.$ Actually, using the Lipschitzianity bound in variable $\Upsilon_N$ in Theorem \ref{minimisationHermitian}, one can show similarly that not only $h_t^\omega(.,\upsilon):(\mathcal{L}_P^\omega)^{m(i+1)}\to\R$ is defined but also :$$h_t^\omega(.,.):(\mathcal{L}_P^\omega)^{m(i+1)}\times \mathcal{U}((M_N(\C))^\omega)\to\R.$$
Indeed, it suffices to note that a unitary in $\mathcal{U}((M_N(\C))^\omega)$ can be represented by a sequence of unitaries (and $\mathcal{U}((M_N(\C))^\omega)$ can even be identified with the ultraproduct of groups $\mathcal{U}(M_N(\C))$ see e.g. \cite[Ex 2.11.6]{CapraroLupini}) and the lipschitzness bound implies the function is well defined on the ultraproduct. We won't really use the second argument except at the fixed value $\upsilon.$

\begin{step} Regularity of the limit $h_t^\omega$.\end{step}

First, $h_t^\omega$ is Lipschitz on bounded sets, from the inequality obtained by taking the limit of the one obtained for $h_{t,N}$ in the previous step 2:
$$|h_t^\omega(X,\upsilon)-h_t^\omega(Y,\upsilon)|\leq ||X-Y||_2\sqrt{3C^2||X||_2^2+3C^2||Y||_2^2+3D^2}.$$

Let us recall that $h_t^\omega(.,\upsilon):(\mathcal{L}_P^\omega)^{m(i+1)}\to \R$ for $t_i<t\leq t_{i+1}$ is a convex function % with %constant $2\ell(1-t)$.
%Indeed, for $\lambda\in [0,1]$, the relation $$E(h_{t,N}(\lambda X^N+(1-\lambda) Y^N),\Upsilon_N)\leq \lambda E(h_{t,N}(X^N),\Upsilon_N)+(1-\lambda)E(h_{t,N}(Y^N),\Upsilon_N) +2\ell (1-t)\lambda(1-\lambda)E(||X^N- Y^N||_2^2) $$
 from the convexity of $h_{t,N}$ and a limit $N\to \omega$.

From equation \eqref{SubquadE} and the uniformity of the constants in propositions \ref{UniformN} and \ref{OptimalValueFct}, one gets:
$$h_t^\omega(X,\upsilon)\leq C(1+||X||_2^2).$$
Finally, since equation \eqref{CalphaE} is checked with $D_2=0$ and constants independent of $N$, from proposition \ref{UniformN}, one deduces from proposition \ref{OptimalValueFct}:
\begin{equation}\label{C11htomega}
h_t^\omega(X+Y,\upsilon)-2h_t^\omega(X,\upsilon)+h_t^\omega(X-Y,\upsilon)\leq C||Y||_2^2.
\end{equation} 

Thus from proposition \ref{C11}, $h_t^\omega$ is differentiable on $(\mathcal{L}_P^\omega)^{m(i+1)}$ with Lipschitz derivative. Using \eqref{Incrementsht}, one also deduces (remembering that in our case $D_{2,k,i}=0$) that there are constants $C,D$ such that for all $t,t+s\in ]t_i,t_{i+1}[:$
\begin{align*}&\left|E(h_{t,N}(X_1,...,X_i,X,\Upsilon_N)-h_{t,N}(X_1,...,X_i,Y,\Upsilon_N)\right.\\&\left.-h_{t+s,N}(X_1,...,X_i,X,\Upsilon_N)+ h_{t+s,N}(X_1,...,X_i,Y,\Upsilon_N)\right| \\&\leq  ||Y-X||_2\sqrt{s}\times (C+D ||(X_1,...,X_i,Y)||_2+D||(X_1,...,X_i,X)||_2).\end{align*}
Thus one deduces in taking the limit $N\to\omega$:
\begin{align*}&\left|h_t^\omega(X_1,...,X_i,X,\upsilon)-h_t^\omega(X_1,...,X_i,Y,\upsilon)-h_{t+s}^\omega(X_1,...,X_i,X,\upsilon)+ h_{t+s}^\omega(X_1,...,X_i,Y,\upsilon)\right| \\&\qquad\leq  ||Y-X||_2\sqrt{s}\times (C+D ||(X_1,...,X_i,Y)||_2+D||(X_1,...,X_i,X)||_2).\end{align*}
and thus for all $t,t+s\in ]t_i,t_{i+1}[,s>0:$\begin{equation}\label{Holderhtomega}\left\|\nabla_Xh_t^\omega(X_1,...,X_i,X,\upsilon)
-\nabla_Xh_{t+s}^\omega(X_1,...,X_i,X,\upsilon)\right\|_2 %\\&\qquad
\leq  \sqrt{s} (C+2D ||(X_1,...,X_i,X)||_2).\end{equation}

 This $h_t^\omega$ will be the function to which we will apply the previous proposition, we now need the alternative formulas needed there.

\begin{step} Formula for $\nabla h_t^\omega$.\end{step}

First note that by the fundamental theorem of calculus, the identity $\mathscr{D}_{i+1}^k(h_{t,N})(x_1,...,x_i,x,\Upsilon_N)=\frac{\sqrt{N}}{N^2}\mathscr{D}_{i+1}^k(h_t^{N,\Upsilon_N})(\sqrt{N}x_1,...,\sqrt{N}x_i,\sqrt{N}x)$  and the bounds on gradients of $h_{t,N}$ from step 3 and proposition \ref{C11}, one gets:

\begin{align*}&\left|h_{t,N}(x_1,...,x_i,x,\Upsilon_N)-h_{t,N}(x_1,...,x_i,y,\Upsilon_N)\right.\\&\left.-\sum_{k=1}^m\frac{1}{N\sqrt{N}}Tr(\mathscr{D}_{i+1}^k(h_t^{N,\Upsilon_N})](\sqrt{N}x_1,...,\sqrt{N}x_i,\sqrt{N}x)(y-x)_k)\right|%\leq \int_0^1du\\& \left|\sum_{k=1}^mTr(\mathscr{D}_{i+1}^k(h_{t,N})](x_1,...,x_i,x+u(y-x))(y-x)_k)-\sum_{k=1}^mTr(\mathscr{D}_{i+1}^k(h_{t,N})](x_1,...,x_i,x)(y-x)_k)\right|
\leq \frac{C}{2}||y-x||_2^2.
\end{align*}
with the constant $C$ of equation \eqref{C11htomega}.

We now want to use, for $t_i<t\leq t_{i+1}$ fixed, the two last formulas in corollary \ref{minhthermitian}.
We thus consider a random variable $X^N=(X_1^N,...,X_i^N,X_{i+1}^N)$ a vector of hermitian matrices and $\mathcal{F}_t$ -measurable (i.e. adapted to the filtration of hermitian brownian motion), and consider the solution $X^{G,N,t}(s)$ in this corollary starting from $X^{G,N,\Upsilon_N,t}(s)=X_{i+1}^N$ which is independent of the noise appearing in the equation. %As before in (i)%(the density of finite dimensional distributions of the process in the proof  in proposition \ref{minht};
One defines $Y^t_s=(X^{G,N,\Upsilon_N,t}(s))^\omega\in\mathcal{L}_P^\omega$ and then note that $u^t_s=u^{t,G}_s=( b^{G,N,\Upsilon_N,t}(u, X^{G,N,\Upsilon_N,t}(u)))^\omega\in\mathcal{L}_P^\omega$ (from the uniform bound coming from lipschitzness of $h_t{t,\ell}^{N,\Upsilon_N}$).

From the a priori H\"older continuity bounds, we know for $s\geq t$: 
$$Y^t_s=(X_{i+1}^N)^\omega+\int_t^sdv u^t_v+(S_s-S_t).$$
Using the 2-paraconvexity of $h^\omega$, one deduces in using the characterization of the gradient as Clarke-subdifferential and the formula in  proposition \ref{alphaparaconvex}, that for $t_I<s<t_{I+1}:$ $$u^t_s=\nabla_{Y^t_s}h_s^\omega(X_1,...,X_i,Y^t_{t_i},...,Y^t_{t_{I}},Y^t_s,\upsilon)$$

Similarly, one takes $$v^t_l=\sum_{j=i+1}^k\left(\frac{1}{\sqrt{N}}(\mathscr{D}_{j}^lg_N
  (\sqrt{N}X_1^N,...,\sqrt{N}X_i^N,\sqrt{N}X^{G,N,\Upsilon_N,t}(t_{i+1},X_{i+1}^N),...,\sqrt{N}X^{G,N,\Upsilon_N,t}(t_{k},X_{i+1}^N)))\right)^\omega\in \mathcal{L}_P^\omega$$ (the a priori boundedness from \eqref{lipE}).
  First note that as for $h_{N,t}$, $$G^\omega(X^\omega,\upsilon)=\lim_{N\to \omega}E(\frac{1}{N^2}g_N
  (\sqrt{N}X^N))$$ defines a well-defined convex function on the ultraproduct (this is actually the special case $G^\omega=h_1^\omega$).
Considering the convexity relation:
\begin{align*}&\sum_{l=1}^m \frac{1}{N}Tr\left(H_l[\sum_{j=i+1}^k(\frac{1}{\sqrt{N}}(\mathscr{D}_{j}^lg_N
  (\sqrt{N}X_1^N,...%,\sqrt{N}X_i^N
  ,\sqrt{N}X^{G,N,\Upsilon_N,t}(t_{i+1},X_{i+1}^N),...,\sqrt{N}X^{G,N,\Upsilon_N,t}(t_{k},X_{i+1}^N)]\right)\\&\leq-\frac{1}{N^2}g_N
  (\sqrt{N}X_1^N,...,\sqrt{N}X_i^N,\sqrt{N}X^{G,N,\Upsilon_N,t}(t_{i+1},X_{i+1}^N),...,\sqrt{N}X^{G,N,\Upsilon_N,t}(t_{k},X_{i+1}^N)) \\&\qquad + \frac{1}{N^2}g_N
  (\sqrt{N}X_1^N,...%,\sqrt{N}X_i^N
  ,\sqrt{N}[X^{G,N,\Upsilon_N,t}(t_{i+1},X_{i+1}^N)+H],...,[\sqrt{N}X^{G,N,\Upsilon_N,t}(t_{k},X_{i+1}^N)+H]).\end{align*}
From which, replacing $H$ by $H^N$ random and then defining $H=(H^N)^\omega$, taking the expectation and the limit $N\to \omega$ one deduces:
$$\sum_{l=1}^m\langle H_l,v^t_l\rangle\leq -G^\omega(X_1^\omega,...,X_i^\omega,Y^t_{t_{i+1}},...,Y^t_{t_{k}},\upsilon)+ G^\omega(X_1^\omega,...,X_i^\omega,Y^t_{t_{i+1}}+H,...,Y^t_{t_{k}}+H,\upsilon).$$
and as a consequence one deduces $$v^t=\sum_{j=i+1}^k\nabla_{Y^t_{t_{j}}}G^\omega(X_1^\omega,...,X_i^\omega,Y^t_{t_{i+1}},...,Y^t_{t_{k}},\upsilon).$$

Now, one can take the limit $N\to\omega$ in the first relation obtained in this point (vii) for $X_l^\omega,Y,Y_l$ $\mathcal{L}_{P,t}^\omega$ -measurable after taking expectations and using the next-to-last equation in corollary \ref{minhthermitian} and get:
\begin{align*}&\left|h_t^\omega(X_1^\omega,...,X_i^\omega,X_{i+1}^\omega,\upsilon)-h_t^\omega(Y_1,...,Y_i,Y,\upsilon)-\sum_{l=1}^m\langle E_{\mathcal{L}_{P,t}^\omega}(v^t_l)%-2\ell\int_t^1 Y^t_{s,l}ds
),(Y-X_{i+1}^\omega)_l\rangle\rangle\right|\\& \qquad \qquad \qquad \qquad \qquad \qquad \qquad \qquad \leq \frac{C}{2}||X_{i+1}^\omega-Y||_2^2\end{align*}

and thus $$\nabla_{X_{i+1}^\omega}h_t^\omega(X_1^\omega,...,X_i^\omega,X_{i+1}^\omega,\upsilon)=E_{\mathcal{L}_{P,t}^\omega}
(\sum_{j=i+1}^k\nabla_{Y^t_{t_{j}}}G^\omega(X_1^\omega,...,
X_i^\omega,Y^t_{t_{i+1}},...,Y^t_{t_{k}},\upsilon)%-2\ell\int_t^1 Y^t_{s,l}ds
).$$
The formula for $\nabla_{X_{j}^\omega}h_t^\omega,j\leq i$ is obtained similarly.

\begin{step} Computation of $G^\omega, \nabla_XG^\omega$ on $M=W^*(\upsilon, S_t)$.\end{step}

Finally, we will need a way to compute $G^\omega$, since they appear in the above formula for $\nabla_{X_{i+1}^\omega}h_t^\omega$. We start with the value on $(X_1,...,
X_k)\in L^2(M)^{km}, X_i=X_i^*$ such that there is $X_0=X_0^*\in M$ with $W^*(\upsilon,X_0,X_1,...,X_k)=W^*(\upsilon,X_0, u(X_1),...,u(X_k))$ is a factor. Then for any model with $(X_i^N)^\omega=X_i$ which implies as above $(u(X_i^N))^\omega=u(X_i)$. Let $Y=(X_0,u(X_1)+u(X_1)^*,i( u(X_1)-u(X_1)^*),...,i( u(X_k)-u(X_k)^*)), Y^N$ similarly, $\tau_{Y,\upsilon}\in \mathcal{S}_R^{(2k+1)m}\star \mathcal{T}(\mathcal{F}^\nu_{\mu})$ is extremal and we can apply proposition \ref{ConcentrationPoulsen} so that from $\lim_{N\to\omega}d_{2,0}(E\circ \tau_{Y^N,\Upsilon_N},\tau_{Y,\upsilon})=0,$ one deduces $\lim_{N\to\omega}E(d_{2,0}(\tau_{Y^N,\Upsilon_N},\tau_{Y,\upsilon}))=0$ and then in rewriting variables in terms of $X_1,...,X_k$ and using the lipschitzness of $G$:
$$G^\omega(X_1,...,X_k,\upsilon)=G(\tau_{X_1,...,X_k,\upsilon})\ \ if \ \ W^*(X_0,X_1,...,X_k,\upsilon)\ factor.$$
We now establish a variant for the derivative. From the approximation property in the definition of $\mathcal{E}^{1,1}_{app}(   \mathcal{T}_{2,0}^c(\mathcal{F}^m_{[0,1]}*\mathcal{F}^{\nu}_{\mu}),d_{2,0})$, for $\epsilon>0$, one can fix $P_1^1,...,P_L^m\in \mathcal{F}^m_{k}*\mathcal{F}^{\nu}_{\mu}$,  $f_1^1,...,f_L^m, g_{1,1,1},...,g_{k,m,m}\in C^0(  \mathcal{T}_{2,0}(\mathcal{F}^m_{k}*\mathcal{F}^{\nu}_{\mu}),d_{2,0}) )$ such that 
inserting in the first equation of (vii) for $t=1$, $X=(x_1,...,x_k)$
\begin{align*}&\left|G(\tau_{x_1,...,x_k,\Upsilon_N})-G(\tau_{x_1,...,x_o+y ,...,x_k,\Upsilon_N})-\sum_{K=1}^m\right.\\&\left.\frac{1}{N}Tr\left([\sum_{i=1}^L P_i^K(u(X),\Upsilon_N)f_i^K(\tau_{X,\Upsilon_N})+ \sum_{i=1}^m\sum_{j=1}^kx_j^{(i)}g_{j,i,K}(\tau_{X,\Upsilon_N})]y_K\right)\right|%\leq \int_0^1du\\& \left|\sum_{k=1}^mTr([\mathscr{D}_{i+1}^k(h_{t,N})](x_1,...,x_i,x+u(y-x))(y-x)_k)-\sum_{k=1}^mTr(\mathscr{D}_{i+1}^k(h_{t,N})](x_1,...,x_i,x)(y-x)_k)\right|
\leq \frac{C}{2}||y||_2^2+\epsilon.
\end{align*}
Taking the limit $N\to \omega$, treating each function of trace by concentration as before, and finally letting $\epsilon\to 0$, one gets 
if $W^*(X_0,X_1,...,X_k,\upsilon)$ factor, then for all $o$: $\nabla_{X_o}G^\omega(X_1,...,X_k)=\nabla_{X_o}G(\tau_{X_1,...,X_k})\in L^2(W^*(\upsilon,X_1,...,X_k))^m.$

Finally, let us extend this to any $(X_1,...,X_k)\in L^2(M)^{mk}, X_i=X_i^*$.

By Clarck Ocone's formula \cite{BianeSpeicher}, for $B=W^*(\upsilon)$, each $X_i=E_B(X_i)+\int_0^1 U_s \# dS_s$ thus there is $X_{i,\epsilon}=E_B(X_i)+\int_{\eta(\epsilon)}^1 U_s dS_s$ with $||X_{i,\epsilon}-X_i||_2\leq \epsilon$. Moreover if $X_{0,\epsilon}=S_{\eta(\epsilon)}$, let us check that $W^*(X_{0,\epsilon},X_{1,\epsilon} ,...,X_{k,\epsilon},\upsilon)$ is a factor and thus satisfies our previous assumption. Indeed from \cite{V5}, $X_{0,\epsilon}$ have bounded first and second order conjugate variables, thus, if $m\geq 2$ (the case $m=1$ is similar in taking two increments $S_{\eta(\epsilon)/2},S_{\eta(\epsilon)}-S_{\eta(\epsilon)/2}$ instead of one) from \cite[Rmk 11]{Dab08}, $X_{0,\epsilon}$ is a non-$\Gamma$ set thus a non-amenability set (\cite{Connes76}, see e.g. \cite[Def 2.4, lemma 2.10]{DabrowskiIoana}) and $W^*(X_{0,\epsilon},X_{1,\epsilon} ,...,X_{k,\epsilon},\upsilon)\subset W^*(X_{0,\epsilon})*A=M$ is contained in some free product (using the free brownian motion property and $\upsilon$ free from $X_{0,\epsilon}$). For any $Z$ in the center, $E_{W^*(X_{0,\epsilon})}(Z)$ is in the center of $W^*(X_{0,\epsilon})$ which is a factor (since it is non-$\Gamma$, even a free group factor in our example) thus $Z-\tau(Z)=Z-E_{W^*(X_{0,\epsilon})}(Z)\in L^2(W^*(X_{0,\epsilon})*A)\ominus L^2((W^*(X_{0,\epsilon}))$ which is well-known to be a (countable) direct sum of coarse correspondences over $W^*(X_{0,\epsilon})$. But the non-amenability set property implies :
$$||Z-E_{W^*(X_{0,\epsilon})}(Z)||_2\leq K\sum_{i=1}^m||[X_{0,\epsilon},Z-E_{W^*(X_{0,\epsilon})}(Z)]||_2=0.$$
Thus $Z=\tau(Z)$ and we deduce the expected factoriality. Thus, from the first case, we deduce: $G^\omega(X_{1,\epsilon} ,...,X_{k,\epsilon},\upsilon)=G(\tau_{X_{1,\epsilon} ,...,X_{k,\epsilon},\upsilon})$ and 
$\nabla_{X_{o,\epsilon}}G^\omega(X_{1,\epsilon} ,...,X_{k,\epsilon},\upsilon)=\nabla_{X_{o,\epsilon}}G(\tau_{X_{1,\epsilon} ,...,X_{k,\epsilon},\upsilon}).$

Since $G,G^\omega$ and their gradients are continuous (even Lipschitz) by assumption, one deduces in taking $\epsilon\to 0$ that for any $(X_1,...,X_k,X)\in L^2_{sa}(M)^{(k+1)m}$
\begin{align}\label{GomegaG}G^\omega(X_1,...,X_k,\upsilon)&=G(\tau_{X_1,...,X_k,\upsilon}),
\\\label{DGomegaG}\nabla_{X_o}G^\omega(X_1,...,X_k,\upsilon)&=\nabla_{X_o}G(\tau_{X_1,...,X_k,\upsilon})\in L^2(W^*(X_1,...,X_k,\upsilon))^m.\end{align}

 \begin{step} Conclusion.\end{step} 
We are now ready to apply proposition \ref{convexFBSDE}. We apply it to $\mathcal{G}_t=\mathcal{M}_{P,t}^\omega$ the filtration of ultraproducts and $\mathcal{F}_t=W^*(S_s,s\leq t,\upsilon)$ inside it.
 This is indeed a subfiltration by Example \ref{ExFiltration}.(3). 
 
Note also that from the convergence in law of $(\Upsilon_N)$ which are deterministic, it is known that $S_t$ is a free brownian motion in the filtration  $(\mathcal{F}_s)$. We take $G^\omega$ as our function $G$ 
 so that \eqref{DGomegaG} enables to check the expected assumption on its gradient. We take $h_t^\omega$ (restricted to a power of $L^2(\mathcal{G}_t)$) as $h_t$ and it is convex and 2-paraconcave with the right uniformity and right regularity in time of its gradient by step 3. Step 4 then checks the SDE  (and note that since from step 1, $Y_s^t\in\mathcal{G}_s^m$ the drift is also by differentiation in $L^2(\mathcal{G}_s)^m$
  and the relations on gradients expressed as conditional expectation after projection on $L^2(\mathcal{G}_t)^m$ (which gives for the left hand sides the gradient of the restriction and the expected projection for the right hand sides). The conclusion of the proposition then gives the space of values for $Y^t_s,u^G_s.$
% [ADD LIPSCHITZ SOL TO Check independence of representative seuence of initial conditions.]
\end{proof}

\section{Partial Laplace principal for convex functions for Hermitian Brownian motion} 
%\subsection{Limit for convex functions}
    \setcounter{Step}{0}
     We first want to define the candidate for the limiting function in the Laplace principle for $f\in \mathcal{E}_{reg,p}(   \mathcal{T}_{2,0}^c(\mathcal{F}^m_{[0,1]}*\mathcal{F}^\nu_{\mu}),d_{2,0}).$
     
   We fix  $\Upsilon_N\in \mathcal{U}(M_N(\C))$ (deterministic).
Assume finally that the non-commutative law $\tau_{\Upsilon_N}$ converges to some $\mu_\Upsilon\in (\mathcal{T}(\mathcal{F}^\nu_{\mu}),d)$. Recall that we aim at proving a Laplace principle for $\widehat{\sigma}^N_\Upsilon$ the joint law of $\Upsilon_N$ and an hermitian brownian motion.

Recall 
 $\gamma_{sa,N,m}=\gamma_N$ is the law of the $m$ hermitian brownian motions.

We can consider the von Neumann algebra $M_{N}(L^\infty(\mathbb{W}_{sa, N},\gamma_N))$ in which lives random matrix processes over this probability space. This is a finite non-commutative probability space with trace $\tau_{\gamma_N}=E_{\gamma_N}\circ \frac{1}{N}Tr$. We recall we use $\beta\N-\N$ the set of non-principal ultrafilters $\omega\in \beta\N-\N$ on $\N$. One  considers the tracial ultraproduct $\mathcal{M}_P^\omega=(M_{N}(L^\infty(\mathbb{W}_{sa, N},\gamma_N)),\tau_{\gamma_N})^\omega$. Of course there is a natural filtration $\mathcal{M}_{P,s}^\omega=(M_{N}(L^\infty(X_t, t\leq s)),\tau_{\gamma_N})^\omega,$ where $X_t$ is the coordinate process of our hermitian process and $L^\infty(X_t, t\leq s)$ is the generated commutative von Neumann algebra. We will also need $\mathcal{L}_P^\omega=[L^2_{sa}(M_{N}(L^\infty(\mathbb{W}_{sa, N},\gamma_N)),\tau_{\gamma_N})]^\omega$ and $\mathcal{L}_{P,s}^\omega=[L^2_{sa}(M_{N}(L^\infty(X_t, t\leq s)),\tau_{\gamma_N})]^\omega.$

Let $M=W^{*}(\upsilon,S_t^i, i=1,...,m, t\in [0,1])$   the von Neumann algebra of a free brownian motion free from $\upsilon$ of law $\mu_\Upsilon$ and $M_s=W^{*}(u,S_t^i, i=1,...,m, t\in [0,s])$ its canonical filtration. Fix a sequence of times $\mathbf{t}=(t_0=0<t_1<....< t_k\leq t_{k+1}=1)$ We consider a set of adapted paths which are sufficiently H\"older continuous except at a sequence of times $$\mathcal{P}_{ad,1/8,\mathbf{t}}=\{u\in L^\infty_{ad}([0,1], L^2_{sa}(M)^m): \exists C>0, \forall (s,t)\in ]t_i,t_{i+1}]^2 ||u_t-u_s||_2\leq C(t-s)^{1/8}\}$$
\begin{align*}&\mathcal{P}_{ad,1/8loc,\mathbf{t}}=\{u\in L^\infty_{ad}([0,1], L^2_{sa}(M)^m):\\& \qquad\forall [a,b]\subset]t_i,t_{i+1}[,\ \exists C>0, \ \forall (s,t)\in ]a,b[^2,\  ||u_t-u_s||_2\leq C(t-s)^{1/8}\}\end{align*}
     
We consider another family of paths such that one requires more von Neumann algebraic regularity at this sequence of times:  
     \begin{align*}&\mathcal{P}_{ad,factor,\mathbf{t}}=\{u\in L^\infty_{ad}([0,1], L^2_{sa}(M)^m):\\&\qquad \forall i=1,...,k, %\exists \tau_i<t_i<T_i,\ \forall t\in [\tau_i,T_i],
     \ Y_{t_i}=S_{t_i}+\int_0^{t_i}u_sds\in M^m \ and \ W^*(\upsilon,Y_{t_1},...,Y_{t_k})\ \mathrm{\ is\  a\ factor}\ \}.
     \end{align*}
 and the variant     \begin{align*}&\mathcal{P}_{ad,factor,\mathbf{t},b}=\{u\in L^\infty_{ad}([0,1], L^2_{sa}(M)^m):\\&\qquad \exists C \forall i=1,...,k, \exists \tau_i<t_i<T_i,\ \forall t\in [\tau_i,T_i],\ Y_{t}=S_{t}+\int_0^{t}u_sds\in M^m, ||Y_{t,l}||\leq C\\ &\qquad and \ W^*(\upsilon,Y_{t_1},...,Y_{t_k})\ \mathrm{\ is\  a\ factor}\ \}.
     \end{align*}

     We are finally ready to define our candidate to be a limiting function for $f\in \mathcal{E}_{reg,p}( \mathcal{T}_{2,0}^c(\mathcal{F}^m_{[0,1]}*\mathcal{F}^\nu_{\mu}),d_{2,0}))$ for $p\in[ 2,\infty]$.  We even define it for any $ f\in C^0_{(k)}(   \mathcal{T}_{2,0}^c(\mathcal{F}^m_{[0,1]}*\mathcal{F}^\nu_{\mu}),d_{2,0}),$ bounded from below. %$ f\in \mathcal{E}(   \mathcal{T}_2^c(\mathcal{F}^m_{[0,1]}),d_2)$.
      For we call $\mathbf{t}(f)=(t_1<...<t_k)$ the minimal set of times such that $f=g\circ (I_{t_1,...t_k}*Id).$
          $$\Lambda_\upsilon(f):=\inf\{\frac{1}{2}\int_0^1||u_s||^2ds+f(\tau_{S+\int_0^{.}u_sds,\upsilon}): u\in \mathcal{P}_{ad,1/8loc,\mathbf{t}(f)}\cap \mathcal{P}_{ad,factor,\mathbf{t}(f)}\}.$$
          $$\Lambda_{b,\upsilon}(f):=\inf\{\frac{1}{2}\int_0^1||u_s||^2ds+f(\tau_{S+\int_0^{.}u_sds,\upsilon}): u\in \mathcal{P}_{ad,1/8,\mathbf{t}(f)}\cap \mathcal{P}_{ad,factor,\mathbf{t}(f),b}\}.$$
  
 This second definition will be more important for us. We give a name to the piecewise linear part \begin{align*}&\mathcal{P}_{ad,1/8loc,\mathbf{t},pl}=\{u\in\mathcal{P}_{ad,1/8loc,\mathbf{t}}:\\& \exists s_0^1=0<s_1^1<...<s_n^1<... s_\omega^1=s_0^2<...<s_\omega^2<... s_\omega^k=s_0^{k+1}<s^{k+1}_1<...s_l^{k+1}\\&\qquad
 \forall i,\  s_\omega^i=t_i, s^i_{n+1}-s^i_n=O(\frac{1}{2^n}),\  \forall i,j\geq 1,\ u_{s_i^{j}}\in L^2_{sa}(M_{s_{i-1}^j})^m, u_{s_0^{j}}=u_{s_1^{j}} \\&\qquad\exists C,K>0, \forall i>K, Y_{s_i^l}=S_{s_i^l}+\int_0^{s_i^l}u_sds\in M^m, ||Y_{s_i^l,l}||\leq C, \\&\qquad\forall s\in [s_i^j,s_{i+1}^j],\quad u_s=\frac{s-s_i^j}{s_{i+1}^j-s_i^j} u_{s_{i+1}^j}+\frac{s_{i+1}^j-s}{s_{i+1}^j-s_i^j}u_{s_{i}^j}\}\end{align*}
   We call $\mathbf{s}(u)$ the minimal sequence of times appearing in the definition.
Note that for $u\in   \mathcal{P}_{ad,1/8,\mathbf{t},pl}$, we have: 
\begin{equation}\label{PLInt}\int_{s_i^l}^{s_{i+1}^l}u_sds=(s_{i+1}^l-s_i^l)\frac{ u_{s_{i+1}^l}+u_{s_{i}^l}}{2}.\end{equation}

We start by finding a first estimate for an alternative formula for    $\Lambda_{b,\upsilon}(f)$ that will be more convenient for the Laplace deviation upper bound since a piecewise linear functional depends locally on finitely many values and are easier to make converge in ultraproducts.

\begin{lemma}\label{EquivLambda} For any $ f\in C^0_{(k)}(   \mathcal{T}_{2,0}^c(\mathcal{F}^m_{[0,1]}*\mathcal{F}^\nu_{\mu}),d_{2,0}),$ bounded from below: 
$$\Lambda_{\upsilon}(f)\leq \inf\{\frac{1}{2}\int_0^1||u_s||^2ds+f(\tau_{S+\int_0^{.}u_sds,\upsilon}): u\in \mathcal{P}_{ad,1/8loc,\mathbf{t}(f)%(t_0=0<t_1=1)
,pl}\cap \mathcal{P}_{ad,factor,\mathbf{t}(f)}\}\leq \Lambda_{b,\upsilon}(f).$$
\end{lemma}

\begin{proof}
Clearly, we have the first $\leq$ since $\mathcal{P}_{ad,1/8loc,\mathbf{t}(f),pl}\subset \mathcal{P}_{ad,1/8loc,\mathbf{t}(f)}.$ For the second inequality, fix $\mathbf{t}(f)=(t_0=0<t_1<....< t_k\leq t_{k+1}=1),$ %$\delta=\min\{|t_{l}-t_{l-1}|, l\geq 1\}$. Let $\epsilon_n=1/n\leq \delta/10$, 

Fix $n\geq 2$ and define for $k\geq l\geq 1,$ $$s_i^l=t_{l-1}+(t_l-t_{l-1})\frac{i}{n},\ 0\leq i\leq n-1$$
and $$s_i^l=t_{l-1}+(t_l-t_{l-1})\frac{n-1}{n} + (t_l-t_{l-1})\frac{1}{n}\sum_{K=1}^{i-n+1}\frac{1}{2^K},\ i> n-1.$$

   Fix $v\in \mathcal{P}_{ad,1/8,\mathbf{t}(f)}\cap \mathcal{P}_{ad,factor,\mathbf{t}(f),b}$, with Holder constant $C$ and uniform bound $||v||_\infty=\sup_{t\in [0,1]}||v_t||_2,$ and define $u^{(n)}\in \mathcal{P}_{ad,1/8,\mathbf{t}(f),pl}$ as follows.
   First we take $$u^{(n)}_{s_i^l}=v_{s_{i-1}^l},\ k\geq l\geq 1,\ 1\leq i\leq n-1 $$ and $u^{(n)}_{s_0^l}=v_{s_{0}^l}$ which is compatible with the measurability constraint $i,j\geq 1,\ u_{s_i^{j}}\in L^2_{sa}(M_{s_{i-1}^j})^m.$ Thanks to the H\"older continuity of $v$ this will guaranty a good uniform approximation. Of course we take a piecewise linear interpolation. We now want to guaranty properties near $t_i$, and for that we want $$\int_{t_{l-1}}^{s_i^l}u^{(n)}_{s}ds =\int_{t_{l-1}}^{s_{i-1}^l}v_{s}ds ,\ k\geq l\geq 1,\ i\in [\![ n, \omega]\!] .$$

Since we expect a piecewise linear interpolation, one can use \eqref{PLInt} to obtain :   
      $$\int_{t_{l-1}}^{s_i^l}u^{(n)}_{s}ds=\sum_{K=0}^{i-1}(s_{K+1}^l-s_K^l)\frac{u_{s_{K+1}^l}^{(n)}+u_{s_{K}^l}^{(n)}}{2}$$
This determines:
$$u_{s_{n}^l}^{(n)}=\frac{2}{(s_{n}^l-s_{n-1}^l)}\left(\int_{t_{l-1}}^{s_{n-1}^l}v_{s}ds-\sum_{K=0}^{n-2}\frac{(t_l-t_{l-1})}{n}\frac{u_{s_{K+1}^l}^{(n)}+u_{s_{K}^l}^{(n)}}{2}\right)-u_{s_{n-1}^l}^{(n)}$$
and then inductively  $u_{s_{i+1}^l}^{(n)}$ for $i\geq n$:
$$u_{s_{i+1}^l}^{(n)}=\frac{2}{(s_{i+1}^l-s_i^l)}\int_{s_{i-1}^l}^{s_{i}^l}v_sds-u_{s_{i}^l}^{(n)}.$$

Note that those inductive definitions are compatible with the measurability constraint for $i,j\geq 1,\ u_{s_i^{j}}\in L^2_{sa}(M_{s_{i-1}^j})^m.$

Especially, we can bound inductively using the H\"older continuity of $v$, for $i\geq n$: \begin{align*}&||u_{s_{i+1}^l}^{(n)}-v_{s_{i}^l}||_2\\&\leq \frac{1}{(s_{i+1}^l-s_i^l)}\int_{s_{i-1}^l}^{s_{i}^l}||v_s-v_{s_{i}^l}||_2ds +\frac{1}{(s_{i+1}^l-s_i^l)}\int_{s_{i-1}^l}^{s_{i}^l}||v_s-v_{s_{i-1}^l}||_2ds +||u_{s_{i}^l}^{(n)}-v_{s_{i-1}^l}||_2%\\&\leq \frac{C}{(s_{i+1}^l-s_i^l)}\int_{s_i^l}^{s_{i+1}^l}|s-s_{i+1}^l|^{1/8}ds +\frac{C}{(s_{i+1}^l-s_i^l)}\int_{s_i^l}^{s_{i+1}^l}|s-s_{i+1}^l|^{1/8}ds +||u_{s_{i}^l}^{(n)}-v_{s_{i}^l}||_2
\\&\leq 4C|s_{i-1}^l-s_i^l|^{1/8} +||u_{s_{i}^l}^{(n)}-v_{s_{i-1}^l}||_2
\leq 4C\sum_{I=n}^i|s_{I-1}^l-s_I^l|^{1/8}+||u_{s_{n}^l}^{(n)}-v_{s_{n}^l}||_2
\\&\leq \frac{2C (t_l-t_{l-1})^{1/8}}{n^{1/8}}\sum_{I=n}^i\frac{1}{2^{(I-n+1)/8}}+||u_{s_{n}^l}^{(n)}-v_{s_{n-1}^l}||_2.\end{align*}

  One thus gets from the crude bound $||u_{s_{n}^l}^{(n)}||_2\leq  (8n+1)||v||_\infty$ and by the converging geometric series in $i$ above and convex combinations that $u_{s}^{(n)}$ is bounded in $L^2_{sa}(M)$.

From our construction we have:
   $$\int_{t_{l-1}}^{t_l}u^{(n)}_{s}ds =\int_{t_{l-1}}^{t_l}v_{s}ds ,\ k\geq l\geq 1%,\ i\in [\![ n, \omega]\!]
    ,$$
and then by induction:
   $$\int_{0}^{s_i^l}u^{(n)}_{s}ds =\int_{0}^{s_{i-1}^l}v_{s}ds ,\ k\geq l\geq 1,\ i\in [\![ n, \omega]\!] .$$

   Thus for any $n$, one can deduce from $Y_{t_i}(v)=Y_{t_i}(u^{(n)})\in M^m,$ and the factoriality condition is thus also kept. We also deduce $Y_{s_i^l}(v)=Y_{s_i^l}(u^{(n)})\in M^m,$ for $i$ large enough and the bounded constraint in $M^m$ near $t_l$ also follows from the one for $v$.  
 We also have for the same reason   $f(S+\int_0^{.}v_sds)=f(S+\int_0^{.}u_s^{(n)}ds)$ and from the bounds before and H\"older continuity of $v$ one also easily prove $$\frac{1}{2}\int_0^1||u_s^{(n)}||^2ds\to_{n\to\infty} \frac{1}{2}\int_0^1||v_s||^2ds.$$   
   
   It only remains to check the H\"older continuity condition for $u^{(n)}.$
First if $s_i^l$ is the smallest value above or the highest value below s, the linear interpolation implies $||u^{(n)}_{s}-  u^{(n)}_{s_i^l}||\leq M ||u^{(n)}||_\infty |s-s_i^l|$ where $M$ is finite as soon as extreme points remain within $[a,b]\subset ]t_{l-1},t_l[$ and we get lipschitzness within intervals of $[s_i^l,s_{i+1}^l]$ similarly. The H\"older continuity thus easily follows from the one with endpoints at times $s_i^l, s_I^l$. If these points are those for which $u^{(n)}$ coincides with values of $v$ we are done, otherwise, we can take either $i=n-1, I\geq n$ or $I>i\geq n$ but there are finitely many ratios with such endpoints in  $[a,b]\subset ]t_{l-1},t_l[$, thus there is necessarily even a Lipschitz bound.\end{proof}   
   
   We now want to prove the Laplace principle we aim at getting for convex functionals. This is our main technical result in the convex case. 
    
\begin{theorem}\label{MainTechnical}
Fix  $\Upsilon_N\in \mathcal{U}(M_N(\C))^{\mu\nu}$ (deterministic).
Assume that the non-commutative law $\tau_{\Upsilon_N}$ converges to some $\mu_\Upsilon\in (\mathcal{T}(\mathcal{F}^\nu_{\mu}),d)$. \textbf{We assume either $m\geq 2$ or $m=1$ and $W^*(\mu_\Upsilon)$ diffuse.}

Let $\gamma_{sa,N,m}=\gamma_N$ the law of hermitian $N\times N$ brownian motion $W_s^N\in (M_N(\C))^m$, then, for any 
$f\in \mathcal{E}_{reg,p}(   \mathcal{T}_{2,0}^c(\mathcal{F}^m_{[0,1]}*\mathcal{F}^\nu_{\mu}),d_{2,0}),$ for $p\in[ 2,\infty]$  or $f\in  \mathcal{E}^{1,1}_{app}(   \mathcal{T}_{2,0}^c(\mathcal{F}^m_{[0,1]}*\mathcal{F}^\nu_{\mu}),d_{2,0})$ the following limit exists and is given by our formula above :
$$\lim_{N\to \infty} -\frac{1}{N^2}\log E_{\gamma_{sa,N,m}}(e^{-N^2 f(\tau_{W,\Upsilon_N})})=\Lambda_{b,\upsilon}(f).$$ 
\end{theorem}    
    
     The reader should note that what we will call lower bound (as in \cite{BD}) corresponds to the usual large deviation upper bound.
\begin{proof}
\begin{step}Reduction of the case $p=\infty$ to the limit for functionals in $f\in \mathcal{E}_{reg,p}(   \mathcal{T}_{2,0}^c(\mathcal{F}^m_{[0,1]}*\mathcal{F}^\nu_{\mu}),d_{2,0}),$ for $p\in[ 2,\infty[$\end{step}

First given a functional in $f\in \mathcal{E}_{reg,\infty}(   \mathcal{T}_{2,0}^c(\mathcal{F}^m_{[0,1]}*\mathcal{F}^\nu_{\mu}),d_{2,0}),$ one can find $g_i$ as in the definition and consider $f_{(p)}\in \mathcal{E}_{reg,p}(   \mathcal{T}_{2,0}^c(\mathcal{F}^m_{[0,1]}*\mathcal{F}^\nu_{\mu}),d_{2,0}),$ $$f_{(p)}
(\overline{ \tau})= g_{(p)}(\overline{ \tau}\circ (I_{t_1,...,t_k}*Id))\ \ and \ g_{(p)}(\tau)=D+\left(\sum_{i=1,...,l}\left(g_i(\tau)\right)^p\right)^{1/p}.$$

(the convexity condition on $g_{(p)}$ is easily implied by those on $g_i$.) From the two next steps, we can assume the Laplace principle is satisfied for $f_{(p)}.$

Then by standard estimates between norms in finite dimension $$g(\tau)\leq g_{(p)}(\tau)\leq l^{1/p} g(\tau).$$

As a consequence, \begin{align*}\limsup_{N\to \infty} -\frac{1}{N^2}\log E_{\gamma_{sa,N,m}}(e^{-N^2 f(\tau_{W,\Upsilon_N})})&\leq  \limsup_{N\to \infty} -\frac{1}{N^2}\log E_{\gamma_{sa,N,m}}(e^{-N^2 f_{(p)}(\tau_{W,\Upsilon_N})})\\&=\Lambda_{b,\upsilon}(f_{(p)})\end{align*} 
and taking an infimum over $p$ one easily get the upper bound $\Lambda_{b,\upsilon}(f)$ since the set over which one takes the infimum does not depend on the argument of $\Lambda,$  only on $\mathbf{t}(f)$ and $f_{(p)}\to f.$
Conversely, we obtain using the other inequality:
\begin{align*}\liminf_{N\to \infty} \frac{1}{N^2}\log E_{\gamma_{sa,N,m}}(e^{-N^2 f(\tau_{W,\Upsilon_N})})&\geq  \liminf_{N\to \infty} -\frac{1}{N^2}\log E_{\gamma_{sa,N,m}}(e^{-N^2 f_{(p)}(\tau_{W,\Upsilon_N})/l^{1/p}})\\&=\Lambda_{b,\upsilon}(f_{(p)}/l^{1/p})\geq \Lambda_{b,\upsilon}(f/l^{1/p})\end{align*}

For a control close enough to the infimum defining $\Lambda_{b,\upsilon}(f/l^{1/p})$, one can assume $$f(\tau_{S+\int_0^{.}u_sds,\upsilon})/l^{1/p}\leq \frac{1}{2}\int_0^1||u_s||^2ds+\left(f(\tau_{S+\int_0^{.}u_sds,\upsilon})\right)/l^{1/p}\leq f(\tau_{S,\upsilon})/l^{1/p}\leq f(\tau_{S,\upsilon})=C(f)$$ the value at $u=0$, which is a constant depending only on $f$ since $f$ depends only of the law and $S$ is always a free brownian motion.

Thus one obtains the following concluding lower bound: $$\Lambda_{b,\upsilon}(f/l^{1/p})\geq \Lambda_{b,\upsilon}(f)-(l^{1/p}-1)f(\tau_{S,\upsilon})\to_{p\to \infty}\Lambda_{b,\upsilon}(f).$$

\begin{step} Lower bound for $f\in \mathcal{E}^{1,1}_{app}(\mathcal{T}_{2,0}^c(\mathcal{F}^m_{[0,1]}*\mathcal{F}^\nu_{\mu}),d_{2,0})$ including the case $ \mathcal{E}_{reg,p}(\mathcal{T}_{2,0}^c(\mathcal{F}^m_{[0,1]}*\mathcal{F}^\nu_{\mu}),d_{2,0})
, p\in [2,\infty[.$\end{step}

Consider an ultrafilter 
$\omega\in \beta\N-\N$. It suffices to show that :
\begin{equation}\label{Lowerboundomega}\lim_{N\to \omega}-\frac{1}{N^2}\log E_{\gamma_{sa,N,m}}(e^{-N^2 f(\tau_{W,\Upsilon_N})})\geq \Lambda_{b,\upsilon}(f).\end{equation}

%Fix a probability space $(\Omega,P)$ containing random processes of law $\gamma_{sa,N,m}$

We will use the tracial von Neumann algebra ultraproduct and the Hilbert space ultraproduct:
 $$\mathcal{M}_P^\omega=(M_{N}(L^\infty(\mathbb W_{sa,N},\gamma_{N})),\tau_{\gamma_N})^\omega\subset\mathcal{L}_P^\omega=[L^2_{sa}(M_{N_n}(L^\infty(\mathbb W_{sa,N},\gamma_{N})),\tau_{\gamma_N})]^\omega.$$

We can consider $S_t=(W_t^N)^\omega\in \mathcal{L}_P^\omega$ and we know from example \ref{ExFiltration} (2) that $S_t\in \mathcal{M}_P^\omega$.  Of course we have $\upsilon=(\Upsilon_N)^\omega\in (\mathcal{M}_{P,0}^\omega)^{\mu\nu}$ which has consistently the law $\mu_\Upsilon$.

Unfortunately $S_t$ is NOT a free brownian motion adapted to the canonical filtration $\mathcal{M}_{P,s}^\omega$ (which is not a factor so that the covariance map of the process is a centred valued conditional expectation and not the trace). However, $S_t$ is a free brownian motion adapted to its own filtration $M_s=W^*(\upsilon,S_t,t\leq s)\subset \mathcal{M}_{P,s}^\omega$ (this is for instance a standard freeness result between GUE and constant matrices or one can use Theorem \ref{LevyBCG} and the concentration result proposition \ref{ConcentrationNorm} for the computation of the covariance).

Note that we thus have a canonical (adapted) embedding  $I:M\subset \mathcal{M}_P^\omega$ which extends to $I:L^2_{sa}(M)\subset L^2_{sa}( \mathcal{M}_P^\omega)\subset\mathcal{L}_P^\omega.$

 Define $G$ with $f=G\circ (I_{t_1,...t_k}*Id)$ and the associated $X^{G,N,\Upsilon_N}$ from 
Theorem \ref{minimisationHermitian}. We know the finite dimensional distribution of $X^{G,N,\Upsilon_N}$ which is of the form assumed in proposition \ref{ConcentrationNorm} so that, as before for $S_s$, $Y_s=(X^{G,N,\Upsilon_N}_s)^\omega \in \mathcal{L}_P^\omega$ is actually $Y_s\in \mathcal{M}_P^\omega.$ Moreover, once fixed $Y_{t_j}$, the process is a free brownian bridge in between (as seen its limit law from the same proposition), hence we have the uniform in time operator norm bound on$ ||Y_s||$ that we wanted.

\noindent \textbf{(i) First bounds on $u_s=(b^{G,N,\Upsilon_N}(s,X^{G,N,\Upsilon_N}(s)))^\omega\in \mathcal{L}_P^\omega$.}

We know that $E(||b^{G,N,\Upsilon_N}(s,X^{G,N,\Upsilon_N}(s))||_2^2)\leq C$ independently of $N,s$ so that $||u_s||_2\leq C$.

Then from the bound in Theorem \ref{minimisationHermitian} we have  for $t,s \in ]t_i, t_{i+1}]$: $$||u_t-u_s||_2^2=\lim_{N\to\omega}E_P(||b^{G,N,\Upsilon_N}(s,X^{G,N,\Upsilon_N}(s))-b^{G,N,\Upsilon_N}(t,X^{G,N,\Upsilon_N}(t))||_2^2)\leq C_4\sqrt[4]{|t-s|},$$
so that $u$ is $1/8$-H\"older continuous on $]t_i,t_{i+1}]$ as expected and especially $u\in L^\infty_{ad}([0,1], \mathcal{L}_P^\omega)$. 

\noindent \textbf{(ii) Limit of the value function along $\omega$.}
As another consequence, $\frac{1}{2}\int_0^1||u_s||^2ds$ can be approximated by Riemann sums, and the same approximations holds for $\frac{1}{2}\int_0^1||b^{G,N,\Upsilon_N}(s,X^{G,N,\Upsilon_N}(s))||_2^2ds$ uniformly in $N$ so that :
$$\lim_{N\to \omega}\frac{1}{2}\int_0^1||b^{G,N,\Upsilon_N}(s,X^{G,N,\Upsilon_N}(s))||_2^2ds=\frac{1}{2}\int_0^1||u_s||^2ds.$$
By convexity of $G$ in the trace and Jensen's inequality, we also have :
$$E(G(\tau_{(X_{t_1}^{G,N,\Upsilon_N},...,X_{t_k}^{G,N,\Upsilon_N},\Upsilon_N)})\geq G(E\circ \tau_{(X_{t_1}^{G,N,\Upsilon_N},...,X_{t_k}^{G,N,\Upsilon_N},\Upsilon_N)}).$$
Moreover, using again the uniform H\"older continuity to approximate integrals, we have $Y_s=S_s+\int_0^su_tdt$, and the relation $(u(X_{t_k}^{G,N,\Upsilon_N})-1)(X_{t_k}^{G,N,\Upsilon_N}-4i)=8i$ imply the corresponding relation in the ultraproduct so that $u(Y_s)=(u(X_{t_k}^{G,N,\Upsilon_N}))^\omega$ and it is thus easy to see from the definition of product and trace in the ultraproduct that $$\lim_{N\to \omega} d_2(E\circ \tau_{(X_{t_1}^{G,N,\Upsilon_N},...,X_{t_k}^{G,N,\Upsilon_N},\Upsilon_N)},\tau_{(Y_{t_1},...,Y_{t_k},\upsilon)})=0.$$
Thus since $G\in C^0( \mathcal{T}_{2,0}(\mathcal{F}^m_{k}*\mathcal{F}^\nu_{\mu}),d_{2,0}),$ one deduces :
$$\lim_{N\to \omega}E(G(\tau_{X^{G,N,\Upsilon_N},\Upsilon_N}))\geq G(\tau_{(Y_{t_1},...,Y_{t_k}),\upsilon}).$$
Thus combining all our results and the formula from theorem \ref{minimisationHermitian}, one gets :
$$\lim_{N\to \omega}-\frac{1}{N^2}\log E_{\gamma_{sa,N,m}}(e^{-N^2 f(\tau_{W,\Upsilon_N})})\geq f(\tau_{Y,\upsilon})+\frac{1}{2}\int_0^1||u_s||^2ds.$$ 
To conclude with a bound below by $\Lambda_{b,\upsilon}(f)$, it only remains to check the two last conditions on $u_s$, namely $u_s\in L^2_{sa}(M)^m$ and $W^*(\upsilon,Y_{t_1},...,Y_{t_k})$ is a factor. 

\noindent \textbf{(iii) Factoriality of $W^*(\upsilon,Y_{t_1},...,Y_{t_k})$.}

For factoriality, by \cite[Th 4]{Dab08} and using the assumption $m\geq 2$ or $m=1$ and $B=W^*(\upsilon)$ diffuse, it suffices to check that $(Y_{t_1},...,Y_{t_k})$ have finite free Fisher information in the sense of \cite{V5} relative to $B=W^*(\upsilon)$.

From the explicit knowledge of the law of $(X_{t_1}^{G,N,\Upsilon_N},...,X_{t_k}^{G,N,\Upsilon_N})$ in Theorem \ref{minimisationHermitian}, one knows the classical score function of this random matrix model deduced from the density $e^{-N^2G(\tau_{x,\Upsilon_N})-N^2G_{2,\mathbf{t}(f)}(\tau_x)-C}$ with respect to Lebesgue measure on Hermitian matrices. The score functions $(-\Xi_{t_1}^{G,N},...,-\Xi_{t_k}^{G,N})$ written in matrix form are thus  (for $i\geq 1$, recall $t_0=0$, $G^N(x)=G(\tau_{x,\Upsilon_N})$) given by:
\begin{align*}\Xi_{t_i,l}^{G,N}&=\frac{N}{t_i-t_{i-1}}(X_{t_i,l}^{G,N,\Upsilon_N}-X_{t_{i-1},l}^{G,N,\Upsilon_N})+\frac{N}{t_{i+1}-t_{i}}(X_{t_i,l}^{G,N,\Upsilon_N}-X_{t_{i+1},l}^{G,N,\Upsilon_N})\\&+N^2\mathscr{D}_i^l G^N(X_{t_1}^{G,N,\Upsilon_N},...,X_{t_k}^{G,N,\Upsilon_N}).\end{align*}
Consider a non-commutative polynomial $P\in \C\langle X_1^1,...,X_1^m,X_2^1,...,X_k^m,u\rangle$ in $km$ self-adjoint indeterminates and $\mu\nu$ unitary indeterminates and let us write $\partial_{X_j^l}$ the corresponding free difference quotient (with value $0$ on $u_i^j$).

As a consequence, one can write the integration by parts formula characterizing score functions:
\begin{equation}\label{FinDimConjVar}E\left(\frac{1}{N^2}Tr(\Xi_{t_i,l}^{G,N}P(X_{t_1}^{G,N},...,X_{t_k}^{G,N},\Upsilon_N))\right)
=E\left(\frac{1}{N^2}(Tr\otimes Tr)(\partial_{X_{i}^l}P(X_{t_1}^{G,N},...,X_{t_k}^{G,N},\Upsilon_N))\right).
\end{equation}

Note that $N\mathscr{D}_i^l G^N=\frac{1}{\sqrt{N}}\mathscr{D}_i^lg_N$ and from the dimension independence of equation \eqref{lipE} obtained in lemma \ref{UniformN}, one gets for some constants $C,D>0$ independent of $N$:
$$E\left(\left\|\frac{\Xi_{t_i,l}^{G,N}}{N}\right\|_2^2\right)=E\left(\frac{1}{N^3}Tr((\Xi_{t_i,l}^{G,N})^*\Xi_{t_i,l}^{G,N})\right)\leq CE\left(\sum_{i=1}^k||X_{t_i}^{G,N}||_2^2\right)+D.$$
Thus $\xi_{t_i}^l=\left(\frac{\Xi_{t_i,l}^{G,N}}{N}\right)^\omega\in \mathcal{L}_P^\omega$ is well-defined.

Using the second concentration result from proposition \ref{ConcentrationNorm}, one can take the limit $N\to \omega$ in the right hand side of \eqref{FinDimConjVar} and get the equation in $\mathcal{L}_P^\omega$ (using also $\xi_{t_i}^l=(\xi_{t_i}^l)^*$):
\begin{equation}\label{ConjVar}\langle \xi_{t_i}^l, P(Y_{t_1},...,Y_{t_k},\upsilon)\rangle
=(\tau\otimes \tau)\left((\partial_{X_{i}^l}P)(Y_{t_1},...,Y_{t_k},\upsilon)\right),
\end{equation}
and this gives that $E_{L^2(W^*(\upsilon,Y_{t_1},...,Y_{t_k}))}(\xi_{t_i}^l)$ are exactly the conjugate variables in $L^2(W^*(\upsilon,Y_{t_1},...,Y_{t_k}))$ we were looking for.

\noindent \textbf{(iv) Conclusion}

We know that :
$$Y_s=S_s+\int_0^su_tdt.$$
This is moreover the process of corollary \ref{convexFBSDEmatricial} with  $t=0$. One thus obtains that $Y_s\in M_s$ and $u_t\in L^2(W^*(\upsilon, Y_t,Y_{t_j}, t_j\leq t)).$
Especially $u_t\in (L^2_{sa}(M_t,\tau))^{m}$ %(for almost every time, this uses the right continuity of the filtration of free brownian motion) 
and thus $u\in L^\infty_{ad}([0,1],(L^2_{sa}(M,\tau))^{m})$. Thus, we can consider $Y$ as one of the processes entering in the infimum for $\Lambda_{b,\upsilon}(f)$ and we deduce :
  $$f(\tau_{Y})+\frac{1}{2}\int_0^1||u_s||^2ds\geq \Lambda_{b,\upsilon}(f),$$
and, from (ii), \eqref{Lowerboundomega} is satisfied.  
  
\begin{step} Upper bound for $f\in \mathcal{E}^{1,1}_{app}(\mathcal{T}_{2,0}^c(\mathcal{F}^m_{[0,1]}*\mathcal{F}^\nu_{\mu}),d_{2,0}), p\in [2,\infty[.$ %[Reprendre Detail dans mail à Jonathan]
\end{step}

Consider an ultrafilter 
$\omega\in \beta\N-\N$. It suffices to show that :
$$\lim_{N\to \omega}-\frac{1}{N^2}\log E_{\gamma_{sa,N,m}}(e^{-N^2 f(\tau_{W,\Upsilon_N})})\leq \Lambda_{b,\upsilon}(f).$$ 
Thus take $u\in \mathcal{P}_{ad,1/8loc,\mathbf{t}(f)%(t_0=0<t_1=1)
,pl}\cap \mathcal{P}_{ad,factor,\mathbf{t}(f)}$ so that, from lemma \ref{EquivLambda}, it suffices to show: 
\begin{equation}\label{Upperv7}\lim_{N\to \omega}-\frac{1}{N^2}\log E_{\gamma_{sa,N,m}}(e^{-N^2 f(\tau_{W,\Upsilon_N})})\leq \frac{1}{2}\int_0^1||u_s||^2ds+f(S+\int_0^{.}u_sds).\end{equation}%$$\lim_{N\to \omega}-\frac{1}{N^2}\log E_{\gamma_{sa,N,m}}(e^{-N^2 f(\tau_W)})\leq \frac{1}{2}\int_0^1||v_s||^2ds+f(S+\int_0^{.}v_sds).$$

We can see $u_s\in L^2_{sa}(M)^m\subset (\mathcal{L}_P^\omega)^{m}$  and we have a set of times $$\mathbf{s}(u)=\{s_1^1<...<s_n^1<... s_\omega^1=t_1<s_1^2<...<s_\omega^2=t_2<... s_\omega^k=t_k<s^{k+1}_1<...s_l^{k+1}\}$$%\supset \mathbf{t}(f)$$ 
enabling to realize the piecewise linear property.

We want to define a continuous adapted process $V_s^N \in M_{N}(L^2(\mathbb{W}_{sa, N},\gamma_N))$. 

Take $U_{s_{k}^l}^N\in M_N(L^2(W_t^N, t\leq s_{k-1}^l))$ such that  $$(U_{s_{k}^l}^N)^\omega=\int_{0}^{s_{k}^l}u_sds.$$
One can assume $||U_{s_{k}^l}^N||_2\leq ||\int_{0}^{s_{k}^l}u_sds||_2$. In a standard way, we will realize operator norm inequalities simultaneously with $L^2$ inequalities in taking $f_C$ lipschitz bounded functional calculus equal to identity on a huge ball and replacing if necessary $U_{s_{k}^l}^N$ by $f_C(U_{s_{k}^l}^N)$ and choosing $C$ to keep the same value in the ultraproduct.

For each $l$ for $k$ large enough, one can take $U_{s_{k}^l}^N\in M_N(L^\infty(W_t^N, t\leq s_{k-1}^l))^m$ since then $\int_{0}^{s_{k}^l}u_sds\in (\mathcal{M}_P^\omega)^{m}$ and then assume instead the operator norm bound $||U_{s_{k}^l}^N||\leq ||\int_{0}^{s_{k}^l}v_sds||+1\leq C+1$ where $C$ is given by the definition of $\mathcal{P}_{ad,1/8loc,\mathbf{t}(f)%(t_0=0<t_1=1)
,pl}$.

Guided by the relation \eqref{PLInt}, we define $V_{s_{0}^l}^N=V_{s_{1}^l}^N=\frac{1}{s_{1}^l-s_{0}^l}(U_{s_{1}^l}^N-U_{s_{0}^l}^N)$ and then:
$$V_{s_{i+1}^l}^N=\frac{2}{s_{i+1}^l-s_{i}^l}(U_{s_{i+1}^l}^N-U_{s_{i}^l}^N)-V_{s_{i}^l}^N.$$
Finally we define the linear interpolation $V_{s}^N.$
Note that taking the ultraproduct of defining relations we have 
$$(V_{s_{0}^l}^N)^\omega=(V_{s_{1}^l}^N)^\omega=\frac{1}{s_{1}^l-s_{0}^l}\left((U_{s_{1}^l}^N)^\omega-(U_{s_{0}^l}^N)^\omega\right)=\frac{1}{s_{1}^l-s_{0}^l}\int_{s_{0}^l}^{s_{1}^l}u_sds=u_{s_{0}^l}$$
and similarly $$(V_{s_{i+1}^l}^N)^\omega=\frac{2}{s_{i+1}^l-s_{i}^l}(\int_{s_{i}^l}^{s_{i+1}^l}v_sds)-(V_{s_{i}^l}^N)^\omega=u_{s_{i+1}^l}.$$
The last equality is by induction on $i$ using \eqref{PLInt}. Finally by linear interpolation one finds $(V_s^N)^\omega =u_s$ and $V_s^N$ are adapted processes and by construction $\int_{0}^{s_{k}^l}V_s^Nds=U_{s_{k}^l}^N$ so that one deduces the operator norm bound $||\int_{0}^{t_k}V_s^Nds||\leq C+1.$ (almost surely).

Moreover, we have a uniform continuity (lipschitzness) in time with value $L^2$, uniformly in $N$ outside of a small interval where we can use a uniform bound in $L^2$ uniform in $N$ and thus as in step 2, we have convergence of Riemann integrals (which are indefinite in each $t_k$):
\[\lim_{N\to \omega} E(\frac{1}{2}\int_0^1||V_s^N||_2^2 ds)=\frac{1}{2}\int_0^1||v_s||^2ds, \ \forall t\in [0,1], \int_0^{t}v_sds=(\int_0^{t}V_s^Nds)^\omega.\]

As at the beginning of step 2, one can find $A_N$ so that $(W_t^N1_{A_N})^\omega=S_t$ and  $W_{t_i}^N1_{A_N}$ is operator norm bounded uniformly in $N$.
Let us call $Y_t^N=\int_{0}^{t}V_s^Nds+W_t^N1_{A_N}$ and 
$Z_t^N=\int_{0}^{t}V_s^Nds+W_t^N.$ Let us call $R$ the operator norm uniform bound of $Y_{t_i}^N$.
Since $(Y_t^N)^\omega=Y_t =\int_{0}^{t}v_sds+S_t$ and since $ W^*(\upsilon,Y_{t_1},...,Y_{t_{k}})$ is a factor, it gives rise to an extremal state in $\mathcal{S}_R^{mk}$ and thus one can apply proposition \ref{ConcentrationPoulsen}. Indeed, the ultraproduct relation above $(Y_t^N)^\omega=Y_t$ (and the corresponding relation for their unitary transforms) implies by definition $$\lim_{N\to \omega} d(E\circ \tau_{(Y_{t_1}^N,...,Y_{t_{k}^N},\Upsilon_N)}, \tau_{(Y_{t_1},...,Y_{t_{k}},\upsilon)})=0.$$
Thus, using also \eqref{LipReg}, one gets:
$$\lim_{N\to \omega}\left|E(G(\tau_{(Y_{t_1}^N,...,Y_{t_{k}^N},\Upsilon_N)})-G(\tau_{(Y_{t_1},...,Y_{t_{k},\upsilon})})\right|\leq \lim_{N\to \omega}C(G)E(d_2(\tau_{(Y_{t_1}^N,...,Y_{t_{k}^N},\Upsilon_N)},\tau_{(Y_{t_1},...,Y_{t_{k},\upsilon})})=0.$$
Moreover, from the uniform bound in lemma \ref{UniformN} for \eqref{lipE}, one deduces as in step 2 (with normalised euclidean norms):
\begin{align*}&\left|E(G(\tau_{(Y_{t_1}^N,...,Y_{t_{k}^N},\Upsilon_N)})-G(\tau_{(Z_{t_1}^N,...,Z_{t_{k}^N},\Upsilon_N)})\right|\\&\qquad\leq  \sqrt{E(||Y^N-Z^N||^2)}\times 
\sqrt{E\left(3C^3||Y^N||_2^2+3C^2||Z^N||_2^2 +3D^2\right)},\end{align*}
and thus since $$E(||Y^N-Z^N||^2)=E(||W_t^N||_2^2(1_{A_N}-1))\to_{N\to \infty} 0$$ one gets a limit $0$ for our previous expression when $N\to \omega$.

Finally, we obtained :
$$\lim_{N\to \omega}E\left(G(\tau_{(Z_{t_1}^N,...,Z_{t_{k}^N},\Upsilon_N)})+\frac{1}{2}\int_0^1||V_s^N||_2^2 ds\right)
=G(\tau_{(Y_{t_1},...,Y_{t_{k}},\upsilon)})+\frac{1}{2}\int_0^1||u_s||_2^2 ds.$$

But the infimum characterization of theorem \ref{ustunel} includes $V_s^N$ as adapted process and thus 
$$E\left(G(\tau_{(Z_{t_1}^N,...,Z_{t_{k}^N},\Upsilon_N)})+\frac{1}{2}\int_0^1||V_s^N||_2^2 ds\right)\geq -\frac{1}{N^2}\log E_{\gamma_{sa,N,m}}(e^{-N^2 f(\tau_{W,\Upsilon_N})}).$$
Taking the limit $N\to \omega$ concludes to \eqref{Upperv7}, and thus to 
$$\lim_{N\to \omega}-\frac{1}{N^2}\log E_{\gamma_{sa,N,m}}(e^{-N^2 f(\tau_{W,\Upsilon_N})})= \Lambda_{b,\upsilon}(f).$$
Since the limit does not depend on the ultrafilter $\omega$, the limit exists as stated.
\end{proof}     
     
We will also need a more technical consequence of the proof in order to compute Voiculescu's microstates free entropy later in some cases.
     
     \begin{corollary}\label{OptimalFree}
     Let $f\in \mathcal{E}_{reg,p}(\mathcal{T}_{2,0}^c(\mathcal{F}^m_{[0,1]}*\mathcal{F}^\nu_{\mu}),d_{2,0})
\cup \mathcal{E}^{1,1}_{app}(\mathcal{T}_{2,0}^c(\mathcal{F}^m_{[0,1]}*\mathcal{F}^\nu_{\mu}),d_{2,0})
, p\in [2,\infty[,$ $f=G\circ (I_{t_1,...t_k}*Id)$
 and consider $X^{G,N,\Upsilon_N}$ the solution in Theorem \ref{minimisationHermitian}. Then, for any ultrafilter $\omega\in \beta\N-\N, \upsilon=(\Upsilon_N)^\omega, S_t=(W_t^N)^\omega$, $Y_t=(X_t^{G,N,\Upsilon_N})^\omega\in W^*(\upsilon,S_s, s\leq t)$ satisfies an SDE with respect to the canonical brownian motion in $S_s\in\mathcal{M}_P^\omega$. There is $u_s=u_s^G:=(b^{G,N,\Upsilon_N}(s,X^{G,N,\Upsilon_N}(s)))^\omega\in L^2(W^*(\upsilon,Y_{t_1}, ...,Y_{t_i},Y_s))$ for $t_i<s\leq t_{i+1}$ ($t_0=0$, $u_s=0$ if $s>t_k$) such that $$Y_t=S_t+\int_0^tu_s^Gds.$$
 
Finally, the infimum in the definition of $\Lambda_b(f)$ is reached at $u_s$, i.e.:
$$\Lambda_{b,\upsilon}(f)=f(\tau_{Y,\upsilon})+\frac{1}{2}\int_0^1||u_s^G||_2^2ds.$$
\end{corollary}
\begin{proof}
All the properties of $Y_s$ were obtained in step 2 of the proof of Theorem \ref{MainTechnical}. For instance, the $L^2$ space containing $u_s$ was obtained in (x) based on (ix) (since the drift of $Z_s=Y_s$ and $u_s$ coincide). The inequality $$\lim_{N\to \omega}-\frac{1}{N^2}\log E_{\gamma_{sa,N,m}}(e^{-N^2 f(\tau_{W,\Upsilon_N})})\geq f(\tau_{Y,\upsilon})+\frac{1}{2}\int_0^1||u_s^G||^2ds$$ was obtained in (ii),
but since from the definition and the upper bound of step 3:
$$f(\tau_{Y,\upsilon})+\frac{1}{2}\int_0^1||u_s^G||^2ds\geq \Lambda_{b,\upsilon}(f)\geq \lim_{N\to \omega}-\frac{1}{N^2}\log E_{\gamma_{sa,N,m}}(e^{-N^2 f(\tau_{W,\Upsilon_N})}).$$
 One deduces the stated equality.
\end{proof}

  \section{Laplace Principal Lower and upper bounds}
   \subsection{Ultrafilter limit for non-convex functions 
   }
 
We now define the candidate for ultrafilter limits in the Laplace principle.
Recall we defined : $\mathcal{M}_P^\omega=(M_{N}(L^\infty(\mathbb{W}_{sa, N},\gamma_N)),\tau_{\gamma_N})^\omega$ and  $\mathcal{L}_P^\omega=[L^2_{sa}(M_{N}(L^\infty(\mathbb{W}_{sa, N},\gamma_N)),\tau_{\gamma_N})]^\omega$.

Let $\mathcal{G}([0,1],L^2_{sa}(N)^m)$ the space of left continuous functions with right limits.

We consider a set of adapted path in $\mathcal{G}([0,1],L^2_{sa}(N)^m)$ corresponding to a filtration $\mathcal{F}$ on $\mathcal{M}_P^\omega$:%with a piecewise martingale property :

\begin{align*}&\mathcal{P}_{ad,g}^{\omega,t}(\mathcal{F})=\{u\in L^\infty_{ad}([t,1], L^2(N)_{sa}^m): M\subset N\subset \mathcal{F}%\mathcal{M}_P^\omega
\ separable, u\in \mathcal{G}([t,1],L^2_{sa}(N)^m)\\& \qquad\mathrm{with\ at\ most\ countably\ many \ discontinuity\ points}
\}\end{align*}
     Similarly $\mathcal{P}_{ad,d}^{\omega,t}(\mathcal{F})$ is the variant with right continuous controls. 
Here and later, the copy of $M$ is understood to come from the canonical embedding and construction of $S_t\in    \mathcal{M}_P^\omega$  .
     
We consider another family of paths such that one requires more von Neumann algebraic regularity at this sequence of times:  
   \begin{align*}&\mathcal{P}_{ad,b}^{\omega,t}(\mathcal{F})=\{u\in \Lambda^\infty_{ad}([t,1], N_{sa}^m), M\subset N\subset \mathcal{M}_P^\omega\ separable%:\\&\qquad \exists C \forall i=1,...,k, \ \forall \tau\in [t,1],\ Y_{\tau}=S_{\tau}+\int_t^{\tau}u_sds\in N^m, ||Y_{\tau,l}||\leq C\ 
   \}.
     \end{align*}Here $\Lambda^\infty$ is the set of weak-ù measurable processes, weak-* essentially bounded.
  
Recall also that from the argument in \cite[Corol 4.3]{FarahI}, we know the center $$\mathcal{Z}_P^\omega:= \mathcal{Z}(\mathcal{M}_P^\omega)=(\mathcal{Z}(M_{N}(L^\infty(\mathbb{W}_{sa, N},\gamma_N)),\tau_{\gamma_N}))^\omega\simeq (L^\infty(\mathbb{W}_{sa, N},\gamma_N))^\omega.$$ Moreover since $\gamma_N$ is diffuse, and $(\mathbb{W}_{sa, N},\gamma_N)$ is a standard probability space (as any Polish space with a completed Borel $\sigma$-field, e.g. \cite[Thm 2-3]{delaRue}), thus $L^\infty(\mathbb{W}_{sa, N},\gamma_N)\simeq L^\infty([0,1[,Leb)\simeq L^\infty(\{0,1\}^{\N})$ \cite[Thm 4-3]{delaRue}. As a consequence from the proof \cite[Prop 4.6]{FarahI}, we know that $\mathcal{Z}_P^\omega$ is the $L^\infty$ space of a Maharam-homogeneous algebra of Maharam character ${\aleph_0}^{\aleph_0}$, thus $\mathcal{Z}_P^\omega\simeq L^\infty(\{0,1\}^{2^{\N}}).$ Let us call $\Omega=\{0,1\}^{2^{\N}}$ with its standard product measure of symmetric Bernoulli laws and fix an isomorphism above. We will soon use the lifting theory first developed by Dorothy Maharam in the form given by \cite{Ionescu-Tulcea}.

% and fix an isomorphism above.
We also fix the measure $\mu^\omega$ corresponding to the faithful normal state $E^\omega=(E_{\gamma_N})^\omega$.

 Of course, for any $z\in \Omega$, any process $X_t=S_t+\int_0^tu_sds, u\in L^2([0,1],L^2(\mathcal{M}_P^\omega)^m)$, for $P\in\mathcal{F}^m_{[0,1]}*\mathcal{F}^\nu_{\mu}$ we define $P(X,\upsilon)=P(u(X),\upsilon)$, one can define the random variable $\tau_{X,\upsilon:\mathcal{Z}}:z\mapsto\tau_{X,\upsilon:z}(P)=[E_{\mathcal{Z}(\mathcal{M}_P^\omega)}(P(X,\upsilon))](z)$ obtained by evaluating at $z$ the central conditional expectation (after taking a lifting of the random variable).
We will use this based on the following lemma. Note that even though at this stage it may seem that we will need only $u$ essentially bounded, we will need the square integrable case since only the use of those will give a semicontinuous rate function.

\begin{lemma}\label{agreementUltraprod} For  $X_t=S_t+\int_0^tu_sds, u\in L^2([0,1],L^2(\mathcal{M}_P^\omega)^m))$%for $N\subset\mathcal{M}_P^\omega$ separable and $u$ essentially bounded with value $N^m$ 
,$\upsilon\in \mathcal{U}((\mathcal{M}_P^\omega))^{\mu\nu}$  as above, there is $\tau_{X,\upsilon:\mathcal{Z}}$ Borel measurable with value $(\mathcal{T}_{2,0}^c(\mathcal{F}^m_{[0,1]}*\mathcal{F}^\nu_{\mu}),d_{1,0})$ with almost everywhere: $$\tau_{X,\upsilon:\mathcal{Z}}(P)=[E_{\mathcal{Z}(\mathcal{M}_P^\omega)}(P(X,\upsilon))].$$

Moreover, there is $L$ square-integrable measurable on $(\Omega,E^\omega)$ such that % for almost every $\omega$, 
$\tau_{X,\upsilon:\mathcal{Z}}(\omega)\in \mathcal{T}_{L(\omega)}\ a.e$.

Finally, for any $f\in \mathcal{C}_{reg,p,C}( \mathcal{T}_{2,0}^c(\mathcal{F}^m_{[0,1]}*\mathcal{F}^\nu_{\mu}),d_{2,0}))$, we have $f(\tau_{X,\upsilon:\mathcal{Z}})\in L^{1/1_{C\neq 0}}(\mathcal{Z}_P^\omega)$ and $$\lim_{N\to \omega} E_{\gamma_N}(f(\tau_{X^N,\upsilon^N}))=E^\omega(f(\tau_{X,\upsilon:\mathcal{Z}})).$$
\end{lemma}
  \begin{proof}
  We start from $T:P\mapsto [E_{\mathcal{Z}(\mathcal{M}_P^\omega)}(P(X,\upsilon))]$ which defines a continuous operator $T\in L(C,L^\infty(\Omega,E^\omega))$.
 Let $C=\mathcal{F}^m_{[0,1]}*\mathcal{F}^\nu_{\mu}$  and recall that $L(C,L^\infty(\Omega,E^\omega))\simeq (C\widehat{\o}_\pi L^1(\Omega,E^\omega))^*$ (see e.g. \cite[p 24,29]{Ryan}), namely the dual of the space of Bochner-integrable maps. Then from \cite[VII.4 Corol. p 95]{Ionescu-Tulcea}, $(L^1(\Omega,E^\omega;C))^*$ equals the set $\Lambda^\infty(\Omega,E^\omega;C^*)$ of (equivalence classes of) weak-* measurable  functions which are weak-* essentially bounded. But since $C$ is separable (recall we took variables only at rational times), this means that the norm itself is essentially bounded.  We write $I:L(C,L^\infty(\Omega,E^\omega))\to \Lambda^\infty(\Omega,E^\omega;C^*)$ the above isomorphism.

 Using \cite[Thm 3 p 46,Prop 1 p 77]{Ionescu-Tulcea} there are liftings:
 \begin{equation}  \sigma:L^\infty(\mathcal{Z}_P^\omega)\to M^\infty(\Omega,E^\omega), \rho:\Lambda^\infty(\Omega,E^\omega;C^*)\to M^\infty(\Omega,E^\omega;C^*)
 \end{equation}
 to the space of essentially bounded measurable maps with supremum norm. They satisfy that $\sigma$ is linear unital order preserving homomorphism, section to equivalence class and $\rho$ is a linear norm preserving section with intertwining relation $\langle \rho(f),c\rangle=\sigma(\langle f,c\rangle),c\in C$ hence $||\rho (x)||\leq \sigma (||x||)$. 
  $\rho(I(T))$  as $\tau_{X,\upsilon:\mathcal{Z}}$ satisfies the requirements with value $\mathcal{T}(\mathcal{F}^m_{[0,1]}*\mathcal{F}^\nu_{\mu})$. {We must check the smaller space of value $\mathcal{T}_{2,0}^c$.} To prove that, we will actually change slightly our choice. Note that for $\alpha\in \Q_+$, $R_{\alpha,t}^i=\alpha((X_t^i)^*X_t^i+\alpha)^{-1}, X_{\alpha,t}^i=\sqrt{\alpha}X_t^i((X_t^i)^*X_t^i+\alpha)^{-1}$ are bounded by $1$ in operator norm. % Fix $K=\sup||X_t^i||_2<\infty$. 
  Recall that $\tau_{X,\upsilon}=\tau_{u(X),\upsilon}$ and we can also consider, following the notation of section 2.8, $\sigma_{\alpha(X^*X+\alpha)^{-1},\sqrt{\alpha}X(X^*X+\alpha)^{-1},u(X),\upsilon}\in \mathcal{T}(C'), C'=C^0([-1,1])^{*2m\times([0,1]\cap \Q)\times\Q_+}*\mathcal{F}^m_{[0,1]}*\mathcal{F}^\nu_{\mu}$ for some tracial state space of a universal $C^*$-algebra containing also self-adjoint variables for contractions $R_{\alpha,t}^i,X_{\alpha,t}^i$. Hence one gets: $T':P\mapsto [E_{\mathcal{Z}(\mathcal{M}_P^\omega)}(P((R_\alpha),(X_\alpha),u(X),\upsilon))]$ and $T'\in L(C',L^\infty(\Omega,E^\omega))$. One gets similar maps 
  $I':L(C',L^\infty(\Omega,E^\omega))\to \Lambda^\infty(\Omega,E^\omega;(C')^*),$ $\rho'$ a lifting for the last space $J:C\to C'$ gives the weak- continuous $J^*$ and one can take:$\sigma_{(R_\alpha),(X_\alpha),u(X),\upsilon:\mathcal{Z}}=\rho'(I'(T'))$ and change our choice in profit of $\tau_{X,\upsilon:\mathcal{Z}}=J^*\circ(\rho'(I'(T')))1_A+
  1_{A^c}\sigma_{\alpha(1+\alpha)^{-1},\sqrt{\alpha}(1+\alpha)^{-1},u(1),\upsilon}$ for a full measure set $A$ we will now specify in order to impose everywhere the expected relations we want on $R_\alpha,X_\alpha,u(X)$ and will depend on the variable $L$ of the statement we will fix afterwards  and with the centre valued norm $||(1-R_{\alpha,t})\alpha ||_{1,\mathcal{Z}}\leq||X_t||_{2,\mathcal{Z}}^2\leq L^2$ on $A$. $(u(X^j_t)-1)/8i$ is supposed to be the inverse of $X^j_t-4i$. Note that $||\sqrt{\alpha}X_{\alpha,t}-X_t||_1=|| (X_t^*X_t)|X_t|(\alpha+X_t^*X_t)^{-1}||_{1,\mathcal{Z}}\leq L^2/\sqrt{\alpha}.$ For the countably many $Q=P(X_\alpha,R_\alpha,u(X),\upsilon)$ NC-monomials in $X_\alpha,R_\alpha,u(X),\upsilon$, one considers the set $A\subset\{L<\infty,\forall t\in \Q\cap[0,1],n\in\N:\sigma(\min(n,||X_t||_{2,\mathcal{Z}}^2))\leq L^2\}$ where for all such $P$ \begin{align*}&|\rho'(I'(T'))(Q(u(X^i_t)-1)(X^i_{\alpha,t}-4i))-8i\rho'(I'(T'))(Q)|\leq 2||P||_\infty L^2/\sqrt{\alpha}, 
  \\&|\rho'(I'(T'))(Q(X^i_{\alpha,t}-4i)(u(X^i_t)-1))-8i\rho'(I'(T'))(Q)|\leq 2||P||_\infty L^2/\sqrt{\alpha}, 
  \\&|\rho'(I'(T'))(Q((X^i_{\alpha,t})^*X^i_{\alpha,t}-R^i_{\alpha,t}(1-R^i_{\alpha,t})))|=0
  \\&|\rho'(I'(T'))(Q(R^i_{\alpha,t}\alpha^{-1}-R^i_{\beta,t}\beta^{-1}
  -(\beta-\alpha)R^i_{\alpha,t}\alpha^{-1}R^i_{\beta,t}\beta^{-1}))|=0,\\&|\rho'(I'(T'))(Q(X^i_{\alpha,t}\alpha^{-1/2}-X^i_{\beta,t}\beta^{-1/2}-(\beta-\alpha)X^i_{\alpha,t}
  \alpha^{-1/2}R^i_{\beta,t}\beta^{-1}))|=0.\end{align*}
  $A$  is a full measure set since the relations with $T'$ instead of $\rho'(I'(T'))$ are satisfied almost everywhere and from the choice of $L$. 
  Hence the equality expected from $\tau_{X,\upsilon:\mathcal{Z}}$ is not altered from adding $1_A$. But restricted to $A$, 
    $\sigma_{(R_\alpha),(X_\alpha),u(X),\upsilon:\mathcal{Z}}$ is valued in states such that in their GNS representation $R^i_{\alpha,t}\alpha^{-1}$ is a strongly continuous contraction resolvent on $L^2$ (in the sense of \cite[Def I.1.4]{MaR}), with generator say $(X^i_t)^*X^i_t$ affiliated and in $L^1$ (from $||(1-R_{\alpha,t})\alpha ||_{1,\mathcal{Z}}\leq L^2$) with the polar decomposition of $X^i_{\alpha,t}$ determined appropriately (especially the last relation imposes a unique partial isometry in all those polar decompositions) giving rise to a limit $X^i_t=\lim_{\alpha\to\infty}\sqrt{\alpha}X_{\alpha,t}^i\in L^2$  with the inverse relation $(u(X^i_t)-1)=8i(X^i_t-4i)^{-1}$ in their GNS representation so that the belonging to $\mathcal{T}_{2,0}^c$ can be tested in the GNS representation for $\sigma_{(R_\alpha),(X_\alpha),u(X),\upsilon:\mathcal{Z}}$ which is the same as the one for its restriction $\tau_{X,\upsilon:\mathcal{Z}}$ and it suffices to prove the second statement about $L$. This requires a little care since %by a result of von Neumann 
    there is no lifting of $L^p(\Omega),p<\infty$ \cite[Th IV.4.6]{Ionescu-Tulcea}.
    
     %Let $X_t=S_t+\int_0^{t}u_sds$ 
     For the center-valued trace $E_{Z}$ one can estimate the center valued norm $||x||_{2,Z}^2= E_{Z}(x^2)$ and get (recall $||S_t-S_s||_{2,Z}=\sqrt{m(t-s)}$ by the proof a.s. convergence of GUE giving equality in ultraproduct of centers and $E_{Z}(ab)= E_{Z}(ba)$) $\sqrt{\alpha}||X_{\alpha,t}-X_{\alpha,s}||_{2,Z}\leq 3||X_t-X_s||_{2,Z}$ and \begin{equation}\label{TauL}||X_t-X_s||_{2,Z}\leq ||S_t-S_s||_{2,Z}+ ||\int_s^{t}u_vdv||_{2,Z}\leq\sqrt{m(t-s)}+ \sqrt{t-s}\sqrt{\int_0^{1}||u_v||_{2,Z}^2 dv}.\end{equation}

Indeed $u$ is assumed Bochner measurable so that $u_.^2$ is Bochner-measurable with value in $L^1$, $||u_.||_{2,Z}^2$ is Borel-measurable with value in the $L^1$-space of the center (since conditional expectation is continuous hence Borel measurable). Since $[0,1]$ is a Radon measure space, one deduces from Fremlin's theorem (\cite[Thm 2B]{Fremlin}, see also the review \cite[Thm 4.1]{KupkaPrikry}, or \cite[418G,451S]{FremlinBook})  that $||u_.||_{2,Z}^2$ is also $L^1$ valued Bochner measurable.  Thus for $a\geq 0$ in the center by a standard Cauchy-Schwartz inequality (following the use of triangular inequality for  $\tau(a||\cdot||_{2,Z}^2)^{1/2}$): $$\tau(a||\int_s^{t}u_vdv||_{2,Z}^2)%= \tau(a (\int_s^{t}u_vdv)^2)
\leq \Big(\int_s^{t}dv\tau(a||u_v||_{2,Z}^2)^{1/2}\Big)^2\leq (t-s)\int_s^{t}\tau(a||u_v||_{2,Z}^2)dv=(t-s)\tau\Big(a\int_s^{t}||u_v||_{2,Z}^2dv\Big)$$
 where the last equality is identification of Bochner and Pettis integral in the Bochner measurable case.
 Note also that $\int_s^{t}||u_v||_{2,Z}^2dv$ is in the $L^1$ space of the center with norm $\int_s^{t}||u_v||_{2}^2dv.$ %But we know better since $||||u_v||_{2,Z}^2||_\infty\leq ||u_v||_\infty\leq ||u||_\infty$ hence $\int_0^{1}||u_v||_{2,Z}^2dv$ is essentially bounded and one can define $L=\sigma(\int_0^{1}||u_v||_{2,Z}^2dv)$.
 We can now define $$L=\sup_{n\in\N}\sigma\left(\sqrt{m}+\min(n,\sqrt{\int_0^{1}||u_v||_{2,Z}^2 dv})\right).$$
We claim this gives (using also $X_0=0$ and the properties of liftings) that  $\tau_{X,\upsilon:\mathcal{Z}}(\omega)\in \mathcal{T}_{3L(\omega)}
 $  for almost every $\omega\in A$. % with the above random $L$ which is in $M^\infty$ of the center. 
 Indeed, $L_n=\sigma\left(\min(n,\sqrt{\int_0^{1}||u_v||_{2,Z}^2 dv})\right)$ is an increasing sequence of random variable with $L_m=\min(m,L_n)$ for $m\leq n$ since $\sigma$ is a lifting and thus preserves $\min$. Hence for $\omega\in A$, since $\sup_nL_n(\omega)\leq N<\infty$ this gives $N=N(\omega)$ with  $L_n(\omega)=L_N(\omega), n\geq N$.
    Hence on $A$, $\sqrt{\int_0^{1}||u_v||_{2,Z}^2 dv}=L-\sqrt{m}$ almost everywhere. 
Moreover, since $||X_t-X_s||_{2,Z}\in L^2(\Omega,E^\omega)$, by the lifting property, $\sigma(\min(   ||X_t-X_s||_{2,Z},(\sqrt{m}+n)\sqrt{t-s}))\leq \sqrt{t-s}L_n\leq \sqrt{t-s}L$. and thus taking the $\sup_n$ one gets a variable a.e. equal to $||X_t-X_s||_{2,Z}$ so that almost everywhere on $A$ for every $t,s\in \Q\cap[0,1]$: $\sqrt{\alpha}||X_{\alpha,t}-X_{\alpha,s}||_{2,Z}/3\leq ||X_t-X_s||_{2,Z}\leq \sqrt{t-s}L$
   which means that on the same set taking the limit $\alpha\to \infty$ in each GNS representation, one knows 
    $\tau_{X,\upsilon:\mathcal{Z}}(\omega)\in \mathcal{T}_{3L(\omega)}
 .$

  Let us turn to equality of limits. From the definition of the state, it suffices to get $(f(\tau_{X^N,\upsilon^N})^\omega)=f(\tau_{X,\upsilon:\mathcal{Z}})$. This relation extends by linearity, order preservation and algebra morphism in $f$ (as measurability does). Especially, if it true for $f\geq 1$, it is true for $|f|^{-1}$ and then for $\ln(|f|)=\int_1^{|f|}\frac{1}{t}dt$ by uniform approximation by Riemann sums and  for any $f$ similarly it is true for $\exp(f)$, and combining this as a consequence for $|f|^p, p>0.$ From the form of functions  in  $\mathcal{C}_{reg,p,C}$, it suffices to consider $f$ given by $f(\tau)=\tau ((4i\frac{u_j^l+1}{u_j^l-1})^*(4i\frac{u_j^l+1}{u_j^l-1}))$ (measurability in this case comes from limit of $X_{t,\alpha}$) or $f(\tau)=\tau((u_{j_1^i(I)}^{l_1^i(I)})^{\epsilon_1^i(I)}...(u_{j_{m_i}^i(I)}^{l_{m_i}^i(I)})^{\epsilon_{m_i}^i(I)}))$. Since we already saw canonical unitaries commutes with ultraproducts, this is then the definition of $\tau_{X,\upsilon:.}$ since $E_{\mathcal{Z}(\mathcal{M}_P^\omega)}$ is given by the ultraproduct of traces.
  
 We finally prove Borel measurability with value $(\mathcal{T}_{2,0}^c(\mathcal{F}^m_{[0,1]}*\mathcal{F}^\nu_{\mu}),d_{1,0})$ . Since this is a separable metric space it suffices to prove Baire measurability, i.e. the measurability of any composition with a continuous real valued function \cite[4A3N,4A3L]{FremlinBook}. Take such a bounded continuous $f$, then  $f(\tau_{X,\upsilon:\mathcal{Z}})=f(\tau_{X,\upsilon:\mathcal{Z}})1_{\{L<\infty\}},\ a.e.$ hence it suffices to check $f(\tau_{X,\upsilon:\mathcal{Z}})1_{\{L\leq n\}}$ measurable. But since $\tau_{X,\upsilon:\mathcal{Z}}$ is valued on a compact set $\mathcal{T}_n$, $f$ it is a limit of functions $f_m\in Min(\mathcal{C}_{reg,p,C=0})$ from the well-separating property of this set. Hence $f(\tau_{X,\upsilon:\mathcal{Z}})1_{\{L\leq n\}}=\lim_{m}f_m(\tau_{X,\upsilon:\mathcal{Z}})1_{\{L\leq n\}}$ is measurable since we checked the case of $f_m$ earlier.\end{proof}

     We are finally ready to define our candidate to be a limiting function for $f\in \mathcal{C}_{reg,p,C}=\mathcal{C}_{reg,p,C}( \mathcal{T}_{2,0}^c(\mathcal{F}^m_{[0,1]}*\mathcal{F}^\nu_{\mu}),d_{2,0}))$ for $p\in[ 1,\infty], C\in [0 ,\infty[$.  We also define it for $ f\in C^0_{(k)}(   \mathcal{T}_{2,0}^c(\mathcal{F}^m_{[0,1]}*\mathcal{F}^\nu_{\mu}),d_{2,0}),$ bounded from below.
      %For we call $\mathbf{t}(f)=(t_1<...<t_k)$ the minimal set of times such that $f=g\circ (I_{t_1,...t_k}*Id).$
         
          $$\Lambda_{b,\Upsilon}^\omega(f):=\inf\{\frac{1}{2}\int_0^1||u_s||^2ds+E^\omega(f(\tau_{S+\int_0^{.}u_sds,\upsilon:z})): u\in \mathcal{P}_{ad,g}^{\omega,0}((\mathcal{M}_{P,s}^\omega)_{s\geq 0})\cap \mathcal{P}_{ad,b}^{\omega,0}((\mathcal{M}_{P,s}^\omega)_{s\geq 0})\}.$$
   
We will also need in the second paper of this series a value function analogue of this expected limit.
Let $t\in [t_i, t_{i+1}[$,  $X_1,..., X_{i+1}\in L^2(\mathcal{M}_{P,t}^\omega)_{sa},$ we define for $t\leq s$ and a function $f$ on the ultraproduct for $s\in [t_l,t_{l+1}[$:

\begin{align*}&\Lambda_{b,\Upsilon}^{\omega,s}(f; t, X_1,..., X_{i+1}):=\inf\{\frac{1}{2}\int_t^s||u_v||^2dv\\&+f(X_1,..., X_i,X_{i+1}+S_{t_{i+1}}-S_t+\int_t^{t_{i+1}}u_vdv,...,X_{i+1}+S_{t_l}-S_t+\int_t^{t_{l}}u_vdv X_{i+1}+S_s-S_t+\int_t^{s}u_vdv )\\&\qquad \qquad \qquad  :u\in \mathcal{P}_{ad,g}^{\omega,[t,s]}((\mathcal{M}_{P,s}^\omega)_{s\geq 0})\cap \mathcal{P}_{ad,b}^{\omega,[t,s]}((\mathcal{M}_{P,s}^\omega)_{s\geq 0})\}.\end{align*}

We then write for $f\in \mathcal{C}_{reg,p,C}$:$$\Lambda_{b,\Upsilon}^{\omega}(f; t, X_1,..., X_{i+1})=\Lambda_{b,\Upsilon}^{\omega,1}(E^\omega(f(\tau_{.,\upsilon:z})); t, X_1,..., X_{i+1}).$$  We may also write for simplicity $\mathcal{V}_t^{f,\omega}(X_1,..., X_{i+1})=\Lambda_{b,\Upsilon}^\omega(f; t, X_1,..., X_{i+1}).$

We start by finding a first estimate for an alternative formula for    $\Lambda_{b,\Upsilon}^\omega(f)$ that will be more convenient for the Laplace deviation upper bound since a piecewise linear functional depends locally on finitely many values and are easier to make converge in ultraproducts.

 We give a name to the piecewise linear part :%for $t\in [t_i, t_{i+1}[$ 
 \begin{align*}&\mathcal{P}_{ad,d,pl}^{\omega,t}(\mathcal{F})=\{u\in\mathcal{P}_{ad,d}^{\omega,t}(\mathcal{F})\cap \mathcal{P}_{ad,b}^{\omega,t}(\mathcal{F}):\\& \exists\ an \ ordinal\ \nu<\omega_1,  s_\lambda, \lambda \leq \nu, s_0=t,s_\nu=1, s_l\leq s_\mu \ if\ l\leq \mu,\  \forall l\leq \nu,\ u_{s_{l+1}}\in L^2_{sa}(M_{s_{l}})^m,  \\&\qquad \forall l <\nu \ \mathrm{limit} \ u_{s_l}=u_{s_{1+1}}, \forall s\in [s_l,s_{l+1}],\quad u_s=\frac{s-s_l}{s_{l+1}-s_l} u_{s_{l+1}}+\frac{s_{l+1}-s}{s_{l+1}-s_l}u_{s_{l}^j}\}\end{align*}
   We call $\mathbf{s}(u)$ the sequence of times appearing in the definition with  minimal ordinal $\nu(u)$.

\begin{lemma}\label{EquivLambdaX} For any $ f\in \mathcal{C}_{reg,p,C}( \mathcal{T}_{2,0}^c(\mathcal{F}^m_{[0,1]}*\mathcal{F}^\nu_{\mu}),d_{2,0})),$  for any $t\in ]t_i, t_{i+1}]$,  $X=(X_1,..., X_{i+1})\in L^2(\mathcal{M}_{P,t}^\omega)_{sa}^{m(i+1)},$ we have :
$$\Lambda_{b,\Upsilon}^{\omega}(f; t, X)\geq \inf\{\frac{1}{2}\int_t^1||u_s||^2ds+E^\omega(f(\tau_{(X_1,..., X_i,X_{i+1}+S+\int_t^{.}u_sds,\upsilon:z})): u\in \mathcal{P}_{ad,d,pl}^{\omega,t}((\mathcal{M}_{P,s+}^\omega)_{s\geq 0})%\mathcal{P}_{ad,g,\mathbf{t}(f)}^{\omega,t}\cap \mathcal{P}_{ad,b}^{\omega,t}((\mathcal{M}_{P,s+}^\omega)_{s\geq 0})
\}.$$
\end{lemma}

\begin{proof}
First note that $\Lambda_{b,\upsilon}^{\omega,t}(f; t, X_1,..., X_{i+1})$
is larger than the corresponding quantity where the filtration is replaced by the associated larger right continuous filtration $((\mathcal{M}_{P,s+}^\omega)_{s\geq 0})$.

   Fix $v\in \mathcal{P}_{ad,d}^{\omega,t}((\mathcal{M}_{P,s+}^\omega)_{s\geq 0})\cap \mathcal{P}_{ad,b}^{\omega,t}((\mathcal{M}_{P,s+}^\omega)_{s\geq 0})$. Moreover, we can replace $v$ by the right continuous process with left limits which gives the same value and which by right continuity is still adapted to the same filtration.
   
Fix $\epsilon=1/n>0$ and build by ordinal induction the following sequence of times.  Take $s_0=t$, Pick $s_1$ with $||v_s-v_t||_2\leq \epsilon/2$ for $s\in [t,s_1]$ by right continuity at $t$, and then at successor step $s_{l+1}>s_l$ with $||v_s-v_{s_l}||_2\leq \epsilon/2$ for $s\in [s_l,s_{l+1}]$ and, at limit step, take $s_{\lambda}=\sup_{l<\lambda} s_l.$
Stop at $\nu$ when $s_\nu=1$

 We have uniform bound $||v||_\infty=\sup_{t\in [0,1]}||v_t||_\infty,$ and define $u^{(n)}\in \mathcal{P}_{ad,d,pl}$ as follows.
   First we take for $l\leq \nu$ successor ordinal, $\lambda$ limit ordinal or $0$, $$u^{(n)}_{s_l}=v_{s_{l-1}}, u^{(n)}_{s_\lambda}=v_{s_{\lambda}},$$ which is compatible with the measurability constraint $ u_{s_l}\in L^2_{sa}(M_{s_{l-1}})^m.$  We then interpolate by linearity as we said.
   
   Note that for $t\in  [s_\lambda,s_{\lambda+1}]$, $u^{(n)}_t=v_{s_{\lambda}}$ and therefore $   ||v_t-u^{(n)}_t||_2\leq \epsilon/2$ by assumption on the sequence of times. If $t\in  [s_l,s_{l+1}]$, $u^{(n)}_t$ is a convex combination of $v_{s_{l-1}}$  and $v_{s_{l}}$ and therefore $   ||v_t-u^{(n)}_t||_2\leq \max(||v_t-v_{s_{l-1}}||_2, ||v_t-v_{s_{l}}||_2)\leq  \epsilon$.
   
   Finally note that we still have $u\in L^\infty_{ad}$, included in a separable filtration, continuous at all but at most countably many times ($s_\lambda$, $\lambda$ limit ordinal) at which it still has left limits from this property of $v$ and is obviously right continuous.
   
   Of course, we have:  $$\Big| \frac{1}{2}\int_t^1||v_s||^2ds-\frac{1}{2}\int_t^1||u_s^{(n)}||^2ds\Big|\leq 2\epsilon ||v||_\infty(1-t)\leq 2\epsilon ||v||_\infty.$$
   A similar bound and the continuity of $f$ insures that the value at $u_s^{(n)}$ is as close as we want from the one at $v$.
   \end{proof}

\begin{theorem}\label{MainTechnicalNonConvex}
Fix  $\Upsilon_N\in \mathcal{U}(M_N(\C))^{\mu\nu}$ (deterministic).
Assume that the non-commutative law $\tau_{\Upsilon_N}$ converges to some $\mu_\Upsilon\in (\mathcal{T}(\mathcal{F}^\nu_{\mu}),d)$. %\textbf{We assume either $m\geq 2$ or $m=1$ and $W^*(\mu_\Upsilon)$ diffuse.}

Let $\gamma_{sa,N,m}=\gamma_N$ the law of hermitian $N\times N$ brownian motion $W_s^N\in (M_N(\C))^m$, then, for any 
$f\in \mathcal{C}_{reg,p,C=0}(   \mathcal{T}_{2,0}^c(\mathcal{F}^m_{[0,1]}*\mathcal{F}^\nu_{\mu}),d_{1,0}),$ for $p\in[ 2,\infty]$   then for any non-principal ultrafilter $\omega\in\beta\N-\N$, one can compute the limit :
$$\lim_{N\to \omega} -\frac{1}{N^2}\log E_{\gamma_{sa,N,m}}(e^{-N^2 f(\tau_{W,\Upsilon_N})})=\Lambda_{b,\Upsilon}^\omega(f).$$ 
Moreover, for any $p\in[ 2,\infty[$, with the definition of $h_t^{N,\Upsilon_N}$ of Theorem \ref{minimisationHermitianNonconvex}, $t\in ]t_i, t_{i+1}]$,  $X_1,..., X_{i+1}\in L^2(\mathcal{M}_{P,t}^\omega),$ $X_l=(X_{l,N})^\omega$ any representative, we have: $$\lim_{N\to \omega}\frac{1}{N^2}E(h_t^{N,\Upsilon_N}(X_{1,N},\cdots, X_{i+1,N}))=\Lambda_{b,\Upsilon}^\omega(f; t, X_1,..., X_{i+1}).$$
Finally, the infimum in the definition of $\Lambda_{b,\Upsilon}^\omega(f; t, X_1,..., X_{i+1})$ is attained, at a control $u$ which is a martingale and such that $||u_s||_\infty\leq C(f)$ for any $s$.
\end{theorem}    
    
\begin{proof}\setcounter{Step}{0}
\begin{step}Reduction of the case $p=\infty$ to the limit for functionals in $f\in \mathcal{C}_{reg,p,0}(   \mathcal{T}_{2,0}^c(\mathcal{F}^m_{[0,1]}*\mathcal{F}^\nu_{\mu}),d_{1,0}),$ for $p\in[ 2,\infty[$\end{step}
 Same as in Theorem \ref{MainTechnical}

\begin{step} Lower bound for $f\in \mathcal{C}_{reg,p,0}(\mathcal{T}_{2,0}^c(\mathcal{F}^m_{[0,1]}*\mathcal{F}^\nu_{\mu}),d_{1,0}), p\in [2,\infty[.$\end{step}

The statement for $\Lambda_{b,\Upsilon}^\omega(f)$ is the case $t=0$ of the statement for $\Lambda_{b,\Upsilon}^\omega(f; t, X_1,..., X_{i+1})$, we thus focus on this second case.

We already know $S_t=(W_t^N)^\omega\in (\mathcal{M}_P^\omega)^m$ and $\upsilon=(\Upsilon_N)^\omega\in (\mathcal{M}_{P,0}^\omega)^{\mu\nu}$ which has consistently the law $\mu_\Upsilon$.

 Define $G$ with $f=G\circ (I_{t_1,...t_k}*Id)$ and the associated $X^{G,N,\Upsilon_N,t}_s=X^{G,N,\Upsilon_N,t}(s,X_{1,N},..., X_{i+1,N})$ from 
Theorem \ref{minimisationHermitianNonconvex} (composed with our random initial values). %We know the finite dimensional distribution of $X^{G,N,\Upsilon_N,t}$ and the argument in step 1 of Corollary \ref{convexFBSDEmatricial}  (the case $\ell>0$ that also works for $\ell\to 0$) 
The formula in this Theorem for the drift as a conditional expectation of a cyclic derivative of $G$
gives that $Y_s-X_{i+1}=(X^{G,N,\Upsilon_N,t}_s-X_{i+1,N})^\omega \in \mathcal{L}_P^\omega$ is actually $Y_s-X_{i+1}\in \mathcal{M}_P^\omega.$ (The bound for the brownian motion is known and the bound for the stochastic integral comes from a uniform bound for derivatives of $G$ when $f\in \mathcal{C}_{reg,p,0}(\mathcal{T}_{2,0}^c(\mathcal{F}^m_{[0,1]}
*\mathcal{F}^\nu_{\mu}),d_{1,0}))$.

As in this Theorem recall that we consider $X^{G,N,\Upsilon_N,t}_{t_j}=X_j$ for $j\leq i.$

\noindent \textbf{(i) First bounds on $u_s=(b^{G,N,\Upsilon_N,t}(s,X^{G,N,\Upsilon_N}(s)))^\omega\in \mathcal{L}_P^\omega$.}

We know that $E(||b^{G,N,\Upsilon_N}(s,X^{G,N,\Upsilon_N,t}(s))||_2^2)\leq C$ independently of $N,s$ so that $||u_s||_2\leq C$ for instance using the conditional expectation formulas in Theorem \ref{minimisationHermitianNonconvex} again.
The argument above even gives $||u_s||_\infty\leq C(f).$

 From this martingale property for each $N$, one also deduces for $s\leq t$, $(s,t)\in ]t_i,t_{i+1}]$, $E_{\mathcal{L}_{P,s}^\omega}(u_t)=u_s$ so that $||u_t-u_s||_2^2=||u_t||_2^2-||u_s||_2^2$ and since $||u_s||_2^2$ is non-decreasing, it is continuous except at (at most) countably many jump points, and thus so is $u_t$ in $\mathcal{L}_{P}^\omega$.

Fix $s\in ]t_l,t_{l+1}[, l\geq i, s\geq t$ such a continuity point, then, for $S\geq s$ within an interval $]t_l,t_{l+1}]$, then from the SDE and the conditional expectation property:

\begin{align*}||&X^{G,N,\Upsilon_N,t}(S))-X^{G,N,\Upsilon_N,t}(s))-(H_S^N-H_s^N)-(S-s)b^{G,N,\Upsilon_N,t}(s,X^{G,N,\Upsilon_N,t}(s))||_2\\&\leq \int_s^Sdu||b^{G,N,\Upsilon_N,t}(u,X^{G,N,\Upsilon_N,t}(u))-b^{G,N,\Upsilon_N,t}(s,X^{G,N,\Upsilon_N,t}(s))||_2\\&\leq (S-s)||b^{G,N,\Upsilon_N,t}(S,X^{G,N,\Upsilon_N,t}(S))-b^{G,N,\Upsilon_N,t}(s,X^{G,N,\Upsilon_N,t}(s))||_2.\end{align*}

Thus in taking ultraproducts, one gets $||Y_S-S_S-(Y_s-S_s)-(S-s)u_s||_2\leq (S-s)||u_S-u_s||_2=o(t-s)$ from the choice of $s$ as continuity point, and thus especially, $Y_t-S_t$ is right differentiable at $s$ with derivative $u_s$. Let us call $ N=W^*(Y_t,S_t, t\in [0,1], \upsilon)$ which has by construction separable predual since both $Y,S$ are continuous.
As a consequence for any such continuity point $s$, $u_s\in L^2_{sa}(N)^m$.

We thus define $v$ the left continuous version of $u$ (which has right limits automatically) and then $v_s\in L^2_{sa}(N)^m$ for any $s$, it is clearly adapted as $u$ is and thus in $L^\infty_{ad}([t,1],L^2_{sa}(N)^m)\cap \mathcal{G}([t,1],L^2_{sa}(N)^m)$. We thus have a $v\in \mathcal{P}_{ad,g}^{\omega,t}.$

Finally, We know the Lipschitz bound $||Y_t-S_t-(Y_s-S_s)||_2\leq (t-s)C$ in taking the ultraproduct of the corresponding matrix relation and since $L^2_{sa}(N)$ is a reflexive Banach space, thus the derivative is Bochner integrable (and thus so is $v$) and (see e.g. \cite[Thm 1]{Bochenek}): $$Y_t-S_t-(Y_s-S_s)=\int_s^tv_udu.$$

 %expected and especially $u\in L^\infty_{ad}([0,1], \mathcal{L}_P^\omega)$. 
 
 From the bounds on $Y$, we also have $v\in \mathcal{P}_{ad,g}^{\omega,t}\cap \mathcal{P}_{ad,b}^{\omega,t}.$

\noindent \textbf{(ii) Limit of the value function along $\omega$.}
%As another consequence, $\frac{1}{2}\int_0^1||u_s||^2ds$ can be approximated by Riemann sums, and the same approximations holds for $\frac{1}{2}\int_0^1||b^{G,N,\Upsilon_N}(s,X^{G,N,\Upsilon_N}(s))||_2^2ds$ uniformly in $N$ so that :
Let us define $$f_\omega(T)=\lim_{N\to \omega}\frac{1}{2}\int_t^T||b^{G,N,\Upsilon_N,t}(s,X^{G,N,\Upsilon_N,t}(s))||_2^2ds$$

Note that for $(s,T)\in ]t_{l-1},t_l]^2,l\geq i$, we have from the conditional expectation properties \begin{align*}&\left|\int_s^T||b^{G,N,\Upsilon_N,t}(u,X^{G,N,\Upsilon_N,t}(u))||_2^2du-(T-s)||b^{G,N,\Upsilon_N,t}(s,X^{G,N,\Upsilon_N,t}(s))||_2^2
\right|\\&\leq \int_s^T ||b^{G,N,\Upsilon_N,t}(u,X^{G,N,\Upsilon_N,t}(u))||_2||b^{G,N,\Upsilon_N,t}(s,X^{G,N,\Upsilon_N,t}(s))-b^{G,N,\Upsilon_N,t}(u,X^{G,N,\Upsilon_N,t}(u))||_2
\\&+ \int_s^T ||b^{G,N,\Upsilon_N,t}(s,X^{G,N,\Upsilon_N,t}(s))||_2||b^{G,N,\Upsilon_N,t}(s,X^{G,N,\Upsilon_N,t}(s))-b^{G,N,\Upsilon_N,t}(u,X^{G,N,\Upsilon_N,t}(u))||_2\\&\leq 2||b^{G,N,\Upsilon_N,t}(T,X^{G,N,\Upsilon_N,t}(T))-b^{G,N,\Upsilon_N,t}(s,X^{G,N,\Upsilon_N,t}(s))||_2||b^{G,N,\Upsilon_N,t}(T,X^{G,N,\Upsilon_N,t}(T))||_2(T-s)\end{align*}

and thus taking the ultrafilter limit:
$$\left|2f_\omega(T)-2f_\omega(s)-(T-s)||u_s||_2^2\right|\leq 2||u_T-u_s||_2||u_T||_2(T-s)=o(T-s)$$
if $s$ is a continuity point for $u$. As a consequence the right derivative of $2f_\omega$ at $s$ is $||u_s||_2^2$ and since $f_\omega$ is Lipschitz and this is the case at all but at most countably many points, one gets :

$$f_\omega(T)=\lim_{N\to \omega}\frac{1}{2}\int_t^T||b^{G,N,\Upsilon_N,t}(s,X^{G,N,\Upsilon_N,t}(s))||_2^2ds=\frac{1}{2}\int_t^T||u_s||_2^2ds.$$

Using Lemma \ref{agreementUltraprod}, one gets that 
%E^\omega(f(\tau_{S+\int_0^{.}u_sds,\upsilon:z}))

$$\lim_{N\to \omega}E(f(\tau_{X^{G,N,\Upsilon_N,t},\Upsilon_N}))= E^\omega(f(\tau_{Y,\upsilon:z})).$$
Thus combining all our results and the formula from Theorem \ref{minimisationHermitianNonconvex}, one gets :
$$\lim_{N\to \omega}-\frac{1}{N^2}\log E_{\gamma_{sa,N,m}}(e^{-N^2 f(\tau_{W,\Upsilon_N})})\geq E^\omega(f(\tau_{Y,\upsilon:z}))+\frac{1}{2}\int_0^1||v_s||^2ds.$$ 
We thus conclude to a bound below by $\Lambda_{b,\Upsilon}^\omega(f; t, X_1,..., X_{i+1})$, since we checked the required conditions on $v$.  Once the other inequality obtained this will imply that the infimum is reached at $v$.

\begin{step} Upper bound for $f\in \mathcal{C}_{reg,p,0}(\mathcal{T}_{2,0}^c(\mathcal{F}^m_{[0,1]}*\mathcal{F}^\nu_{\mu}),d_{1,0}), p\in [2,\infty[.$\end{step}

We proceed as in step 3 of Theorem \ref{MainTechnical}, but based on lemma \ref{EquivLambdaX}.  It suffices to show that :
$$\lim_{N\to \omega}\frac{1}{N^2}E(h_t^{N,\Upsilon_N}(X_{1,N},\cdots, X_{i+1,N}))\leq \Lambda_{b,\Upsilon}^\omega(f; t, X_1,..., X_{i+1}).$$

%$$\lim_{N\to \omega}-\frac{1}{N^2}\log E_{\gamma_{sa,N,m}}(e^{-N^2 f(\tau_{W,\Upsilon_N})})\leq \Lambda_{b,\upsilon}(f).$$ 
Thus take $u\in \mathcal{P}_{ad,g,pl}(f)$ so that, from lemma \ref{EquivLambdaX}, it suffices to show: 
\begin{equation}\label{Upperv}\lim_{N\to \omega}\frac{1}{N^2}E(h_t^{N,\Upsilon_N}(X_{1,N},\cdots, X_{i+1,N}))\leq \frac{1}{2}\int_t^1||u_s||^2ds+E^\omega(f(\tau_{(X_1,..., X_i,X_{i+1}+S+\int_t^{.}u_sds,\upsilon:z})).\end{equation}%$$\lim_{N\to \omega}-\frac{1}{N^2}\log E_{\gamma_{sa,N,m}}(e^{-N^2 f(\tau_W)})\leq \frac{1}{2}\int_0^1||v_s||^2ds+f(S+\int_0^{.}v_sds).$$

By definition of the sequence of times, $\sum_{l<\nu(u)}(s_{l+1}-s_l)=1-t$ is a summable family, thus, for any $\epsilon>0$ one can find a finite set $\{l_1,...l_p\}$ with  $\sum_{l=1}^p(s_{l+1}-s_l)=1-t-\epsilon.$

Then, one gets: $$\Big|\frac{1}{2}\int_t^1||u_s||^2ds-\frac{1}{2}\sum_{l=1}^p\int_{s_l}^{s_{l+1}}||u_s||^2ds\Big|\leq \epsilon ||u||_\infty$$ and this sum depends on finitely many times by the linear interpolation property, and therefore, choosing representatives of each $v_{s_l}$, one gets a path  $U_{s_{k}^l}^N\in M_N(L^2(W_t^N, t\leq s_{k-1}^l))$ such that  $(U_{s_{k}^l}^N)^\omega=u_{s_{k}^l}$ and then $(U_{s}^N)^\omega=u_{s}$. Thus we have $\sum_{l=1}^p\int_{s_l}^{s_{l+1}}||u_s||^2ds=\lim_{N\to \omega}\sum_{l=1}^p\int_{s_l}^{s_{l+1}}||U_s^N||^2ds$ and therefore, from the same uniform bound at level $N$: $$\frac{1}{2}\int_t^1||u_s||_2^2= \lim_{N\to \omega}\frac{1}{2}\int_t^1||U_s^N||_2^2.$$

Arguing similarly one obtains: $ \int_t^{.}u_sds=(\int_t^{.}U_s^Nds)^\omega$.

Finally we obtained :
$$\lim_{N\to \omega}E_{\gamma_N}\left(G(\tau_{(Z_{t_1}^N,...,Z_{t_{k}}^N,\Upsilon_N)})+\frac{1}{2}\int_0^1||V_s^N||_2^2 ds\right)
=\frac{1}{2}\int_t^1||u_s||^2ds+E^\omega(f(\tau_{(X_1,..., X_i,X_{i+1}+S_.-S_t+\int_t^{.}u_sds,\upsilon:z}).$$

But the infimum characterization of theorem \ref{ustunel} includes $U_s^N$ as adapted process and thus 
$$E\left(G(\tau_{(Z_{t_1}^N,...,Z_{t_{k}}^N,\Upsilon_N)})+\frac{1}{2}\int_t^1||V_s^N||_2^2 ds\right)\geq \frac{1}{N^2}E(h_t^{N,\Upsilon_N}(X_{1,N},\cdots, X_{i+1,N})).$$
Taking the limit $N\to \omega$ concludes to \eqref{Upperv}, and thus combining with the first step.
$$\lim_{N\to \omega}\frac{1}{N^2}E(h_t^{N,\Upsilon_N}(X_{1,N},\cdots, X_{i+1,N}))=\Lambda_{b,\Upsilon}^\omega(f; t, X_1,..., X_{i+1})$$
%Since the limit does not depend on the ultrafilter $\omega$, the limit exists as stated.
\end{proof}     
     
\subsection{Consequence for Laplace principle}
Since our Laplace principle will be in $
(\mathcal{T}_{2,0}^c(\mathcal{F}^m_{[0,1]}*\mathcal{F}^\nu_{\mu}),d_{1,0})$
we write it in short $\mathcal{T}$ in this section. We call also $\mathcal{T}_L=K_{L\sqrt{\cdot},2}\cap \Gamma_L$ the compact subset, $\mathcal{T}_\infty=\cup_{L>0}\mathcal{T}_L$ is $\sigma$-compact.  $M(\mathcal{T})\subset(C^0_b( \mathcal{T}))^*$ is the space of finite Radon measures  equipped with the narrow topology (induced from the weak-* topology) containing $M(\mathcal{T}_L)=(C^0( \mathcal{T}_L))^*$ equipped with its own weak-* topology as dual, and which is by compactness a set of Radon measures, continuously included into $M(\mathcal{T})$.

Note that $\mathcal{T}_\infty$ is  $\sigma$-compact separable metric space hence a countable union of Polish spaces in a Hausdorff space and thus a Lusin space \cite[Cor 2 of Thm 5 p 102]{Schwartz} (in Bourbaki's sense as continuous injective image of a Polish space). Indeed, separability comes from identificaton of $\mathcal{T}$ to a subspace of a countable product of spaces of continuous functions (on a compact space) which are separable metric. As a consequence again $\mathcal{T}_\infty$ has a stronger topology which is Polish and Borel sets for the two topologies coincide (see \cite{Schwartz}). Moreover, $\mathcal{T}_\infty$ is also a Radon space \cite[Thm 9 p 122]{Schwartz}, hence any Borel probability measure is Radon and they coincide for any stronger Polish topology on $\mathcal{T}_\infty$.  

We can consider $M(\mathcal{T}_\infty)$ the set of finite Radon measures (which coincides here with the set of finite Borel measures) equipped with its narrow topology. From \cite[Prop 2 p 371]{Schwartz}, it is Hausdorff, hence $C^0_b(\mathcal{T}_\infty)$ separates points (since the space is metric, hence completely regular)  or said otherwise $M(\mathcal{T}_\infty)\hookrightarrow (C^0_b(\mathcal{T}_\infty))^*$ is a topological embedding (a continuous injection which gives the induced topology here with the target space having its weak-* topology). Note that from \cite[Prop 3 p 372]{Schwartz} we have a continuous linear map:$M(\mathcal{T}_\infty)\to M(\mathcal{T})\subset (C^0_b( \mathcal{T}))^*$ compatible with the previous injections.

Let us call $P(\mathcal{T})\subset M(\mathcal{T})$ the set of Radon probability measures %positive functionals $m$ with $m(1)=1,$ 
$P(\mathcal{T}_L)=P(\mathcal{T})\cap M(\mathcal{T}_L)$.
 We have a continuous inclusion $i:\mathcal{T}\to P(\mathcal{T})$ given by $i(\tau)(f)=f(\tau)$. 
 $i(\mathcal{T}_\infty)$ is the set of Dirac measures, its preannihilator in $C^0_b( \mathcal{T}_\infty)$ is $i(Vect(\mathcal{T}))_{\perp}=\{0\}$ hence its narrowly closed generated space $(i(Vect(\mathcal{T}))_{\perp})^{\perp}=M(\mathcal{T}_\infty)$
 hence rational convex combinations of $i(A)$, for a countable dense subset of $A\subset\mathcal{T}_\infty$, form a narrowly dense subset making it narrowly separable. The same reasoning with polars gives that the narrowly closed absolutely convex hull of  $i(\mathcal{T}_\infty)$ is the unit ball, which is therefore separable too.
 
It is known that  $P(\mathcal{T}_\infty),M(\mathcal{T}_\infty)$ are Lusin spaces (in Bourbaki's sense) and that for a stronger metric $d$ making $\mathcal{T}_\infty$ Polish, $P(\mathcal{T}_\infty,d),M(\mathcal{T}_\infty,d)$ are  Polish spaces \cite[Thm 7 p 385]{Schwartz} (and stability by $G_\delta$ sets, or directly \cite[Thm 60 p 73-III]{DellacherieMeyer}). The same result  \cite[Thm 60 p 73-III]{DellacherieMeyer} also explains that since $\mathcal{T}_\infty$ is a separable metric space, so is $P(\mathcal{T}_\infty)$ with its narrow topology (be aware that Lusin as a different meaning though in this result of \cite{DellacherieMeyer}), hence $P(\mathcal{T}_\infty)$ is a Lusin separable metric space as $\mathcal{T}_\infty$ is. 

Finally we define $\mathcal{F}(\beta\N)$ the smallest set of filtrations with a standard martingale (for matricial free BM)%(model of $T_{0,m,C,mart}$)
 containing $\mathcal{M}_P^\omega$ for any $\omega \in \beta\N-\N$ and stable by ultraproducts by ultrafilters in $\beta\N-\N$. For these filtrations, the center can be understood as before and one can define for $\mathcal{F}\in \mathcal{F}(\beta\N)$ as before $\Lambda_{b,\Upsilon}^{\mathcal{F}}(f)$ in replacing $\mathcal{M}_P^\omega$ by $\mathcal{F}.$

We also need a more ultraproduct friendly description of square-integrable functions with value in a Hilbert space:

For an absolutely continuous process $U_s$, we define
$$||U||_{BV2}^2=\sup_{\mathbf{t}\ partition}\sum_i\frac{||U_{t_i}-U_{t_{i-1}}||_2^2}{t_i-t_{i-1}},$$   
which may be infinite (note also the value increases with the partition% use CauchySchwartz with factor for r\leq s\leq t ! s-r/t-s on the double produit
) and note that when $U_t=\int_0^tu_sds$ with $\int_0^1||u_s||_2^2ds<\infty$ then $||U||_{BV2}^2\leq\int_0^1||u_s||_2^2ds$ is obvious by Cauchy-Schwarz inequality and we have an equality first when $u$ say lipschitz  %$\left|||U_{t_i}-U_{t_{i-1}}||_2^2- \int_{t_{i-1}}^{t_i}||u_s||_2^2ds(t_i-t_{i-1})\right|=|\int_{t_{i-1}}^{t_i}dv\int_{t_{i-1}}^{t_i}ds\langle u_v-u_s,u_s\rangle|\leq\int_{t_{i-1}}^{t_i}dv\int_{t_{i-1}}^{t_i}ds ||u_v-u_s||_2||u_s||_2\leq C\sqrt{\int_{t_{i-1}}^{t_i}ds||u_s||_2^2\int_{t_{i-1}}^{t_i}ds(\int_{t_{i-1}}^{t_i}dv|v-s|)^2 }=(t_{i}-t_{i-1})^{3/2}C\sqrt{\int_{t_{i-1}}^{t_i}ds||u_s||_2^2}$ so that the difference $\int_0^1||u_s||_2^2ds-\sum_i\frac{||U_{t_i}-U_{t_{i-1}}||_2^2}{t_i-t_{i-1}}\leq \sqrt{\int_{0}^{1}ds||u_s||_2^2}\sqrt{\sum_i(t_{i}-t_{i-1})^{3}}\to 0,$
 and then by density since $||.||_{BV2}$ is a semi-norm on the space where it is finite. Conversely, any $U$ with finite $||U||_{BV2}<\infty$ is absolutely continuous (especially its derivative is Bochner-measurable) %is absolutely continuous
  by an application of Cauchy-Schwarz inequality with derivative in $L^p$ for $p<2$% [Detail]
  . And $\int_0^1||u_s||_2^pds\leq ||u||_{BV2}^p$ implies by monotone convergence theorem (for part of integral with $||u_s||_2\geq 1$) and dominated convergence theorem (for part of integral with $||u_s||_2\leq 1$) that $\int_0^1||u_s||_2^2ds\leq ||u||_{BV2}^2$ concluding to the identification of spaces.

  \begin{theorem}\label{LaplaceNonConvex}
Fix $\Upsilon_N\in \mathcal{U}(M_N(\C))^{\mu\nu}$ (deterministic).
Assume that the non-commutative law $\tau_{\Upsilon_N}$ converges to some $\mu_\Upsilon\in (\mathcal{T}(\mathcal{F}^\nu_{\mu}),d)$. %\textbf{We assume either $m\geq 2$ or $m=1$ and $W^*(\mu_\Upsilon)$ diffuse.}

Let $\gamma_{sa,N,m}=\gamma_N$ the law of hermitian $N\times N$ brownian motion $W_s^N\in (M_N(\C))^m_{sa}$, then, for any 
$f\in \mathcal{C}_{reg,p,C=0}(   \mathcal{T}_{2,0}^c(\mathcal{F}^m_{[0,1]}*\mathcal{F}^\nu_{\mu}),d_{1,0}),$ for $p\in[ 2,\infty]$. We have a Laplace principal upper bound:
$$\limsup_{N\to\infty} -\frac{1}{N^2}\log E_{\gamma_{sa,N,m}}(e^{-N^2 f(\tau_{W,\Upsilon_N})})=\sup_{\omega\in\beta\N-\N}\Lambda_{b,\Upsilon}^\omega(f).$$ 
and a Laplace principal lower bound:
$$\liminf_{N\to\infty} -\frac{1}{N^2}\log E_{\gamma_{sa,N,m}}(e^{-N^2 f(\tau_{W,\Upsilon_N})})=\inf_{\omega\in\beta\N-\N}\Lambda_{b,\Upsilon}^\omega(f)\geq \inf_{\mathcal{F}\in \mathcal{F}(\beta\N)}\Lambda_{b,\Upsilon}^{\mathcal{F}}(f)= \inf_{\tau \in P(\mathcal{T}_\infty)}
\tau(f)+\overline{I}_\Upsilon(\tau).$$ 
Especially, $\widehat{\sigma}^N_{\Upsilon_N}$ satisfy a Large deviation lower bound   in $ 
\mathcal{T}%_{2,0}^c(\mathcal{F}^m_{[0,1]}*\mathcal{F}^\nu_{\mu}),d_{1,0})
$ (resp. upper bound in $P(\mathcal{T}_\infty)$ and $\mathcal{T}_\infty$) with good rate function $$\underline{I}_\Upsilon(\tau)=\sup_{f\in \mathcal{C}_{reg}}-f(\tau)+\sup_{\omega\in\beta\N-\N}\Lambda_{b,\Upsilon}^\omega(f)$$
$$(resp.\ \overline{I}_\Upsilon(\tau)=
\inf_{\mathcal{F}\in \mathcal{F}(\beta\N)}I^{\mathcal{F}}_\Upsilon(\tau)\ \ \ and \ \overline{I}_\Upsilon\circ i),$$ where we wrote for brevity $\mathcal{C}_{reg}=\mathcal{C}_{reg,\infty,C=0}(   \mathcal{T}_{2,0}^c(\mathcal{F}^m_{[0,1]}\star\mathcal{F}^\nu_{1}),d_{1,0}).$, $Min(\mathcal{C}_{reg})$ the space stable by minimum generated by $\mathcal{C}_{reg}$,  $L^2_{ad+}(\mathcal{F})=L^2_{ad}([0,1],L^2(\mathcal{F}_+))$ for adapted processes to the associated right continuous filtration  and with for $\tau$ in $P(\mathcal{T}_\infty)$:$$ I^{\mathcal{F}}_\Upsilon(\tau)=%\sup_{f\in \mathcal{C}_{reg,\infty,C=0}}-f(\tau)+\Lambda_{b,\Upsilon}^{\mathcal{F}}(f)=
\inf\Big\{\frac{1}{2}\int_0^1||u_s||^2_2ds:  u\in L^2_{ad+}(\mathcal{F})^m \ and \ \forall g\in Min(\mathcal{C}_{reg}): E^{\mathcal{F}}(g(\tau_{S+\int_0^{.}u_sds,\upsilon:z}))=\tau(g)\Big\}.%=\sup_{f\in \mathcal{C}_{reg}\cup \mathcal{E}^{1,1}_{app}}-f(\tau)+\Lambda_{b,\Upsilon}^{\mathcal{F}}(f)
%\ \ \ )
.$$
\end{theorem}  
     
\begin{proof}
The two estimates on Laplace functionals are direct consequences of Theorem \ref{MainTechnicalNonConvex}.
$\underline{I}_\Upsilon%,\overline{I}_\upsilon
$ is a rate function as a supremum of continuous functions. From lemma \ref{Exptight}, we already know 
$\widehat{\sigma}^N_{\Upsilon_N}$ is exponentially tight, the LDP lower bound will thus imply $\underline{I}_\upsilon$ is a good rate function \cite[lemma 1.2.18 (b)]{DZ}. By (a) of the same lemma, we only need to check the upper bound for compact sets.%Moreover since $\overline{I}_\upsilon\leq\underline{I}_\upsilon$  $\{x: \overline{I}_\upsilon(x)\leq \alpha\} \supset  \{x: \underline{I}_\upsilon(x)\leq \alpha\}$
\setcounter{Step}{0}

\begin{step}Lower bound\end{step}
Then if one replaces our class of functions by all continuous bounded functions, this is a standard part of the proof of Bryc's inverse Varadhan lemma, the proof is similar to \cite[lemma 4.4.6]{DZ}. It remains to see the sup is the same when restricted to our well-separating class of functions (based on lemma \ref{wellsep}).

We reproduce an argument essentially present in \cite{DZ} and start by the lower LDP. %Let $X=(\mathcal{T}_{2,0}^c(\mathcal{F}^m_{[0,1]}*\mathcal{F}^\nu_{\mu}),d_{1,0}) $ 
One takes $x\in O\subset \mathcal{T}$ an open set. Then one takes $f$ a continuous function with value $[0,1]$ with $f(x)=1$ and $f(y)=0$ for $y\in O^c$. Then one defines $f_K=K(f-1)\leq 0$.
 Take by exponential tightness a compact set $\Gamma$ with 
$P(\widehat{\sigma}^N_{\Upsilon_N}\not\in\Gamma)\leq exp(-4KN^2).$

From \cite[lemma 4.4.9]{DZ}, a continuous bounded function as  $-f_K$ is $\epsilon$-approximated uniformly on $\Gamma$ by a $\kappa_{K}=min(g_1,...,g_{k_K})$ in our well-separating class $g_i\in\mathcal{C}_{reg}$, with $-g_i\leq 0$. 
Reproducing their estimate on this compact:
\begin{align*}\int_{\mathcal{T}}e^{-\kappa_{K}(x)N^2}dP_{\widehat{\sigma}^N_{\Upsilon_N}}(x)&\leq\int_{\Gamma}e^{f_K(x)N^2+\epsilon N^2}dP_{\widehat{\sigma}^N_{\Upsilon_N}}(x)+P(\widehat{\sigma}^N_{\Upsilon_N}\in\Gamma^c)\\&\leq e^{\epsilon N^2-KN^2}P(\widehat{\sigma}^N_{\Upsilon_N}\in\Gamma\cap G^c)+e^{\epsilon N^2}P(\widehat{\sigma}^N_{\Upsilon_N}\in G)+exp(-4KN^2)%P(\widehat{\sigma}^N_{\Upsilon_N}\in\Gamma^c)
\end{align*}

Hence one gets:
$$\min(\limsup_{N\to\infty} -\frac{1}{N^2}\log(P(\widehat{\sigma}^N_{\Upsilon_N}\in G))-\epsilon,4K,K-\epsilon)\leq \sup_{\omega\in\beta\N-\N}\Lambda_{b,\Upsilon}^\omega(\kappa_{K})\leq \min_i\sup_{\omega\in\beta\N-\N}\Lambda_{b,\Upsilon}^\omega(g_i)$$ 
where we used we know $\Lambda_{b,\Upsilon}^\omega(min(g_1,...,g_{k_K}))=\min_i\Lambda_{b,\Upsilon}^\omega(g_i)$
and 
$\sup_{\omega\in\beta\N-\N}\min_i\Lambda_{b,\Upsilon}^\omega(g_i)\leq \min_i\sup_{\omega\in\beta\N-\N}\Lambda_{b,\Upsilon}^\omega(g_i).$
But in noting $f_K(x)=0$ hence $%-g_i(x)\leq -\kappa_K(x)\leq \epsilon,
-\kappa_K(x)\geq -\epsilon$ one deduces:
$$\min(\limsup_{N\to\infty} -\frac{1}{N^2}\log(P(\widehat{\sigma}^N_{\Upsilon_N}\in G))-\epsilon,4K,K-\epsilon)\leq -\kappa_K(x)+\min_i\sup_{\omega\in\beta\N-\N}\Lambda_{b,\Upsilon}^\omega(g_i)+\epsilon\leq \underline{I}_\upsilon(x)+\epsilon.$$ 
Taking the limits $K\to \infty, \epsilon \to 0$ and then an infimum on $x\in G$ concludes.

\begin{step}Formula and goodness for the upper rate function\end{step}
We first need for any control to identify the nature of the object defined by 

 $$\forall g\in Min(\mathcal{C}_{reg}): \tau(g)=E^{\mathcal{F}}(g(\tau_{S+\int_0^{.}u_sds,\upsilon:z}))$$
 and see this defines uniquely  $\tau\in P(\mathcal{T}_\infty)$. To simplify we only treat the case of $\mathcal{F}=\mathcal{M}_P^\omega$ and leave the almost identical general case to the reader. We saw in Lemma \ref{agreementUltraprod} that there is a square integrable $L$ with $\tau_{S+\int_0^{.}u_sds,\upsilon:\mathcal{Z}}\in \mathcal{T}_L$ and Borel-measurability of this variable.
 As a consequence, the formula $\tau(g)=E^\omega(g(\tau_{X,\upsilon:\mathcal{Z}}))$ defines a continuous positive unital linear functional on 
 $C^0_b( \mathcal{T}_\infty)$, or on bounded measurable functions. Since $\mathcal{T}_n$ compact, the value $\tau(1_{\mathcal{T}_n^c})\leq E^\omega(\{L\geq n\})\leq  E^\omega(L^2)/n^2$, hence $\tau$ is Radon.
 Let us see it is uniquely determined by the value on $Min(\mathcal{C}_{reg})$. % hence determining a unique Radon measure $\tau$ as claimed.
  Indeed, by well-separation, bounded sets in $Min(\mathcal{C}_{reg})$ are dense with the topology of uniform convergence on compact sets % (hence on $\mathcal{T}_L$'s) 
  in balls of $C^0_b( \mathcal{T}_\infty)$, and for a Radon measure $\tau$, one can thus arbitrarily well approximate $\tau(g)$   by the value on $Min(\mathcal{C}_{reg})$. %By the almost sure value in $\tau_{L(z)}$ with $L$ square integrable using Markov inequality this is also true of $ E^\omega(g(\tau_{X,\upsilon:z}))$, hence both sides must coincide.

 Based on this let us identify our formulas for our Laplace Lower functional.
 Let us write 
$ \mathcal{P}^{\omega,0}_{L^\infty}=\mathcal{P}_{ad,g,}^{\omega,0}((\mathcal{M}_{P,s}^\omega)_{s\geq 0})\cap \mathcal{P}_{ad,b}^{\omega,0}((\mathcal{M}_{P,s}^\omega)_{s\geq 0}),$ $\mathcal{P}^{\mathcal{F},0}_{L^2}=\mathcal{P}_{ad,g,}^{\mathcal{F},0}((\mathcal{M}_{P,s}^{\mathcal{F}})_{s\geq 0})%\cap \mathcal{P}_{ad,\mathbf{t},b}^{\omega,0}((\mathcal{M}_{P,s}^\omega)_{s\geq 0})
$ and recall that by definition for $f\in Min(\mathcal{C}_{reg})$:
\begin{align*} &\inf_{\omega\in\beta\N-\N}\Lambda_{b,\Upsilon}^\omega(f)= \inf_{\omega\in\beta\N-\N}\inf\{E^\omega(f(\tau_{S+\int_0^{.}u_sds,\upsilon:z}))+\frac{1}{2}\int_0^1||u_s||^2ds:  u\in \mathcal{P}^{\omega,0}_{L^\infty} %\ and \ \forall g\in Min(\mathcal{C}_{reg}): E^\omega(g(\tau_{S+\int_0^{.}u_sds,\upsilon:z}))=\tau(g)
\}\\&= \inf_{\omega\in\beta\N-\N}\inf\{E^\omega(f(\tau_{S+\int_0^{.}u_sds,\upsilon:z}))+\frac{1}{2}\int_0^1||u_s||^2ds:  u\in \mathcal{P}^{\omega,0}_{L^2} %\ and \ \forall g\in Min(\mathcal{C}_{reg}): E^\omega(g(\tau_{S+\int_0^{.}u_sds,\upsilon:z}))=\tau(g)
\}\\&\geq \inf_{\mathcal{F}\in\mathcal{F}(\beta\N)}\inf\{E^{\mathcal{F}}(f(\tau_{S+\int_0^{.}u_sds,\upsilon:z}))+\frac{1}{2}\int_0^1||u_s||^2ds:  u\in L^2_{ad+}(\mathcal{F}) %\ and \ \forall g\in Min(\mathcal{C}_{reg}): E^\omega(g(\tau_{S+\int_0^{.}u_sds,\upsilon:z}))=\tau(g)
\}\\&= \inf_{\mathcal{F}\in\mathcal{F}(\beta\N)}\inf_{\tau\in P(\mathcal{T}_\infty)}\inf\{f(\tau)+\frac{1}{2}\int_0^1||u_s||^2ds:  u\in L^2_{ad+}(\mathcal{F}) \\&\qquad \qquad \qquad \qquad \qquad \qquad and \ \forall g\in Min(\mathcal{C}_{reg}): E^{\mathcal{F}}(g(\tau_{S+\int_0^{.}u_sds,\upsilon:z}))=\tau(g)
\}%\\&= \inf_{\tau \in P(\mathcal{T}_\infty)}
%\tau(f)+\overline{I}_\Upsilon(\tau)
\end{align*}
 where the first equality is by standard spectral theory approximation, the second is obvious % by extra standard approximation (anyway, we only need $\geq$ which is automatic) 
 and the third equality is permitted by our previous consideration giving that any control has a unique associated $\tau\in P(\mathcal{T}_\infty)$.% and the last one is merely indentifications and inversion of infima.
 
Similarly,  one gets after switching of infima: $\inf_{\mathcal{F}\in \mathcal{F}(\beta\N)}\Lambda_{b,\Upsilon}^{\mathcal{F}}(f)= \inf_{\tau \in P(\mathcal{T}_\infty)}
\tau(f)+\overline{I}_\Upsilon(\tau).$%\inf_{}\Lambda_{b,\Upsilon}^\omega(f)=\inf_{\tau \in P(\mathcal{T}_\infty)}
%\tau(f)+\overline{I}_\Upsilon(\tau)
 
 Finally, let us see that the level set $I_\alpha=\{\tau\in P(\mathcal{T}_\infty): \overline{I}_\Upsilon(\tau)\leq \alpha\}$ is tight hence relatively compact by Prokhorov's Theorem (beyond the Polish case, see \cite[Thm 3 p 379]{Schwartz}). 
Note that $$I_\alpha\subset \{\tau\in P(\mathcal{T}_\infty): \forall g\in Min(\mathcal{C}_{reg}):\exists \omega, \exists u\in \mathcal{P}^{\mathcal{F},0}_{L^2} \int_0^1||u_s||_2^2ds\leq 1+\alpha: E^\omega(g(\tau_{S+\int_0^{.}u_sds,\upsilon:z}))=\tau(g)\}$$

but we explained that for such kind of $u$, $P(\int_0^{1}||u_v||_{2,Z}^2dv> C)\leq \frac{\int_0^1||u_s||_2^2ds}{C}\leq \frac{\alpha+1}{C}$.

Moreover, the relation $E^\omega(g(\tau_{S+\int_0^{.}u_sds,\upsilon:z}))=\tau(g)$ extends to $g$ continuous bounded and hence by monotone convergence theorem to indicator functions of open sets so that: $$P(\tau \not\in \mathcal{T}_{C+1})=P(\tau_{S+\int_0^{.}u_sds,\upsilon:.} \not\in \mathcal{T}_{C+1})\leq P(\int_0^{1}||u_v||_{2,Z}^2dv> C)\leq  \frac{\alpha+1}{C}\to_{C\to \infty} 0$$
and gives the expected tightness since $\mathcal{T}_{C+1}$ is compact.

To conclude that $\overline{I}_\Upsilon$ is a good rate function, it suffices to see that $I_\alpha$ is closed.

Thus consider $\tau_n\to \tau$, take $\omega_n\in\beta\N-\N$, $u^n\in \mathcal{P}^{\mathcal{F}_n,0}_{L^2}$ with $\forall g\in Min(\mathcal{C}_{reg}): E^{\mathcal{F}_n}(g(\tau_{S+\int_0^{.}u^n_sds,\upsilon:z}))=\tau_n(g)$
and 
$\frac{1}{2}\int_0^1||u_s^n||^2_2ds\leq\overline{I}_\Upsilon(\tau_n) +1/n$%:  u\in \mathcal{P}^{\omega,0}_{L^2} \ and \ \forall g\in Min(\mathcal{C}_{reg}): E^\omega(g(\tau_{S+\int_0^{.}u_sds,\upsilon:z}))=\tau(g)\}%=\sup_{f\in \mathcal{C}_{reg}\cup \mathcal{E}^{1,1}_{app}}-f(\tau)+\Lambda_{b,\Upsilon}^\omega(f)

Let us consider $X^n_t=S_t+\int_0^{t}u^n_sds$ and $\omega\in \beta\N-\N$ and consider $\mathcal{F}=(\mathcal{F}_n)^\omega\in \mathcal{F}(\beta\N)$. One can consider $S_t^i=(S_t^i)^\omega\in \mathcal{F}_t$, $u_t^i=[u(X^{n,i}_t)]^\omega\in \mathcal{F}_t$ (in general $(X^n_t)^\omega$ is in an ultraproduct of $L^2$ spaces).% by compactness of this space one can consider $\Omega=\lim_{n\to \omega} \omega_n.$ 
As usual, one recovers $X_t^j$ as a limit of $X_{t,\epsilon}^j=4i(u_t^j+1)(u_t^j-1)^*((u_t^j-1)(u_t^j-1)^*+\epsilon)^{-1}$ when $\epsilon\to 0$ The bound on $((u_t^j-1)(u_t^j-1)^*+\epsilon)^{-1}$ in $L^1$ uniform in $\varepsilon$ goes to the ultraproduct by functional calculus enabling to see that their is an $L^2$ limit of its square-root $|u_t^j-1|^{-1}\in L^2(\mathcal{F}_t)$ by monotone convergence theorem. The relation $((u_t^j-1)(u_t^j-1)^*+\epsilon)^{-1}((u_t^j-1)(u_t^j-1)^*+\epsilon)=1$ extended to the ultraproduct then gives $((u_t^j-1)(u_t^j-1)^*+\epsilon)^{-1}(u_t^j-1)(u_t^j-1)^*-1\to 0$ in norm in $L^1$ hence the fact that $|u_t^j-1|^{-1}$ is indeed the right inverse which gives $X_t^j$ first in $L^1$ and then $L^2$.

Let also $U_t=X_t-S_t$. Using bounds for regularized version first, one sees $||U||_{BV}^2\leq \lim_{n\to\omega} \int_0^1||u^n_s||_2^2ds<\infty$
so that $U$ is almost separably valued and can be written for $u\in L^2_{ad}([0,1],L^2(\mathcal{F}))$  in the form $U_t=\int_0^t u_sds$ with :
$$\int_0^1||u_s||_2^2ds=||U||_{BV}^2\leq \lim_{n\to\omega} \int_0^1||u^n_s||_2^2ds\leq \alpha<\infty.$$
Finally, looking at the specific form of functions in $ Min(\mathcal{C}_{reg})$ one sees they are stable by ultraproduct and by assumption continuous for $d_{1,0}$, giving : $$E^{\mathcal{F}}(g(\tau_{S+\int_0^{.}u_sds,\upsilon:z}))=\lim_{n\to \omega}E^{\mathcal{F}^n}(g(\tau_{S+\int_0^{.}u^n_sds,\upsilon:z}))=\tau(g).$$

Finally%with an extra standard approximation argument for the equality
, we obtained:
\begin{align*}\overline{I}_\Upsilon(\tau)&\leq I^{\mathcal{F}}_\Upsilon(\tau)=\inf\{\frac{1}{2}\int_0^1||v_s||^2_2ds:\\&
  v\in L^2_{ad}([0,1],L^2(\mathcal{F}_+)) \ and \ \forall g\in Min(\mathcal{C}_{reg}): E^{\mathcal{F}}(g(\tau_{X^v,\upsilon:z}))=\tau(g)\}\leq \int_0^1||u_s||_2^2ds\leq \alpha.\end{align*}

This concludes the proof of $\overline{I}_\Upsilon$ good rate function on $P(\mathcal{T}_\infty)$.

\begin{step}Extension of the Laplace lower bound to $P(\mathcal{T}_\infty)$\end{step}

We first need to bootstrap the Laplace lower bound to get the LDP upper bound.
 
We start by replacing $f\in MIN(\mathcal{C}_{reg})$ by $f\in C^0_b(\mathcal{T}_\infty)$. Indeed, fix $1/4>\epsilon>0$ and for such an $f$, fix $-C=\min(0,\inf f)$, $E=\sup |f|$  from step 2 a compact set $\Gamma=\mathcal{T}_D$ such that $K=\{\tau: \overline{I}_\Upsilon(\tau)\leq \alpha +C+1\}\subset \{\tau :\tau(\mathcal{T}_D^c)\leq \epsilon/(E+C)\}$ with $\alpha=\inf_{\tau \in P(\mathcal{T}_\infty)}
\tau(f)+\overline{I}_\Upsilon(\tau)$ (e.g. $D\geq 1+(\alpha+C+2)(E+C)/\epsilon$). 
Note that in enlarging $D$ we can ensure for future purposes  by exponential tightness that $P(\widehat{\sigma}^N_{\Upsilon_N}\not\in\mathcal{T}_D)\leq exp(-N^2/\epsilon-N^2C)$ at least for $N$ large enough.
Note also that as a consequence 
$\alpha=\inf_{\tau \in K}
\tau(f)+\overline{I}_\Upsilon(\tau)$ since $\inf_{\tau \in K^c}
\tau(f)+\overline{I}_\Upsilon(\tau)\geq \alpha+1>\alpha.$

 Then we $\epsilon$-approximate $f$ on $\Gamma$ by $g\in MIN(\mathcal{C}_{reg})$ with $inf g\geq -C=\min(0,\inf f)$ thanks to \cite[lemma 4.4.9]{DZ} then  $$\inf_{\tau \in P(\mathcal{T}_\infty)}
\tau(f)+\overline{I}_\Upsilon(\tau)\leq \inf_{\tau \in P(\mathcal{T}_\infty)}
\tau(f1_\Gamma)+\overline{I}_\Upsilon(\tau)+E\tau(1_{\Gamma^c})
\leq\inf_{\tau \in P(\mathcal{T}_\infty)}
\tau(g1_\Gamma)+\overline{I}_\Upsilon(\tau)+\epsilon +E\tau(1_{\Gamma^c})%\leq \inf_{\tau \in P(\mathcal{T}_\infty)}
%\tau(g)+C \tau(1_{\Gamma^c})+\overline{I}_\Upsilon(\tau)+2\epsilon 
$$
From our choices, one can decompose the last infimum in the minimum of 
$$\inf_{\tau \in K}
\tau(g1_\Gamma)+\overline{I}_\Upsilon(\tau)+\epsilon+E\tau(1_{\Gamma^c})\leq \inf_{\tau \in K}
\tau(f)+C\tau(1_{\Gamma^c})+\overline{I}_\Upsilon(\tau)+3\epsilon
\leq \alpha+4\epsilon<\alpha +1
$$
and 
$$\inf_{\tau \in K^c}
\tau(g1_\Gamma)+\overline{I}_\Upsilon(\tau)+\epsilon
+E\tau(1_{\Gamma^c})\geq \inf_{\tau \in K^c}
\tau(g1_\Gamma) +\alpha+C+1+\epsilon\geq \alpha+1+\epsilon$$
The minimum of the two is clearly attained at the first value for which we can also bound by choice of $K$:
$$\inf_{\tau \in K}
\tau(g1_\Gamma)+\overline{I}_\Upsilon(\tau)+\epsilon+E\tau(1_{\Gamma^c})\leq \inf_{\tau \in K}
\tau(g)+(C+E) \tau(1_{\Gamma^c})+\overline{I}_\Upsilon(\tau)+\epsilon
\leq \inf_{\tau \in K}\tau(g)+\overline{I}_\Upsilon(\tau)+2\epsilon%\geq \inf_{\tau \in K}\tau(g1_\Gamma)+\overline{I}_\Upsilon(\tau)+3\epsilon
$$
But again we have a lower bound 
$$\inf_{\tau \in K^c}
\tau(g)+\overline{I}_\Upsilon(\tau)
\geq \inf_{\tau \in K^c}
\tau(g) +\alpha+C+1\geq \alpha+1$$
so that gathering up with
$\alpha=\inf_{\tau \in P(\mathcal{T}_\infty)}
\tau(f)+\overline{I}_\Upsilon(\tau)\leq\inf_{\tau \in K}\tau(g)+\overline{I}_\Upsilon(\tau)+2\epsilon$, one gets the bound :
$$\alpha=\inf_{\tau \in P(\mathcal{T}_\infty)}
\tau(f)+\overline{I}_\Upsilon(\tau)\leq\inf_{\tau \in P(\mathcal{T}_\infty)}\tau(g)+\overline{I}_\Upsilon(\tau)+2\epsilon$$

%$$\inf_{\tau \in K^c}
%\tau(g)+C \tau(1_{\Gamma^c})+\overline{I}_\Upsilon(\tau)+2\epsilon
%\geq \inf_{\tau \in K^c}
%\tau(g) +\alpha+2C+1+2\epsilon\geq \alpha+C+1+2\epsilon$$
%Now take $\tau_0$ with $\tau(g1_\Gamma)+\overline{I}_\Upsilon(\tau)$

We are ready to extend the Laplace principle and we start by a decomposition estimate:
$$E_{\gamma_{sa,N,m}}(e^{-N^2 f(\tau_{W,\Upsilon_N})} %(1_\Gamma +1_{\Gamma^c})(\tau_{W,\Upsilon_N})
)
\leq e^{N^2\epsilon}E_{\gamma_{sa,N,m}}(1_\Gamma (\tau_{W,\Upsilon_N}) e^{-N^2 g(\tau_{W,\Upsilon_N})})+e^{CN^2 } E_{\gamma_{sa,N,m}}(1_{\mathcal{T_D}^c}(\tau_{W,\Upsilon_N}))$$
Hence, using the choice of $D$, one deduces:
\begin{align*}\liminf_{N\to\infty} -\frac{1}{N^2}\log E_{\gamma_{sa,N,m}}(e^{-N^2 f(\tau_{W,\Upsilon_N})})&\geq \min(1/\epsilon ,-\epsilon+\liminf_{N\to\infty} -\frac{1}{N^2}\log E_{\gamma_{sa,N,m}}(e^{-N^2 g(\tau_{W,\Upsilon_N})}))\\&\geq \min(1/\epsilon ,-\epsilon+\inf_{\tau \in P(\mathcal{T}_\infty)}
\tau(g)+\overline{I}_\Upsilon(\tau)) \\&\geq \min(1/\epsilon ,-3\epsilon+\inf_{\tau \in P(\mathcal{T}_\infty)}
\tau(f)+\overline{I}_\Upsilon(\tau)) .\end{align*}
This concludes in taking the limit $\epsilon\to 0$.

We call $J:C^0_b(\mathcal{T}_\infty)\to (M(\mathcal{T}_\infty))^*\subset C^0_b( P(\mathcal{T}_\infty))$ given by for $\tau\in P(\mathcal{T}_\infty),$ $f\in C^0_b(\mathcal{T}_\infty)$, $J(f)(\tau)=\tau(f)$ since $f$ is bounded, $J(f)$ is bounded on probabilities and from the definition of the narrow topology it is continuous on $P(\mathcal{T}_\infty)$ as claimed.

Consider $\mathcal{D}_{reg}=MAX(%\{\R\}\cup 
J(C^0_b(\mathcal{T}_\infty)))$ the class stable by maximum and containing constants obtained from the previous family and then $MIN( \mathcal{D}_{reg})$ the class stable by minimum generated. If we extend the Laplace bound to $\mathcal{D}_{reg}$ it is standard to extend it automatically to $MIN( \mathcal{D}_{reg})$ and moreover this will be useful since $\mathcal{D}_{reg}$ is well-separating since the narrow topology is Hausdorff on $P(\mathcal{T}_\infty)$.

But for $J(f_1),...,J(f_n)\in J(C^0_b(\mathcal{T}_\infty))$ then  for $F=Max( J(f_1),...,J(f_n))\in \mathcal{D}_{reg}$ a typical element we have: $$Max( J(f_1),...,J(f_n))(i(\widehat{\sigma}^N_{\Upsilon_N}))=
Max(i(\widehat{\sigma}^N_{\Upsilon_N})[f_1],...,
i(\widehat{\sigma}^N_{\Upsilon_N})[f_n])=Max(f_1,...,
f_n)[\widehat{\sigma}^N_{\Upsilon_N}].$$
Thus if we write $G=Max(f_1,...,
f_n)$, one gets the concluding inequality:
\begin{align*}\liminf_{N\to\infty} -\frac{1}{N^2}\log E_{\gamma_{sa,N,m}}(e^{-N^2 F(i(\tau_{W,\Upsilon_N}))})&=\liminf_{N\to\infty} -\frac{1}{N^2}\log E_{\gamma_{sa,N,m}}(e^{-N^2 G(\tau_{W,\Upsilon_N})})\\&\geq \inf_{\tau \in P(\mathcal{T}_\infty)}
\tau(G)+\overline{I}_\Upsilon(\tau)\geq \inf_{\tau \in P(\mathcal{T}_\infty)}
F(\tau)+\overline{I}_\Upsilon(\tau) ,\end{align*}
since by positivity of the law $\tau$: $$\tau(G)=\tau(Max(f_1,...,
f_n))\geq Max(\tau(f_1),...,\tau(f_n))=Max( J(f_1),...,J(f_n))(\tau)=F(\tau).$$

\begin{step}Upper bound in $P(\mathcal{T}_\infty)$\end{step}

The proof is a modification and will use part of the proof of \cite[Th 1.2.3]{DupuisEllis}. They use Polish spaces while $P(\mathcal{T}_\infty)$ is only Lusin separable metric space, but it is easy to see they only use for this proof a metric space and compactness of the level sets of the rate function.

Take $K$ a compact set, $\epsilon\in ]0,1[$ and $h_j(x)=j(d(x,K)\wedge 1)$ the continuous bounded function. Take $\alpha= \inf_{\tau\in K}\overline{I}_\Upsilon(\tau)$ and then consider the compact set $L=\{\tau: \overline{I}_\Upsilon(\tau)\leq \alpha+1\}$. Note that for $\tau\not\in L\cup K$, $h_j(\tau)+\overline{I}_\Upsilon(\tau)>\overline{I}_\Upsilon(\tau)
>\alpha= \inf_{\tau\in K}\overline{I}_\Upsilon(\tau)\geq \inf_{\tau\in K\cup L}h_j(\tau)+\overline{I}_\Upsilon(\tau)$
hence:$$\inf_{\tau\in K\cup L}h_j(\tau)+\overline{I}_\Upsilon(\tau)= \inf_{\tau\in P(\mathcal{T}_\infty)}h_j(\tau)+\overline{I}_\Upsilon(\tau)= h_j(\tau_0)+\overline{I}_\Upsilon(\tau_0)\leq \alpha$$
for some $\tau_0\in K\cup L$ since we minimize a lower-semicontinuous function on a compact set.
Take also $C>0$ from the exponential tightness result such that $P(\widehat{\sigma}^N_{\Upsilon_N}\not\in\mathcal{T}_C)\leq exp(-N^2/\epsilon).$ Let us call $\Gamma=K\cup L\cup\mathcal{T}_C$ the above compact set and find using \cite[lemma 4.4.9]{DZ}, an uniform $\epsilon$-approximation on $\Gamma$ of $h_j$ by a $\kappa_{\Gamma}=min(g_1,...,g_{k_\Gamma})$ in our well-separating class $g_i\in\mathcal{D}_{reg}$, with $-g_i\leq 0$. 
Then, one bounds (with identification of $i(\widehat{\sigma}^N_{\Upsilon_N})$ and $\widehat{\sigma}^N_{\Upsilon_N}$):
\begin{align*}P(\widehat{\sigma}^N_{\Upsilon_N}\in K)&\leq   E\Big(e^{-N^2h_j(\widehat{\sigma}^N_{\Upsilon_N})))}\Big)\leq   E\Big(1_{\Gamma}(\widehat{\sigma}^N_{\Upsilon_N})e^{N^2(\epsilon -\kappa_\Gamma(\widehat{\sigma}^N_{\Upsilon_N})))}\Big)+  P(\widehat{\sigma}^N_{\Upsilon_N}\in \Gamma^c)\\&\leq   e^{N^2\epsilon}E\Big(e^{-N^2\kappa_\Gamma(\widehat{\sigma}^N_{\Upsilon_N}))}\Big)+  exp(-N^2/\epsilon)\end{align*}
Using the standard \cite[lemma 1.2.15]{DZ}, one gets:
\begin{align*}&\limsup_{N\to\infty}\frac{1}{N^2}\log P(\widehat{\sigma}^N_{\Upsilon_N}\in K)\leq  \max(\epsilon+ \limsup_{N\to\infty}\frac{1}{N^2}\log E\Big(e^{-N^2\kappa_\Gamma(\widehat{\sigma}^N_{\Upsilon_N}))}\Big),  -1/\epsilon)
,\end{align*}\begin{align*}&\limsup_{N\to\infty}\frac{1}{N^2}\log P(\widehat{\sigma}^N_{\Upsilon_N}\in K)\leq \max\Big(-1/\epsilon,\epsilon-\inf_{\tau \in P(\mathcal{T}_\infty)}
\tau(\kappa_\Gamma)+\overline{I}_\Upsilon(\tau)\Big)\\&\leq \max\Big(-1/\epsilon,\epsilon-\inf_{\tau \in K\cup L}
\tau(\kappa_\Gamma)+\overline{I}_\Upsilon(\tau),\epsilon-\inf_{\tau \in K^c\cap L^c}
\tau(\kappa_\Gamma)+\overline{I}_\Upsilon(\tau)\Big)
\\&\leq \max\Big(-1/\epsilon,2\epsilon-\inf_{\tau \in K\cup L}
\tau(h_j)+\overline{I}_\Upsilon(\tau),\epsilon-(\alpha+1)\Big),\end{align*}
with the last inequality from the choice of $L$, non-negativity of $\kappa_\Gamma$ and its uniform approximation again.

But the middle term has been chosen such that $2\epsilon-\inf_{\tau \in K\cup L}
\tau(h_j)+\overline{I}_\Upsilon(\tau)\geq 2\epsilon-\alpha>\epsilon-\alpha-1$, so that we can get rid of the last term and obtain in taking the limit $\epsilon\to 0$: \begin{align*}\limsup_{N\to\infty}\frac{1}{N^2}\log P(\widehat{\sigma}^N_{\Upsilon_N}\in K)&\leq   -\inf_{\tau \in K\cup L}
\tau(h_j)+\overline{I}_\Upsilon(\tau)= -\inf_{\tau \in P(\mathcal{T}_\infty)}
\tau(h_j)+\overline{I}_\Upsilon(\tau)\end{align*}
One can finally use \cite[(1.4) p 8]{DupuisEllis} which only uses again that $\overline{I}_\Upsilon$ is a good rate function on a metric space to  take the limit $j\to \infty$ and get the expected upper bound:
$$\limsup_{N\to\infty}\frac{1}{N^2}\log P(\widehat{\sigma}^N_{\Upsilon_N}\in K)\leq -\inf_{\tau\in K}
\overline{I}_\Upsilon(\tau).$$
\begin{step}Upper bound in $\mathcal{T}_\infty$\end{step}
Finally the LDP upper bound on $\mathcal{T}_\infty$ is obtained by inverse contraction. %Indeed first the random vaariable is valued in $i(\mathcal{T}_\infty)$ and exponentially tight there so that it suffices to prove the upper bound for compact sets which are also compact in $P(\mathcal{T}_\infty)$ so that our previous bound applies. Moreover, 
The map $i:\mathcal{T}_\infty\to P(\mathcal{T}_\infty)$ is a continuous injection, our random variable is exponentially tight in the source space so that the upper bound part of \cite[Thm 4.2.4]{DZ} and the following remark applies and give an upper LDP in $\mathcal{T}_\infty$ with the induced function which is a rate function (not necessarily good).

We saw in step 2 that the level sets $$\{\tau:\overline{I}_\Upsilon(\tau)\leq \alpha\}\subset \{\tau: P(\tau \not\in \mathcal{T}_{C+1})\leq  \frac{\alpha+1}{C}\}$$
so that for $\tau=i(x)\in i( \mathcal{T}_\infty)$
a Dirac mass, one gets $i(x)\in \mathcal{T}_{C+1}$ as soon as $\frac{\alpha+1}{C}<1$ hence $\{x\in\mathcal{T}_\infty:\overline{I}_\Upsilon(i(x))\leq \alpha\}\subset \mathcal{T}_{\alpha+3}$ is again compact and $\overline{I}_\Upsilon\circ i$ is also a good rate function on $\mathcal{T}_\infty$ as claimed.
\end{proof}

%\subsection{Proof of the partial Laplace deviation principle}
 
\section{Applications to free entropy}     
     By the contraction principle of large deviation theory (see e.g. \cite[Th 4.2.1]{DZ}) for the projection $\pi_1:(\mathcal{T}_{2,0}^c(\mathcal{F}^m_{[0,1]}*\mathcal{F}^\nu_{\mu}),d_{1,0})
\to (\mathcal{T}(\mathcal{F}^m_{1}*\mathcal{F}^\nu_{\mu}),d_{0,0}))
$ at time $1$, one deduces 
\begin{theorem}\label{contraction}Fix  $\Upsilon_N\in \mathcal{U}(M_N(\C))^{\nu}$ (deterministic $\mu=1$ with previous notation).
Assume that the non-commutative law $\tau_{\Upsilon_N}$ converges to some $\mu_\Upsilon\in (\mathcal{T}(\mathcal{F}^\nu_{1}),d_{1,0})$. %\textbf{We assume either $m\geq 2$ or $m=1$ and $W^*(\mu_\Upsilon)$ diffuse.}
Then $\mathfrak{G}^N_{\Upsilon_N}=\pi_1(\widehat{\sigma}^N_{\Upsilon_N})$ satisfies a Large Deviation Principle lower bound (resp upper bound) 
in $(\mathcal{T}(\mathcal{F}^m_{1}*\mathcal{F}^\nu_{1}),d_{1,0})$ with Good rate function $J_\Upsilon\ (resp.\  \overline{J}_\Upsilon):(\mathcal{T}(\mathcal{F}^m_{1}*\mathcal{F}^\nu_{1}),d_{1,0})\to [0,\infty]$ given by $$J_\Upsilon(\tau)=\inf\{\underline{I}_\Upsilon(\sigma): \pi_1(\sigma)=\tau, \sigma \in (\mathcal{T}_{2,0}^c(\mathcal{F}^m_{[0,1]}*\mathcal{F}^\nu_{1}),d_{1,0})
\},$$
$$\mathrm{(resp.}\qquad  \qquad   \overline{J}_\Upsilon(\tau)=\inf\big\{\overline{I}_\Upsilon
(i(\sigma)): \pi_1(\sigma)=\tau, \sigma \in (\mathcal{T}_{2,0}^c(\mathcal{F}^m_{[0,1]}*\mathcal{F}^\nu_{1}),d_{1,0})
\big\}\ \ ).$$
\end{theorem}     
\subsection{Equality $\chi=\chi^*$ for free Gibbs states with convex potential}  
We have to find where the infima in the definition of $J_{\upsilon},\Lambda_{b,\upsilon}$ and the supremum in the definition of $\underline{I}_{\upsilon}$ are reached. We first use the argument in \cite[Th 7.3]{BCG}.% Let $g\in \mathcal{E}^{1,1}(  \mathcal{T}_2(\mathcal{F}^m_{1}),d_2 )$
     
    Let  $\mu\in \mathcal{T}_2(\mathcal{F}^m_{1}\star\mathcal{F}^\nu_{1})$, we follow \cite{BCG} and define $\tau_\mu\in \mathcal{T}_2^c(\mathcal{F}^m_{[0,1]}\star\mathcal{F}^\nu_{1})$ the law of the brownian bridge, i.e. if $\{S^1,...,S^m\}$ is the law of a free brownian motion  free from  $X=\{X^1,...,X^m\},\upsilon$ with law (of the unitary $u(X),\upsilon$) $\tau_{X,\upsilon}=\mu$, the law $\tau_{U,\upsilon}$ of the process:$$\left\{U_t^l=u(t X^l+(1-t)S^l_{\frac{t}{1-t}}), 1\leq l\leq m, t\in [0,1]\right\}.$$
\begin{proposition}\label{marginal}
Fix the assumption of Theorem \ref{contraction}. For any $\mu\in \mathcal{T}_2(\mathcal{F}^m_{1}\star\mathcal{F}^\nu_{1})$, we have $$\overline{I}_\Upsilon(\tau_\mu)\leq -\widetilde{{\chi}}_\infty(\Psi(X)|(\Upsilon_N)_{N\in\N})%\leq \overline{J}_\Upsilon(\mu)
\leq J_\Upsilon(\mu)\leq \underline{I}_\Upsilon(\tau_\mu),$$
%as soon as $\overline{I}_\upsilon(\tau_\mu)=\underline{I}_\upsilon(\tau_\mu).$
\end{proposition}    
    \begin{proof}
    Clearly from  the definition  in Theorem \ref{contraction}, $J_\Upsilon(\mu)\leq \underline{I}_\Upsilon(\tau_\mu).$ But let us recall the straightforward variant of formula (19) in \cite{BCG} for any $\delta>0$:\begin{align*}&\limsup_{\epsilon\to 0}\limsup_{N\to \infty}\frac{1}{N^2}\log P(d(\pi_1(\widehat{\sigma}^N_{\Upsilon_N}),\mu)\leq \epsilon)\\&=\limsup_{\epsilon\to 0}\limsup_{N\to \infty}\frac{1}{N^2}\log P(d(\pi_1(\widehat{\sigma}^N_{\Upsilon_N}),\mu)\leq \epsilon,d(\widehat{\sigma}^N_{\Upsilon_N})),\tau_\mu)\leq \delta).\end{align*}
Thus, from the lower Laplace deviation bound in Theorem \ref{LaplaceNonConvex}, one deduces (from the fact that $\overline{I}_\upsilon$ is a rate function in choosing a sequence with $\tau_n\to \tau_\mu$ and $\lim_{\delta\to 0}\inf_{d(\tau,\tau_\mu)\leq \delta}\overline{I}_\Upsilon(\tau)=\liminf _n\overline{I}_\Upsilon(\tau_n)\geq \overline{I}_\Upsilon(\tau_\mu)$):
$$\widetilde{{\chi}}_\infty(\Psi(X)|(\Upsilon_N)_{N\in\N})=\limsup_{\epsilon\to 0}\limsup_{N\to \infty}\frac{1}{N^2}\log P(d(\pi_1(\widehat{\sigma}^N_{\Upsilon_N}),\mu)\leq \epsilon)\leq -\lim_{\delta\to 0}\inf_{d(\tau,\tau_\mu)\leq \delta}\overline{I}_\Upsilon(\tau)\leq-\overline{I}_\Upsilon(\tau_\mu)$$
But from the large deviation principle in Theorem \ref{contraction}% and the fact that $J$ is as a consequence a good rate function
, we also have:$$\limsup_{\epsilon\to 0}\liminf_{N\to \infty}\frac{1}{N^2}\log P(d(\pi_1(\widehat{\sigma}^N_{\Upsilon_N}),\mu)\leq \epsilon)\geq \limsup_{\epsilon\to 0}-\inf_{d(\tau,\mu)<\epsilon} J_\Upsilon(\tau)\geq-J_\Upsilon(\mu).$$
Finally, we thus obtained the missing $-J_\Upsilon(\mu)\leq -\overline{I}_\Upsilon(\tau_\mu).$    
    \end{proof}
\begin{theorem}\label{chichistar}Fix the assumption of Theorem \ref{contraction}.
Let $g\in \mathcal{E}^{1,1}_{app}(  \mathcal{T}_{2,0}(\mathcal{F}^m_{1}\star\mathcal{F}^\nu_{1}),d_{2,0} ).$ %satisfying for $c_g>0$ the strict convexity condition for any $X,Y$ in  any $L^2_{sa}(M,\tau)^m,\upsilon\in\mathcal{U}(M)^\nu$ (any tracial $(M,\tau)$):
%$$g(\tau_{X+Y,\upsilon})+g(\tau_{X-Y,\upsilon})-2g(\tau_{X,\upsilon})\geq c_g||Y||_2^2.$$
Let $\tau_g$ the unique solution of $(SD_g)$ obtained in Theorem \ref{SDVg} with fixed law of the unitary part $\upsilon$, and $X=X_1,...X_m,\upsilon$ having this law.
%Assume also  for the upper bound part $\chi\leq \chi^*$ that either $W^*(\upsilon)=\C$ or $g$ satisfies the assumption of $h$ in Theorem \ref{OptControlFree}.
 Then, we have the inequalities: $$\chi(X_1,...,X_m|\upsilon)\geq\underline{\chi}(X_1,...,X_m|\upsilon)\geq \chi^*(X_1,...,X_m|W^*(\upsilon)),$$ $$\chi^G(X_1,...,X_m|\upsilon)\geq  \underline{\chi}^G(X_1,...,X_m|\upsilon)\geq\chi^{G*}(X_1,...,X_m|W^*(\upsilon))$$
 with equality if $W^*(\upsilon)=\C.$
\end{theorem}    
Since the upper bound is known in the case $W^*(\upsilon)=\C$ from \cite{BCG}, and probably virtually known via similar techniques in the general case too, the main new feature is the lower bound. Hence, we don't try to extend here the upper bound since it requires supplementary techniques from the second paper of this series.%get the most general assumption for the upper bound and stick to an easy enough application of our control techniques.

    \begin{proof}
    The first equality comes from the second. Let $\mu_g=\tau_{X,\upsilon}$ %the law of $u(X_g),\upsilon$ if $X_g,\upsilon$ has law $\tau_g.$ 
    By our previous results, \eqref{chiGchiU}, %corollary \ref{Main}, Theorem \ref{liminflimsup}, 
    obvious relations and proposition \ref{marginal}, we know that%if $\underline{I}_\Upsilon(\tau_{\mu_g})=\overline{I}_\Upsilon(\tau_{\mu_g}))$%for $f=g\circ J_1$
   %then
    : \begin{align*}\chi^G(X|\upsilon)&=\widetilde{\chi}(\Psi(X)|\upsilon)\geq \widetilde{\underline{\chi}}(\Psi(X)|\upsilon)\geq \widetilde{\underline{\chi}}(\Psi(X)|(\Upsilon_N)_{N\in \N})=-J_\Upsilon(\mu_g)\geq-\underline{I}_\Upsilon(\tau_{\mu_g})\\&=-\sup_{f\in \mathcal{C}_{reg,\infty,C=0}(   \mathcal{T}_{2,0}^c(\mathcal{F}^m_{[0,1]}\star\mathcal{F}^\nu_{1}),d_{1,0})}-f(\tau_{\mu_g})+\sup_{\omega}\Lambda_{b,\Upsilon}^\omega(f).\end{align*}

   Thus take  $f\in \mathcal{C}_{reg,\infty,C=0}(   \mathcal{T}_{2,0}^c(\mathcal{F}^m_{[0,1]}\star\mathcal{F}^\nu_{1}),d_{1,0})$, and fix $\mathbf{t}(f)=(t_0=0<t_1<...<t_k)\leq 1$, $F\in \mathcal{E}_{reg,\infty}(   \mathcal{T}_{2,0}(\mathcal{F}^m_{k}\star\mathcal{F}^\nu_{1}),d_{2,0})$ so that $f=F\circ (I_{t_1,...,t_k}*Id).$ We know from Theorem \ref{MainTechnicalNonConvex} that:$$\liminf_{N\to \infty} -\frac{1}{N^2}\log E_{\gamma_{sa,N,m}}(e^{-N^2 f(\tau_{W,\Upsilon_N})})=\sup_\omega\Lambda_{b,\upsilon}(f).$$
   
If $t_k=1$, call $\mathbf{t}= \mathbf{t}(f)$ and $K=k$ and otherwise if $t_k\neq 1$  define  $\mathbf{t}=(t_0=0<t_1<...<t_k<t_{k+1}=1)$ and $K=k+1$. Then consider $G(\tau_{x_1,...,x_{K},\upsilon})=g(\tau_{x_{K},\upsilon})$ so that $G\in \mathcal{E}^{1,1}(  \mathcal{T}_{2,0}(\mathcal{F}^m_{K}\star\mathcal{F}^\nu_{1}),d_{2,0})$ then $\mu_{G,\mathbf{t},N}$ from proposition \ref{ConcentrationNorm} can be seen (after linear change of variable) as a law of the form $\mu_{g',N}$ as in Theorem \ref{SDVg} for $g'\in \mathcal{E}^{1,1}(  \mathcal{T}_{2,0}(\mathcal{F}^{Km}_{1}\star\mathcal{F}^\nu_{1}),d_{2,0})$ and thus converges in law, since the law is a marginal of a Hermitian Brownian bridge, it is easy to see (using standard freeness results for instance the characterization of free brownian motion in Theorem \ref{LevyBCG} and concentration from proposition \ref{ConcentrationNorm}) the limit law is a marginal of the brownian bridge $\tau_{\mu_g}$ since $\mu_{G,\mathbf{t},N}$ is itself the finite dimensional distribution of a brownian bridge (see e.g. \cite{Karatzas} (5.6.28) (5.6.29)) namely of the process $t X^l+(1-t)H^l_{\frac{t}{1-t}}$ with $H$ an hermitian brownian motion independent of $X$, following the law $\mu_{g,N}$. As a consequence, one deduces (using again the concentration result in proposition \ref{ConcentrationNorm} and the lipschitzness of $F$ with respect to $d_{2,0}$):
$$f(\tau_{\mu_g})=\lim_{N\to \infty} E_{\mu_{G,\mathbf{t},N}}(F(\tau_{.,\Upsilon_N})).$$

But of course we can compare this value to    
\begin{align*} -&E_{\mu_{G,\mathbf{t},N}}(F(\tau_{.,\Upsilon_N}))-\frac{1}{N^2}\log E_{\gamma_{sa,N,m}}(e^{-N^2 f(\tau_{W,\Upsilon_N})})\\&\leq \sup_{f\in \mathcal{C}_{sq}(\R^{N^2mk})} E_{\gamma_{sa,N,m}}(f(\tau_{W_{t_1},...,W_{t_k},\Upsilon_N})\frac{e^{-N^2 g(\tau_{W_1,\Upsilon_N})}}{Z_{G,\mathbf{t},N}})%\\&\qquad
 -\frac{1}{N^2}\log E_{\gamma_{sa,N,m}}(e^{N^2 f(\tau_{W_{t_1},...,W_{t_k},\Upsilon_N})})
\\& =-\frac{1}{N^2}Ent(\frac{1}{Z_{G,\mathbf{t},N}}e^{-N^2 g(\tau_{W_1,\Upsilon_N})}d\gamma_{sa,N,m}|\gamma_{sa,N,m})
\\& =\mathbf{E}  \left(
  %G(\tau_{X^{G,N}})+
  \frac{1}{2}\int_0^1||b^{G,N,\Upsilon_N}(t,X^{G,N,\Upsilon_N}(t))||_2^2\ dt\right)\end{align*}
   where the next-to-last equality comes from \eqref{dualEntsq} and the last one from Theorem \ref{minimisationHermitian} (with its notation) and its proof. Taking the limit $N\to\omega$ and using corollary \ref{OptimalFree} one thus gets:
      $-f(\tau_{\mu_g})+\sup_\omega\Lambda_{b,\Upsilon}^\omega(f)\leq\frac{1}{2}\int_0^1dt||u_t^G||_2^2$
      and taking the supremum over $f$:

       $$-\underline{I}_\Upsilon(\tau_{\mu_g})\geq   -\frac{1}{2}\int_0^1dt||u_t^G||_2^2.$$
       
%But from our choice of $g$ and    the upper Laplace deviation bound in Theorem \ref{LaplaceNonConvex} and the limit in Theorem \ref{MainTechnical}, we know $$\overline{I}_\Upsilon(\tau_{\mu_g})\geq -g(\tau_{\mu_g})+\inf_\omega\Lambda_{b,\Upsilon}^\omega(g)=-g(\tau_{\mu_g})+\Lambda_{b,\upsilon}(g)=\frac{1}{2}\int_0^1dt||u_t^G||_2^2.$$
   Hence we deduce %the equality     \begin{equation}\label{uIoIInt}\underline{I}_\Upsilon(\tau_{\mu_g})
  % =\overline{I}_\Upsilon(\tau_{\mu_g})=\frac{1}{2}\int_0^1dt||u_t^G||_2^2\end{equation} needed for our first inequality in this proof and hence
  :
    $$\chi^G(X_1,...,X_m|\upsilon)\geq \underline{\chi}^G(X_1,...,X_m|\upsilon)\geq   -\frac{1}{2}\int_0^1dt||u_t^G||_2^2.$$
    
    It remains to identify $u_t^G$. Note that we know from Theorem \ref{minimisationHermitian} that the law of $(X^{G,N,\Upsilon_N}(t),X^{G,N,\Upsilon_N}(1))$ is $\mu_{g,(t,1),N}$ with the notation of proposition  \ref{ConcentrationNorm} and thus the density of $X^{G,N,\Upsilon_N}(t)$ is the integral on the second variable:
    
    $$\int_{(M_N(\C)_{sa})^m}\mu_{g,(t,1),N}(dx_1,dx_2)=\exp\left(-NTr(\frac{x_1^2}{2t})-h_t^{N,\Upsilon_N}(\sqrt{N}x_1)\right)dLeb_{(M_N(\C)_{sa})^{m}}(dx_1)$$
We can thus compute the score function   of  $X^{G,N,\Upsilon_N}(t)$ to be $-\frac{N}{t}X^{G,N,\Upsilon_N}(t)-\sqrt{N}\mathscr{D}h_t^{N,\Upsilon_N}(\sqrt{N}X^{G,N,\Upsilon_N}(t))$ and one thus deduces by integration by parts the usual characteristic equation for a non-commutative polynomial $P\in  \C\langle X_1,...,X_m,\upsilon\rangle$ as in the proof of Theorem \ref{SDVg}
    \begin{align*}\ E&\left(\frac{1}{N}Tr(\frac{1}{t}X^{G,N}_i(t)P(X^{G,N,\Upsilon_N}(t),\Upsilon_N))%\right.\\&\left.
    + \frac{1}{N\sqrt{N}}Tr(\mathscr{D}h_t^{N,\Upsilon_N}(\sqrt{N}X^{G,N,\Upsilon_N}(t))P(X^{G,N,\Upsilon_N}(t),,\Upsilon_N))\right)\\&=E\left((\frac{1}{N}Tr\otimes\frac{1}{N}Tr)(\partial_{X_i}P)(X^{G,N,\Upsilon_N}(t),\Upsilon_N)\right)\end{align*}
and thus in terms of $b^{G,N,\Upsilon_N}$ this can be written:    
    \begin{align*}\ E&\left(\frac{1}{N}Tr(\frac{1}{t}X^{G,N,\Upsilon_N}_i(t)P(X^{G,N,\Upsilon_N}(t),\Upsilon_N))%\right.\\&\left.
    - \frac{1}{N}Tr(b_i^{G,N,\Upsilon_N}(t,X^{G,N,\Upsilon_N}(t))P(X^{G,N,\Upsilon_N}(t),\Upsilon_N))\right)\\&=E\left((\frac{1}{N}Tr\otimes\frac{1}{N}Tr)(\partial_{X_i}P)(X^{G,N,\Upsilon_N}(t),\Upsilon_N)\right)\end{align*}
Note that for further use in the newt proof, our previous inequality before the limit can be written if $-N\Xi^{G,N,\Upsilon_N}(t)=\frac{1}{t}X^{G,N,\Upsilon_N}(t)-b^{G,N,\Upsilon_N}(t,X^{G,N,\Upsilon_N}(t))$ is the score function of $X^{G,N,\Upsilon_N}(t)$ (from the previous integration by parts formula extended bejond non-commutative polynomials):    
    \begin{equation}\label{MatrixScoreBound} -E_{\mu_{G,\mathbf{t},N}}(F(\tau_{.,\Upsilon_N}))-\frac{1}{N^2}\log E_{\gamma_{sa,N,m}}(e^{-N^2 f(\tau_{W,\Upsilon_N})})
\leq \mathbf{E}  \left(
  %G(\tau_{X^{G,N}})+
  \frac{1}{2}\int_0^1\|\frac{1}{t}X^{G,N,\Upsilon_N}(t)+N\Xi^{G,N,\Upsilon_N}(t)\|_2^2\ dt\right)\end{equation}
From the concentration result in Theorem \ref{ConcentrationNorm} for     $\mu_{g,(t,1),N}$, one can take the limit $N\to \omega$ of the score function equation and obtain (recall the notation of corollary \ref{OptimalFree} $Y_t=(X^{G,N}(t))^\omega$)
$$\tau_\omega\left( \frac{   Y_t^{(i)}}{t}P(Y_t,\upsilon)-(u_t^G)^{(i)}P(Y_t,\upsilon)\right)=(\tau_\omega\otimes \tau_\omega)((\partial_{X_i}P)(Y_t,\upsilon)).$$
   But note the key result in corollary \ref{OptimalFree} that in our situation $(u_t^G)\in L^2(W^*(Y_t))$ implying that $\frac{   Y_t^{(i)}}{t}-(u_t^G)^{(i)}$ is Voiculescu's  i-th conjugate variable \cite{V5} for $Y_t$: $\xi_{i}^t$, and thus our previous inequality reads:
   $$\underline{\chi}^G(X_1,...,X_m|\upsilon)\geq   -\frac{1}{2}\int_0^1dt\left\|\frac{   Y_t}{t}-\xi^t\right\|_2^2=\chi^{G*}(X_1,...,X_m|W^*(\upsilon)).$$
   
    Since the other inequality was known from \cite{BCG} in the case $W^*(\upsilon)=\C$, one concludes.
    \end{proof}

\subsection{Proof of Theorem \ref{ThmC}}
\setcounter{Step}{0}
The first equality $\chi(X_1,...,X_m)=\chi^*(X_1,...,X_m)$ is in Theorem \ref{chichistar}. We keep the notation of its proof. We know that $Y_t$ as the same law as $tX+(1-t)S_{\frac{t}{1-t}},t\in [0,1)$ for $S_t$ a free brownian motion free from $X=(X_1,...,X_m)$ and we know that the conjugate variables are $\frac{   Y_t}{t}-(u_t^G).$ Most of remaining proof will boil down to a change of time and linear change of variable to induce what we learned in our previous proofs from free brownian bridge to free brownian motion.

\begin{step} $\chi(X_1+\sqrt{t} S_1,...,X_m+\sqrt{t} S_m)=\chi^*(X_1+\sqrt{t} S_1,...,X_m+\sqrt{t} S_m).$\end{step}
The result does not follow from our theorem by lack of a way of producing a universal function $g_t$ such that the law above satisfies $(SD_{g_t})$ but the proof will follow closely the one of Theorem \ref{chichistar}. Since both $\chi$ (see \cite[Prop 3.6]{V2}) and $\chi^*$ (see \cite[Prop 7.7, 7.8]{V5}) have the same formula under linear change of variable, it suffices to prove $\chi^G(Y_t)=\chi^{G*}(Y_t).$ Using \cite{BCG}, we are even content to prove : $\chi^{G*}(Y_t)\leq \chi^G(Y_t).$ Let us call $\mu_t$ the law of $Y_t$. Recall that $\tau_{\mu_t}$ is the law of brownian bridge $sY_t+(1-s)S_{\frac{s}{1-s}}$ which is the same law as $stX+ s(1-t)S_{\frac{t}{1-t}}'+ (1-s)S_{\frac{s}{1-s}}=\frac{1}{u}\left(ustX+ us(1-t)S_{\frac{t}{1-t}}'+ u(1-s)S_{\frac{s}{1-s}}\right)$ with $u=u(s,t)=\frac{t}{1-s(1-t)}$. But since %st-us^2t^2=st-(1-s(1-t))s^2t=st-s^2t+s^3t(1-t) which is the same in law as 
$u^2[s^2t(1-t)+(1-s)s]=-u^2s^2t^2+ u^2s[st+1-s]=ust-u^2s^2t^2$
%s^2t(1-t)+(1-s)s= -s^2t^2+s^2(t-1)+ 
this has the same law as $\frac{1}{u}\left(ustX+ (1-ust)S_{\frac{ust}{1-ust}}\right)$ which is a time change and linear change of variable of a free Brownian bridge. We call $L_{t}(\tau), t\in ]0,1[$ the law of the process $U_s=\frac{V_{u(s,t)st}}{u(s,t)},s\in [0,1]$ for $V_t,t\in [0,1]$ a process of law $\tau$. Recall that $\chi^{G*}(X_1,...,X_m)=-\frac{1}{2}\int_0^1dt\left\|\frac{   Y_t}{t}-\xi^t\right\|_2^2$ where $\xi^t$ is the conjugate variable for $Y_t$ thus, since the conjugate variable for $sY_t+(1-s)S_{\frac{s}{1-s}}$ is $u(s,t)\xi^{u(s,t)st}$ one deduces that  
\begin{align*}\chi^{G*}(Y_t)&=-\frac{1}{2}\int_0^1ds\left\|\frac{ Y_{u(s,t)st}}{u(s,t)s}-u(s,t)\xi^{u(s,t)st}\right\|_2^2%\\&=-\frac{1}{2}\int_0^1ds\left\|Y_{u(s,t)st}\left(\frac{1 }{u(s,t)s}-\frac{1 }{st}\right)-u(s,t)(\xi^{u(s,t)st}-\frac{ Y_{u(s,t)st}}{u(s,t)st})\right\|_2^2
\\&=-\frac{1}{2}\int_0^1ds\left\|-Y_{u(s,t)st}\left(\frac{1 }{t}-1\right)-u(s,t)(\xi^{u(s,t)st}-\frac{ Y_{u(s,t)st}}{u(s,t)st})\right\|_2^2
.\end{align*}
%If v=u(s,t)st dv=\frac{t^2}{1-s(1-t)} ds+\frac{t}{(1-s(1-t))^2}(1-t)st= t^2[1-s(1-t)]+(1-t)st/(1-s(1-t))^2ds=t^2+ (1-t)^2st
Of course we start by the same result as in the proof of our Theorem \ref{chichistar}:$$\chi^G(Y_t)\geq=-\sup_{f\in \mathcal{C}_{reg,\infty,C=0}(   \mathcal{T}_{2,0}^c(\mathcal{F}^m_{[0,1]}\star\mathcal{F}^\nu_{1}),d_{1,0})}-f(\tau_{\mu_t})+\sup_{\omega}\Lambda_{b,\Upsilon}^\omega(f)%=-\sup_{f\in \mathcal{E}_{reg,\infty}(   \mathcal{T}_{2,0}^c(\mathcal{F}^m_{[0,1]}\star\mathcal{F}^\nu_{1}),d_{2,0})}-f(\tau_{\mu_t})+\Lambda_{b,\upsilon}(f)
.$$
Fixing $f$ as in the supremum, and notation similar as in the proof that we follow, we deduce if we call for simplicity $\mu_{G,N}$ the law of the full brownian bridge process of marginals $\mu_{G,\mathbf{t},N}$ (in order not to fix the relevant times after the above deterministic change of time):$$\lim_{N\to \omega} -E_{\mu_{G,N}}(F(L_t(\tau_{.})))-\frac{1}{N^2}\log E_{\gamma_{sa,N,m}}(e^{-N^2 f(\tau_{W})})=-f(\tau_{\mu_t})+\Lambda_{b,\Upsilon}^\omega(f).$$
   
But of course, interpreting the key inequality \ref{MatrixScoreBound} in the proof of Theorem \ref{chichistar}  in terms of the same computation of the score function as before along an hermitian brownian bridge, on gets:
\begin{align*} -&E_{\mu_{G,N}}(F(L_t(\tau_{.})))-\frac{1}{N^2}\log E_{\gamma_{sa,N,m}}(e^{-N^2 f(\tau_{W})})
\\& \leq\mathbf{E}  \left(
  %G(\tau_{X^{G,N}})+
  \frac{1}{2}\int_0^1\left\|u(s,t)b^{G,N}(u(s,t)st,X^{G,N}(u(s,t)st))-X^{G,N}(u(s,t)st)\left(\frac{1 }{t}-1\right)\right\|_2^2\ ds\right).\end{align*}
Taking the limit $N\to \omega$ after a supremum over $f$ as in the proof  of Theorem \ref{chichistar}, one thus gets the expected inequality.

\begin{step} Use of Time reversal of free brownian motion.\end{step}
From step 2 (vi) of the proof of Theorem \ref{MainTechnical}, we know that $u_t^G=-\nabla_{Y_t} h_t^\omega(Y_t)$, which is from \eqref{C11htomega} a $C$-Lipschitz function of $Y_t$. From the linear change of variable for conjugate variables \cite[Proof of Corol 3.9]{V5}, in dividing by $t$ the previous variables,  $Y_t+(t\nabla_{Y_t} h_t^\omega(Y_t))= t \frac{   Y_t}{t}+(t\nabla_{Y_t} h_t^\omega(t \frac{   Y_t}{t}))$ is the conjugate variable of $\frac{   Y_t}{t}$ which is of same law as $X+ S_{\frac{1}{t}-1}$ for $t\leq 1.$ 

For $X\in L^2(W^*(Y_s, s\leq t))^m$, let us define   $H_t^\omega(X)=t\frac{X^2}{2}+h_t^\omega(tX)$ (using that $h_t^\omega$ is defined on the same space, even on a huger ultraproduct). Then, the conjugate variable of $\frac{   Y_t}{t}$ is given by $\nabla_{Y_t/t}H_t^\omega(\frac{   Y_t}{t})$. Note also that from the bounds obtained on $h_t^\omega$ in step 2.(vi) of the proof of Theorem \ref{MainTechnical}, we deduce that $H_t^\omega$ satisfies \eqref{Holderht} ($\alpha=1/2$), \eqref{Holderhtspace} ($\beta=1$), the lipschitz, convex and subquadratic behaviour assumed in Theorem \ref{FreeMonotone}. Note also that from %step 2.(ix) of the 
same proof, we know that $\nabla_XH_t^\omega(X)\in L^2(W^*(X)).$ %Since this will be important after change of time, let us point out that the constants in \eqref{Holderht} have the following form for $t,s\in ]0,1[$:
%\begin{equation}\label{HolderHt}||\nabla_{X} H_t^\omega(X)-\nabla_{X} H_s^\omega(X)||_2\leq |t-s|^{1/2} \max(\sqrt{t},\sqrt{s})(C+D||X||_2).\end{equation}
This conjugate variable is a $t+t^2C$ Lipschitz function of $\frac{   Y_t}{t}$ defined on $L^2(W^*(Y_s, s\leq t))^m$. Thus moving to the variable $s=\frac{1}{t}-1$, the brownian motion $X_s=X+ S_{s}$ has a $\frac{1}{1+s}+C\frac{1}{(1+s)^2}$ Lipschitz conjugate variable (in the sense it is given by evaluation at $X_s$ of a Lipschitz map on $L^2(W^*(X_u, u\geq s))^m$, given by the gradient of a function $\mathcal{H}_s^\omega$ with properties similar to $H_t^\omega$ on this space, thus satisfying the assumptions of Theorem \ref{FreeMonotone}). Applying, \cite[Prop 15,19]{Dab14}, and fixing any $T\geq s$ the reversed process $\overline{X}_s=X_{T-s}$ satisfies, for $\overline{\xi}_s=\xi_{T-s}$ the conjugate variable , the SDE:
$$\overline{X}_t
=\overline{X}_0-\int_0^t\overline{\xi}_udu+\overline{S}_t,$$
for a free brownian motion $\overline{S}_t$ adapted to the reversed filtration $L^2(W^*(X_u, u\geq {T-t}))^m=L^2(W^*(\overline{X}_u, u\leq t))^m$. We want to apply Theorem \ref{FreeMonotone} since $\overline{\xi}_u=\nabla \mathcal{H}_u^\omega(\overline{X}_u)$.

Note that from \eqref{Holderht} ($\alpha=1/2$), \eqref{Holderhtspace} ($\beta=1$), we already have $$||\overline{\xi}_s-\overline{\xi}_t||_2\leq |t-s|^{1/2}(C+D||\overline{X}_s||_2)+||\overline{X}_s-\overline{X}_t||_2(C||\overline{X}_s||_2+C||\overline{X}_t||_2+D),$$
and since $||\overline{X}_s-\overline{X}_t||_2=||S_{T-s}-S_{T-t}||_2=\sqrt{t-s}$ and $||\xi_t||_2$ is bounded, one deduces the stated H\"older continuity of $\Phi^*(\overline{X}_t).$
Since $\overline{\mathcal{F}}_t=L^2(W^*(\overline{X}_0, \overline{S}_u, u\leq t))\subset L^2(W^*(\overline{X}_u, u\leq t))$, we can consider the restriction $\mathcal{H}_u^\omega|_{\overline{\mathcal{F}}_u}$ and apply it Theorem \ref{FreeMonotone} to obtain a solution adapted to $\overline{\mathcal{F}}_u$ : 
$$\overline{Z}_t
=\overline{X}_0-\int_0^tP_{\overline{\mathcal{F}}_u}(\nabla \mathcal{H}_u^\omega(\overline{Z}_u))du+\overline{S}_t.$$
But from the property $\nabla \mathcal{H}_u^\omega(\overline{Z}_u))\in L^2(W^*(\overline{Z}_u))\subset \overline{\mathcal{F}}_u$, it also satisfies:
$$\overline{Z}_t
=\overline{X}_0-\int_0^t\nabla \mathcal{H}_u^\omega(\overline{Z}_u)du+\overline{S}_t.$$
From the uniqueness in Theorem \ref{FreeMonotone}, this time applied to the unrestricted $\mathcal{H}_u^\omega$, we have $\overline{X}_t=\overline{Z}_t\in W^*(\overline{S}_u, u\leq t, \overline{X}_0).$
Now as in the proof of corollary 22 in \cite{Dab14}, for $s\leq T$: $\xi_{i,s}-\xi_{i,T}-\int_0^{T-s}\delta_{T-u}\xi_{i,T-u}\#d\overline{S}_u$ is a martingale in the reversed filtration, with increments orthogonal to any stochastic integral adapted to this filtration by proposition 21.(3) in \cite{Dab14}. But from our adaptedness result to $W^*(\overline{S}_u, u\leq t, \overline{X}_0)$ and Clarke-Ocone formula \cite{BianeSpeicher}, it is such a stochastic integral, thus it is $0$ and one deduces in taking the $L^2$ norm:
$$||\xi_{i,s}||_2^2=||\xi_{i,T}||_2^2+\int_0^{T-s}||\delta_{T-u}\xi_{i,T-u}||_2^2du.$$
Summing over $i$ gives the stated integral equation.

\begin{step}For $T>0,$ $\chi(X_1+\sqrt{T} S_1,...,X_m+\sqrt{T} S_m:\sqrt{T} S_1,...,\sqrt{T} S_m)=\chi(X_1+\sqrt{T} S_1,...,X_m+\sqrt{T} S_m).$\end{step}
The hardest part is to control entropy in presence which does not seem to be studied at all by usual large deviation techniques. Hopefully our previous construction of a strong solution to the time reversal reduces it to standard results on entropy in presence. It suffices to prove $\geq$. Recall that for a subalgebra $B$, $$\chi(X_1+\sqrt{T} S_1,...,X_m+\sqrt{T} S_m:B)=\inf_{n\in \N,Y_1,...,Y_n\in B}\chi(X_1+\sqrt{T} S_1,...,X_m+\sqrt{T} S_m:Y_1,...,Y_n).$$
From \cite[Prop 3.4]{EntPow}, since $B=W^*(\overline{S}_t,t\leq T)$ is free from $X_1+\sqrt{T} S_1,...,X_m+\sqrt{T} S_m$, one deduces:
\begin{align*}\chi&(X_1+\sqrt{T} S_1,...,X_m+\sqrt{T} S_m)=\chi(X_1+\sqrt{T} S_1,...,X_m+\sqrt{T} S_m:B)
\\&=\chi(X_1+\sqrt{T} S_1,...,X_m+\sqrt{T} S_m:X_1+\sqrt{T} S_1,...,X_m+\sqrt{T} S_m,B)
\\&=\chi(X_1+\sqrt{T} S_1,...,X_m+\sqrt{T} S_m:X_1+\sqrt{T} S_1,...,X_m+\sqrt{T} S_m,B,\sqrt{T} S_1,...,\sqrt{T} S_m)
\\&\leq\chi(X_1+\sqrt{T} S_1,...,X_m+\sqrt{T} S_m:\sqrt{T} S_1,...,\sqrt{T} S_m).
\end{align*}
Indeed the second line comes from \cite[Prop 1.7]{V2} and the third line comes from the fact that $\sqrt{T} S_i=\overline{X}_0-\overline{X}_T\in W^*(X_1+\sqrt{T} S_1,...,X_m+\sqrt{T} S_m,B)$ from our result in step 2 and from \cite[Corol 1.8]{V2} (slightly extended to the non-finitely generated algebra case). The last inequality is then a trivial and concluding inequality.

 \bibliographystyle{amsalpha}

\end{document}